\documentclass[reqno,11pt,a4paper]{amsart}
\usepackage{amsmath}
\usepackage{amssymb}
\usepackage{graphicx}
\usepackage{mathrsfs}
\usepackage{amssymb}
\usepackage{amsfonts,bm,bbm}
\usepackage{amsthm}
\usepackage[colorlinks,linkcolor=blue,anchorcolor=blue,citecolor=blue,urlcolor=purple]{hyperref}
\usepackage{orcidlink}
\usepackage{geometry}
\usepackage{fullpage}
\usepackage{color}
\usepackage{underscore}
\usepackage{cite}
\usepackage{doi}

\DeclareFontFamily{U}{mathx}{}
\DeclareFontShape{U}{mathx}{m}{n}{<-> mathx10}{}
\DeclareSymbolFont{mathx}{U}{mathx}{m}{n}
\DeclareMathAccent{\widecheck}{0}{mathx}{"71}

\usepackage[open,openlevel=2]{bookmark}
\newtheorem{Thm}{Theorem}[section]

\newtheorem{Lem}[Thm]{Lemma}
\newtheorem{Coro}[Thm]{Corollary}
\newtheorem{Rem}[Thm]{Remark}

\numberwithin{equation}{section}

\newcommand{\1}{\mathbf{1}}
\newcommand{\R}{\mathbb{R}}

\newcommand{\Rd}{{\mathbb{R}^d}}
\newcommand{\Z}{\mathbb{Z}}
\renewcommand{\Re}{\text{Re}}

\newcommand{\F}{\mathcal{F}}
\newcommand{\N}{\mathbb{N}}

\renewcommand{\F}{\mathcal{F}}
\newcommand{\ve}{\varepsilon}

\newcommand{\vp}{\varphi}
\newcommand{\vpi}{\varpi}
\newcommand{\om}{\omega}
\newcommand{\ga}{\gamma}
\newcommand{\Ga}{\Gamma}
\renewcommand{\th}{\theta}
\newcommand{\Th}{\Theta}
\newcommand{\wt}{\widetilde}

\newcommand{\ti}{\tilde}
\newcommand{\ol}{\overline}

\newcommand{\pa}{\partial}
\newcommand{\na}{\nabla}
\newcommand{\de}{\delta}
\newcommand{\De}{\Delta}
\newcommand{\al}{\alpha}
\newcommand{\ka}{\kappa}
\newcommand{\si}{\sigma}

\newcommand{\lam}{\lambda}
\newcommand{\Lam}{\Lambda}
\newcommand{\wh}{\widehat}
\newcommand{\fr}{\frac}
\renewcommand{\S}{\mathbb{S}}

\newcommand{\<}{\langle}
\renewcommand{\>}{\rangle}

\newcommand{\Op}{\operatorname{Op}}

\newcommand{\vertiii}[1]{{\left|\kern-0.25ex\left|\kern-0.25ex\left|#1 \right|\kern-0.25ex\right|\kern-0.25ex\right|}}

\usepackage{tikz,pgfplots}
\usepackage{tikz-cd}
\usepackage{tikz-3dplot}
\usepackage{amsmath,amssymb,amsthm,amsfonts,bm}
\usepackage{color}
\usepackage{tcolorbox}
\usepackage{pdfpages}
\usepackage{textcomp}
\usetikzlibrary{patterns}
\usetikzlibrary{shapes.geometric}
\usetikzlibrary{arrows.meta,arrows}
\usetikzlibrary{decorations.pathreplacing,decorations.pathmorphing}
\usetikzlibrary{hobby}
\usetikzlibrary{math}
\usepgfplotslibrary{fillbetween}
\pgfplotsset{compat=1.18}
\usepackage{graphicx}
\usepackage{multicol}
\usepackage{float}
\usepackage{subfigure}
\usepackage{cancel}

\allowdisplaybreaks

\usepackage{enumitem}

\makeatletter
\def\@tocline#1#2#3#4#5#6#7{\relax
\ifnum #1>\c@tocdepth 
\else
\par \addpenalty\@secpenalty\addvspace{#2}%
\begingroup \hyphenpenalty\@M
\@ifempty{#4}{%
	\@tempdima\csname r@tocindent\number#1\endcsname\relax
}{%
	\@tempdima#4\relax
}%
\parindent\z@ \leftskip#3\relax \advance\leftskip\@tempdima\relax
\rightskip\@pnumwidth plus4em \parfillskip-\@pnumwidth
#5\leavevmode\hskip-\@tempdima
\ifcase #1
\or\or \hskip 1em \or \hskip 2em \else \hskip 3em \fi%
#6\nobreak\relax
\hfill\hbox to\@pnumwidth{\@tocpagenum{#7}}\par
\nobreak
\endgroup
\fi}
\makeatother

\mathcode`l="8000
\begingroup
\makeatletter
\lccode`\~=`\l
\DeclareMathSymbol{\lsb@l}{\mathalpha}{letters}{`l}
\lowercase{\gdef~{\ifnum\the\mathgroup=\m@ne \ell \else \lsb@l \fi}}%
\endgroup

\begin{document}

\title[Uniqueness and Zeroth-Order of Boltzmann]{Uniqueness and Zeroth-Order Analysis of Weak Solutions to the Non-cutoff Boltzmann equation
}
\author{Dingqun Deng
\,\orcidlink{0000-0001-9678-314X}
}
\address{
Graduate School of Engineering Science, Akita University, ORCID: \href{https://orcid.org/0000-0001-9678-314X}{0000-0001-9678-314X}}
\curraddr{}
\email{dingqun.deng@s.akita-u.ac.jp}
\email{dingqun.deng@gmail.com}

\author{Shota Sakamoto
\,\orcidlink{0000-0002-8509-0484}
}
\address{
    Faculty of Mathematics, Kyushu University
ORCID: \href{https://orcid.org/0000-0002-8509-0484}{0000-0002-8509-0484}}
\curraddr{}
\email{sakamoto.shota.588@m.kyushu-u.ac.jp}
\date{\today}

\begin{abstract}
We establish the uniqueness of large solutions to the non-cutoff Boltzmann equation with moderate soft potentials. Specifically, the weak solution $F=\mu+\mu^{\frac{1}{2}}f$ is unique under mild regularity conditions: any two weak solutions $\phi_1,\phi_2$ coincide provided that $\phi_1$ admits a lower bound and possesses logarithmic regularity in $L^\infty_x$, while $\phi_2\in L^\infty_t L^{r}_{x,v}\cap L^\infty_t L^2_{x,v}$ for some sufficiently large $r>2$. As a byproduct, we establish $L^2_{t,x,v}$ stability for initial data $f_0\in L^r_{x,v}\cap L^2_{x,v}$. Our approach employs dilated dyadic decompositions in phase space $(v,\xi,\eta)$ to capture hypoellipticity and to reduce the fractional derivative structure $(-\Delta_v)^{s}$ of the Boltzmann collision operator to zeroth order. The main novelties are the zeroth-order reduction and the negative-order hypoelliptic estimate, which gains integrability in $(t,x)$.
\end{abstract}




\maketitle

{
\fontsize{10}{11.5}\selectfont
\tableofcontents
}

\fontsize{10.42}{12.21}\selectfont

\section{Introduction}

\subsection{Non-cutoff Boltzmann model and difficulties of uniqueness}
Let $d\ge 2$ be the dimension and $T>0$ be a given time.
The \emph{Boltzmann equation} for the particle distribution function $F(t,x,v):[0,T]\times\R^d_x\times\R^d_v\to \R$ at time $t$, position $x$, and velocity $v$ reads 
\begin{align}
	\label{B}
	\pa_tF(t,x,v) + v\cdot\na_xF(t,x,v)=Q(F,F)(t,x,v), \quad \text{ in } (0,T)\times\R^d_x\times\R^d_v,
\end{align}
with initial data $F(0,x,v)=F_0(x,v)$,
where $Q$ is the \emph{Boltzmann collision operator} given by
\begin{align}\label{QFGDef}
	Q(F,G)=\int_{\R^d_{v_*}}\int_{\S^{d-1}_\si}B(v-v_*,\sigma)\big(F(v'_*)G(v')-F(v_*)G(v)\big)\,d\sigma dv_*,
\end{align}
where $(v,v_*)$ and $(v',v_*')$ are pre- and post-collisional velocities, satisfying the conservation laws of momentum and energy: 
\begin{align}\label{vprime}
		v'=\frac{v+v_*}{2}+\frac{|v-v_*|\sigma}{2}, \quad
		v'_*=\frac{v+v_*}{2}-\frac{|v-v_*|\sigma}{2},
\end{align}
with $\sigma\in\S^{d-1}:=\{x\in\R^d\,:\,|x|=1\}$.
The \emph{cross-section} $B(v-v_*,\sigma)$ is a measure of the probability for the event of a collision (or scattering), given by
\begin{align*}
	B(v-v_*,\sigma)=\Th(v-v_*)b(\cos\th)=|v-v_*|^\ga b(\cos\th),
\end{align*}
which depends on the relative velocity $|v-v_*|$ and the deviation angle $\th$ through $\cos\th=\mathbf{k}\cdot\sigma$, where $\mathbf{k}=\frac{v-v_*}{|v-v_*|}$ and $\cdot$ is the usual scalar product in $\R^d$. Without loss of generality, we assume that $B(v-v_*,\sigma)$ is supported on $\mathbf{k}\cdot\sigma\ge 0$, i.e., $0\le\th\le\frac{\pi}{2}$, since one can reduce to this situation by standard symmetrization: $\ol B(v-v_*,\si)=[B(v-v_*,\si)+B(v-v_*,-\si)]\1_{\mathbf{k}\cdot\si\ge 0}$. Moreover, we assume the angular non-cutoff assumption:
\begin{align}
	\label{ths}
	\frac{1}{C_b}\th^{-1-2s}\le\sin^{d-2}\th\, b(\cos\th)\le C_b \th^{-1-2s},
\end{align}
for some $C_b>0$, derived from the inverse power law (for long-range interactions) according to a spherical intermolecular repulsive potential $\phi(r)=r^{-(p-1)}$ $(p>2)$. For $d=3$, it is known that $\ga=\frac{p-5}{p-1}$ and $s=\frac{1}{p-1}$; see for instance \cite{Alexandre2000,Maxwell1867,Villani2002}. The factor $\sin^{d-2}\th$ corresponds to the Jacobian factor for integration in spherical coordinates. 
In this paper, we assume that the indices $(\ga,s)$ satisfy
\begin{align}\label{AssGas}
	-d<\ga<-2s, \quad s\in(0,1),\quad \ga+s>-\frac{d}{2}. 
\end{align}
We refer to 
the case $\ga\ge -2s$ as \emph{hard potentials} and the case $-\frac{d}{2}-s<\ga<-2s$ as \emph{moderate soft potentials}.
It is direct to check that the global Maxwellian
	$\mu(v):=(2\pi)^{-\frac{d}{2}}e^{-\frac{|v|^2}{2}}$
is a steady solution to equation \eqref{B}.
In this work, we apply the exponential decomposition: 
	$\displaystyle F=\mu+\mu^{\frac{1}{2}}f$.
Then the Boltzmann equation can be rewritten as
\begin{align}
	\label{B1}
	\begin{cases}
		\pa_tf+ v\cdot\na_xf = \Gamma(\mu^{\frac{1}{2}}+f,f)+\Gamma(f,\mu^{\frac{1}{2}}), &\text{ in } (0,T)\times\R^d_x\times\R^d_v, \\
		f(0,x,v)=f_0, &\text{ in }\R^d_x\times\R^d_v,
	\end{cases}
\end{align}
where, by \eqref{QFGDef}, the standard Boltzmann collision operator is given by
\begin{align}\label{GaBolt}
	\Gamma(f,g)=\mu^{-\frac{1}{2}}Q(\mu^{\frac{1}{2}}f,\mu^{\frac{1}{2}}g)=\int_{\R^d_{v_*}\times\S^{d-1}_\si}B(v-v_*,\sigma)\mu^{\frac{1}{2}}(v_*)\big(f'_*g'-f_*g\big)\,d\sigma dv_*.
\end{align}
The weak solution of the Boltzmann equation \eqref{B1} is a function $f$ that satisfies  
\begin{align}\label{weakf}
&(f(t),\Phi(t))_{L^2_{x,v}
}\Big|_{t=0}^{t=T}
-(f,(\pa_t+v\cdot\na_x)\Phi)_{L^2_t([0,T])L^2_{x,v}}
=(\Ga(\mu^{\frac{1}{2}}+f,f)+\Ga(f,\mu^{\frac{1}{2}}),\Phi)_{L^2_t([0,T])L^2_{x,v}}, 
\end{align}
for any test function $\Phi\in C^2_c(\R^{1+2d})$ and for a.e. $T\ge 0$.
Moreover, the actual weak form of the collision operator $\Ga(f,g)$ can be obtained by pre-post velocity change of variables as in \eqref{CollBoltzDef}. 
If $\phi_1$ and $\phi_2$ are two solutions to the Boltzmann equation \eqref{B1} with initial data $\phi_{1,0}$ and $\phi_{2,0}$, respectively, then $f:=\phi_1-\phi_2$ solves 
\begin{align}\label{eqh}\begin{cases}
\pa_tf+v\cdot\na_xf=\Ga(\mu^{\frac{1}{2}}+\phi_1,f)+\Ga(f,\mu^{\frac{1}{2}}+\phi_2), & \text{ in }(0,T)\times\R^d_x\times\R^d_v,\\
f(0,x,v)=\phi_{1,0}-\phi_{2,0},&\text{ in }\R^d_x\times\R^d_v.
\end{cases}
\end{align}
Moreover, we need to capture the dissipation effect produced by the solution $\phi_1$. For the sake of simplicity, inspired by the self-generating lower bounds in \cite{Henderson2020b}, we assume that the solution $\mu+\mu^{\fr12}\phi_1$ has a uniform exponential lower bound satisfying \eqref{lowerphi} below.

\subsubsection{Difficulties of uniqueness for nonlinear kinetic equations}
	The uniqueness of weak or low-regularity solutions to nonlinear kinetic equations remains a fundamental open problem, particularly for nonlocal collision operator. 
The main difficulties lie in the following aspects. 
\begin{itemize}[leftmargin=1em]
	\item \textbf{Nonlinearity.} The non-cutoff Boltzmann collision operator is bilinear and nonlocal. Roughly speaking, the collision operator behaves like the fractional operator $(-\De_v)^{s}$ and satisfies
	\begin{align*}
		(\Ga(f,g),h)_{L^2_v}\le \|f\|_{L^2_v}\|g\|_{H^s_v}\|h\|_{H^s_v}.
	\end{align*}
	However, when considering the $L^2$ energy estimate, the nonlinear terms in \eqref{eqh} are, roughly,
	\begin{align}\label{observation1}
		\big(\Ga(\phi_1,f)+\Ga(f,\phi_2),f\big)_{L^2_{t,x,v}}
		&\le \|\phi_1\|_{L^\infty_{t,x}L^2_v}\|f\|_{L^2_{t,x}H^s_v}^2
		+\|f\|_{L^2_{t,x,v}}\|\phi_2\|_{L^\infty_{t,x}H^s_v}\|f\|_{L^2_{t,x}H^s_v},
	\end{align}
where the $L^\infty_{t,x}$ norm of $\|\phi_2\|_{H_v^s}$ is not a natural structure for low-regularity solutions to the non-cutoff Boltzmann equation, such as $L^2$ solutions.

	\smallskip 
	\item \textbf{Lack of regularity.} As observed in \eqref{observation1}, due to the $L^\infty_{t,x}$ norm, it is difficult to obtain the uniqueness unless the regular norm $\|\phi_2\|_{L^\infty_{t,x}H^s_v}$ is small.
	Even for short times, the lack of $L^\infty_{t,x}$ boundedness prevents us from establishing uniqueness. 

	\smallskip \item \textbf{Commutator against advection $v\cdot\na_x$.} 
	For spatially inhomogeneous kinetic equations, applying the pseudo-differential operator $P_k$ (or the negative Bessel potential $\<D_v\>^{-2s}$) in the $v$ variable leads to the commutator $[P_k, v\cdot\na_x]$. This commutator generates the undesirable derivative $|\na_x|$, which does not correspond to the natural regularity for kinetic equations, whereas hypoellipticity yields only regularity of order $|\na_x|^{\frac{s}{1+2s}}$.

	\end{itemize}

\subsection{Main results and notations}
We present the main results, followed by some observations and notations.
Let $\{\De_j^{\mathrm{std}}\}_{j\ge 0}$ denote the standard Littlewood--Paley decomposition in $x$ associated with a dyadic partition of unity in the Fourier variable $\xi$; see, for example, \eqref{DejQkDef} in the case $\vpi=1$. 
\begin{Thm}[Uniqueness and $L^2_{t,x,v}$ stability]
	\label{MainThm1}
	Let $d\ge 2$, and fix any $M_0,T_*>0$. Consider the moderate soft potential
case, in which the parameters $(\gamma,s)$ satisfy \eqref{AssGas}, and let
$C_{\mathrm{low}},L_0>0$ denote the fixed constants used throughout the paper.
Then there exist a sufficiently large exponent $r=r(d,\gamma,s)\in(2,\infty)$
and a constant $C_{\gamma,s,d}>0$, both depending only on $d,\gamma,s$, for which
the following holds. Let $\phi_1,\phi_2$ be any two solutions to the non-cutoff
Boltzmann equation \eqref{B1} on $[0,T_*]$, with respective initial data
$\phi_{1,0},\phi_{2,0}$, and suppose that
\begin{align}
	\label{lowerphi}
	&\Phi_1(t,x,v):=\mu+\mu^{\fr12}\phi_1\ge C_{\mathrm{low}}^{-1}\mu^{L_0},\quad\text{ a.e. $(t,x,v)\in[0,T]\times\R^{2d}_{x,v}$},
\end{align}
and 
\begin{align}
	\label{assumption}
	&\begin{aligned}
&\|\<v\>^{C_{\ga,s,d}}\phi_2\|_{L^\infty_t([0,T_*])L^{r}_{x,v}}
+\|\<v\>^{C_{\ga,s,d}}\phi_2\|_{L^\infty_t([0,T_*])L^{2}_{x,v}}
\le M_0,
\\
&\sum_{j\ge 0}(j+1)^2\|\De^{\mathrm{std}}_j\<v\>^{|\ga|+d+1}\phi_{1}\|^2_{L^\infty_t([0,T])L^\infty_{x}L^r_v}\le M_0.
	\end{aligned}
\end{align}
Then the pair satisfies the $L^2_{t,x,v}$ stability estimate (see
\eqref{EsContDepend} for the precise formulation):
\begin{align}\label{ThmContiDepend}
	\|\phi_1-\phi_2\|_{L^2_t([0,T_*])L^2_{x,v}}
	\le C_{T_*,M_0,\ga,s,d,C_{\mathrm{low}},L_0}\|\phi_{1,0}-\phi_{2,0}\|_{L^2_{x,v}}.
\end{align}
Consequently, if $\phi_{1,0}=\phi_{2,0}$, we obtain the following uniqueness:
	\begin{align}\label{ThmUnique}
		\phi_1(t,x,v)=\phi_2(t,x,v)\quad\text{ for a.e. }\  (t,x,v)\in[0,T]\times\R^d_x\times\R^d_v.
	\end{align}
\end{Thm}

\begin{Rem}
	\begin{enumerate}[leftmargin=2em]
		\item Considering the decomposition $F-\mu$ could be mathematically (and physically) natural for kinetic equations on $\R^d_x$, as reflected in the trend to global equilibrium \cite{Desvillettes2000}. 
		\item This approach bypasses the $L^\infty_x$ norm issue of $\phi_2$.
		Moreover, for $\phi_1$, we require logarithmic regularity in $x$, measured by $\log(\langle D_x\rangle)$ in $L^\infty$.
		\item The assumption $\ga+s>-\frac{d}{2}$ arises from the collisional estimates in Subsections \ref{SecSmallQc} and \ref{SecCOmmSmall} for small relative velocity; in other words, if one considers the non-singular cross-section $\Th=\<v-v_*\>^{\ga}$, then the assumption $\ga+s>-\fr{d}{2}$ is not required. Moreover, if one considers the Landau collision operator, such a restriction is unnecessary, and the \textbf{Landau-Coulomb} case $\ga=-d$ should also be covered.
		\item The constant $C_{\gamma,s,d}>0$ is constructive throughout the proof in Section~\ref{SecLPkinetic} and is uniform with respect to $\ga,s$; that is, it is bounded independently of $\ga,s$, as ensured by the choice we will make in \eqref{choice10} later. 
		\item The largeness of $r>0$ reflects the small gain of integrability in $x$; see Corollary \ref{CoroNegativeHypo}. 
		\item The constants $C_{\mathrm{low}},L_0>0$ in assumption \eqref{lowerphi} can be arbitrary and our method applies whenever $\phi_1$ has an exponential lower bound. One may also consider decompositions with tails other than exponential, or with a rougher lower bound; we may leave this to future work. 
		\item The constants in this work may depend on the parameters $\ga, s, d$ without further mention. 
		\item One may also use $(j+1)^{1^+}$ instead of $(j+1)^2$ in \eqref{assumption}, but in this case the constants in the estimates depend on $1^+-1$, as reflected in \eqref{ges1a}.
	\end{enumerate}
\end{Rem}

\subsubsection{Miscellaneous Notations}
Throughout this paper, $C$ and $N$ denote positive constants (generally large), and $c_0$ denotes a constant (small) that may vary from line to line. 
$(a_1,a_2,\dots,a_n)$ denotes any $n$-dimensional vector. $[P,Q]=PQ-QP$ denotes the Lie bracket of two operators $P,Q$. $\N = {1, 2, \dots}$ denotes the set of positive natural numbers, and $\R$ denotes the set of real numbers.
Let $\1_A$ be the indicator function on the set $A$.  
The notation $a\approx b$ (resp. $a\gtrsim b$, $a\lesssim b$) for positive real functions $a,b$ means $C^{-1}a\le b\le Ca$ (resp. $a\ge C^{-1}b$, $a\le Cb$) on their domains where $C>0$ is a constant not depending on possible free parameters. 
A constant $C=C_{a_1,a_2,\dots}=C(a_1,a_2,\dots)$ denotes a positive constant depending on $a_1,a_2,\dots$ which may vary from line to line.
The notation $a\ll b$ means that $a\le C^{-1}b$ for some sufficiently large constant $C>0$ depending only on fixed parameters. 
The set $C^\infty_c$ consists of smooth functions with compact support. 
We may also use $P\sim Q$ to illustrate the formal (and rough) equivalence of $P$ and $Q$. For operators, we simply list them in parallel, e.g., $\De_jP_kR_rf\triangleq\De_j\big(P_k(R_rf)\big)$. 
Moreover, $m^-<m$ and $m^+>m$ denote constants chosen arbitrarily close to $m$.
For brevity, we denote by $\chi_{a\le|\eta|\le b},\chi_{|\eta|\le a},\chi_{|\eta|>b}$ smooth cutoff functions (which may vary from place to place) satisfying 
\begin{align}\label{smoothcutoff}\notag
	&\qquad\qquad\quad\chi_{a\le|\eta|\le b}=
	\begin{cases}
		1,&\text{ when }a\le|\eta|\le b,\\
		0,&\text{ when }|\eta|<\frac{3}{4}a\text{ or }|\eta|>\frac{4}{3}b, 
	\end{cases}\\
	&\chi_{|\eta|\le a}=
	\begin{cases}
		1,&\text{ when }|\eta|\le a,\\
		0,&\text{ when }|\eta|>\frac{4}{3}a, 
	\end{cases}
	\qquad
	\chi_{|\eta|>b}=
	\begin{cases}
		1,&\text{ when }|\eta|>b,\\
		0,&\text{ when }|\eta|\le \frac{3}{4}b. 
	\end{cases}
\end{align}

\subsubsection{$L^p$ spaces}
We denote $\|\cdot\|_{L^p}$ the Lebesgue $L^p$ norm and write $L^p_tL^q_xL^r_v=L^p_t(L^q_x(L^r_v))$ in short. We take $\mathbb{R}^d_x \times \mathbb{R}^d_v$ as the underlying domain for $(x,v)$, while the time interval will be specified in the relevant sections. For instance,
\begin{align*}
	\|f\|_{L^p_tL^q_xL^r_v([0,T]\times\R^d_x\times\R^d_v)}:=\Big\{\int_0^T\Big[\int_{\R^d_x}\Big(\int_{\R^d_v}|f|^r\,dv\Big)^{\frac{q}{r}}\,dx\Big]^{\frac{p}{q}}\,dt\Big\}^{\frac{1}{p}}.
\end{align*}
Throughout this paper, we adopt shorthand notation for integrals; for example, we write $\int_x$ for $\int_{\R^d_x}$ whenever no confusion can arise. 
The Sobolev space $W^{k,p}(\R^d)$ denotes the set of tempered distributions whose $W^{k,p}$ norm $\|f\|_{W^{k,p}(\R^d)}:=\sum_{j=0}^k\|\na^jf\|_{L^p(\R^d)}$ is finite. 

\subsubsection{Bessel potential}
Let $(I-\De_x)^{-\frac{\ka}{2}}$ be the Bessel potential operator of order $\ka$ ($0<\Re\ka<\infty$) in $\R^d_x$. It can also be represented by the Bessel potential: for any suitable function $f$, 
\begin{align*}
	(I-\De_x)^{-\frac{\ka}{2}}f=G_\ka*f,\quad
	\text{where} \ \ G_\ka(x) = \big((1+4\pi^2|\xi|^2)^{-\frac{\ka}{2}}\big)^\vee(x).
\end{align*}
Here and below, $\wh{f}$ and $f^\vee$ are the Fourier transform and inverse Fourier transform in the sense of tempered distribution, respectively.
By \cite[Prop. 1.2.5, p. 13]{Grafakos2014a}, we know that if $\ka>0$ is real, then $G_\ka$ is strictly positive.
Furthermore, 
\begin{align}
	&\begin{cases}
		\|G_\ka\|_{L^1_x(\R^d)}=1,\\
		G_\ka(x)\le C_{\ka,d} e^{-\frac{|x|}{2}}\ (\text{if }|x|\ge 2),\\
		G_{\ka,d}(x)\approx_{\ka,d}H_s(x)
		\ 
		(\text{if }|x|\le 2),
	\end{cases}
\label{GsDecay}
	\text{where }
	H_\ka(x)=\begin{cases}
		|x|^{\ka-d}+1+O(|x|^{\ka-d+2}),&\text{if }0<\ka<d,\\
		\ln\frac{2}{|x|}+1+O(|x|^{2}),&\text{if }\ka=d,\\
		1+O(|x|^{\ka-d}),&\text{if }\ka>d.
	\end{cases}
\end{align}
By Young's convolution inequality, one has $\|(I-\De_x)^{-\frac{\ka}{2}}f\|_{L^p_x(\R^d)}\le\|G_\ka\|_{L^1_x}\|f\|_{L^p_x}=\|f\|_{L^p_x(\R^d)}$ ($1\le p\le \infty$).
Moreover, we have the equivalence relation: for
$m,n\in\R$ and $1\le p\le \infty$,
\begin{align}
	\label{equiv}\|\<v\>^m\<D_v\>^nf\|_{L^p(\R^d)}\approx \|\<D_v\>^n\<v\>^mf\|_{L^p(\R^d)}.
\end{align}
This follows from \cite[Lemma 2.1 and Section 8]{Alonso2022} for the case $1\le p\le\infty$ and \cite[Prop. 5.5 and its Corollary, pp. 251--252]{Stein1993} for the case $1<p<\infty$.
%
Moreover, for $k,l\in\R$, we define the Sobolev space $H^{k,p}_l(\R^d)$ as follows, with the conventions $H^{k}=H^{k,2}_0$, $H^{k}_l=H^{k,2}_l$ and $H^{\infty,p}=\cap_{k=1}^\infty H^{k,p}$:
\begin{align*}
	H^{k,p}_l(\R^d)&=\big\{f \text{ is tempered distribution}\,:\, \|f\|_{H^k_l(\R^d)}:=\|\<v\>^l\<D_v\>^kf\|_{L^p_v(\R^d)}<\infty\big\}. 
\end{align*}


\subsubsection{Dissipation norm}
Some standard estimates of non-cutoff Boltzmann collision operators can be found in \cite{Gressman2011,Alexandre2012,Global2019}. 
To describe its behavior, \cite{Alexandre2012} introduces the triple norm $\vertiii{f}$:
\begin{align}
	\label{vertiii}
	\vertiii{f}^2 & :=\int_{v,v_*,\si} B(v-v_*,\sigma)\Big(\mu_*(f'-f)^2+f^2_*((\mu')^{1/2}-\mu^{1/2})^2\Big)\,d\sigma dv_*dv,
\end{align}
while \cite{Gressman2011} introduces the anisotropic norm $N^{s,\gamma}$:
\begin{align*}
	\|f\|^2_{N^{s,\gamma}}: & =\|\<v\>^{\gamma/2+s}f\|^2_{L^2_v}+\int_{\R^{2d}}(\<v\>\<v'\>)^{\frac{\gamma+2s+1}{2}}\frac{(f'-f)^2}{d(v,v')^{d+2s}}\1_{d(v,v')\le 1}\,dvdv',
\end{align*}
where $d(v,v'):=\sqrt{|v-v'|^2+\frac{1}{4}(|v|^2-|v'|^2)^2}$.
Moreover, \cite{Global2019} use the pseudo-differential norm
\begin{align}\label{tiaa}
	\|(\tilde{a}^{1/2})^wf\|_{L^2_v}, \quad \ti a(v,\eta)=\<v\>^{\ga}\big(1+|v|^2+|v\wedge\eta|^2+|\eta|^2\big)^s+K_0\<v\>^{\gamma+2s},
\end{align}
where $(\cdot)^w$ is the Weyl quantization and $K_0>0$ is a sufficiently large constant; see also \cite{Lerner2010,Deng2020a}. 
The symbolic calculus for the non-cutoff Boltzmann equation, as presented in \cite{Global2019}, can be extended to a general $d$-dimensional space by defining the wedge product $|v\wedge\eta|$ as $|v||\eta|\sin\vartheta$, with $\cos\vartheta=\frac{v\cdot\eta}{|v||\eta|}$. 
(In the proof of \cite[Prop. 3.1]{Global2019}, only $|v||\eta|\sin\vartheta$ is used essentially.)
Then one has from \cite[(2.13) abd (2.15)]{Gressman2011}, \cite[Prop. 2.1]{Alexandre2012} and \cite[Theorem 1.2]{Global2019} that these dissipation norms are all equivalent. We denote this equivalent norm by $L^2_D$:
\begin{align}\label{tia}
	\|f\|_{L^2_D}:=\|(\ti a^{\frac{1}{2}})^wf\|_{L^2_v}\approx \|f\|^2_{N^{s,\gamma}}\approx\vertiii{f}^2.
\end{align}
One also has from \cite[Lemma 2.4]{Deng2020a} that $\|\<v\>^l(\ti a^{\frac{1}{2}})^wf\|_{L^2_v}\approx\|(\ti a^{\frac{1}{2}})^w(\<v\>^lf)\|_{L^2_v}$ for any $l\in\R$.
From \cite[Prop. 2.2]{Alexandre2012} or \cite[(2.15)]{Gressman2011} (or simply \eqref{tiaa}), we have
\begin{align}
	\label{esD}
	\|\<v\>^{\frac{\ga+2s}{2}}f\|_{L^2_v}+\|\<v\>^{\frac{\ga}{2}}\<D_v\>^sf\|_{L^2_v}\lesssim \|f\|_{L^2_D}\lesssim \|\<v\>^{\frac{\ga+2s}{2}}\<D_v\>^sf\|_{L^2_v}.
\end{align}

\subsubsection{The weak solution}
In this work, we apply the Littlewood-Paley decomposition to the Boltzmann equation to ensure sufficient smoothness, so that $\De_jP_kR_r\phi_i\in L^2_t([0,T];H^\infty_{x,v}(\R^{2d}))$. Then, by the weak formulation \eqref{weakf} and assumption \eqref{assumption}, $\pa_t\De_jP_kR_r\phi_i\in L^2_t([0,T];H^\infty_{x,v}(\R^{2d}))$. Hence, standard approximation arguments yield that $\De_jP_kR_r\phi_i\in C([0,T];H^\infty_{x,v}(\R^{2d}))$ and $t\mapsto \|\De_jP_kR_r\phi_i\|_{L^2_{x,v}}^2$ is absolutely continuous, which suffices for our energy estimates. 

Moreover, if $\phi_i$ has finite energy in $L^\infty_t([0,T_*];L^2_{x,v}(\R^{2d}_{x,v}))$ as assumed in \eqref{assumption}, and satisfies the weak form \eqref{weakf}, which implies $\phi_i\in L^2((0,T_*);H^{-N}_{x,v}(\R^{2d}_{x,v}))$ for sufficiently large $N=N(d)>0$, then the Aubin-Lions lemma yields $\phi_i\in C([0,T_*];H^{-\de}_{x,v}(\Omega))$ for any $\de>0$ and any open bounded Lipschitz domain $\Omega\subset\R^{2d}_{x,v}$, so that the term $(f(t),\Phi(t))_{L^2_{x,v}}|_{t=0}^{t=T}$ on the left-hand side of \eqref{weakf} is well-defined for all $T\in [0,T_*]$, rather than a.e. $T\in [0,T_*]$.

\subsection{Main novelties and ideas for the uniqueness of large solutions}

\begin{itemize}[leftmargin=1.5em]
\item \textbf{Zeroth-order frequency.}
First, observe the term $\|f\|_{L^2_{t,x,v}}$ in \eqref{observation1} requires less regularity than $L^2_{t,x}H^s_v$. To exploit this, we reduce the derivative level in the kinetic equation from $H^s_v$ to $L^2_v$, so that $\|f\|_{L^2_{t,x,v}}$ aligns with the zeroth-order frequency energy estimate. This can be done, roughly speaking, by applying $\<D_v\>^{-2s}$ to the kinetic equation, reducing the fractional derivative $\<D_v\>^{2s}$ implicit in the collision $Q(f,g)$ to zeroth order.

\smallskip 
\item \textbf{Dyadic decomposition and Littlewood-Paley theory.} 
As implied above, our goal is to extract all the derivative structures of the kinetic equation and reduce them to zeroth order. To do so, a natural tool is the dyadic decomposition of the frequency space and Littlewood-Paley theory. By applying suitable pseudo-differential operators to the equation (defined later in \eqref{DejQkDef}), the advection $v\cdot\na_x$ and fractional derivative $\<D_v\>^{2s}$ implicit in the collision $Q(f,g)$ behave as constants in frequency space. For instance, under dyadic decomposition
\begin{align*}
	\De_jf(x)&=\int_{\R^{d}_{\xi}\times\R^{d}_{y}}e^{2\pi i \xi\cdot(x-y)}\wh{\Psi_j}(\varpi^{-1}\xi)f(y)\,dyd\xi,\\
	P_kf(v)&=\int_{\R^d_\eta\times\R^d_u}e^{2\pi i\eta\cdot(v-u)}\wh{\Psi_{k}}(\om^{-1}\eta)f(u)\,dud\eta,
\end{align*}
with dilations $\varpi,\om>0$, 
$\na_x$ behaves like $\varpi2^j$ while $\<D_v\>^{2s}$ behaves like $\<\om2^k\>^{2s}$. Furthermore, since the kinetic equation \eqref{B} involves $v,\pa_t,\na_x,\<D_v\>^{2s}$, it's very natural to utilize the dyadic decomposition on $(v,\tau,\xi,\eta)$, corresponding to velocity, time frequency, spatial frequency, and velocity frequency, respectively. 
The details will be given in Section \ref{SecLP}. 

\smallskip 
\item \textbf{Gain of integrability and negative-order hypoellipticity.} 
We will further utilize the structure of $\<D_v\>^{-s}$ within the first argument $f$ of collision $Q(f,g)$. To this end, as in \cite{Jabin2022,Deng2022}, one can introduce a weight function in the frequency space $(\xi,\eta)$:
	$\displaystyle W(\xi,\eta)=\exp\big\{\frac{\xi}{|\xi|}\cdot\frac{\eta}{\om}\big\}$. 
on the support of $P_0$. 
Substituting such a weight into the kinetic equation and noting that 
$\xi\cdot \nabla_\eta W(\xi,\eta)=W(\xi,\eta)\frac{|\xi|}{\om}$, one obtains a gain of regularity in $x$.
Using this idea, we establish the negative-order hypoellipticity in Theorem \ref{ThmNegativeHypo}. This enables us to recover regularity, and hence integrability, in $(t,x)$ at the $L^2$ level (or more generally in $L^p$ for $p$ close to $2$). As a consequence, the energy estimate can be closed when $\phi_2\in L^\infty_tL^r_{x,v}$ rather than $\phi_2\in L^\infty_{t,x}L^r_v$.

\smallskip 
\item
\textbf{Approximation of solutions.} Utilizing the above gain of integrability from negative-order hypoellipticity, we could treat $L^p$ norms ($p<\infty$) instead of $L^\infty$ norms, which allows us to use approximate identities in $L^p$. In fact, for any $p\in[1,\infty)$ and $\phi_1,\phi_2\in L^p_t([0,T])L^p_{x,v}$, we decompose
\begin{align*}
\phi_i=\phi_{i,\de_j}+\wt\phi_{i,\de_j}\quad(i=1,2),
\end{align*}
where $\phi_{i,\de_j}$ is the regular part and $\wt\phi_{i,\de_j}$ is the irregular but small part, i.e., $\|\<v\>^C\wt\phi_{i,\de_j}\|_{L^p_t([0,T])L^p_{x,v}}=o_{\de_j}(1)\to 0$ as $\de_j\to 0$ (with possibly different convergence rates $\de_j$ in different places). Moreover, for $\phi_1$ we use a mollifier in $x$ so that $\phi_{1,\de_j}$ also satisfies the lower bound \eqref{lowerphi} while for $\phi_2$ we use a mollifier in both $x$ and $v$. A similar approximation holds for $\phi_1$ in the $\log(\<D_x\>)$ space; see \eqref{Propphide} -- \eqref{es718} for details. 
\end{itemize}

\smallskip Applying the above ideas, one can consider the nonlinear kinetic diffusion-type equation 
\begin{align}\label{eqIntro0}
	\pa_tf+v\cdot\na_xf=Lf+\Ga(\phi_1,f)+\Ga(f,\phi_2)\quad\text{ in }\R_t\times\R^d_x\times\R^d_v,
\end{align}
where the linearized operator $L$ behaves as a (possibly fractional) Laplacian in velocity, like
$$L \sim -\<D_v\>^{2s}+\text{lower-derivative order terms}$$ for some $s \in (0,1)$. Applying the dyadic decomposition to \eqref{eqIntro0} and performing the energy estimate with suitable parameters $\varpi,\om>0$ (to be determined later), one has the key estimate: 
\begin{align*}\notag
	&\|P_k\De_jf\|_{L^\infty_tL^2_{x,v}}^2+c_0\|P_k\De_jf\|_{L^2_{t,x}H^s_v}^2\\
	&\quad\le C\varpi2^j(\om2^k)^{-1}\|P_k\De_jf\|_{L^2_{t,x,v}}^2
	+\big(P_k\De_j(\Ga(\phi_1,f)+\Ga(f,\phi_2)),P_k\De_jf\big)_{L^2_{t,x,v}}
\end{align*}
for any $j,k\ge 0$, 
where the underlying time interval is $[0,T]$ and the first-right term is obtained from the commutator $[P_k,v\cdot\na_x]$. Moreover, for the nonlinear collision terms, we intend to deduce a negative fractional derivative $\<D_v\>^{-s}$ to the first argument of $\Ga$, i.e., 
\begin{align*}
	\big(\Ga(f,g),h\big)_{L^2_v}\le C\|\<D_v\>^{-s}f\|_{L^2_v}\|\<D_v\>^{s}g\|_{L^2_v}\|\<D_v\>^{s}h\|_{L^2_v}. 
\end{align*}
Therefore, by summing in $j,k$ via suitable commutator estimates $[P_k\De_j,\Ga]$, and extracting the fractional derivative via dyadic decomposition, we aim to obtain the estimate
\begin{align}\notag \label{eq124}
	\|f\|_{L^2_{t,x,v}}^2&\le C\sum_{j,k\ge 0}\varpi2^j(\om2^k)^{-1-2s}\|P_k\De_jf\|_{L^2_{t,x,v}}^2
	+C\|\<D_v\>^{-s}\phi_1\|_{L^\infty_{t,x}L^2_v}\|f\|_{L^2_{t,x,v}}^2\\
	&\quad
	+C\|\<D_v\>^{-s}f\|_{L^{r}_{t,x}L^2_v}\|\phi_2\|_{L^{r^*}_{t,x}L^2_v}\|f\|_{L^2_{t,x,v}}+\text{lower-order terms},
\end{align}
for some $r,r^*\in(2,\infty)$ with $\frac{1}{r}+\frac{1}{r^*}=\frac{1}{2}$. The precise energy estimate will be given in \eqref{esSumEverything} and \eqref{es21c}.
To eliminate the derivative $\na_x\sim \varpi2^j$, we set $\om=\om_02^{\frac{j}{1+2s}}$ in the Littlewood-Paley operator $P_k$ with a large $\om_0>1$. This corresponds one derivative in $v$ matching to a $\frac{1}{1+2s}$-derivative in $x$. For the term $\|\<D_v\>^{-s}f\|_{L^r_{t,x}L^2_v}$, we apply the hypoelliptic estimate in Theorem \ref{ThmNegativeHypo} to gain integrability in $(t,x)$. This ensures that, on a short time interval $[0,T]$, the right-hand side of \eqref{eq124} can be absorbed by the left-hand side, implying $f=0$ and proving the uniqueness.

\smallskip 
For the non-cutoff Boltzmann equation, the nonlocal collision structure in \eqref{tiaa} couples velocity and velocity frequency, so we further apply a dyadic decomposition $R_r(v)$ in the velocity variable $v$. The same approach can be directly extended to local collision kinetic equations, such as the Landau and Fokker-Planck equations. We emphasize to the reader that, in this work, the \textbf{leading-order} term requires particular care, in contrast to the \textbf{lower-order} term.

\smallskip\noindent {\bf Some Literature.} 
Even in fluid dynamics, the (non-)uniqueness of weak solutions to the incompressible Navier-Stokes equations remains a fundamental and challenging problem; see e.g., \cite{Buckmaster2019,Cheskidov2022,Coiculescu2025}. 
For the spatially homogeneous kinetic equations, the uniqueness of non-cutoff-type equations has been widely studied; see e.g., \cite{Tanaka1978,Fournier2008,Desvillettes2009,Fournier2010}.

\smallskip 
For the spatially inhomogeneous non-cutoff case, a renormalized solution is obtained~\cite{Alexandre2001}, but its uniqueness is not known.
However, in regular Sobolev spaces $H^m_{x}$ with $m>\frac{d}{2}$, thanks to the Sobolev embedding $\|f\|_{L^\infty_x}\le C\|f\|_{H^m_x}$, the uniqueness of solutions to the non-cutoff Boltzmann equation (and other kinetic equations) can be obtained; see \cite{Alexandre2012,Gressman2011}. For the case of regular Besov spaces, we refer to \cite{Morimoto2016}. On the other hand, if $f\in L^\infty_{t,x}H^{2s}_v$, the uniqueness follows as in \cite[Theorem 1.2]{Alexandre2011}; but $L^\infty_{t,x}H^{2s}_v$ is rather strong in both $(t,x)$ and $v$ variables. 

\smallskip 
For the existence theory of large-data solution, \cite{Henderson2020,Henderson2025} establishes the local-in-time existence of large-data solutions (in $L^\infty$ or $C^2_{\text{kin}}$) to the non-cutoff Boltzmann equation. However, the $L^\infty_{x,v}$ norm appears insufficient to provide a regularity framework capable of establishing uniqueness. 
On the other hand, \cite{Chen2024} showed that the cutoff Boltzmann equation is ill-posed in $H^s_x$ for small $s\ge 0$ in the sense that uniform continuous dependence on the initial data fails. In contrast, we establish the $L^2_{t,x,v}$ stability (weaker than uniform continuous dependence) in the non-cutoff setting. 
Beyond classical solutions, \cite{Duan2020} studies the well-posedness of mild solutions to the non-cutoff Boltzmann and Landau equations in the $L^1_{\xi}$ framework, which forms a Banach algebra under the convolution product, by using the Fourier transform in the spatial variable $x$. 

\smallskip
Recently, Alonso-Gualdani-Sun \cite{Alonso2025} studied uniqueness for Landau-type equations using a multiplier operator $\mathcal{M}$, under the assumption of either solution continuity or the addition of a regularizing term $\nu \Delta_v f$ with small $\|\cdot\|_{L^\infty_{t,x}}$ norm. Moreover, Chen-Nguyen-Yang \cite{Chen2025} established well-posedness for non-cutoff kinetic equations in critical (scaling-invariant) spaces, assuming the initial data is small in a suitable norm and exploiting the regularization effect. 
Moreover, the De Giorgi-Nash-Moser theory offers a powerful tool to obtain $L^\infty$ bounds and regularity for kinetic equations; see \cite{Imbert2019,Imbert2021,Alonso2022,FernandezReal2025}. 
In \cite{Henderson2025}, uniqueness of the non-cutoff Boltzmann equation is established under the assumptions $\phi_1\in C^\al$ and $\phi_2\in L^1_tL^\infty_{x,v}$.

\smallskip\noindent
{\bf Outline of the paper.}
The remainder of the paper is organized as follows. In Appendix \ref{SecToolbox}, we will summarize basic tools for pseudo-differential calculus and standard properties of the Boltzmann collision operator. Section \ref{SecLP} introduces the pseudo-differential operators used in this work and presents the Littlewood-Paley theorem along with key commutator estimates. In Section \ref{SecHypoelliptic}, we prove the hypoelliptic estimate with negative derivative $\langle D_v\rangle^{-s}$ in Theorem \ref{ThmNegativeHypo}. Sections \ref{SecUpper} and \ref{SecComm} establish upper bound and commutator estimates for the Boltzmann collision operator with the velocity-related pseudo-differential operators $P_k$ and $R_r$ defined in \eqref{DejQkDef}. Section \ref{SecLPkinetic} proves the main uniqueness result, Theorem \ref{MainThm1}, for the non-cutoff Boltzmann equation by applying the Littlewood-Paley decomposition and combining the previous estimates. In particular, Subsection \ref{SecSpatialNonlinear} applies $\Delta_j$ to the collision operator $\Gamma(f,g)$ and establishes nonlinear estimates in spatially dyadic blocks. 

\section{The Littlewood-Paley framework}\label{SecLP}
In this section, we introduce the Littlewood-Paley decomposition with mixed $v$, $(t,x)$-frequency variable $(\tau,\xi)$ (used only for hypoellipticity analysis), $x$-frequency variable $\xi$, and $v$-frequency $\eta$. 


\subsection{The pseudo-differential operator with mixed $(v,\tau,\xi,\eta)$}
We state several types of pseudo-differential operators that are used for the Littlewood-Paley decomposition in this work. 

\smallskip  \noindent
{\bf The general case with kernel $\wh{\Xi}(a^{-1}\eta)$.} 
Let $\widehat{\Xi}(\eta) \in C_c^\infty$ be a function with compact support in ${|\eta| \le 1}$ (which will be taken as either $\widehat{\Psi}$ or $\widehat{\Psi_0}$ defined in \eqref{Psi0def} below). Let $a \ge 1$ be a constant that may depend on other parameters. We consider the symbol $\widehat{\Xi}(a^{-1}\eta)$ satisfying
\begin{align}\label{XiDef111}
	\begin{aligned}
	&\wh{\Xi}(a^{-1}\eta)\ \text{ 
that is supported in the region $\{|\eta|\le 2a\}$},\\
&\text{$\na^k\wh{\Xi}(a^{-1}\eta)$ $(k\ge 1)$ is supported in the region $\{\frac{6}{7}a\le |\eta|\le 2a\}$.}
	\end{aligned}
\end{align}
Then we can define the corresponding pseudo-differential operator $P_a$ given by $\wh{P_af}(\eta)=\wh{\Xi}(a^{-1}\eta)\wh{f}(\eta)$. 
Then Young's convolution inequality yields the uniformly-in-$a$ $L^p$ bound for $P_a$ ($p \in [1,\infty]$): 
\begin{align}
	\label{boundDea}
	\|P_{a}f\|_{L^p_v}&\le\|\Xi(v)\|_{L^1_v}\|f(v)\|_{L^p_v}.  
\end{align}

\noindent
{\bf The $v$, $x$-frequency, and $v$-frequency cases: $R_r\sim \wh{\Psi_r}(\<v\>)$,  $\De_j\sim\wh{\Psi_j}(2^{-rl_0}\xi)$ and $P_k\sim\wh{\Psi_k}(2^{-rl_1}2^{-\frac{j}{1+2s}}\eta)$.}
We fix any Schwartz function $\Psi$ on $\R^d$ whose Fourier transform is non-negative, supported in the annulus $\frac{6}{7}\le |\xi|\le 2$, equal to $1$ on the smaller annulus $1\le|\xi|\le2-\frac{2}{7}$. 
Note that we assume only compact support and the Schwartz property; the partition of unity property is imposed only when needed.
To consider the dyadic decomposition, we define a sequence of radial Schwartz functions $\Psi_k$ $(k\ge 1)$ and one radial Schwartz function $\Psi_0$ that takes values in $[0,1]$, by 
\begin{align}\label{Psi0def}
	\wh{\Psi_k}(\eta)=\wh{\Psi}(2^{-k}\eta),\quad k\ge 1,\quad\text{ and }\quad
	\wh{\Psi_0}(\eta)=
	\begin{cases}
		1,&|\eta|<\frac{6}{7}\\
		0,&|\eta|\ge 2. 
	\end{cases}
\end{align}
Henceforth, the radial function notation $\wh{\Psi_k}$ may denote functions in different dimensions. 
We now introduce the pseudo-differential operators associated with the dyadic decomposition in velocity $v$, velocity-frequency $\eta$, and time-space frequency $(\tau,\xi)$, respectively: 
\begin{align}\label{DejQkDef}
	\begin{cases}
	\displaystyle
	R_r(v)=\wh{\Psi_r}(\<v\>),\\
	\displaystyle
	\F_x\De_jf(\xi)=\wh{\Psi_j}(\varpi^{-1}\xi)\F_xf(\xi),\\
	\F_{t,x}\De^{t,x}_jf(\tau,\xi)=\wh{\Psi_j}(\varpi^{-1}\xi)\F_{t,x}f(\tau,\xi),\\
	\displaystyle
	\F_vP_kf(\eta)=\wh{\Psi_{k}}(\om^{-1}\eta)\F_vf(\eta),\\
	\displaystyle
	\F_vQ_kf(\eta)=\wh{\Phi_{k}}(\om^{-1}\eta)\F_vf(\eta),~\text{ where }\wh{\Phi_k}=|\wh{\Psi_k}|^{\frac12}.
	\end{cases}
\end{align}
For brevity, we use the same notation $\wh{\Psi}$ for the indices $r$, $j$, and $k$, as well as for both the $(d+1)$- and the $d$-dimensional dyadic decompositions. Note that $\De_j$ commutes with $P_k$ and $R_r$. Here, the dilation parameters $\varpi=\varpi(r)>0$ and $\om=\om(j,r)\ge 1$ are given by 
\begin{align}\label{Choice}
	\om=\om_02^{\al j+rl_1},\ \
	\varpi=\varpi_02^{rl_0},\ \ 
	\om_0,\varpi_0\ge 1, \quad l_0,l_1\ge 0. 
\end{align}
This indicates that one derivative in $v$ corresponds to an $\al$-derivative in $x$ at the cost of the weight $\<v\>^{l_1}$.
We will also denote, $\wt{R_r}, \wt{\Delta_j}, \wt{\De}_j^{t,x}, \wt{P_k}$, etc., as the operators obtained by replacing $\widehat{\Psi}$ with another Schwartz function $\wh{\wt{\Psi}}$ that equals to $1$ on the support of $\wh{\Psi}$ and whose support is slightly larger than that of $\wh{\Psi}$, so that 
\begin{align}\label{SlightlyLarger}
	R_r=\wt{R_r}R_r,\quad \Delta_j=\wt{\Delta}_j \Delta_j,\quad \De_j^{t,x}=\wt{\De}_j^{t,x}\De_j^{t,x}, \quad P_k=\wt{P_k} P_k,\ \  etc.
\end{align}
The parameters $\varpi$ and $\omega$ will be specified separately in Subsection~\ref{SubsecLPkinetic} and in the negative-order hypoellipticity Theorem~\ref{ThmNegativeHypo} and Corollary~\ref{CoroNegativeHypo}. 

\medskip \noindent
{\bf Partition of unity.} For coercive estimates, we employ a partition of unity associated with $\Psi$, choosing $\wh{\Psi}(\xi)$ such that $\wh{\Psi}(\xi)+\wh{\Psi}(\xi/2)=1$ on the annulus $1\le|\xi|\le 4-\frac{4}{7}$. Then 
\begin{align}\label{sum1}
	&\sum_{j\in\Z}\wh{\Psi}(2^{-j}\xi)=1\ \text{ for all }\xi\neq 0, 
	\ 
	\text{ and set }\wh{\Psi_0}(\xi)=\begin{cases}
		\Big(\sum_{j\le 0}\wh{\Psi}(2^{-j}\xi)\Big)^{\frac{1}{2}}
		&\text{ when }\xi\ne 0,\\
		1&\text{ when }\xi=0. 
	\end{cases}
\end{align}
Then we have dyadic decomposition
Therefore, when \eqref{sum1} is assumed, using the notations in \eqref{DejQkDef}, we have Littlewood-Paley decomposition as 
\begin{align}\label{fequivDejPk}
	f=\sum_{r=0}^\infty\sum^{\infty}_{j=0}\sum_{k=0}^\infty \De_jP_kR_rf,
\end{align}
where the summation is taken in the order $k,j,r$, by expanding the pseudo-differential operators as in \eqref{DejQkDef}, even when $\De_j$ depends on $r$ and $P_k$ depends on $j,r$.
The assumption \eqref{sum1} is used only when explicitly stated (for instance, in Littlewood-Paley Theorem \ref{LPThm} and in the estimate of $\De_j\Gamma(f,g)$ in Subsection \ref{SecSpatialNonlinear}); most other results require only that $\Psi$ be a Schwartz function.

\medskip \noindent
{\bf 
	Parameter table}
Here, we provide a table to illustrate the meaning of the parameters in the Littlewood-Paley (LP) decomposition. 
\begin{table}[ht]
\centering
\begin{tabular}{c|l}
\hline
Parameter & Description \\
\hline
$r$ & Velocity weight index in LP decomposition \\
$j$ & Spatial (or space--time) frequency index in LP decomposition \\
$k$ & Velocity frequency index in LP decomposition \\
$\om$ & Velocity-frequency dilation: $\om=\om_0 2^{\al j + r l_1}$ \\
$\vpi$ & Spatial-frequency dilation: $\vpi=\vpi_0 2^{r l_0}$ \\
$\om_0$ & Large constant (velocity-frequency base) \\
$\vpi_0$ & Large constant (spatial-frequency base) \\
$\al$ & Transfer of regularity: velocity frequency $\rightarrow$ spatial frequency \\
$l_1$ & Transfer of weight: velocity frequency $\rightarrow$ weight $\<v\>$ \\
$l_0$ & Transfer of weight: space frequency $\rightarrow$ weight $\<v\>$ \\
\hline
\end{tabular}
\caption{Description of the parameters.}
\end{table}

\subsection{The Littlewood-Paley theorem}
For the dilated pseudo-differential operators defined in \eqref{DejQkDef}, we still have the corresponding Littlewood-Paley theorem for $P_k$, $Q_k$, and $\Delta_j$ (the same estimates also hold for $R_r(v)$ by Plancherel's theorem).
\begin{Thm}\label{LPThm}
	Denote the pseudo-differential operators $\De_j$, $P_k$, and $Q_k$ as in \eqref{DejQkDef}, with dilation parameters $\varpi=\varpi(r)>0$ and $\om=\om(j,r)\ge 1$. 
	Then, for any $n\in\R$ and suitable function $f$,  
\begin{align}\label{LPThmes}
	\sum_{k=0}^\infty\|\<v\>^n(P_kf,Q_kf)\|_{L^2_v}^2
	&\le C\|\<v\>^nf\|_{L^2_v}^2,\qquad
	\sum_{j=0}^\infty\|\De_jf\|_{L^2_x}^2\le C\|f\|_{L^2_x}^2,
\end{align}
with some constant $C>0$ that is independent of $f,r$ (for both) and $j$ (for the first one).

\smallskip Conversely, by assuming 
	$\sum_{k\ge 0}\wh{\Psi_k}(\eta)=1$ (for all $\eta\neq 0$), 
for any $n\in\R$ and suitable function $f$, 
\begin{align*}
	\|\<v\>^nf\|_{L^2_v}^2&\le C\sum_{k=0}^\infty\|\<v\>^nP_kf\|_{L^2_v}^2,\qquad
	\|f\|_{L^2_x}^2
	\le C\sum_{j=0}^\infty\|\De_jf\|_{L^2_x}^2, 
\end{align*}
\end{Thm}
\begin{proof}
For the classical Littlewood-Paley theorem on $\R^d_x$, we refer to \cite[Theorem 6.1.2]{Grafakos2014}.
Although we have inserted a dilation constant in the kernels, the classic proof is still valid via a simple dilation in the function $f$ when $n=0$.
For instance, by letting $\ol f(x)=f(\varpi^{-1}x)$, we have 
\begin{align*}
	\De_jf(x)&=\int_{\R^{d}_{\xi}\times\R^{d}_{y}}e^{2\pi i\xi\cdot(x-y)}\wh{\Psi_j}(\vpi^{-1}\xi)f(y)\,dyd\xi
	\\
	&=\int_{\R^{d}_{\xi}\times\R^{d}_{y}}e^{2\pi i\xi\cdot(\varpi x-y)}\wh{\Psi_j}(\xi)f(\varpi^{-1}y)\,dyd\xi
	=\De^{\mathrm{std}}_j\,\ol{f}(\varpi x),
\end{align*}
where $\De^{\mathrm{std}}_j$ is the classical Littlewood-Paley operator without dilation. Therefore, by the classical Littlewood-Paley theorem, we have 
\begin{align}\label{rhoolf}
	\varpi^{\frac{d}{2}}\|f\|_{L^2_{x}}=\|\ol{f}\|_{L^2_{x}}\approx \Big(\sum_{j=0}^\infty\|\De^{\mathrm{std}}_j\,\ol{f}\|_{L^2_{x}}^2\Big)^{\frac{1}{2}}=\varpi^{\frac{d}{2}}\Big(\sum_{j=0}^\infty\|\De_jf(x)\|_{L^2_{x}}^2\Big)^{\frac{1}{2}}, 
\end{align}
whenever $f\in L^2_{x}$. 
The same proof applies to $P_k$ and $Q_k$, which implies \eqref{LPThmes} for the case $n=0$.

\smallskip For case $n\neq 0$, we borrow the commutator estimate $[\<v\>^n,P_k]$ in \eqref{CommawPa} below. That is, 
\begin{align*}
	\sum_{k=0}^\infty\|\<v\>^nP_kf\|_{L^2_v}^2
	&\le \sum_{k=0}^\infty\big(\|[\<v\>^{n},P_k]f\|_{L^2_v}+\|P_k\<v\>^nf\|_{L^2_v}\big)^2\\
	&\le C_n\sum_{k=0}^\infty\big(\|\<v\>^{n-1}(\om2^k)^{-1}f\|_{L^2_v}^2+\|P_k\<v\>^nf\|_{L^2_v}^2\big)
	\le C_n\|\<v\>^nf\|_{L^2_v}^2, 
\end{align*}
since $\om\ge 1$ (this is the only place where we use this assumption). 
The same estimate holds for $Q_k$. Conversely, following the standard Littlewood-Paley theorem, if $\<v\>^nf\in L^2_v$, using $\sum_kP_kf=f$ and $P_k=P_k\wt{P_k}$, where $\wt{P_k}$ is given by \eqref{SlightlyLarger}, one can deduce 
\begin{align*}
	\|\<v\>^nf\|_{L^2_v}&=\Big\|\sum_{k\ge 0}\<v\>^nP_kf\Big\|_{L^2_v}=\sup_{\|g\|_{L^2_v}=1}\Big(\sum_{k\ge 0}[\<v\>^n,\wt{P_k}]P_kf+\wt{P_k}\<v\>^nP_kf,g\Big)_{L^2_v}\\
	&\le \sup_{\|g\|_{L^2_v}=1}\sum_{k\ge 0}\Big(\|\<v\>^{n-1}(\om2^k)^{-1}P_kf\|_{L^2_v}\|g\|_{L^2_v}+|(\<v\>^nP_kf,\wt{P_k}g)_{L^2_v}|\Big)\\
	&\le \sup_{\|g\|_{L^2_v}=1}\Big(\sum_{k\ge 0}\|\<v\>^nP_kf\|_{L^2_v}^2\Big)^{\frac{1}{2}}\Big(\|g\|_{L^2_v}^2+\sum_{k\ge 0}\|\wt{P_k}g\|_{L^2_v}^2\Big)^{\frac{1}{2}}
	\\
	&
	\le C\Big(\sum_{k\ge 0}\|\<v\>^nP_kf\|_{L^2_v}^2\Big)^{\frac{1}{2}}.  
\end{align*}
where we used the commutator estimate \eqref{CommawPa}. 
The case of $Q_k$ follows in the same manner.
This completes the proof of Theorem \ref{LPThm} for the case $n\neq 0$. 
\end{proof}

\subsection{Some Commutator estimates}\label{SecLPcommutator}
In this Subsection, we derive some useful commutator estimates related to the pseudo-differential operators defined in \eqref{DejQkDef}. 
Within this Subsection, $\wh{\cdot}$ denotes the Fourier transform with respect to $v$. 


\subsubsection{``Compact support'' of $P_k,\De_j$}
Using the support of the kernels of $P_k$ and $\De_j$, we have 
\begin{align}
	\label{supportPk}
	\|\<v\>^l\<D_v\>^{n}P_kf\|_{L^2_v}&\le C\<\om2^k\>^{n}\|\<v\>^lP_kf\|_{L^2_v},\quad \text{ when }k\ge 0,\,n\ge0,\\
	\notag\|\<v\>^l\<D_v\>^{n}P_kf\|_{L^2_v}&\le C\<\om2^k\>^{n}\|\<v\>^lP_kf\|_{L^2_v},\quad \text{ when }k\ge 1,\,n\in\R,\\
	\||D_x|^n\De_jf\|_{L^2_x}&\le C(\varpi2^j)^n\|\De_jf\|_{L^2_x},\quad\text{ when }j\ge 0,\,n\ge 0,\notag
\end{align}
for any $l\in\R$, where we used \eqref{SlightlyLarger} and the fact that $\<v\>^l\<D_v\>^{n}\wt{P_k}\in \operatorname{Op}(\<v\>^l\<\om2^k\>^{n})$. 

\subsubsection{Commutators \texorpdfstring{$[P_k,\<v\>^l]$}{[Pk,<v>l]} and \texorpdfstring{$[P_k,\ti{a}^w]$}{[Pk,aw]}}
Recall the pseudo-differential structure of $L^2_D$: 
$\|f\|_{L^2_D}=\|\ti{a}^wf\|_{L^2_v}$, where $\ti{a}$ is given by \eqref{tiaa}. It follows from \cite[Prop. 1.4]{Global2019} that $\ti{a}^w$ is an invertible pseudo-differential operator whose symbol belongs to $S(\ti{a}^{-1})$. 
Moreover, by \cite[Prop. 3.7]{Global2019} and \cite[Lemma 2.1]{Deng2020a}, for any multi-indices $\al,\beta$, 
\begin{align*}
	\pa^{\al}_v\pa^{\beta}_\eta\ti{a}&\in S(\ti{a}), \ \ \text{ and }\ \ 
	(\ti{a}^w)^{-1}=G_1(\ti{a}^{-1})^w=(\ti{a}^{-1})^wG_2,
\end{align*}
for some bounded operator $G_1,G_2$ on $H(M)$ with any $\Ga$-admissible weight $M$ (the space $H(M)$ is defined in \eqref{HmDef}). In particular, $G_1,G_2$ are bounded linear operators on $H^k_l$. 
On the other hand, the symbols $\wh{\Psi_k}(\om^{-1}\eta)$, $R_r(v)=\wh{\Psi_r}(\<v\>)$ and the trivial symbols $\<v\>^n$, $\<\eta\>^n$ satisfy 
\begin{align}
		\label{symbolic}
	\begin{aligned}
	\pa_\eta\wh{\Psi_k}(\om^{-1}\eta)\in S(\<\eta\>^{-s'}(\om2^{k})^{s'-1}),\quad
	\pa_v R_r(v)&\in S(2^{-r}),\\ 
	\pa_v\<v\>^n\in S(\<v\>^{n-1}),\quad\pa_\eta\<\eta\>^n\in S(\<\eta\>^{n-1}), 
	\end{aligned}
\end{align}
respectively, for any $s'\in\R$. 
Thus, by using the commutator estimate \eqref{Liebracket} and Lemma \ref{basicCommuLem}, for any $k\ge 0$ and $n,m,l\in\R$, (noting that $[P_k,\<v\>^{l}]\in\text{Op}(\<v\>^{l-1}(\om2^k)^{n-1})$)
\begin{align}
	\label{CommawPa}
	\begin{cases}
	\|
		[P_k,\<v\>^{l}]
		f\|_{L^2_v}
    \le C(\om2^k)^{-1}\|\<v\>^{l-1}f\|_{L^2_v},\\
	\|[\ti{a}^w,\<D_v\>^{l}]f\|_{L^2_v}
	\le C\|\<D_v\>^{l-1}\ti{a}^wf\|_{L^2_v},\\ 
	\|\<v\>^n\<D_v\>^s[P_k,R_r]f\|_{L^2_v}\le C2^{-r}(\om2^k)^{-1}\|\<v\>^n\<D_v\>^sf\|_{L^2_v}.
	\end{cases}
\end{align}

\subsubsection{Upper and lower bounds of some dissipation}
Using the Littlewood-Paley theorem and the commutator estimates \eqref{CommawPa}, we derive the following bounds. 
Due to the mixed presence of $P_k$, $R_r$, and $\tilde a^w$, the commutator estimates only yield lower-order terms, but at the cost of losing the pseudo-differential operators $P_k$ or $R_r$. Therefore, the strategy is to estimate each quantity by a {\bf leading-order} term plus a {\bf lower-order} term.

\smallskip Using the support of the kernel of $P_k$, we write $P_k=P_k\chi_{\frac{6}{7}2^k\le|D_v|\le2^{k+1}}$ and $\wt\chi:=\chi_{\frac{6}{7}2^k\le|D_v|\le2^{k+1}}=\chi_{\frac{6}{7}2^k\le|D_v|\le2^{k+1}}(P_{k-1}+P_k+P_{k+1})$ for $k\ge 0$ (where $P_{-1}:=0$ and $\wt\chi$ is as in \eqref{smoothcutoff}). 
Applying \eqref{esD} and the lower bound $|\eta| \ge \frac{6}{7}\omega 2^k$ for $k \ge 1$, we obtain (note that $P_0$ cannot be included in the dissipation estimate, since $|\eta|$ has no positive lower bound on the support of its kernel)
\begin{align*}
	&\sum_{k\ge 1}(\om2^k)^{-2s}\|P_kf\|^2_{L^2_D}\\
	&\notag\ge \frac{1}{C}\sum_{k\ge 1}(\om2^k)^{-2s}\|\<v\>^{\frac{\ga}{2}}\<D_v\>^{s}P_k\chi_{\frac{6}{7}\om2^k\le|D_v|\le\om2^{k+1}}f\|^2_{L^2_v}\\
	&\notag\ge \frac{1}{C}\sum_{k\ge 1}(\om2^k)^{-2s}\|\<D_v\>^{s}P_k\<v\>^{\frac{\ga}{2}}\wt\chi f\|^2_{L^2_v}-C\sum_{k\ge 1}(\om2^k)^{-2s}\|\underbrace{[\<v\>^{\frac{\ga}{2}},\<D_v\>^{s}P_k]}_{\Op(\<v\>^{\frac{\ga}{2}-1}(\om2^k)^{s-1})}\wt\chi f\|^2_{L^2_v}\\
	&\notag\ge \frac{1}{C}\sum_{k\ge 1}\|\<v\>^{\frac{\ga}{2}}P_kf\|^2_{L^2_v}
	-\frac{1}{C}\sum_{k\ge 1}\|\underbrace{[P_k,\<v\>^{\frac{\ga}{2}}]}_{\Op(\om2^k)^{-1}\<v\>^{\frac{\ga}{2}-1}}\wt\chi f\|^2_{L^2_v}
	-C\sum_{k\ge 1}(\om2^{k})^{-2}\|\<v\>^{\frac{\ga}{2}-1}\wt\chi f\|^2_{L^2_v}\\
	&\ge \frac{1}{C}\sum_{k\ge 1}\|\<v\>^{\frac{\ga}{2}}P_kf\|^2_{L^2_v}
	-C\sum_{k\ge 0}(\om2^{k})^{-2}\|\<v\>^{\frac{\ga}{2}-1}P_kf\|^2_{L^2_v}.
\end{align*}
Let $\wt{P_k}$ be the pseudo-differential operator defined by \eqref{SlightlyLarger}. 
Then $P_k=\wt{P_k}P_k$ and hence, for any $s\in\R$ and $s'\ge 0$, we can extract the Bessel potential $\<D_v\>$ as  
\begin{align}
	\label{dissUpperPk}\notag
    &\|\<D_v\>^s\<v\>^{n}P_kR_rf\|_{L^2_v}^2=\|\<D_v\>^s\<v\>^{n}\wt{P_k}\<D_v\>^{s'-s}\wt{P_k}\<D_v\>^{s-s'}P_kR_rf\|_{L^2_v}^2\\
    &\notag \le C\|\<D_v\>^s\wt{P_k}\<v\>^{n}P_kR_rf\|_{L^2_v}^2
		\\
    &
		\notag\ \ +C\|\<D_v\>^s\underbrace{[\<v\>^{n},\wt{P_k}]}_{\text{Op}((\om2^k)^{-1}\<v\>^{n-1})}\underbrace{\<D_v\>^{s'-s}\wt{P_k}}_{\text{Op}((\om2^k)^{s'}\<\eta\>^{-s})}\<D_v\>^{s-s'}P_kR_rf\|_{L^2_v}^2\\
    & \le C(\om2^k)^{2s'}\|\<D_v\>^{s-s'}\<v\>^{n}P_kR_rf\|_{L^2_v}^2. 
\end{align}
Thus, by \eqref{dissUpperPk} and Littlewood-Paley theorem \ref{LPThm}, since $\om\ge 1$, if $s'\ge 0$, then 
\begin{align*}
	\notag
    \sum_{k\ge0}(\om2^k)^{-2s'}\|\<D_v\>^s\<v\>^{n}P_kR_rf\|_{L^2_v}^2
		&\le C\|\<D_v\>^{s-s'}\<v\>^{n}R_rf\|_{L^2_v}^2. 
\end{align*}
In this fashion, let $\wt{R_r}$ be defined as in \eqref{SlightlyLarger}, satisfying $R_r\le \wt{R_r}\le R_{r-1}+R_r+R_{r+1}$ (where we set $R_{-1}=0$ for brevity).
To extract the weight $\<v\>$, for any $s',s\in\R$ satisfying $s\ge s'$, we have 
\begin{align*}\notag
    &2^{2rl}\|\<v\>^{n}\<D_v\>^sP_kR_rf\|_{L^2_v}^2\\
    &\notag\le C2^{2rl}\|\<v\>^{n}\wt{R_r}\<D_v\>^sP_kR_rf\|_{L^2_v}^2+C2^{2rl}\|\<v\>^n\underbrace{[\<D_v\>^sP_k,\wt{R_r}]}_{\text{Op}((\om2^k)^{s-s'-1}2^{-rl}\<\eta\>^{s'}\<v\>^{l-1})}R_rf\|_{L^2_v}^2\\
		&\le C\|\<v\>^{l+n}\<D_v\>^sP_kR_rf\|_{L^2_v}^2+C(\om2^k)^{2s-2s'-2}\|\<v\>^{l+n-1}\<D_v\>^{s'}R_rf\|_{L^2_v}^2. 
\end{align*}
Moreover, without the operator $P_k$, for any $l,s\in\R$, we simply have 
\begin{align}\notag
	\label{dissUpper1temp}
    &2^{2rl}\|\<D_v\>^s\<v\>^{n}R_rf\|_{L^2_v}^2\\
		\notag
    &\le C2^{2rl}\|\<v\>^{n}\wt{R_r}\<D_v\>^sR_rf\|_{L^2_v}^2+C2^{2rl}\|\underbrace{[\<D_v\>^s,\<v\>^{n}\wt{R_r}]}_{\text{Op}(\<\eta\>^{s-1}2^{-rl-1}\<v\>^{l+n})}R_rf\|_{L^2_v}^2\\
    &\notag\le C\|\wt{R_r}\<v\>^{l+n}\<D_v\>^sR_rf\|_{L^2_v}^2+C2^{-2r}\|\<v\>^{l+n}\<D_v\>^{s-1}R_rf\|_{L^2_v}^2\\
		&\le C\|\<v\>^{l+n}\<D_v\>^sR_rf\|_{L^2_v}^2. 
\end{align}
In the following, we consider a pseudo-differential operator $P^{(1)}_k$ (and similarly $Q^{(1)}_k$) whose kernel has compact support in the annulus $\{\frac{6}{7}\om2^k\le|\eta|\le\om2^{k+1}\}$. The rigorous definition will be given in \eqref{Pkderi1} and \eqref{wtPk} later. 
Let $\wt{R_r}=\wt{R_r}(\<v\>)$ be defined as in \eqref{SlightlyLarger} such that $R_r=\wt{R_r}R_r$.
 Then, for any $n,l,m,s'\in\R$, we can consider  
\begin{align*}
	2^{rl}\<v\>^{n}\<D_v\>^{m}P_k^{(1)}R_rf
	&=2^{rl}\underbrace{[\<v\>^{n},\<D_v\>^{m}P_k^{(1)}]}_{\text{Op}((\om2^k)^{m+s'-1}\<\eta\>^{-s'}\<v\>^{n-1})}R_rf
	+2^{rl}\<D_v\>^{m}P_k^{(1)}\<v\>^{n}R_rf,
\end{align*}
and, by using $P_k^{(1)}=P_k^{(1)}\chi_{\frac{6}{7}\om2^k\le|D_v|\le\om2^{k+1}}$ (when $k\ge 0$), 
\begin{align*}
	2^{rl}\<D_v\>^{m}P_k^{(1)}\<v\>^{n}\wt{R_r}R_rf
	&=2^{rl}\underbrace{[\<D_v\>^{m}P_k^{(1)},\<v\>^{n}\wt{R_r}]}_{\text{Op}((\om2^k)^{m+s'-1}\<\eta\>^{-s'}\<v\>^{n-1})}R_rf
	+\underbrace{2^{rl}\<v\>^{n}\wt{R_r}\<D_v\>^{m}P_k^{(1)}}_{\text{Op}(\<v\>^{n+l}(\om2^k)^{m-s}\<\eta\>^s)}R_rf.
\end{align*}
As in the symbolic calculus in \eqref{symbolic}, we use the fact that $\pa_\eta(\<\eta\>^{m}\wh{\Psi_k^{(1)}}(\eta))\in S((\om2^k)^{m-1})$ and that $\pa_v\wt{R_r}\in S(2^{-r})$. 
These commutator estimates, together with \eqref{dissUpper1temp}, yield
\begin{align*}\notag
	&2^{rl}\|\<D_v\>^{m}\<v\>^{n}P_k^{(1)}R_rf\|_{L^2_v}\\
	&\notag\le C(\om2^k)^{m+s'-1}\|2^{rl}\<v\>^{n-1}\<D_v\>^{-s'}R_rf\|_{L^2_v}
	+2^{rl}\|\<D_v\>^{m}P_k^{(1)}\<v\>^{n}R_rf\|_{L^2_v}\\
	&\notag\le C(\om2^k)^{m+s'-1}\|\<v\>^{l+n-1}\<D_v\>^{-s'}R_rf\|_{L^2_v}
	\\
	&\quad+C(\om2^k)^{m-s}\|\<v\>^{l+n}\<D_v\>^{s}(P_{k-1}+P_k+P_{k+1})R_rf\|_{L^2_v},
\end{align*}
where, for brevity of notations, we let $P_{-1}=0$. 

\smallskip As a summary, we have obtained estimates: for any $n,l,m,s,s',s''\in\R$ satisfying $s'\le s$, 
\begin{align}
	\label{dissipationLower}
	&\sum_{k\ge 1}\|\<v\>^{\frac{\ga}{2}}P_kf\|^2_{L^2_v}
	\le C\sum_{k\ge 1}(\om2^k)^{-2s}\|P_kf\|^2_{L^2_D}
	+C\sum_{k\ge 0}(\om2^{k})^{-2}\|\<v\>^{\frac{\ga}{2}-1}P_kf\|^2_{L^2_v},\\
		\label{dissUpper}
    &\sum_{k\ge0}(\om2^k)^{-2s'}\|\<D_v\>^s\<v\>^{n}P_kR_rf\|_{L^2_v}^2
		\le C\|\<D_v\>^{s-s'}\<v\>^{n}R_rf\|_{L^2_v}^2,\\
		\label{dissUpper1b}\notag
		&2^{2rl}\|\<D_v\>^s\<v\>^{n}P_kR_rf\|_{L^2_v}^2
		\le C\|\<v\>^{l+n}\<D_v\>^sP_kR_rf\|_{L^2_v}^2\\
		&\qquad\qquad\qquad\qquad\qquad\qquad+C(\om2^k)^{2s-2s'-2}\|\<v\>^{l+n-1}\<D_v\>^{s'}R_rf\|_{L^2_v}^2,\\
		\label{dissUpper1}
    &2^{2rl}\|\<D_v\>^s\<v\>^{n}R_rf\|_{L^2_v}^2
    \le C\|\<v\>^{l+n}\<D_v\>^sR_rf\|_{L^2_v}^2,\\
\label{LowerSupportComm}\notag
	&2^{2rl}\|\<D_v\>^{m}\<v\>^{n}P_k^{(1)}R_rf\|_{L^2_v}^2\\
	&\ \ \notag\le C(\om2^k)^{2m-2s}\|\<v\>^{l+n}\<D_v\>^{s}(P_{k-1}+P_k+P_{k+1})R_rf\|_{L^2_v}^2
	\\
	&\quad
	+C(\om2^k)^{2m+2s''-2}\|\<v\>^{l+n-1}\<D_v\>^{-s''}R_rf\|_{L^2_v}^2 \ \ \text{(where $P_{-1}=0$)}. 
\end{align}
The estimate \eqref{LowerSupportComm} also holds with $P_k$ replacing $P_k^{(1)}$ for $k\ge 1$. Similar estimates hold for $R_rP_k$ replacing $P_kR_r$ via the commutator estimate \eqref{CommawPa}.

\section{The hypoellipticity method}\label{SecHypoelliptic}
In this Section, we take 
$\displaystyle\wh{f}(\tau,\xi)=\int_{\R_t\times\R^d_x}f(t,x)e^{-i(t,x)\cdot(\tau,\xi)}\,dxdt$,
as the non-unitary Fourier transform in $(t,x)$, 
and we aim to derive the hypoelliptic estimate, for the case of low-velocity frequency $k=0$ and large spatial-frequency $j\ge 1$, of the kinetic transport equation in the whole space:
\begin{align}\label{eqtrans}
	(\pa_t+v\cdot\na_x)f=g,\quad\text{ in }\ \R_t\times\R^d_x\times\R^d_v,
\end{align}
in the sense of distribution. 
We begin by estimating the following basic pattern of an integral.
\begin{Lem}\label{basicpattern}
	Let $d\ge 1$ and $\phi(v)\in W^{1,1}(\R^d)$. For any constants $a,b>0$ and $l\ge 1$, we have   
	\begin{align}\label{eqPhivp0}
		&\int_{\R^d_{v_*}}\phi(v-v_*)
		\Big|\frac{1}{b+i(\tau+v_*\cdot\xi)}\Big|^{2l}
		\,dv_*\le Cb^{-2l+1}|\xi|^{-1}\|\na_v\phi\|_{L^1_v},\\
		&
		\label{eqPhivp0a}
		\int_{\R^d_{v_*}}\phi(v-v_*)
		\Big|\frac{1}{b+i(\tau+v_*\cdot\xi)\<v_*\>}\Big|^{2l}
		\,dv_*\le Cb^{-2l+1}|(\tau,\xi)|^{-1}\|\na_v\phi\|_{L^1_v},
	\end{align}
	where $|(\tau,\xi)|=\sqrt{|\tau|^2+|\xi|^2}$, 
	with some generic constant $C>0$.
\end{Lem}
\begin{proof}
	First, we consider a rotating change of variable: $v\mapsto \ti v=R^{-1}v$ with an orthogonal matrix $R_\xi$ satisfying $R^{-1}_\xi=R^T_\xi$ and $R^T_\xi\frac{\xi}{|\xi|}=e_1\equiv (1,0,\dots,0)$. For example, we choose $R_\xi=\bigl(\frac{\xi}{|\xi|},R_2,\dots,R_d\bigr)$ with some unit vectors $R_i\in\R^d$ orthogonal to $\frac{\xi}{|\xi|}$. Thus, $|\pa_{v_1}(R_\xi v)|\le 1$ and by the Sobolev inequality for one-dimensional $v_1$, 
	\begin{align*}
		\|\psi(R_\xi v)\|_{L^\infty_{v_1}L^1_{v'}}\le C\|\pa_{v_1}(\psi(R_\xi v))\|_{L^1_{v}}\le C\|\na_v\psi(R_\xi v)\|_{L^1_{v}}\le C\|\na_v\psi(v)\|_{L^1_{v}}, 
	\end{align*}
	where $v'=(v_2,\dots,v_d)$. Here, we used the $1$-D Sobolev inequality:
		$\|u\|_{L^\infty(\R)}\le C\|u'\|_{L^1(\R)}$. 
Therefore, fixing any $(\tau,\xi)\in\R^{1+d}$, by rotation $v_*\mapsto R_\xi v_*$, translation, and rotation $v_*\mapsto R^{-1}_\xi v_*$ again, we have 
	\begin{align*}\notag
		\int_{\R^d_{v_*}}\phi(v-v_*)\Big|\frac{1}{b+i(\tau+v_*\cdot\xi)}\Big|^{2l}\,dv_*
		&\notag\le 
		C\|\phi(v-R_\xi v_*)\|_{L^\infty_{v_{*1}}L^1_{v_*'}}
		\Big\|\frac{1}{|b+i(\tau+|\xi|v_{*1})|^{2l}}\Big\|_{L^1_{v_{*1}}}\\
		&\le C\|\na_{v_*}\phi(R_\xi v_*)\|_{L^1_{v_*}}\|(1+|v_{*1}|^2)^{-l}\|_{L^1_{v_{*1}}}|\xi|^{-1}b^{-2l+1}\\
		&\le Cb^{-2l+1}|\xi|^{-1}\|\na_v\phi\|_{L^1_{v_*}},
	\end{align*}
	where $v'=(v_2,\dots,v_d)$. To calculate \eqref{eqPhivp0a}, we apply \cite[Lemma 2.4]{Bouchut2002} to find that 
	\begin{align}
		\label{es45}
		\int_{\R}\frac{dv_1}{(b^2+(\tau+|\xi|v_1)^2\<v_1\>^{2})^{l}}\le \frac{Cb^{-2l+1}}{|(\tau,\xi)|}.
	\end{align} 
	Therefore, 
	\begin{align*}
		\int_{\R^d_{v_*}}\phi(v-v_*)\Big|\frac{1}{b+i(\tau+v_*\cdot\xi)\<v_*\>}\Big|^{2l}\,dv_*
		&\notag\le 
		C\|\phi(v-R_\xi v_{*})\|_{L^\infty_{v_{*1}}L^1_{v_*'}}
		\|(b^2+(\tau+|\xi|v_{1})^2\<v_{1}\>^{2})^{l}\|_{L^1_{v_{1}}}\\
		&\le Cb^{-2l+1}|(\tau,\xi)|^{-1}\|\na_v\phi\|_{L^1_{v}}. 
	\end{align*}
	This completes the proof of Lemma \ref{basicpattern}. 
\end{proof}

\subsection{The hypoelliptic \texorpdfstring{$L^2$}{L2} estimate for dilated pseudo-differential operator}
In this Subsection, we intend to deduce the hypoelliptic estimate in $L^2$. 
Applying $Q_0$ to \eqref{eqtrans}, we have 
\begin{align}\label{eq49}
	\pa_tQ_0f+v\cdot\na_xQ_0f&=[v\cdot\na_x,Q_0]f+Q_0g
	=(\om2^k)^{-1}\na_x\cdot Q_0^{(1)}f+Q_0g,
\end{align}
where we have calculated the commutator as 
\begin{align*}
	[v\cdot\na_x,Q_0]f
	&=\int_{\R^d_{v_*}}\om^d\Phi_0(\om(v-v_*))(v-v_*)\cdot\na_xf(v_*)\,dv_*
	=(\om2^k)^{-1}\na_x\cdot Q^{(1)}_0f. 
\end{align*}
Here, $Q^{(1)}_0$ is a derivative variant of $Q_0$ with no gain in coefficients, given by 
\begin{align}\label{Pkderi1}
	Q_0^{(1)}f(v)&=\int_{\R^d_{v_*}}\om^d\Phi_0^{(1)}(\om(v-v_*))f(v_*)\,dv_*,\  
	~\text{ where }~\Phi_0^{(1)}(v)=v\Phi_0(v),
\end{align}
and thus, $\wh{\Phi_0^{(1)}}(\eta)=\frac{-1}{2\pi i}\na\wh{\Phi_0}(\eta)$ is a derivative variant of $\Phi_0$ .
Taking Fourier transform in \eqref{eq49} with respect to $(t,x)$ and adding a constant $\lam=\lam(\tau,\xi,r)$, we have 
\begin{align}\label{P0FtxWf}
	Q_0\F_{t,x}f=\frac{\lam Q_0\F_{t,x}f+(\om2^k)^{-1}\xi\cdot Q_0^{(1)}\F_{t,x}f+Q_0\F_{t,x}g}{\lam+i(\tau+v\cdot\xi)}.
\end{align}
\begin{Thm}\label{Thmlow}
	Assume that $f,g\in L^2(\R_t\times\R^d_x\times\R^d_v)$ satisfies the pure transport equation 
	\begin{align*}
		(\pa_t+v\cdot\na_x)f=g,
		\quad\text{ in }\ \R_t\times\R^d_x\times\R^d_v,
	\end{align*}
	in the sense of distributions. 
	Let $\Psi_j$ be given as in \eqref{Psi0def} and denote the spatial and velocity pseudo-differential operators $\De_j,P_k,Q_k$ as in \eqref{DejQkDef}. 
Then, for any fixed $j\ge 1$ and any $C_0>0$, the following estimate holds: 
\begin{align}\label{hypoes}\notag
	\|P_0\F_{t,x}{\De_jR_rf}\|_{L^2_v}^2
	&\le C_0^{-1}\|Q_0\F_{t,x}\De_jR_rf\|_{L^2_{v}}^2
	+CC_0\om^{-2}\|\<v\>^{-1}\F_{t,x}\De_jR_rf\|_{L^2_v}^2\notag\\
	&\quad\notag
	+CC_0\om^{-2s}2^{-r\ga}\|\<v\>^{\frac{\ga}{2}}\<D_v\>^{s}(P_0+P_1)\F_{t,x}\De_jR_rf\|_{L^2_v}^2\\
	&\quad
	+CC_0\om^2(\varpi2^j)^{-2}\|Q_0\F_{t,x}\De_jR_rg\|_{L^2_{v}}^2.
\end{align}
\end{Thm}

\begin{proof}
Fix any $j\ge 1$ and set $k=0$. 
Note that $\De_j,R_r$ commutes with $\pa_t$ and $v\cdot\na_x$. We can then apply $\De_jQ_0R_r$ to identity \eqref{P0FtxWf} to obtain 
\begin{align}\label{eq411eq}
	Q_0^2\F_{t,x}{\De_jR_rf}
	&=
	Q_0\Big\{\frac{\lam Q_0\F_{t,x}\De_jR_rf+\om^{-1}\xi\cdot Q_0^{(1)}\F_{t,x}\De_jR_rf+Q_0\F_{t,x}\De_jR_rg}{\lam+i(\tau+v\cdot\xi)}\Big\}.
\end{align}
For the first $f$ term and the last $g$ term, using H\"older's inequality and estimate \eqref{eqPhivp0}, 
\begin{align}\label{eq410a}\notag
	\Big\|Q_0\Big\{\frac{\lam Q_0\F_{t,x}\De_jR_rf}{\lam+i(\tau+v\cdot\xi)}\Big\}\Big\|_{L^2_{v}}^2
	\notag
	&\le \Big\|\Big(\int_{\R^d_{v_*}}\om^{d}|\Psi_0(\om(v-v_*))|\Big|\frac{1}{\lam+i(\tau+v_*\cdot\xi)}\Big|^2\,dv_*\Big)^{\frac{1}{2}}\\
	&\notag\ \ \times\Big(\int_{\R^d_{v_*}}\om^{d}|\Psi_0(\om(v-v_*))||\lam Q_0\F_{t,x}\De_jR_rf(v_*)|^2\,dv_*\Big)^{\frac{1}{2}}\Big\|_{L^2_{v}}^2\\
	&\le C\om\lam|\xi|^{-1}\|Q_0\F_{t,x}\De_jR_rf\|_{L^2_{v}}^2, 
\end{align}
and similarly, 
\begin{align}\label{eq410}
	\Big\|Q_0\Big\{\frac{Q_0\F_{t,x}\De_jR_rg}{\lam+i(\tau+v\cdot\xi)}\Big\}\Big\|_{L^2_{v}}^2
	&\le C\om\lam^{-1}|\xi|^{-1}\|Q_0\F_{t,x}\De_jR_rg\|_{L^2_{v}}^2. 
\end{align}
For the commutator term $Q_0^{(1)}f$, using H\"older's inequality and estimate \eqref{eqPhivp0} again, together with the lower bound of the support of the kernel of $Q_0^{(1)}$ and the commutator estimate as in \eqref{LowerSupportComm} (with $m=n=0$, $l=\frac{\ga}{2}$, and $s'=0$ therein), 
\begin{align}\label{eq410c}\notag
	&\Big\|Q_0\Big\{\frac{\om^{-1}\xi\cdot Q_0^{(1)}\F_{t,x}\De_jR_rf}{\lam+i(\tau+v\cdot\xi)}\Big\}\Big\|_{L^2_{v}}^2
	\le C\om^{-1}\lam^{-1}|\xi|\|Q_0^{(1)}\F_{t,x}\De_jR_rf\|_{L^2_{v}}^2\\
	&\ \le C\om^{-1}\lam^{-1}|\xi|\Big(\om^{-1}\|\<v\>^{-1}\F_{t,x}\De_jR_rf\|_{L^2_v}
	+\om^{-s}2^{-r\frac{\ga}{2}}\|\<v\>^{\frac{\ga}{2}}\<D_v\>^{s}(P_0+P_1)\F_{t,x}\De_jR_rf\|_{L^2_v}\Big)^2.
\end{align}
Here, the first term corresponds to the commutator contribution, while the second term is the leading-order term. Combining estimates \eqref{eq410a}, \eqref{eq410}, and \eqref{eq410c} with \eqref{eq411eq}, using the support property $
\big\{\frac{6}{7}\varpi 2^j\le |\xi|\le \varpi 2^{j+1}\big\}$
of the kernel of $\De_j$ for $j\ge 1$, noting that $P_0=Q_0Q_0$, and choosing 
\begin{align*}
	\lam=C_0^{-1}C^{-1}\om^{-1}(\varpi2^j), 
\end{align*}
we deduce \eqref{hypoes}.
This completes the proof of Theorem \ref{Thmlow}.
\end{proof}

\subsection{The negative-order hypoelliptic \texorpdfstring{$L^p$}{Lp} estimate}
In this subsection, we aim to establish the time-spatial regularity of $\<D_v\>^{-N}f$ for any $N>0$, where $f$ solves a kinetic-type equation. Classical hypoelliptic estimates, such as \cite[Theorem 2.1]{Bouchut2002}, typically focus on the regularity of $f$ itself rather than on negative-order $\<D_v\>^{-N}f$. The averaging lemma involving negative-order $\<D_v\>^{-N}$ was also considered, for instance, in \cite{Jabin2022,Zhu2025}. We also refer to the counterexample in \cite[Theorem 1.3]{DeVore2000}; see also \cite{Arsenio2019}. 
Note that, in this subsection, we employ the \textbf{time-spatial} dyadic decomposition, unlike the purely spatial decomposition used in other parts of this work.  
Thus, we recall the pseudo-differential operators $\De^{t,x}_j,P_k$ defined in \eqref{DejQkDef}: 
\begin{align}\label{Dejtimespatial}
	\begin{cases}
	\De^{t,x}_jf(x)=\int_{\R^{2+2d}_{s,\tau,y,\xi}}e^{2\pi i(\tau,\xi)\cdot(t-s,x-y)}\wh{\Psi_j}(\tau,\xi)f(s,y)\,dydsd\xi d\tau,\\
	P_kf(v)=\int_{\R^d_\eta\times\R^d_u}e^{2\pi i\eta\cdot(v-u)}\wh{\Psi_{k}}(\om^{-1}\eta)f(u)\,dud\eta,\\
	Q_kf(v)=\int_{\R^d_\eta\times\R^d_u}e^{2\pi i\eta\cdot(v-u)}\wh{\Phi_{k}}(\om^{-1}\eta)f(u)\,dud\eta,
	\end{cases}
\end{align}
where $\wh{\Psi_j}(\tau,\xi)$ denotes the $(d+1)$-dimensional dyadic decomposition. Moreover, throughout this subsection we fix dilation $\om=2^{\al j}$ with $\al$ given by Theorem~\ref{ThmNegativeHypo}, which differs from the choice $\om=\om_0 2^{\al j+rl_1}$ and the value $\al=\fr{1}{1+2s}$ used elsewhere in this work.

\smallskip We are now ready to state the negative-order hypoellipticity. 
\begin{Thm}
\label{ThmNegativeHypo}
Let $\om=\om(j)=2^{\al j}$.
	Let $2\le p<\infty$, $m\in(0,\fr{d}{p})$ and $M\ge 0$. Assume that $f,g\in L^p(\R_t\times\R^d_x\times\R^d_v)$ satisfies the pure transport equation 
	\begin{align}
		\label{eqtrans2}
		(\pa_t+v\cdot\na_x)f=g,
		\quad\text{ in }\ \R_t\times\R^d_x\times\R^d_v,
	\end{align}
	in the sense of distributions. 
	Denote the pseudo-differential operators $\De_j,P_k$ as in \eqref{Dejtimespatial}.
	Denote the constants measuring transfer-regularity and time-spatial regularity, respectively, by
	\begin{align*}
		\al=\fr{1}{\th_{m,p}+1+M}, \quad \beta=\fr{\th_{m,p}}{\th_{m,p}+1+M}, 
	\end{align*}
	where $\th_{m,p}=
\begin{cases}
	\min\{\frac{m^-}{2},\frac{1^-}{p}\},&\text{if }p>2,\\
	\min\{\frac{m}{2},\frac{1}{2}\},&\text{if }p=2, 
\end{cases}$ and $m^-$ denotes an arbitrary number in $(0,m)$. 
	For any fixed $j\ge 1$, the following estimates hold:
\begin{align}
	\label{esPkWjDejf}
	\begin{aligned}
	&2^{\beta j}\|\<D_v\>^{-m}P_0\De_jf\|_{L^p_{t,x,v}}
	\le C_{p}
	\Big(\|\De_jf\|_{L^p_{t,x,v}}
	+\|\<D_{v}\>^{-M}\<v\>\De_jg(v)\|_{L^p_{t,x,v}}\Big),\\
	&2^{\beta j}\|\<D_v\>^{-m}(I-P_0)\De_jf\|_{L^p_{t,x,v}}
	\le C_{p}\|\De_jf\|_{L^p_{t,x,v}}. 
	\end{aligned}
\end{align}
For $j=0$ we have
\begin{align}\label{eq442}
	\|\<D_v\>^{-m}\De_0f\|_{L^p_{t,x,v}}\le C\|\De_0f\|_{L^p_{t,x,v}}.
\end{align}
\end{Thm}
\begin{proof}
The estimate \eqref{eq442} in the case $j=0$ is trivial by the boundedness of $\<D_v\>^{-m}$. In the following, we assume $j\ge1$. 
Since $\F_{t,x},\De_j$ commute with $\pa_t,v\cdot\na_x$, applying $\F_{t,x}\De_j$ to \eqref{eqtrans2} implies
\begin{align*}
	\F_{t,x}\De_jf
	=\frac{\lam2^j\<v\>^{-1}\F_{t,x}\De_jf+\F_{t,x}\De_jg}{\lam2^j\<v\>^{-1}+i(\tau+v\cdot\xi)}.
\end{align*}
Applying $P_0\F_{t,x}^{-1}$ implies 
\begin{align}\label{eqFour1}
	P_0\De_jf
	&=
	P_0\F_{t,x}^{-1}\Big\{\frac{\lam2^j\<v\>^{-1}\F_{t,x}\De_jf+\F_{t,x}\De_jg}{\lam2^j\<v\>^{-1}+i(\tau+v\cdot\xi)}\Big\}
	=I_f+I_g.
\end{align}

\smallskip\noindent
{\bf Step 1. $L^2$ estimate of $I_f$.}
For the part $I_f$, applying Plancherel's theorem, $L^2$ boundedness of $P_0$, Bessel potential \eqref{GsDecay}, and Cauchy-Schwarz inequality, we can calculate 
\begin{align}\label{es214}\notag
	\|\<D_v\>^{-m}I_f\|_{L^2_{t,x,v}}
	&\le \big\|G_m*_v\big\{\frac{\lam2^j\<v\>^{-1}\F_{t,x}\De_jf}{\lam2^j\<v\>^{-1}+i(\tau+v\cdot\xi)}\big\}\big\|_{L^2_{\tau,\xi,v}}\\
	&\le \lam2^j\Big\|\Big\{G_m*_v\big(\lam2^j+i(\tau+v\cdot\xi)\<v\>\big)^{-2}\Big\}^{\fr12}\Big\{G_m*_v|\F_{t,x}\De_jf|^2\Big\}^{\fr12}\Big\|_{L^2_{\tau,\xi,v}}. 
\end{align}
For the first factor, we consider the orthogonal matrix $R_\xi$ satisfying $R^{-1}_\xi=R^T_\xi$ and $R^T_\xi\frac{\xi^T}{|\xi|}=(1,0,\dots,0)^T$ as in Lemma \ref{basicpattern}. 
Thus, for any fixed $(\tau,\xi)\in\R^{1+d}$, by rotation $u\mapsto R_\xi u$, we have 
\begin{align}\label{esGmCon}\notag
	G_m*_v\big(\lam2^j+i(\tau+v\cdot\xi)\<v\>\big)^{-2}
	&=\int_{\R^d_{u}}G_m(v-R_\xi^Tu)\big(\lam2^j+i(\tau+|\xi|u_1)\<u\>\big)^{-2}\,du\\
	&\le\|G_m(v-R_\xi^Tu)\|_{L^p_{u_1}L^1_{u'}}\big\|\big(\lam2^j+i(\tau+|\xi|u_1)\<u_1\>\big)^{-2}\big\|_{L^{p'}_{u_1}},
\end{align}
for any $p\ge 1$, where $\fr{1}{p}+\fr{1}{p'}=1$. 
Since the Bessel potential term $G_m$ is a radial function, using the decay in \eqref{GsDecay},
\begin{align*}
	\|G_m(v-R_\xi^Tu)\|_{L^p_{u_1}L^1_{u'}}
	&=\|G_m(R_\xi v-u)\|_{L^p_{u_1}L^1_{u'}}=\|G_m\|_{L^p_{u_1}L^1_{u'}}\\
	&\le \|\1_{|u|\le 2}H_m(u)\|_{L^p_{u_1}L^1_{u'}}+C\|\1_{|u|>2}e^{-\fr{|u|}{2}}\|_{L^p_{u_1}L^1_{u'}},
\end{align*}
where $H_m$ is given by \eqref{GsDecay}. 
For the part near $u=\infty$, the Bessel potential has a good decay. For the part near $u=0$, the asymptotic behavior \eqref{GsDecay} shows that $H_m$ is well behaved for $m\ge d$. On the other hand, when $m\in(0,d)$, 
\begin{align*}
	\|\1_{|u|\le 2}H_m(u)\|_{L^p_{u_1}L^1_{u'}}
	&\le C\|\1_{|u|\le 2}|u|^{m-d}\|_{L^p_{u_1}L^1_{u'}}
	\le C\|\1_{|u_1|\le 2}|u_1|^{m-1}\|_{L^p_{u_1}}\\
	&\le C_m<\infty, 
\end{align*}
provided that $p(m-1)>-1$, i.e., $m>1-\fr1p=\fr1{p'}$ when $p'\in(1,\infty]$ and $m\ge 1$ when $p'=1$. 
For the second factor in \eqref{esGmCon}, we simply apply \eqref{es45} to deduce that 
\begin{align*}
	\big\|\big(\lam2^j+i(\tau+|\xi|u_1)\<u_1\>\big)^{-2}\big\|_{L^{p'}_{u_1}}
	&\le C_p(\lam2^j)^{-2+\fr1{p'}}|(\tau,\xi)|^{-\fr{1}{p'}}. 
\end{align*}
Substituting the above into \eqref{esGmCon}, we have 
\begin{align}\label{esGmCon1}
	G_m*_v\big(\lam2^j+i(\tau+v\cdot\xi)\<v\>\big)^{-2}
	&\le C_{m,p}(\lam2^j)^{-2+\fr1{p'}}|(\tau,\xi)|^{-\fr{1}{p'}},
\end{align}
Plugging \eqref{esGmCon1} into \eqref{es214}, we deduce 
\begin{align}\label{es318}\notag
	\|\<D_v\>^{-m}I_f\|_{L^2_{t,x,v}}
	&\le C_{m,p}\lam2^j\times (\lam2^j)^{-1+\fr1{2p'}}
	\times \Big\||(\tau,\xi)|^{-\fr{1}{2p'}}\Big\{G_m*_v|\F_{t,x}\De_jf|^2\Big\}^{\fr12}\Big\|_{L^2_{\tau,\xi,v}}\\
	&\notag\le C_{m,p}(\lam2^j)^{\fr1{2p'}}\||(\tau,\xi)|^{-\fr{1}{2p'}}\F_{t,x}\De_jf\|_{L^2_{\tau,\xi,v}}\\
	&\le C_{m,p}\lam^{\fr1{2p'}}\|\De_jf\|_{L^2_{t,x,v}}. 
\end{align}
for any $m>0$ and $p\in[1,\infty]$ such that $m>\fr{1}{p'}$ when $p'\in(1,\infty]$, and $m\ge 1$ when $p'=1$, i.e., the optimal regularity is given by $\lam^{\min\{\frac{m^-}{2},\frac{1}{2}\}}$, where $m^-$ denotes an arbitrary number in $(0,m)$.

\smallskip\noindent
{\bf Step 2. $L^p$ estimate of $I_f$.}
We need to analyze the multipliers
\begin{align}\label{m0Def}
	m_f(\tau,\xi)&:=\frac{\lam2^j\<v\>^{-1}}{\lam2^j\<v\>^{-1}+i(\tau+v\cdot\xi)}. 
\end{align}
Moreover, since the translation, dilation, and rotation won't change the multiplier norm; see, for instance, \cite[Prop. 2.5.14]{Grafakos2014}, we can apply dilation $(\tau,\xi)\mapsto 2^{-j}(\tau,\xi)$ and the following $(d+1)$-dimensional rotation. 
Let rotation $B$ be defined by
\begin{align*}
	B\begin{pmatrix}
         1  \\
         v_1\\
         \vdots\\
         v_d
    \end{pmatrix}
    =\begin{pmatrix}
         \<v\>  \\
         0\\
         \vdots\\
         0
    \end{pmatrix}.
\end{align*}
Then, using notation $\<v\>_i=\sqrt{1+v_1^2+\cdots + v_i^2}$ temporarily, $B$ is given by
\begin{align*}
    B=\begin{pmatrix}
     \frac{1}{\<v\>} &  \frac{v_1}{\<v\>}  & \frac{v_2}{\<v\>} & \cdots & \frac{v_d}{\<v\>} \\
      -\frac{v_1}{\<v\>_1}   & \frac{1}{\<v\>_1} & 0 & \cdots &0\\
      -\frac{v_2}{\<v\>_1\<v\>_2} & -\frac{v_1v_2}{\<v\>_1\<v\>_2}  & \frac{\<v\>_1^2}{\<v\>_1\<v\>_2} & \cdots & 0\\
      \vdots & \vdots & \vdots &\ddots & 0\\
      -\frac{v_d}{\<v\>_{d-1}\<v\>} & -\frac{v_1v_d}{\<v\>_{d-1}\<v\>} & -\frac{v_2 v_d}{\<v\>_{d-1}\<v\>} & \cdots & \frac{\<v\>^2_{d-1}}{\<v\>_{d-1}\<v\>}
    \end{pmatrix},
\end{align*}
where the $i$-th row of the matrix ($3\le i\le d+1)$ is
\begin{align*}
    \frac{1}{\<v\>_{i-2}\<v\>_{i-1}}
    \begin{pmatrix}
     -v_{i-1} & -v_1v_{i-1}   & \cdots & -v_{i-2}v_{i-1} & \<v\>_{i-2}^2  &0 &\cdots &0
    \end{pmatrix}.
\end{align*}
Then we consider the change of variables 
\begin{align}\label{changeRotation}
	\begin{pmatrix}
         \wt\tau \\
         \wt\xi
    \end{pmatrix}
    =B^T\begin{pmatrix}
         \tau \\
         \xi
    \end{pmatrix},
\end{align}
which yields that $$\tau+v\cdot\xi=(\tau,\xi)\cdot(1,v)=B^T\begin{pmatrix}
         \wt\tau \\
         \wt\xi
    \end{pmatrix}\cdot B\begin{pmatrix}
         1 \\
         v
    \end{pmatrix}=(\wt\tau,\wt\xi)\cdot (\<v\>,0,\dots,0)=\wt\tau\<v\>,$$
and satisfies that, for any nonzero multi-index $\nu\in\N^{1+d}$, 
\begin{align}\label{boundtauxi}
	|(\tau,\xi)|\le C|(\wt\tau,\wt\xi)|,\quad |\pa_{\wt\tau,\wt\xi}^{\nu}(\tau,\xi)|\le C. 
\end{align}
Therefore, applying the rotation \eqref{changeRotation} to multiplier \eqref{m0Def}, it suffices to consider Fourier multiplier 
\begin{align}\label{m1a}
	m_f'(\tau)&:=\frac{\lam\<v\>^{-1}}{\lam\<v\>^{-1}+i\tau\<v\>}=\frac{\lam}{\lam+i\tau\<v\>^2}.
\end{align}
We then calculate the derivatives of $m'_f,m'_g$ which depends only on $\tau$:
\begin{align*}
	|\pa_{\tau}m_f'|&=\big|\frac{\lam\<v\>^2}{(\lam+i\tau\<v\>^2)^2}\big|
	\le \frac{1}{|\tau|}.
\end{align*}
Therefore, by the Marcinkiewicz multiplier theorem (which utilizes only the first-order derivative in every direction; see, for instance \cite[Corollary 6.2.5]{Grafakos2014}), we deduce that, for any $p\in(1,\infty)$, the $L^p_{t,x}$ multiplier norms of $m_f,m_g$ and $m'_f,m'_g$ (given by \eqref{m0Def} and \eqref{m1a}) satisfy 
\begin{align}\label{m1MultiNorm}
	\|m_f\|_{L^p_{t,x}\to L^p_{t,x}}&=\|m'_f\|_{L^p_{t,x}\to L^p_{t,x}}\le C_{p}.
\end{align}

\smallskip\noindent{\bf Step 3. Interpolation and $L^p$ estimate for $I_f$.}
For the term $I_f$ given by \eqref{eqFour1}, it follows from \eqref{es318} and \eqref{m1MultiNorm} that for any $q\in(1,\infty)$ and $m_0>0$, 
\begin{align*}
	\|\<D_v\>^{-m_0}I_f\|_{L^2_{t,x,v}}
	&\le C\lam^{\min\{\frac{m_0^-}{2},\frac{1}{2}\}}\|\De_jf\|_{L^2_{t,x,v}},\\
	\|I_f\|_{L^q_{t,x,v}}
	&\le C_q\|\De_jf\|_{L^q_{t,x,v}}. 
\end{align*}
where $m_0^-$ denotes an arbitrary number in $(0,m_0)$. 
Utilizing the complex interpolation $(H^{0,p_0}_{l_0},H^{0,p_1}_{l_1})_{[\th]}=H^{0,p}_l$ and $(H^{s_0,p_0}_{},H^{s_1,p_1}_{})_{[\th]}=H^{s^*,p^*}_{}$ for any $p_0,p_1\in(1,\infty)$ and $s_0,s_1\in\R$, where $\frac{l}{p}=\fr{l_0(1-\th)}{p_0}+\fr{l_1\th}{p_1}$, $\frac{1}{p}=\fr{1-\th}{p_0}+\fr\th{p_1}$, and $s^*=(1-\th)s_0+\th s_1$; see e.g., \cite[Theorem 6.4.5 and Theorem 5.5.3]{Bergh1976}, for any $m_0>0$, we have 
\begin{align*}
	\|\<D_v\>^{-m_0\iota}I_f\|_{L^p_{t,x,v}}
	&\le C_{q}\lam^{\iota\min\{\frac{m_0^-}{2},\frac{1}{2}\}}\|\De_jf\|_{L^p_{t,x,v}},
\end{align*}
where $\iota=\frac{1/p-1/q}{1/2-1/q}\in(0,\fr{2}{p})$. 
Since $\iota\to \frac{2}{p}$ as $q\to\infty$, by choosing a sufficiently large $q>2$, we deduce that for any $m\in(0,\fr{d}{p})$ and $p\ge 2$, 
\begin{align}\label{esIf}
	\|\<D_v\>^{-m}I_f\|_{L^p_{t,x,v}}
	&\le C_{p}\lam^{\th_{m,p}}\|\De_jf\|_{L^p_{t,x,v}},
\end{align}
where $\th_{m,p}=
\begin{cases}
	\min\{\frac{m^-}{2},\frac{1^-}{p}\},&\text{if }p>2,\\
	\min\{\frac{m}{2},\frac{1}{2}\},&\text{if }p=2. 
\end{cases}$

\smallskip 
\noindent{\bf Step 4. $L^p$ estimate for $I_g$.}
For the term $I_g$, we rewrite $g=(I-\De_v)^{M/2}\<D_v\>^{-M}g$ with any even integer $M\ge 0$. 
Using $\De_j=\chi_{\frac67 2^j\le|D_{t,x}|\le2^{j+1}}\De_j$, the convolution form of $P_0$, and integrating by parts, we obtain 
\begin{align}\label{Igpri1}\notag
	I_g&=\F_{t,x}^{-1}\int_{\R^d_{v_*}}(I-\De_{v_*})^{M/2}\Big\{\frac{\om^d\Psi_0(\om(v-v_*))\chi_{\frac672^j\le|(\tau,\xi)|\le2^{j+1}}}{\lam2^j\<v_*\>^{-1}+i(\tau+v_*\cdot\xi)}\Big\}\F_{t,x}(\<D_{v_*}\>^{-M}\De_jg(v_*))\,dv_*\\
	&\notag=\sum_{|\ga_1|+|\ga_2|\le M}C_{\ga_1,\ga_2}\F_{t,x}^{-1}\int_{\R^d_{v_*}}\om^{d+|\ga_1|}\pa^{\ga_1}_{v_*}\Psi_0(\om(v-v_*))
	\pa^{\ga_2}_{v_*}\Big\{\frac{1}{\lam2^j\<v_*\>^{-1}+i(\tau+v_*\cdot\xi)}\Big\}\\
	&\qquad\qquad\qquad\qquad\qquad\qquad\ \times\chi_{\frac672^j\le|(\tau,\xi)|\le2^{j+1}}\F_{t,x}(\<D_{v_*}\>^{-M}\De_jg(v_*))\,dv_*. 
\end{align}
To calculate the derivative $\pa^{\ga_2}_{v_*}$, we apply Faà di Bruno's formula or the iterated chain rule, i.e., for any multi-index $\nu\in\N^d$ with $|\nu|\ge 1$, there exist constants $C_{\nu_1,\dots,\nu_k}$ such that
\begin{align*}
	\pa^\nu_v f(g(v))=\sum_{k=1}^{|\nu|}f^{(k)}(v)
	\sum_{\substack{\nu_1+\cdots+\nu_k=\nu\\
	0\ne\nu_n\in\N^d~(\forall n)}}C_{\nu_1,\dots,\nu_k}\prod^k_{n=1}\pa^{\nu_n}g,  
\end{align*}
to deduce that, for any $|\ga_2|\ge 1$, 
\begin{align*}
	\pa^{\ga_2}_{v_*}\Big\{\frac{1}{\lam2^j\<v_*\>^{-1}+i(\tau+v_*\cdot\xi)}\Big\}
	&=
	\sum_{k=1}^{|\ga_2|}\frac{\sum_{\substack{\nu_1+\cdots+\nu_k=\nu\\
	0\ne\nu_n\in\N^d~(\forall n)}}C_{\nu_1,\dots,\nu_k}\prod^k_{n=1}\pa^{\nu_n}_{v_*}(\lam2^j\<v_*\>^{-1}+i(\tau+v_*\cdot\xi))}{(\lam2^j\<v_*\>^{-1}+i(\tau+v_*\cdot\xi))^{k+1}}. 
\end{align*}
Therefore, together with the zero-th derivative, it suffices to establish the multiplier norm of 
\begin{align*}
	m_g(\tau,\xi)&:=
	\frac{\om^{|\ga_1|}\prod^k_{n=1}\pa^{\nu_n}_{v_*}(\lam2^j\<v_*\>^{-1}+i(\tau+v_*\cdot\xi))}{(\lam2^j\<v_*\>^{-1}+i(\tau+v_*\cdot\xi))^{k+1}}\chi_{\frac672^j\le|(\tau,\xi)|\le2^{j+1}},
\end{align*}
with any $\nu_1+\cdots+\nu_k=\nu$, $\nu_n\ne 0$, and $0\le k\le |\ga_2|$. For brevity, we denote $\prod^0_{k=1}(\cdots)=1$. 
By using dilation $(\tau,\xi)\mapsto 2^{-j}(\tau,\xi)$ and rotations $(\wt\tau,\wt\xi)=R_vR(\tau,\xi)$ as in \eqref{changeRotation}, it suffices to find the multiplier norm of 
\begin{align}\label{mgpri}
	m_g'(\wt\tau,\wt\xi)&:=
	\frac{\om^{|\ga_1|}}{2^j}\fr{\prod^k_{n=1}(\lam\pa^{\nu_n}_{v_*}\<v_*\>^{-1}+\1_{|\nu_n|=1}\xi^{\nu_n})}{(\lam\<v_*\>^{-1}+i\wt\tau\<v_*\>)^{k+1}}\chi_{\frac67\le|(\wt\tau,\wt\xi)|\le2},
\end{align}
where $\xi=R^{-1}R_v^{-1}(\wt\tau,\wt\xi)$. 
Using estimate \eqref{boundtauxi}, the derivatives of multiplier $m'_g$ in \eqref{mgpri} satisfies 
\begin{align*}
	|\pa^{\nu}_{\wt\tau,\wt\xi}m_g'(\wt\tau,\wt\xi)|
	\le \fr{\om^{|\ga_1|}\<v_*\>}{\lam^{1+k}2^j|(\wt\tau,\wt\xi)^{\nu}|}
	\le \fr{\<v_*\>}{\lam^{1+|\ga_1|+k}2^j|(\wt\tau,\wt\xi)^{\nu}|}, 
\end{align*}
provided that $\lam^{-1}=\om\ge 1$. 
Therefore, by the Marcinkiewicz multiplier theorem, we deduce the $L^p_{t,x}$ multiplier norms: 
\begin{align}\label{mgMultiNorm}
	\|m_g\|_{L^p_{t,x}\to L^p_{t,x}}=\|m_g'\|_{L^p_{t,x}\to L^p_{t,x}}\le C_{p}\fr{\<v_*\>}{\lam^{1+|\ga_1|+k}2^j},
\end{align}
for any $p\in(1,\infty)$. 
Since $0\le k\le |\ga_2|$ and $|\ga_1|+|\ga_2|\le M$, substituting \eqref{mgMultiNorm} into \eqref{Igpri1} and using $|\pa^{\ga_1}_{v_*}\Psi_0(v_*)|\le C_N\<v_*\>^{-N}$ for any $N\ge 0$, we have 
\begin{align}\label{es441g}\notag
	\|I_g\|_{L^p_{t,x,v}}&\le \fr{C_p}{\lam^{1+M}2^j}
	\Big\|\int_{\R^d_{v_*}}\om^{d}\<\om(v-v_*)\>^{-N}
	\|\<v_*\>\<D_{v_*}\>^{-M}\De_jg(v_*)\|_{L^p_{t,x}}\,dv_*\Big\|_{L^p_v}\\
	&\le \fr{C_p}{\lam^{1+M}2^j}\|\<D_{v}\>^{-M}\<v\>\De_jg(v)\|_{L^p_{t,x,v}}, 
\end{align}
for any even number $M\ge 0$.
By interpolation, \eqref{es441g} also holds for any real $M\ge 0$. 

\smallskip 
Together with the estimate of $I_f$ in \eqref{esIf}, to obtain the optimal regularity, we let $\lam^{1+M}2^j=\lam^{-\th_{m,p}}$, i.e., $\lam=2^{-\fr{1}{\th_{m,p}+1+M}j}$ and $\lam^{\th_{m,p}}=2^{-\fr{\th_{m,p}}{\th_{m,p}+1+M}j}$ with any $m\in(0,\fr{d}{p})$, to deduce 
\begin{align*}
	\|\<D_v\>^{-m}I_f\|_{L^p_{t,x,v}}
	+\|I_g\|_{L^p_{t,x,v}}
	&\le C_{p}2^{-\fr{\th_{m,p}}{\th_{m,p}+1+M}j}
	\Big(\|\De_jf\|_{L^p_{t,x,v}}
	+\|\<D_{v}\>^{-M}\<v\>\De_jg(v)\|_{L^p_{t,x,v}}\Big).
\end{align*}
Substituting this into \eqref{eqFour1}, we obtain 
\begin{align*}
	\|\<D_v\>^{-m}P_0\De_jf\|_{L^p_{t,x,v}}
	&\le C_{p}2^{-\fr{\th_{m,p}}{\th_{m,p}+1+M}j}
	\Big(\|\De_jf\|_{L^p_{t,x,v}}
	+\|\<D_{v}\>^{-M}\<v\>\De_jg(v)\|_{L^p_{t,x,v}}\Big), 
\end{align*}
for any $p\in(1,\infty)$, which implies the first part of \eqref{esPkWjDejf}.

\smallskip For the term $(I-P_0)$, note that $1-\wh{\Psi_0}(\om^{-1}\eta)$, the kernel of $I-P_0$, is supported on $\{|\eta|\ge \om\}$, it's direct to check that 
\begin{align*}
	\Big|\pa^{\ga}_\eta\Big({(1-\wh{\Psi_0}(\om_0^{-1}2^{-\al j}\eta))\<\om\>^{m}}{\<\eta\>^{-m}}\Big)\Big|\le C_\ga|\eta|^{-|\ga|}.
\end{align*}
Thus, by the $L^p$ multiplier theorem and $\om=\lam^{-1}=2^{\fr{p}{\th_{m,p}+1+M}j}$, for any $p\in(1,\infty)$, 
\begin{align*}
	\|\<D_v\>^{-m}(I-P_0)\De_jf\|_{L^p_{t,x,v}}
	&\le C_{p}\<\om\>^{-m}\|\De_jf\|_{L^p_{t,x,v}}
	\le C_{p}2^{\fr{-p\th_{m,p}}{\th_{m,p}+1+M}j}\|\De_jf\|_{L^p_{t,x,v}}.
\end{align*}
This implies the second part of \eqref{esPkWjDejf} and concludes Theorem \ref{ThmNegativeHypo}. 
\end{proof}

\subsection{Corollary of the negative-order hypoellipticity}
As a corollary, let $A_0>0$ and let $f$ satisfy (this extension will be done in Section \ref{SecLPkinetic} and \eqref{eqkinetic12})
\begin{align}\label{es451}
	(\pa_t+v\cdot\na_x)f=
	\begin{cases}
		-A_0f+\Ga(\mu^{\frac{1}{2}},f),&\text{ in }t\in(T,\infty),\\
		\Ga(\mu^{\frac{1}{2}}+\phi_1,f)+\Ga(f,\mu^{\frac{1}{2}}+\phi_2),&\text{ in }t\in(0,T),\\
		A_0f-\Ga(\mu^{\frac{1}{2}},f),&\text{ in }t\in(-\infty,0).
	\end{cases}
\end{align}
Then we can obtain the regularity in $L^p_{t,x}$ for $\<D_v\>^{-m}f$. Note that we will use part of the regularity to derive the embedding on the left-hand side and the remaining regularity for the nonlinear terms on the right-hand side.

\begin{Coro}\label{CoroNegativeHypo}
	Assume the same conditions of Theorem \ref{ThmNegativeHypo}. 
Let $n\in\R$, $p\in[2,\infty)$, $m>0$, and 
\begin{align*}
	M=\frac{d}{2}-\frac{d}{p}+2s,\ \ \beta=\fr{\th_{m,p}}{\th_{m,p}+1+M},\ \  \beta'=\fr{\beta}{2}, \ \ \th_{m,p}=
\begin{cases}
	\min\{\frac{m^-}{2},\frac{1^-}{p}\},&\text{if }p>2,\\
	\min\{\frac{m}{2},\frac{1}{2}\},&\text{if }p=2, 
\end{cases}
\end{align*}
for any $m^-\in(0,m)$. 
Then, if $f$ solves \eqref{es451}, there exists a constant $r^*\in[1,\infty)$ such that 
\begin{align}\label{esNegativeforf}\notag
	&\|\<D_{t,x}\>^{\beta'}\<D_v\>^{-m}\<v\>^{-n}f\|_{L^p_{t}(\R)L^p_{x,v}}
	\\
	&\notag\quad\le C\big(1+\|\phi_1\|_{L^{\infty}_{t}([0,T])L^{\infty}_{x}L^2_v}+\|\<v\>^{1+\max\{\ga+2s,0\}-n}\phi_2\|_{L^{r^*}_{t}([0,T])L^{r^*}_{x}L^2_v}\big)\\
	&\notag\qquad\quad\times\|\<v\>^{1+\max\{\ga+2s,0\}-n}f\|_{L^p_{t}([0,T])L^p_{x}L^2_v}\\
	&\qquad
	+C(1+A_0)\|\<v\>^{1-n}f\|_{L^p_{t}((-\infty,0)\cup(T,\infty))L^p_{x}L^2_v}, 
\end{align}
where $r^*=\fr{d+1}{\beta'}$. 
(For exmaple, when $p=2$, we have $r^*=\frac{2(d+1)(m^-+2(1+M))}{m^-}$.)
\end{Coro}
\begin{proof}
Within this proof, let $\R_t$ be the underlying time interval, and let $\wt{\De}^{t,x}_j$ be the operator defined in \eqref{SlightlyLarger} so that $\De^{t,x}_j=\De^{t,x}_j\wt{\De}^{t,x}_j$. 
First, for any $n\in\R$, using the Embedding Theorem \ref{BesovembeddThm}, we convert the Triebel-Lizorkin norm into the Besov norm with time-spatial pseudo-differential operator $\De^{t,x}_j$ as 
\begin{align}\label{es777z}
	\|\<D_{t,x}\>^{\beta'}\<D_v\>^{-m}\<v\>^{-n}\phi\|_{L^p_t(\R)L^p_{x,v}}
	& \le C\Big(\sum_{j\ge 0}2^{2\beta'j}\|\De^{t,x}_j\<D_v\>^{-m}\<v\>^{-n}\phi\|_{L^{p}_{t}(\R)L^{p}_{x,v}}^2\Big)^{\frac{1}{2}},
\end{align}
Thus, choosing $
2\beta'\le \beta$, applying hypoellipticity Theorem \ref{ThmNegativeHypo} to equation $(\pa_t+v\cdot\na_x)f=G$, where $G$ denotes the right-hand side of \eqref{es451}, we have 
\begin{align}\label{es429}\notag
	&\Big(\sum_{j\ge 0}2^{2\beta'j}\|\<D_v\>^{-m}\De^{t,x}_j\<v\>^{-n}f\|_{L^p_{t,x,v}}^2\Big)^{\frac{1}{2}}\\
	&\notag\le C_p\Big(\sum_{j\ge 0}2^{(2\beta'-2\beta)j}\Big\{\|\De^{t,x}_j\<v\>^{-n}f\|_{L^p_{t,x,v}}^2
  +\|\De^{t,x}_j\<D_v\>^{-M}\<v\>^{1-n}G\|_{L^p_{t,x,v}}^2\Big\}\Big)^{\fr12}\\
	&\quad\notag
	+ C_p\Big(\sum_{j\ge 0}2^{(2\beta'-2\beta)j}\|\De^{t,x}_j\<v\>^{-n}f\|_{L^p_{t,x,v}}^2\Big)^{\fr12}\\
	&\le C_{p}\|\<v\>^{-n}f\|_{L^p_{t,x,v}}+C_{p}\Big(\sum_{j\ge0}2^{(2\beta'-2\beta)j}\|\<D_v\>^{-M}\<v\>^{1-n}\De^{t,x}_jG\|_{L^p_{t,x,v}}^2\Big)^{\frac{1}{2}},
\end{align}
for any $n\in\R$, where, since $\beta'<\beta$, we have used the trivial embedding $L^p\hookrightarrow B^{\beta'-\beta}_{2}[L^{p}]$ for the $f$ term by the embedding Theorem \ref{BesovembeddThm}. 
For the $G$ term, we utilize the regularity factor $2^{(2\beta'-2\beta)j}\le 2^{-\beta'j}$ (since $\beta'=\fr{\beta}{2}$) as follows. Notice $M-2s\ge\frac{d}{2}-\frac{d}{p}\ge 0$ (for the case $p=2$, no embedding is required).
Then, by embedding Theorem \ref{BesovembeddThm}, i.e., $H^{-2s}_v=B^{-2s}_{2}[L^{2}]\hookrightarrow F^{-M}_{2}[L^{p}]=H^{-M,p}$ and $F^{0}_{2}[L^{r_0}]=L^{r_0}_{t,x}\hookrightarrow B^{-\beta'}_{2}[L^{p}]$ provided that $\beta'=(d+1)(\frac{1}{r_0}-\frac{1}{p})>0$ or $r_0=p$, we obtain, for any suitable $g$, that 
\begin{align}\label{es457}\notag
	\Big(\sum_{j\ge 0}2^{-2\beta'j}\|\<D_v\>^{-M}\De^{t,x}_jg\|_{L^p_{t,x,v}}^2\Big)^{\frac{1}{2}}
	&\le C\Big(\sum_{j\ge 0}2^{-2\beta'j}\|\<D_v\>^{-2s}\De^{t,x}_jg\|_{L^p_{t,x}L^2_v}^2\Big)^{\frac{1}{2}}\\
	&\le C\|\<D_v\>^{-2s}g\|_{L^{r_0}_{t,x}L^2_v}. 
\end{align}
Splitting $G$ into time intervals $(-\infty,0)\cup[0,T]\cup(T,\infty)$, and applying \eqref{es457} to the collision $G$ with $(d+1)(\frac{1}{r_0}-\frac{1}{p})=\beta'>0$ for $\phi_2$ term and and $r_0=p$ for other terms, we obtain 
\begin{align}\label{es458}\notag
	&\Big(\sum_{j\ge 0}2^{-2\beta'j}\|\<D_v\>^{-M}\<v\>^{1-n}\De^{t,x}_jG\|_{L^p_{t,x,v}}^2\Big)^{\frac{1}{2}}\\
	&\notag\le C\big\|\<D_v\>^{-2s}\<v\>^{1-n}\Ga(f,\phi_2)\1_{t\in[0,T]}\big)\big\|_{L^{r_0}_{t,x}L^2_v}\\
	&\notag\
	+C\big\|\<D_v\>^{-2s}\<v\>^{1-n}
	\big((-A_0f+\Ga(\mu^{\frac{1}{2}},f))\1_{t\in(T,\infty)}
	\\&\qquad \  +(\Ga(\mu^{\frac{1}{2}}+\phi_1,f)+\Ga(f,\mu^{\frac{1}{2}}))\1_{t\in[0,T]}
	+(A_0f-\Ga(\mu^{\frac{1}{2}},f))\1_{t\in(-\infty,0)}\big)\big\|_{L^p_{t,x}L^2_v}.
\end{align}
provided that $(d+1)(\frac{1}{r_0}-\frac{1}{p})=\beta'>0$. 
Finally, utilizing the negative Bessel potential $\<D_v\>^{-2s}$ and the collision estimate \eqref{64eq}, we continue \eqref{es458} as 
\begin{align}\label{esg457}\notag
	&\Big(\sum_{j\ge 0}2^{-2\beta'j}\|\<D_v\>^{-M}\<v\>^{1-n}\De^{t,x}_jG\|_{L^p_{t,x,v}}^2\Big)^{\frac{1}{2}}\\
	&\notag\le 
	C\Big\|\|\<v\>^{1+\ga+2s-n}f\|_{L^2_v}\Big\|_{L^{p}_{t}([0,T])L^{p}_{x}}+C\Big\|\|f\|_{L^2_v}\|\<v\>^{1+\ga+2s-n}\phi_2\|_{L^2_v}\Big\|_{L^{r_0}_{t}([0,T])L^{r_0}_{x}}\\
	&\notag\quad
	+C\|\<v\>^{1+\max\{\ga+2s,0\}-n}f\|_{L^p_{t}(\R)L^p_{x}L^2_v}
	+C(1+A_0)\|\<v\>^{1-n}f\|_{L^p_{t}((-\infty,0)\cup(T,\infty))L^p_{x}L^2_v}\\
	&\notag\le C\big(1+\|\phi_1\|_{L^{\infty}_{t}([0,T])L^{\infty}_{x}L^2_v}+\|\<v\>^{1+\ga+2s-n}\phi_2\|_{L^{r^*}_{t}([0,T])L^{r^*}_{x}L^2_v}\big)\|\<v\>^{1+\max\{\ga+2s,0\}-n}f\|_{L^p_{t}([0,T])L^p_{x}L^2_v}\\
	&\quad
	+C(1+A_0)\|\<v\>^{1-n}f\|_{L^p_{t}((-\infty,0)\cup(T,\infty))L^p_{x}L^2_v},
\end{align}
where $r^*\in(r_0,\infty)$ is given by $\frac{d+1}{r^*}=(d+1)(\frac{1}{r_0}-\frac{1}{p})=\beta'>0$. 
Combining the estimates \eqref{es777z}, \eqref{es429} and \eqref{esg457} with $\beta'=\fr{\beta}{2}$, we obtain \eqref{esNegativeforf} and conclude Corollary \ref{CoroNegativeHypo}. 
\end{proof}


\section{Upper bound of the collision operator}\label{SecUpper}
In this Section, we address the upper bound of the collision operator, by truncating the angular cross-section $b(\cos\th)$ and kinetic relative velocity kernel $|v-v_*|^\ga$ as in \eqref{Bbarc} and \eqref{QcDef}.

\subsection{Main upper-bound estimate}
We begin by stating the main result of this section.
\begin{Thm}[Upper bound]
	\label{ThmUpper}
	Let $0\le s_1\le s$ and assume $\ga+2s-s_1>-\frac{d}{2}$. Then 
\begin{align}\label{QfghUpper}
	|(Q(f,g),h)_{L^2_v}|
	&\le C\|\<D_v\>^{-s_1}\<v\>^{4|\ga|+4+\frac{d+1}{2}}f\|_{L^2_v}\|g\|_{L^2_D}\|h\|_{L^2_D}.
\end{align}
Consequently, together with commutator estimate \eqref{CommGaAndQ}, we have, 
for any $n,N\in\R$, 
\begin{align}\label{UpperGaNoSum}\notag
	|(\Ga(f,g),h)_{L^2_v}|
	&\le 
	C\|\<D_v\>^{-s_1}\<v\>^{-N}f\|_{L^2_v}\|g\|_{L^2_D}\|h\|_{L^2_D}\\
	&\quad+C\|f\|_{L^2_v}\|\<v\>^{\frac{\ga+2s}{2}+n}g\|_{L^2_v}\|\<v\>^{-n}h\|_{L^2_D},
\end{align}
and, denoting operators $P_k,R_r$ as in \eqref{DejQkDef}, 
\begin{align}\label{UpperGaes}\notag
	&|(\Ga(f,P_kR_rg),P_kR_rh)_{L^2_v}|\\
	&\ \ \notag\le C\|\<D_v\>^{-s_1}\<v\>^{-N}f\|_{L^2_v}\|P_kR_rg\|_{L^2_D}\|P_kR_rh\|_{L^2_D}\ \ \emph{(leading order)}\\
	&\quad
	+C\|f\|_{L^2_v}\|\<v\>^{\frac{\ga+2s}{2}}P_kR_rg\|_{L^2_v}\|P_kR_rh\|_{L^2_D} \ \ \emph{(lower-order commutator)}.
\end{align}
\end{Thm}
Here, we retain part of the dissipation norm because, in the main energy estimate \eqref{es21}, we cannot afford to lose any order of the weight $\<v\>$.
In order to prove Theorem \ref{ThmUpper}, by splitting $Q=Q_c+Q_{\bar{c}}$ as in \eqref{Bbarc}, in the following Subsections (see \eqref{QbarcTemp} and \eqref{QcTemp}), we will obtain 
\begin{Lem}\label{LemUpper}
Let $\ga,s,s_1,s_2,s_3$ be such that $\ga>-d$ and 
\begin{align*}
	\begin{cases}
s_2\ge 0, \ \ 
0\le s_1\le s_3, \ \ 
s_2+s_3=2s, \ \ \ga+2s-s_1>-\frac{d}{2},\\
\ga-s_1+s_2+1>-\frac{d}{2},\quad
\ga-s_1+s_2+s_3>-\frac{d}{2}.
	\end{cases}
\end{align*}
Then 
\begin{align}\label{Qcesfinal}
	|(Q_c(f,g),h)_{L^2_v}|
	\le C\|\<D_v\>^{-s_1}\<v\>^{-\ga}f\|_{L^2_v}\|\<D_v\>^{s_2}\<v\>^{\ga}g\|_{L^2_v}\|\<D_v\>^{s_3}h\|_{L^2_v}, 
\end{align}
and
\begin{align}\label{Qbarfhg}\notag
	\big|\big(Q_{\bar{c}}(f,g),h\big)_{L^2_v}\big|
	&\le C\|\<D_v\>^{-s}\<v\>^{\frac{d+3}{2}}f\|_{L^2_v}
	\|\<D_v\>^{s_2}\<v\>^{\ga}g(v)\|_{L^2_v}
	\|\<D_v\>^{s_3}h\|_{L^2_v}\\
	&\ \notag+C\|\<D_v\>^{-2}\<v\>^{4|\ga|+4+\frac{d+1}{2}}f\|_{L^2_v}\min\Big\{\|g\|_{L^2_D}\|h\|_{L^2_D},\\
	&\qquad\qquad \|\<v\>^{\frac{\ga}{2}}g\|_{L^2_D}\|\<v\>^{-\frac{\ga}{2}}h\|_{L^2_D}+\|\<v\>^{\frac{\ga}{2}}g\|_{L^2_D}\|\<D_v\>^sh\|_{L^2_v}
	\Big\}.
\end{align}
\end{Lem}
\begin{proof}[Proof of Theorem \ref{ThmUpper}]
To apply Lemma \ref{LemUpper}, we let $s_2=s_3=s$ and choose any $s_1>0$ satisfying $0\le s_1\le s$. By splitting 
	\begin{align*}
		(Q(f,g),h)_{L^2_v}
		&
		=(Q(f,\<v\>^{-\frac{\ga}{2}}g),\<v\>^{\frac{\ga}{2}}h)_{L^2_v}
		+\big((Q(f,g),h)_{L^2_v}-(Q(f,\<v\>^{-\frac{\ga}{2}}g),\<v\>^{\frac{\ga}{2}}h)_{L^2_v}\big).
	\end{align*}
	and applying \eqref{Qcesfinal} and \eqref{Qbarfhg} to the first term, and commutator estimate \eqref{CommvlTemp} to the second term, 
	we deduce \eqref{QfghUpper}.
For the estimate of $\Ga$, by \eqref{QfghUpper}, \eqref{equiv}, commutator estimate \eqref{CommGaAndQ}, and decomposition
	\begin{align*}
		\Ga(f,g)=Q(\mu^{\frac{1}{2}}f,g)+\big(\Ga(f,g)-Q(\mu^{\frac{1}{2}}f,g)\big),
	\end{align*} 
	we can obtain \eqref{UpperGaNoSum} and \eqref{UpperGaes}, 
for any $N\ge 0$.
 This concludes Theorem \ref{ThmUpper}. 
\end{proof}

\subsection{Large relative velocity: \texorpdfstring{$Q_{\bar{c}}$}{Q barc}}
For the part of the large relative velocity $Q_{\bar{c}}$ in \eqref{QcDef}, we need to calculate in the frequency space. Integrating over $v$ and utilizing the identities from \eqref{deltafunction}, 
\begin{align}\notag\label{Qbarc101}
	&\big(Q_{\bar{c}}(f,g),h\big)_{L^2_v}=\int_{v,v_*,\si}b(\cos\th)\Th_{\bar{c}}(v-v_*)f(v_*)g(v)\big(\ol{h(v')}-\ol{h(v)}\big)\\
	&\notag=\int_{v,v_*,\si,\zeta,u,\eta}
	b(\cos\th)\Th_{\bar{c}}(u-v_*)f(v_*)g(u)\ol{\wh{h}(\eta)}e^{2\pi i\zeta\cdot(v-u)}\Big(e^{2\pi i\eta\cdot v'}-e^{2\pi i\eta\cdot v}\Big)
	\\
	&=\int_{v_*,\si,u,\eta}
	b\big(\frac{\eta}{|\eta|}\cdot\si\big)\Th_{\bar{c}}(u-v_*)f(v_*)g(u)\ol{\wh{h}(\eta)}\Big(e^{-2\pi i\eta^+\cdot u}e^{-2\pi i\eta^-\cdot v_*}-e^{-2\pi i\eta\cdot u}\Big).  
\end{align}
Using $\<D_{v_*}\>^{2}=(I-\De_{v_*})^{}$, we have 
\begin{align*}
	\big(Q_{\bar{c}}(f,g),h\big)_{L^2_v}
	&=\int_{v_*,\si,u,\eta}
	b\big(\frac{\eta}{|\eta|}\cdot\si\big)\<v_*\>^{M}\<D_{v_*}\>^{-2}f(v_*)\<v_*\>^{-M}g(u)\ol{\wh{h}(\eta)}\\
	&\qquad\times(I-\De_{v_*})^{}\Big\{\Th_{\bar{c}}(u-v_*)\Big(e^{-2\pi i\eta^+\cdot u}e^{-2\pi i\eta^-\cdot v_*}-e^{-2\pi i\eta\cdot u}\Big)\Big\},  
\end{align*}
for any $M\in\R$, where we can calculate 
\begin{align}\notag
	&(I-\De_{v_*})^{}\Big\{\Th_{\bar{c}}(u-v_*)\big(e^{-2\pi i\eta^+\cdot u}e^{-2\pi i\eta^-\cdot v_*}-e^{-2\pi i\eta\cdot u}\big)\Big\}\\
	&\notag\ \ =
	\sum_{|\al_2|=2}C_{\al_2}\Th_{\bar{c}}(u-v_*)
	(\eta^-)^{\al_2}\big(e^{-2\pi i\eta^+\cdot u}e^{-2\pi i\eta^-\cdot v_*}\big)\\
	&\notag\quad+
	\sum_{|\al_1|=|\al_2|=1}C_{\al_1,\al_2}\pa^{\al_1}_{v_*}\Th_{\bar{c}}(u-v_*)
	(\eta^-)^{\al_2}\big(e^{-2\pi i\eta^+\cdot u}e^{-2\pi i\eta^-\cdot v_*}\big)\\
	&\quad+\sum_{|\al_1|\le 2}C_{\al_1}\pa^{\al_1}_{v_*}\Th_{\bar{c}}(u-v_*)\big(e^{-2\pi i\eta^+\cdot u}e^{-2\pi i\eta^-\cdot v_*}-e^{-2\pi i\eta\cdot u}\big),\notag
\end{align}
for some constants $C_{\al_2},C_{\al_1,\al_2},C_{\al_1}>0$ with multi-indices $\al_1,\al_2$.
Thus, integrating over $v_*$ yields 
\begin{align*}\notag
	\big(Q_{\bar{c}}(f,g),h\big)_{L^2_v}
	&=\int_{v_*,\si,u,\eta}
	b\big(\frac{\eta}{|\eta|}\cdot\si\big)\<v_*\>^{M}\<D_{v_*}\>^{-2}f(v_*)\<v_*\>^{-M}g(u)\ol{\wh{h}(\eta)}\\
	&\notag\qquad\times(I-\De_{v_*})^{}\Big\{\Th_{\bar{c}}(u-v_*)\Big(e^{-2\pi i\eta^+\cdot u}e^{-2\pi i\eta^-\cdot v_*}-e^{-2\pi i\eta\cdot u}\Big)\Big\}\\
	&=:Q_{\bar{c},1}+Q_{\bar{c},2}+Q_{\bar{c},3},
\end{align*}
where, by translation $v_*\mapsto v_*-u$ and integrating over $u$, we split (note that $\Th_{\bar{c}}(v_*)=\Th_{\bar{c}}(-v_*)$)
\begin{align}\label{Qbarc123Def}\notag
	Q_{\bar{c},1}&:=\displaystyle\sum_{|\al_2|=2}C_{\al_2}\int_{v_*,\si,\eta}
	b
	\wh{f_2}(\eta_*)\wh{G^{0}}(\eta^--\eta_*,\eta-\eta_*)
	\ol{\wh{h}(\eta)}(\eta^-)^{\al_2},\\
	Q_{\bar{c},2}&:=\displaystyle\sum_{|\al_1|=|\al_2|=1}C_{\al_1,\al_2}\int_{v_*,\si,\eta}
	b
	\wh{f_2}(\eta_*)\wh{G^{\al_1}}(\eta^--\eta_*,\eta-\eta_*)
	\ol{\wh{h}(\eta)}(\eta^-)^{\al_2},\\
	Q_{\bar{c},3}&:=\displaystyle\sum_{|\al_1|\le 2}C_{\al_1}\int_{v_*,\si,\eta}
	b
	\wh{f_2}(\eta_*)\Big(\wh{G^{\al_1}}(\eta^--\eta_*,\eta-\eta_*)-\wh{G^{\al_1}}(-\eta_*,\eta)\Big)
	\ol{\wh{h}(\eta)}.  \notag
\end{align}
and, for any $M>0$, we denote 
\begin{align}
	\label{whGDef}\notag
	f_2(v_*)&:=\<v_*\>^{M}\<D_{v_*}\>^{-2}f(v_*),\\
	G^{\al_1}(v_*,u)&:=\<v_*+u\>^{-M}\pa^{\al_1}_{v_*}\Th_{\bar{c}}(v_*)g(u),\quad 
	\wh{G^{\al_1}}(\eta_*,\eta):=\F_{v_*,u}G^{\al_1}(\eta_*,\eta),
\end{align}
where $\wh{G^{\al_1}}(\eta_*,\eta)$ has a good decay with respect to the first argument $\eta_*$. In fact, since $\Th_{\bar{c}}(v_*)$ has no singularity at $v_*=0$, we have $\<v_*+u\>^{-M}|\pa^{\al_1}\Th_{\bar{c}}(v_*)|\le C\<v_*+u\>^{-M+|\ga|}\<u\>^{\ga}$ and hence, 
\begin{align}\label{eswhG0}\notag
	\|\<\eta_*\>^{N}\<\eta\>^{s'}\wh{G^{\al_1}}(\eta_*,\eta)\|_{L^2_{\eta_*,\eta}}
	&\le C_N\|\<u\>^{\ga}\<D_u\>^{s'}g(u)\|_{L^2_u}\|\<v_*\>^{-M+|\ga|}\|_{L^\infty_uL^2_{v_*}}\\
	&\le C_N\|\<v\>^{\ga}\<D_v\>^{s'}g\|_{L^2_v}, 
\end{align}
for any $|\al_1|\ge 0$, $N\ge 0$ and $s'\in\R$, 
provided that $M>\frac{d+1}{2}+|\ga|$.

\subsubsection{The term $Q_{\bar{c},1}$}
For $Q_{\bar{c},1}$, we have sufficient angular regular factor $|\eta^-|^2$. Using H\"older's inequality, applying translation $\eta_*\mapsto \eta^--\eta_*$
 to the $G^0$ term, we have, for any $s'\in\R$,  
\begin{align}\label{Qbarc1a}\notag
	&Q_{\bar{c},1}
	\le C\Big(\int_{\si,\eta,\eta_*}
	b\big(\frac{\eta}{|\eta|}\cdot\si\big)
	|\wh{f_2}(\eta_*)|^2|\ol{\wh{h}(\eta)}|^2
	\<\eta^--\eta_*\>^{-N}|\eta^-|^{2}\<\eta_*\>^{2}\<\eta\>^{2s'}\Big)^{\frac{1}{2}}
	\notag\\
	&\quad\times
	\Big(\int_{\si,\eta,\eta_*}
	b\big(\frac{\eta}{|\eta|}\cdot\si\big)
	|\<\eta_*\>^{N}\wh{G^0}(\eta_*,\eta-\eta_*)|^2
	\<\eta_*\>^{-N}|\eta^-|^2\<\eta^--\eta_*\>^{-2}\<\eta\>^{-2s'}\Big)^{\frac{1}{2}}, 
\end{align}
By $|\eta^-|^2\le C\<\eta_*\>^2\<\eta^--\eta_*\>^{2}$, $|\eta^-|=|\eta|\sin\frac{\th}{2}$, angular integral \eqref{angulareta}, we compute  
\begin{align*}\notag
	\Big|\int_{\si}b\big(\frac{\eta}{|\eta|}\cdot\si\big)\<\eta^--\eta_*\>^{-N}|\eta^-|^2\<\eta_*\>^{2}\Big|
	&\notag
	\le\int_{\si}\1_{|\eta^-|\le\<\eta_*\>}b\<\eta_*\>^{2}|\eta^-|^2+\int_{\si}\1_{|\eta^-|>\<\eta_*\>}b\<\eta_*\>^{4}\\
	&\le C|\eta|^{2s}\<\eta_*\>^{4-2s}, 
\end{align*}
and, by $|\eta^-|^2\le C\<\eta_*\>^2\<\eta^--\eta_*\>^{2}$, 
\begin{align*}\notag
	\Big|\int_{\si}
	b\big(\frac{\eta}{|\eta|}\cdot\si\big)
	\<\eta_*\>^{-N}|\eta^-|^2\<\eta^--\eta_*\>^{-2}\Big|
	&\notag\le\int_{\si}\1_{|\eta^-|\le\<\eta_*\>}b\<\eta_*\>^{-N}|\eta^-|^2+\int_{\si}\1_{|\eta^-|>\<\eta_*\>}b\<\eta_*\>^{-N+2}\\
	&\le C|\eta|^{2s}\<\eta_*\>^{-N+2-2s}.
\end{align*}
Substituting these into the corresponding $(f_2,h)$ and $G^0$ terms in \eqref{Qbarc1a}, recalling that $f_2(v_*)=\<v_*\>^{M}\<D_{v_*}\>^{-2}f(v_*)$, and using \eqref{eswhG0}, for any $s'\in\R$, we deduce 
\begin{align}\label{Qbarc1es}\notag
	|Q_{\bar{c},1}|
	&\le C\Big(\int_{\eta,\eta_*}\<\eta\>^{2s+2s'}
	|\wh{f_2}(\eta_*)|^2|\ol{\wh{h}(\eta)}|^2\<\eta_*\>^{4-2s}\Big)^{\frac{1}{2}}
	\Big(\int_{\eta,\eta_*}\<\eta\>^{2s-2s'}|\<\eta_*\>^{N}\wh{G^0}(\eta_*,\eta)|^2\Big)^{\frac{1}{2}}\\
	&\le C\|\<D_v\>^{-s}\<v\>^{M}f\|_{L^2_v}\|\<v\>^{\ga}\<D_v\>^{s-s'}g\|_{L^2_v}\|\<D_v\>^{s+s'}h\|_{L^2_v}. 
\end{align}

\subsubsection{The term $Q_{\bar{c},2}$}
For $Q_{\bar{c},2}$ in \eqref{Qbarc123Def}, we need to deduce one more regular factor $|\eta^-|$ from $\wh{G^{\al_1}}$ term. That is, by \eqref{whGDef}, we further decompose  
\begin{align}\label{Qbarc2a}\notag
	Q_{\bar{c},2}
	&\notag=\sum_{|\al_1|=|\al_2|=1}C_{\al_1,\al_2}\int_{v_*,\si,u,\eta}
	b\big(\frac{\eta}{|\eta|}\cdot\si\big)
	\wh{f_2}(\eta_*)\wh{G^{\al_1}}(\eta^--\eta_*,\eta-\eta_*)
	\ol{\wh{h}(\eta)}(\eta^-)^{\al_2}\\
	&=:\sum_{|\al_1|=|\al_2|=1}C_{\al_1,\al_2}\big\{Q_{\bar{c},2}^{s,1}+Q_{\bar{c},2}^{s,2}+Q_{\bar{c},2}^{l}\big\},
\end{align}
where we split 
\begin{align*}
	Q_{\bar{c},2}^{s,1}
	&:=\int_{\si,\eta,\eta_*}
	\1_{|\eta^-|\le\frac{1}{2}\<\eta_*\>^{}}
	b\big(\frac{\eta}{|\eta|}\cdot\si\big)
	\wh{f_2}(\eta_*)\ol{\wh{h}(\eta)}
	\big(\wh{G^{\al_1}}(\eta^--\eta_*,\eta-\eta_*)-\wh{G^{\al_1}}(-\eta_*,\eta-\eta_*)\big)
	(\eta^-)^{\al_2},\\
	Q_{\bar{c},2}^{s,2}
	&:=\int_{\si,\eta,\eta_*}
	\1_{|\eta^-|\le\frac{1}{2}\<\eta_*\>^{}}
	b\big(\frac{\eta}{|\eta|}\cdot\si\big)
	\wh{f_2}(\eta_*)\ol{\wh{h}(\eta)}
	\wh{G^{\al_1}}(-\eta_*,\eta-\eta_*)(\eta^-)^{\al_2},\\
	Q_{\bar{c},2}^{l}
	&:=\int_{\si,\eta,\eta_*}
	\1_{|\eta^-|>\frac{1}{2}\<\eta_*\>^{}}
	b\big(\frac{\eta}{|\eta|}\cdot\si\big)
	\wh{f_2}(\eta_*)\ol{\wh{h}(\eta)}
	\wh{G^{\al_1}}(\eta^--\eta_*,\eta-\eta_*)(\eta^-)^{\al_2}.
\end{align*} 
For $Q_{\bar{c},2}^{s,1}$, by Taylor expansion, we have 
\begin{align*}
	\wh{G^{\al_1}}(\eta^--\eta_*,\eta-\eta_*)-\wh{G^{\al_1}}(-\eta_*,\eta-\eta_*)
	=
	\eta^-\cdot\int^1_0\na_{\eta_*}\wh{G^{\al_1}}(\de\eta^--\eta_*,\eta-\eta_*)\,d\de, 
\end{align*}
where $\na_{\eta_*}$ is taken with respect to the first argument of $\wh{G^{\al_1}}$. Thus, 
\begin{align*}
	Q_{\bar{c},2}^{s,1}&=\int_{\si,\eta,\eta_*}
	\1_{|\eta^-|\le\frac{1}{2}\<\eta_*\>^{}}
	b\big(\frac{\eta}{|\eta|}\cdot\si\big)
	\wh{f_2}(\eta_*)\ol{\wh{h}(\eta)}
	\eta^-\cdot\int^1_0\na_{\eta_*}\wh{G^{\al_1}}(\de\eta^--\eta_*,\eta-\eta_*)\,d\de~
	(\eta^-)^{\al_2}.
\end{align*}
Note that $|\eta^-|\le\frac{1}{2}\<\eta_*\>$ implies $\<\de\eta^--\eta_*\>^2
		\ge 1+\frac{|\eta_*|^2}{2}-|\eta^-|^2\ge \frac{1}{4}\<\eta_*\>^2$. 
Therefore, by using H\"older's inequality and $|\al_1|=|\al_2|=1$, we have 
\begin{align*}
	&|Q_{\bar{c},2}^{s,1}|
	\le C\Big(\int_{\si,\eta,\eta_*,\de}\1_{|\eta^-|\le\frac{1}{2}\<\eta_*\>^{}}
	b
	|\wh{f_2}(\eta_*)\ol{\wh{h}(\eta)}|^2
	\<\eta_*\>^{-N}|\eta^-|^2\<\eta\>^{2s'}\Big)^{\frac{1}{2}}\\
	&\  \times\Big(\int_{\si,\eta,\eta_*,\de}\1_{|\eta^-|\le\frac{1}{2}\<\eta_*\>^{}}
	b
	|\na_{\eta_*}\wh{G^{\al_1}}(\de\eta^--\eta_*,\eta-\eta_*)|^2\<\de\eta^--\eta_*\>^{N}
	|\eta^-|^2\<\eta\>^{-2s'}\Big)^{\frac{1}{2}},
\end{align*}
for any $s'\in\R$. For the $\na_{\eta_*}\wh{G^{\al_1}}$ term, we take the change of variables $\eta_*\mapsto \de\eta^--\eta_*$ as in \eqref{xichangetoplustau}, where $|\eta^-|\le\frac{1}{2}\<\eta_*-\de\eta^-\>$ implies $4|\eta^-|^2
		\le 1+2|\eta_*|^2+2|\eta^-|^2$, and hence, $|\eta^-|^2\le \<\eta_*\>^2$. 
Thus, by angular integral \eqref{angulareta}, 
\begin{align}\label{Qbarc1s1}\notag
	|Q_{\bar{c},2}^{s,1}|
	&\le C\Big(\int_{\eta,\eta_*}\<\eta\>^{2s+2s'}\<\eta_*\>^{-N+2-2s}|\wh{f_2}(\eta_*)\ol{\wh{h}(\eta)}|^2\Big)^{\frac{1}{2}}
	\Big(\int_{\eta,\eta_*}\<\eta\>^{2s-2s'}\<\eta_*\>^{N+2-2s}|\na_{\eta_*}\wh{G^{\al_1}}(\eta_*,\eta)|^2\Big)^{\frac{1}{2}}\\
	&\le C\|\<D_v\>^{-N-s}\<v\>^{M}f\|_{L^2_v}\|\<v\>^{\ga}\<D_v\>^{s-s'}g\|_{L^2_v}\|\<D_v\>^{s+s'}h\|_{L^2_v}, 
\end{align}
where, as in \eqref{eswhG0}, by \eqref{equiv}, for $M>\frac{d+2}{2}$, we have calculated that for any $N\ge 0$ and $|\al_1|=1$, 
\begin{align*}
	\|\<\eta_*\>^{N}\<\eta\>^{s-s'}\na_{\eta_*}G^{\al_1}(\eta_*,\eta)\|_{L^2_{\eta_*,\eta}}
	\notag
	&\le \Big\|\<D_{u}\>^{s-s'}\<D_{v_*}\>^{N}\Big(v_*\<v_*\>^{-M}\pa^{\al_1}_{v_*}\Th_{\bar{c}}(u-v_*)g(u)\Big)\Big\|_{L^2_{v_*,u}}
	\\
	&\le C_N\|\<D_{u}\>^{s-s'}\<u\>^{\ga}g(u)\|_{L^2_u}.
\end{align*}

For $Q_{\bar{c},2}^{s,2}$, since $|\al_2|=1$, we can further decompose $\eta^-$
\begin{align}\label{decompositionofetaminus}
	\eta^-=\frac{\eta-|\eta|\si}{2}=\frac{|\eta|}{2}\Big(\big(\frac{\eta}{|\eta|}\cdot\si\big)\frac{\eta}{|\eta|}-\si\Big)+\Big(1-\big(\frac{\eta}{|\eta|}\cdot\si\big)\Big)\frac{\eta}{2}. 
\end{align}
Using decomposition $\si=\cos\th\,\mathbf{j}+\sin\th\,\omega$ in \eqref{decomsitoj} (where $\mathbf{j}=\frac{\eta}{|\eta|}$ and $\om\cdot\mathbf{j}=0$) and the symmetry of $\omega$ in \eqref{symmetaomega}, since the other integrands in $Q_{\bar{c},2}^2$ are independent of $\omega$, the integral corresponding the first term $\displaystyle\frac{|\eta|}{2}\big((\frac{\eta}{|\eta|}\cdot\si)\frac{\eta}{|\eta|}-\si\big)$ vanishes. Therefore, by $1-\frac{\eta}{|\eta|}\cdot\si=1-\cos\th=2\sin^2\frac{\th}{2}$ and \eqref{angulareta} (note that $\min\{1,|\eta|^{-2+2s}\<\eta_*\>^{2-2s}\}\le |\eta|^{-1+s}\<\eta_*\>^{1-s}$), we have 
\begin{align*}
	|Q_{\bar{c},2}^{s,2}|
	&\le C\int_{\si,\eta,\eta_*}
	\1_{|\eta^-|\le\frac{1}{2}\<\eta_*\>^{}}
	b\sin^2\frac{\th}{2}|
	\wh{f_2}(\eta_*)\ol{\wh{h}(\eta)}\wh{G^{\al_1}}(-\eta_*,\eta-\eta_*)||\eta|\\
	&\le C\int_{\eta,\eta_*}|
	\wh{f_2}(\eta_*)\ol{\wh{h}(\eta)}\wh{G^{\al_1}}(-\eta_*,\eta-\eta_*)|\<\eta\>^{s}\<\eta_*\>^{1-s}. 
\end{align*}
As above, using H\"older's inequality
and estimate \eqref{eswhG0}, for any $s'\in\R$, we have
\begin{align}\label{Qbarc1s2}
	|Q_{\bar{c},2}^{s,2}|\le C\|\<D_v\>^{-N-2}\<v\>^{M}f\|_{L^2_v}\|\<v\>^{\ga}\<D_v\>^{s-s'}g\|_{L^2_v}\|\<D_v\>^{s+s'}h\|_{L^2_v}. 
\end{align}

For $Q^{l}_{\bar{c},2}$ in \eqref{Qbarc2a}, 
applying $|\eta^-|^{2}\le \<\eta_*\>^{2}\<\eta^--\eta_*\>^{2}$ and H\"older's inequality, we have 
\begin{align*}
	|Q^{l}_{\bar{c},2}|
	&\le \Big(\int_{\si,\eta,\eta_*}
	\1_{|\eta^-|>\frac{1}{2}\<\eta_*\>^{}}
	b\big(\frac{\eta}{|\eta|}\cdot\si\big)|\wh{f_2}(\eta_*)\wh{h}(\eta)|^2\<\eta_*\>^{4}\<\eta\>^{2s'}\Big)^{\frac{1}{2}}\\
	&\times
	\Big(\int_{\si,\eta,\eta_*}
	\1_{|\eta^-|>\frac{1}{2}\<\eta_*\>^{}}
	b\big(\frac{\eta}{|\eta|}\cdot\si\big)|\wh{G^{\al_1}}(\eta^--\eta_*,\eta-\eta_*)|^2\<\eta^--\eta_*\>^{4}\<\eta\>^{-2s'}\Big)^{\frac{1}{2}},
\end{align*}
for any $s'\in\R$. 
For the $\wh{G^{\al_1}}$ term, note that $|\eta^-|>\frac{1}{2}\<\eta_*\>$ implies $4|\eta^-|^2>1+|\eta_*|^2\ge 1+\frac{1}{2}|\eta^--\eta_*|^2-|\eta^-|^2$, and hence, $10|\eta^-|^2\ge \<\eta^--\eta_*\>^2$. Applying $\eta_*\mapsto \eta^--\eta_*$ to the $\wh{G^{\al_1}}$ term, using angular integral \eqref{angulareta} and estimate \eqref{eswhG0}, we have 
\begin{align}\label{Qbarc2l}
	|Q^{l}_{\bar{c},2}|\le C\|\<D_{v}\>^{-s}\<v\>^{M}f\|_{L^2_{v}}\|\<D_{v}\>^{s-s'}\<v\>^{\ga-1}g\|_{L^2_v}\|\<D_v\>^{s+s'}h\|_{L^2_v}.
\end{align}

\subsubsection{The term $Q_{\bar{c},3}$}
For $Q_{\bar{c},3}$ in \eqref{Qbarc123Def}, we apply \eqref{Qbarc101} conversely and deduce 
\begin{align*}
	Q_{\bar{c},3}
	&=\sum_{|\al_1|\le 2}C_{\al_1}\int_{v_*,\si,u,\eta}
	b\big(\frac{\eta}{|\eta|}\cdot\si\big)f_1(v_*)g(u)\ol{\wh{h}(\eta)}\pa^{\al_1}_{v_*}\Th_{\bar{c}}(u-v_*)\\
	&\qquad\qquad\qquad\times\big(e^{-2\pi i\eta^+\cdot u}e^{-2\pi i\eta^-\cdot v_*}-e^{-2\pi i\eta\cdot u}\big)\\
	&=\sum_{|\al_1|\le 2}C_{\al_1}\int_{v,v_*,\si}b(\cos\th)\pa^{\al_1}_{v_*}\Th_{\bar{c}}(v-v_*)f_1(v_*)g(v)\big(\ol{h(v')}-\ol{h(v)}\big)\\
	&=:(Q_{\pa^{\al_1}\Th_{\bar{c}}}(f_1,g),h)_{L^2_v}. 
\end{align*}
Since $\pa^{\al_1}_{v_*}\Th_{\bar{c}}$ is smooth, $Q_{\pa^{\al_1}\Th_{\bar{c}}}(f_1,g)$ has similar properties to $Q(f_1,g)$ or $Q_{\bar{c}}(f_1,g)$ regarding upper bounds. 
Applying the upper bound estimate \cite[Lemma 3.2]{Alexandre2013}, we have 
\begin{align}\label{Qbarc3}
	|Q_{\bar{c},3}|
	&\le |(Q_{\pa^{\al_1}\Th_{\bar{c}}}(f_1,g),h)_{L^2_v}|
	\le C\|\<D_v\>^{-2}\<v\>^{3|\ga|+2+\frac{d+1}{2}}f\|_{L^2_v}\|g\|_{L^2_D}\|h\|_{L^2_D},
\end{align}
and by the commutator estimate \cite[Lemma 2.4]{Alexandre2010} (they considered the non-singular truncated cross-section $\Phi(v-v_*)=\<v-v_*\>^{\ga}$, but the same arguments can be applied to $\pa^{\al_1}\Th_{\bar{c}}$; see also \cite[(2.2)]{Alexandre2011}), (note that  $f_1=\<D_v\>^{-2}f$)
\begin{align}\label{Qbarc3a}\notag
	|Q_{\bar{c},3}|
	&\le |(Q_{\pa^{\al_1}\Th_{\bar{c}}}(f_1,\<v\>^{\frac{\ga}{2}}g),\<v\>^{-\frac{\ga}{2}}h)_{L^2_v}|
	+|(\<v\>^{-\frac{\ga}{2}}Q_{\pa^{\al_1}\Th_{\bar{c}}}(f_1,\<v\>^{\frac{\ga}{2}}g)-Q_{\pa^{\al_1}\Th_{\bar{c}}}(f_1,g),h)_{L^2_v}|\\
	&
	\le C\|\<D_v\>^{-2}\<v\>^{4|\ga|+4+\frac{d+1}{2}}f\|_{L^2_v}
	\big(\|\<v\>^{\frac{\ga}{2}}g\|_{L^2_D}\|\<v\>^{-\frac{\ga}{2}}h\|_{L^2_D}+\|\<v\>^{\frac{\ga}{2}}g\|_{L^2_D}\|\<D_v\>^sh\|_{L^2_v}\big).
\end{align}
One can also apply the commutator estimate \eqref{CommGaAndQ} and work directly with the exponential-weighted collision operator $\Gamma(f,g)$, since our final result concerns exponential decompositions.

\smallskip 
Finally, collecting the estimates \eqref{Qbarc1es}, \eqref{Qbarc1s1}, \eqref{Qbarc1s2}, \eqref{Qbarc2l}, \eqref{Qbarc3}, and \eqref{Qbarc3a}
into \eqref{Qbarc123Def}, we can deduce that, for any $l\in\R$, $M>\frac{d+2}{2}$ and $s'\in\R$, 
\begin{align}\notag\label{QbarcTemp}
	\big|\big(Q_{\bar{c}}(f,g),h\big)_{L^2_v}\big|
	&\le C\|\<v\>^{M}\<D_v\>^{-s}f\|_{L^2_v}
	\|\<D_v\>^{s-s'}\<v\>^{\ga}g(v)\|_{L^2_v}
	\|\<D_v\>^{s+s'}h\|_{L^2_v}\\
	&\ \notag+C\|\<D_v\>^{-2}\<v\>^{4|\ga|+4+\frac{d+1}{2}}f\|_{L^2_v}\min\Big\{\|g\|_{L^2_D}\|h\|_{L^2_D},\\
	&\qquad \|\<v\>^{\frac{\ga}{2}}g\|_{L^2_D}\|\<v\>^{-\frac{\ga}{2}}h\|_{L^2_D}+\|\<v\>^{\frac{\ga}{2}}g\|_{L^2_D}\|\<D_v\>^sh\|_{L^2_v}
	\Big\}. 
\end{align}
This concludes the estimate of $Q_{\bar{c}}$ in \eqref{Qbarfhg}.

\subsection{Small relative velocity: $Q_c$}\label{SecSmallQc}
In this part, we will deal with $Q_c(f,g)$ with cross section $\Th_{c}(v-v_*)=|v-v_*|^\ga\phi(v-v_*)$ (given in \eqref{Bbarc}) that has singularity near $0$.
From \eqref{Bbarc}, we have $\phi(v-v_*)=\phi(v-v_*)\chi_{|v-v_*|\le 1}$, where $\chi\in C^\infty_c$ is defined as in \eqref{smoothcutoff}.
 Then we rewrite 
\begin{align}
	\label{Gc}
	\begin{aligned}
	\phi(v-v_*)f(v_*)g(v)&=\phi(v-v_*)\chi_{|v-v_*|\le 1}f(v_*)g(v),\\
	G(v_*,v)&:=\chi_{|v-v_*|\le 1}f(v_*)g(v).
	\end{aligned}
\end{align}
For the fractional derivatives of $G$, by Lemma \ref{basicCommuLem}, for any $M,s_1,s_2\in\R$, we can calculate 
\begin{align}\label{esGS1}\notag
	\|\<\eta_*\>^{-s_1}\<\eta\>^{s_2}\wh{G}(\eta_*,\eta)\|_{L^2_{\eta_*}L^2_{\eta}}
	\notag
	&=\|\underbrace{\<D_{v_*}\>^{-s_1}\<D_v\>^{s_2}\chi_{|v-v_*|\le 1}\<v\>^{-M}\<v_*\>^{M}\<D_{v_*}\>^{s_1}\<D_{v}\>^{-s_2}}_{\in \Op(1)}
	\\
	&\notag\qquad\qquad\quad
	\circ\<D_{v_*}\>^{-s_1}\<v_*\>^{-M}f(v_*)\<D_{v}\>^{s_2}\<v\>^{M}g(v)\|_{L^2_{v_*,v}}\\
	&\le C\|\<D_{v_*}\>^{-s_1}\<v_*\>^{-M}f(v_*)\|_{L^2_{v_*}}\|\<D_{v}\>^{s_2}\<v\>^{M}g(v)\|_{L^2_{v}},
\end{align}
where $\<D_{v_*}\>^{-s_1}\<D_v\>^{s_2}\chi_{|v-v_*|\le 1}\<v\>^{-M}\<v_*\>^{M}\<D_{v_*}\>^{s_1}\<D_{v}\>^{-s_2}$ can be regarded as a pseudo-differential operator whose symbol belongs to $S(1)$, since $\<v_*\>\approx\<v\>$ whenever $|v-v_*|\le 1$, and hence, is bounded on $L^2$. 
By Fourier inversion, integrating in $v$, and the identities in \eqref{deltafunction}, 
\begin{align}\notag\label{QcBobylev}
	&(Q_c(f,g),h)_{L^2_v}
	\notag=\int_{v,v_*,\si}b(\cos\th)\Th_{c}(v-v_*)G(v_*,v)\big(\ol{h(v')}-\ol{h(v)}\big)\\
	&\notag=\int_{v,v_*,\si,\eta,u,\zeta}
	b(\cos\th)\Th_{c}(u-v_*)G(v_*,u)\ol{\wh{h}(\eta)}e^{2\pi i\zeta\cdot(v-u)}\Big(e^{-2\pi i\eta\cdot v'}-e^{-2\pi i\eta\cdot v}\Big)\\
	&\notag=\int_{v_*,\si,\eta,u,\zeta}
	b\big(\frac{\eta}{|\eta|}\cdot\si\big)\Th_{c}(-v_*)G(v_*+u,u)\ol{\wh{h}(\eta)}e^{-2\pi i\eta\cdot u}\Big(
	e^{-2\pi i\eta^-\cdot v_*}-1\Big)\\
	&=\int_{\si,\eta,\zeta}
	b\big(\frac{\eta}{|\eta|}\cdot\si\big)
	\F_{v_*,u}\big(G(v_*+u,u)\big)(\eta_*,\eta)\ol{\wh{h}(\eta)}\Big(
	\wh{\Th_{c}}(\eta_*-\eta^-)-
	\wh{\Th_{c}}(\eta_*)\Big). 
\end{align}
where we have used $\wh{fg}=\wh{f}*\wh{g}$. For the Fourier transform of $G$, one has 
\begin{align}\label{FourGvstar}
	\F_{v_*,u}\big(G(v_*+u,u)\big)(\eta_*,\eta)
	&=\int_{v_*,u}G(v_*,u)e^{-2\pi i\eta_*\cdot v_*-2\pi i(\eta-\eta_*)\cdot u}
	=\wh{G}(\eta_*,\eta-\eta_*).
\end{align}
Therefore, we have 
\begin{align}\label{Qc123}\notag
	(Q_c(f,g),h)_{L^2_v}
    &\notag=\int_{\R^{2d}_{\eta_*,\eta}\times\S^{d-1}_\si}
	b\big(\frac{\eta}{|\eta|}\cdot\si\big)\wh{G}(\eta_*,\eta-\eta_*)\ol{\wh{h}(\eta)}\Big(\wh{\Th_c}(\eta_*-\eta^-)
	-\wh{\Th_c}(\eta_*)\Big)\\
	&=:Q_{c1}+Q_{c2}+Q_{c3},
\end{align}
where, by Taylor expansion $\wh{\Th_c}(\eta_*-\eta^-)-\wh{\Th_c}(\eta_*)
	=-\eta^-\cdot\na_\eta\wh{\Th_c}(\eta_*)+\int^1_0(1-\de)\na^2_\eta\wh{\Th_c}(\eta_*-\de\eta^-):\eta^-\otimes\eta^-\,d\de$, we decompose 
\begin{align*}
	Q_{c1}&=\int_{\eta,\eta_*,\si}\1_{|\eta^-|\le\frac{1}{2}\<\eta_*\>} b\wh{G}(\eta_*,\eta-\eta_*)\ol{\wh{h}(\eta)}(-\eta^-)\cdot\na_\eta\wh{\Th_c}(\eta_*),\\
	Q_{c2}\notag&=\int_{\eta,\eta_*,\si}\1_{|\eta^-|\le\frac{1}{2}\<\eta_*\>}b\wh{G}(\eta_*,\eta-\eta_*)\ol{\wh{h}(\eta)}\int^1_0(1-\de)\na^2_\eta\wh{\Th_c}(\eta_*-\de\eta^-):\eta^-\otimes\eta^-\,d\de,\\
	Q_{c3}&\notag=\int_{\eta,\eta_*,\si}\1_{|\eta^-|\ge\frac{1}{2}\<\eta_*\>}b\wh{G}(\eta_*,\eta-\eta_*)\ol{\wh{h}(\eta)}\big(\wh{\Th_c}(\eta_*-\eta^-)-\wh{\Th_c}(\eta_*)\big),
\end{align*}
For $Q_{c1}$, we also divide $\eta^-$ as in \eqref{decompositionofetaminus}, i.e., 
\begin{align*}
	\eta^-=\frac{\eta-|\eta|\si}{2}=\frac{|\eta|}{2}\Big(\big(\frac{\eta}{|\eta|}\cdot\si\big)\frac{\eta}{|\eta|}-\si\Big)+\Big(1-\big(\frac{\eta}{|\eta|}\cdot\si\big)\Big)\frac{\eta}{2}. 
\end{align*}
Using the decomposition $\si=\cos\th\,\mathbf{j}+\sin\th\,\omega$ in \eqref{decomsitoj} and the symmetry of $\omega$ in \eqref{symmetaomega}, the integral corresponding the first term $\frac{|\eta|}{2}\big((\frac{\eta}{|\eta|}\cdot\si)\frac{\eta}{|\eta|}-\si\big)$ vanishes (note that, by \eqref{bobylevangle}, $|\eta^-|=|\eta|\sin\frac{\th}{2}$ is independent of $\om$). For the second term, we have 
\begin{align*}
	\Big(1-\big(\frac{\eta}{|\eta|}\cdot\si\big)\Big)\frac{\eta}{2}
	&=(1-\cos\th)\frac{\eta}{2}
	=\eta\sin^2\frac{\th}{2}. 
\end{align*}
Thus, by \eqref{angulareta}, we have 
\begin{align*}
    Q_{c1}&=-\int_{\R^{2d}_{\eta,\eta_*}\times\S^{d-1}_\si}\1_{|\eta^-|\le\frac{1}{2}\<\eta_*\>} b\sin^2\frac{\theta}{2}\eta\cdot\na_\eta\wh{\Th_c}(\eta_*)\wh{G}(\eta_*,\eta-\eta_*)\ol{\wh{h}(\eta)},
\end{align*}
By \eqref{nakThc}, we have $|D_\eta^k\wh{\Th_c}(\eta)|\le C_{N,d,\ga}\<\eta\>^{-d-\ga-k}$. 
Then, together with \eqref{angulareta}, we have 
\begin{align*}
	|Q_{c1}|&=\Big|\int_{\R^{2d}_{\eta,\eta_*}\times\S^{d-1}_\si}
	\1_{|\eta^-|\le\frac{1}{2}\<\eta_*\>}b\sin^2\frac{\th}{2}\eta\cdot\na_\eta\wh{\Th_c}(\eta_*)\wh{G}(\eta_*,\eta-\eta_*)\ol{\wh{h}(\eta)}\Big|\\
	&\le C\int_{\R^{2d}_{\eta,\eta_*}}
	\min\{1,|\eta|^{-2+2s}\<\eta_*\>^{2-2s}\}
	\frac{|\eta|}{\<\eta_*\>^{d+\ga+1}}\Big|\wh{G}(\eta_*,\eta-\eta_*)
	\ol{\wh{h}(\eta)}\Big|.
\end{align*}
Similarly, for $Q_{c2}$, on the support of $\1_{|\eta^-|\le\frac{1}{2}\<\eta_*\>}$, when $\ga\ge-d-2$, we have 
\begin{align*}
	|\na^2_\eta\wh{\Th_c}(\eta_*-\de\eta^-)|
	\le C\<\eta_*-\de\eta^-\>^{-d-\ga-2}
	\le C\<|\eta_*|-\de|\eta^-|\>^{-d-\ga-2}
	\le C\<\eta_*\>^{-d-\ga-2}, 
\end{align*} 
and hence, 
\begin{align*}
	|Q_{c2}|&\le C\int_{\R^{2d}_{\eta,\eta_*}\times\S^{d-1}_\si}
	\1_{|\eta^-|\le\frac{1}{2}\<\eta_*\>}b\sin^2\frac{\th}{2}
	\frac{|\eta|^2}{\<\eta_*\>^{d+\ga+2}}
	\Big|\wh{G}(\eta_*,\eta-\eta_*)\ol{\wh{h}(\eta)}\Big|\,d\si d\eta_*d\eta\\
	&\le C\int_{\R^{2d}_{\eta,\eta_*}}\min\{1,|\eta|^{-2+2s}\<\eta_*\>^{2-2s}\}\frac{|\eta|^2}{\<\eta_*\>^{d+\ga+2}}\Big|\wh{G}(\eta_*,\eta-\eta_*)\ol{\wh{h}(\eta)}\Big|\,d\eta_*d\eta.
\end{align*}
The estimates of $Q_{c1}$ and $Q_{c2}$ have similar structure. 
That is, 
\begin{align}\label{Qc12}\notag
	|Q_{c1}|+|Q_{c2}|
	&\notag\le C\int_{{\eta,\eta_*}}\min\{1,|\eta|^{-2+2s}\<\eta_*\>^{2-2s}\}\Big(\frac{|\eta|}{\<\eta_*\>^{d+\ga+1}}+\frac{|\eta|^2}{\<\eta_*\>^{d+\ga+2}}\Big)\Big|\wh{G}(\eta_*,\eta-\eta_*)\ol{\wh{h}(\eta)}\Big|\\
	&\notag\le C\int_{{\eta,\eta_*}}\Big(\frac{|\eta|}{\<\eta_*\>^{d+\ga+1}}+\frac{|\eta|^2}{\<\eta_*\>^{d+\ga+2}}\Big)\Big|\wh{G}(\eta_*,\eta-\eta_*)\ol{\wh{h}(\eta)}\Big|\1_{|\eta|\le 2\<\eta_*\>}\\
	&\notag\qquad+C\int_{{\eta,\eta_*}}\frac{\<\eta_*\>^{2-2s}}{|\eta|^{2-2s}}\Big(\frac{|\eta|}{\<\eta_*\>^{d+\ga+1}}+\frac{|\eta|^2}{\<\eta_*\>^{d+\ga+2}}\Big)\Big|\wh{G}(\eta_*,\eta-\eta_*)\ol{\wh{h}(\eta)}\Big|\1_{|\eta|> 2\<\eta_*\>}\\
	&=:Q_{c1}'+Q_{c2}'.
\end{align}
For $Q_{c1}'$, since $|\eta|\le2\<\eta_*\>$, we have 
\begin{align*}
	Q_{c1}'
	&\le 
	C\int_{{\eta,\eta_*}}\frac{\1_{|\eta|\le 2\<\eta_*\>}|\<\eta_*\>^{-s_1}\<\eta-\eta_*\>^{s_2}\wh{G}(\eta_*,\eta-\eta_*)\<\eta\>^{s_3}\ol{\wh{h}(\eta)}|}{\<\eta_*\>^{d+\ga-s_1}\<\eta-\eta_*\>^{s_2}\<\eta\>^{s_3}}\frac{|\eta|}{\<\eta_*\>}\\
	&\le 
	C\Big\|\frac{\1_{|\eta|\le 2\<\eta_*\>}}{\<\eta_*\>^{d+\ga-s_1+1}\<\eta-\eta_*\>^{s_2}\<\eta\>^{s_3-1}}\Big\|_{L^\infty_{\eta}L^2_{\eta_*}}
	\|\<\eta\>^{s_3}\wh{h}\|_{L^2_{\eta}}
	\|\<\eta_*\>^{-s_1}\<\eta\>^{s_2}\wh{G}(\eta_*,\eta)\|_{L^2_{\eta_*,\eta}}.
\end{align*}
Here, by \eqref{equivetastar}, we calculate the integral as 
\begin{align*}
	&\Big\|\frac{\1_{|\eta|\le 2\<\eta_*\>}}{\<\eta_*\>^{d+\ga-s_1+1}\<\eta-\eta_*\>^{s_2}\<\eta\>^{s_3-1}}\Big\|_{L^\infty_{\eta}L^2_{\eta_*}}^2\\
	&\le 
	\int_{\R^d_{\eta_*}}\frac{\1_{|\eta|\le\frac{1}{2}\<\eta_*\>}\1_{\<\eta-\eta_*\>\approx \<\eta_*\>\gtrsim\<\eta\>}}{\<\eta_*\>^{2d+2\ga-2s_1+2}\<\eta-\eta_*\>^{2s_2}\<\eta\>^{2s_3-2}}\,d\eta_*
	+
	\int_{\R^d_{\eta_*}}\frac{\1_{\frac{1}{2}\<\eta_*\>\le |\eta|\le 2\<\eta_*\>}\1_{\<\eta-\eta_*\>\lesssim\<\eta_*\>\approx\<\eta\>}}{\<\eta_*\>^{2d+2\ga-2s_1+2}\<\eta-\eta_*\>^{2s_2}\<\eta\>^{2s_3-2}}\,d\eta_*\\
	& \le 
	\frac{1}{\<\eta\>^{d+2\ga-2s_1+2s_2+2s_3}}
	+
	\int_{\R^d_{\eta_*}}\frac{\1_{\<\eta-\eta_*\>\lesssim\<\eta\>}}{\<\eta-\eta_*\>^{2d+2\ga-2s_1+2s_2+2s_3}}\,d\eta_*
	\le C,
\end{align*}
provided that $2d+2\ga-2s_1+2+2s_2>d$, $2d+2\ga-2s_1+2s_3\ge 0$, $2d+2\ga-2s_1+2s_2+2s_3>d$, 
where, in the first term, we used the estimate 
\begin{align}\label{integrallargeR}
	\int_{\R^d_{\eta_*}}\frac{\1_{\<\eta_*\>\ge A}}{\<\eta_*\>^k}\,d\eta_*
	\le C\int^\infty_{A}(1+R)^{-k+d-1}\,dR\le CA^{-k+d}\ \text{ for any $k>d$, $A>0$}. 
\end{align}
Thus, by \eqref{esGS1}, 
\begin{align}\label{Qc2es1}
	Q_{c1}'
	&\le C\|\<D_v\>^{-s_1}\<v\>^{-\ga}f\|_{L^2_v}\|\<D_v\>^{s_2}\<v\>^{\ga}g\|_{L^2_v}\|\<D_v\>^{s_3}h\|_{L^2_v}. 
\end{align}
For $Q_{c2}'$, similarly, we have  
\begin{align*}
	&Q_{c2}'\le C\int_{\R^{2d}_{\eta,\eta_*}}\1_{|\eta|> 2\<\eta_*\>}\frac{\<\eta_*\>^{2-2s}}{|\eta|^{2-2s}}\frac{|\eta|^2}{\<\eta_*\>^{d+\ga+2}}\Big|\wh{G}(\eta_*,\eta-\eta_*)\ol{\wh{h}(\eta)}\Big|
    \,d\eta_*d\eta\\
	&\le C\int_{\R^{2d}_{\eta,\eta_*}}\frac{|\eta|^{2s}\1_{|\eta|> 2\<\eta_*\>}}{\<\eta_*\>^{d+\ga+2s-s_1}\<\eta-\eta_*\>^{s_2}\<\eta\>^{s_3}}\Big|\<\eta_*\>^{-s_1}\<\eta-\eta_*\>^{s_2}\wh{G}(\eta_*,\eta-\eta_*)\<\eta\>^{s_3}\ol{\wh{h}(\eta)}\Big|
    \,d\eta_*d\eta\\
	&\le C\Big\|\frac{|\eta|^{2s}\1_{|\eta|> 2\<\eta_*\>}}{\<\eta_*\>^{d+\ga+2s-s_1}\<\eta-\eta_*\>^{s_2}\<\eta\>^{s_3}}\Big\|_{L^\infty_\eta L^2_{\eta_*}}\|\<\eta\>^{s_3}\ol{\wh{h}(\eta)}\|_{L^2_{\eta}}\|\<\eta_*\>^{-s_1}\<\eta\>^{s_2}\wh{G}(\eta_*,\eta)\|_{L^2_{\eta}L^2_{\eta_*}},
\end{align*}
where, by \eqref{equivetastar} ($|\eta|>2\<\eta_*\>$ implies $\<\eta-\eta_*\>\approx\<\eta\>\gtrsim\<\eta_*\>$), we calculate the integral of $\eta_*$ as 
\begin{align*}
	\int_{\R^{d}_{\eta_*}}\frac{|\eta|^{4s}\1_{|\eta|> 2\<\eta_*\>}}{\<\eta_*\>^{2d+2\ga+4s-2s_1}\<\eta-\eta_*\>^{2s_2}\<\eta\>^{2s_3}}
	\,d\eta_*
	&\le\int_{\R^{d}_{\eta_*}}\frac{\<\eta\>^{4s-2s_2-2s_3}\1_{
	\<\eta\>\approx\<\eta-\eta_*\>\gtrsim\<\eta_*\>}}{\<\eta_*\>^{2d+2\ga+4s-2s_1}}\,d\eta_*\\
	&\le C\int_{\R^{d}_{\eta_*}}\frac{1}{\<\eta_*\>^{2d+2\ga+4s-2s_1}}\,d\eta_*\le C, 
\end{align*}
provided that $s_2+s_3\ge 2s$ and $2d+2\ga+4s-2s_1>d$. Together with \eqref{esGS1}, we have 
\begin{align}\label{Qc2es2}
	Q_{c2}'&\le C\|\<D_v\>^{-s_1}\<v\>^{-\ga}f\|_{L^2_v}\|\<D_v\>^{s_2}\<v\>^{\ga}g\|_{L^2_v}\|\<D_v\>^{s_3}h\|_{L^2_v}. 
\end{align}

For $Q_{c3}$ in \eqref{Qc123}, by rewriting 
 $\big(\cdots\big)=b\big(\frac{\eta}{|\eta|}\cdot\si\big)\wh{G}(\eta_*,\eta-\eta_*)\ol{\wh{h}(\eta)}$,
we split it as
\begin{align*}
	Q_{c3}&=\int_{\eta,\eta_*,\si}\1_{|\eta^-|\ge\frac{1}{2}\<\eta_*\>}\big(\wh{\Th_c}(\eta_*-\eta^-)-\wh{\Th_c}(\eta_*)\big)b\big(\frac{\eta}{|\eta|}\cdot\si\big)\wh{G}(\eta_*,\eta-\eta_*)\ol{\wh{h}(\eta)}\\
	&=\int_{\eta,\eta_*,\si}\1_{|\eta^-|\ge\frac{1}{2}\<\eta_*\>}\wh{\Th_c}(\eta_*-\eta^-)\1_{|\eta^-|^2\le2|\eta_*\cdot\eta^-|}
	\1_{|\eta|\le2\<\eta_*\>}\1_{\<\eta-\eta_*\>\le\<\eta_*-\eta^-\>}\Big(\cdots\Big)\\
	&\ +\int_{\eta,\eta_*,\si}\1_{|\eta^-|\ge\frac{1}{2}\<\eta_*\>}\wh{\Th_c}(\eta_*-\eta^-)\1_{|\eta^-|^2\le2|\eta_*\cdot\eta^-|}
	\1_{|\eta|\le2\<\eta_*\>}\1_{\<\eta-\eta_*\>>\<\eta_*-\eta^-\>}\Big(\cdots\Big)\\
	&\ +\int_{\eta,\eta_*,\si}\1_{|\eta^-|\ge\frac{1}{2}\<\eta_*\>}\wh{\Th_c}(\eta_*-\eta^-)\1_{|\eta^-|^2\le2|\eta_*\cdot\eta^-|}
	\1_{|\eta|>2\<\eta_*\>}\Big(\cdots\Big)\\
	&\ +\int_{\eta,\eta_*,\si}\1_{|\eta^-|\ge\frac{1}{2}\<\eta_*\>}\Big(\wh{\Th_c}(\eta_*-\eta^-)\1_{|\eta^-|^2>2|\eta_*\cdot\eta^-|}-\wh{\Th_c}(\eta_*)\Big)\Big(\cdots\Big)\\
	&=:Q_{c3}^{1}+Q_{c3}^{2}+Q_{c3}^{3}+Q_{c3}^{4}.
\end{align*}
Similar to \eqref{supportJ3}, we utilize the following decomposition and their estimates. First, when $|\eta^-|^2\le2|\eta_*\cdot\eta^-|$, we have $|\eta^-|\le2|\eta_*|$ and hence, $|\eta_*-\eta^-|^2=|\eta_*|^2-2\eta_*\cdot\eta^-+|\eta^-|^2\le9|\eta_*|^2$. Similarly, when $|\eta^-|^2>2|\eta_*\cdot\eta^-|$, we have $|\eta_*-\eta^-|^2\ge |\eta_*|^2$. Lastly, when $|\eta|>2\<\eta_*\>$, we have $\<\eta_*-\eta^-\>\le\<\eta_*\>+\<\eta^-\>\le2\<\eta\>$ and $|\eta_*-\eta|\ge |\eta|-|\eta_*|\ge\frac{|\eta|}{2}\approx\<\eta\>$ (note $|\eta|\ge 1$). In summary, together with \eqref{ths}, we have 
\begin{align}\label{supportQc3}
	\begin{cases}
		\text{when}~|\eta^-|\ge\frac{1}{2}\<\eta_*\>,&|\eta|\ge\frac{1}{2}\<\eta_*\>,\\
		\text{when}~|\eta^-|^2\le2|\eta_*\cdot\eta^-|, &\<\eta_*-\eta^-\>\le C\<\eta_*\>\ \text{and} \ |\eta^-|\le2|\eta_*|,\\
		\text{when}~|\eta^-|^2>2|\eta_*\cdot\eta^-|,&\<\eta_*-\eta^-\>\ge \<\eta_*\>,\\
		\text{when}~|\eta|>2{\<\eta_*\>},&\<\eta_*-\eta^-\>\le2\<\eta\>~\text{and}~\<\eta_*-\eta\>\approx\<\eta\>,\\
		\text{in any cases},&b(\cos\th)\le C\th^{-d+1-2s}
		\le \frac{C|\eta|^{d-1+2s}}{|\eta^-|^{d-1+2s}}.
	\end{cases}
\end{align}
Moreover, by \eqref{nakThc}, we know that 
\begin{align*}
|\wh{\Th_c}(\eta_*-\eta^-)|\le C\<\eta_*-\eta^-\>^{-d-\ga}.
\end{align*}
On the support of $Q_{c3}^{1}$, one has ${\<\eta-\eta_*\>\le\<\eta_*-\eta^-\>\lesssim \<\eta_*\>\approx\<\eta\>\approx\<\eta^-\>}$ and hence, $b(\cos\th)\le C$ (there's no singular angular term). Thus, by H\"older's inequality, we have 
\begin{align*}
	Q_{c3}^{1}
	&\le C\Big(\int_{\eta,\eta_*,\si}
	|\<\eta_*\>^{-s_1}\<\eta-\eta_*\>^{s_2}\wh{G}(\eta_*,\eta-\eta_*)|^2
	\,d\si d\eta_*d\eta\Big)^{\frac{1}{2}}\\&\quad\times 
	\Big(\int_{\eta,\eta_*,\si}
	\frac{\1_{|\eta^-|\ge\frac{1}{2}\<\eta_*\>}\1_{|\eta^-|^2\le2|\eta_*\cdot\eta^-|}
	\1_{|\eta|\le2\<\eta_*\>}\1_{\<\eta-\eta_*\>\le\<\eta_*-\eta^-\>}}{\<\eta_*-\eta^-\>^{2d+2\ga}\<\eta_*\>^{-2s_1}\<\eta-\eta_*\>^{2s_2}\<\eta\>^{2s_3}}|\<\eta\>^{s_3}\ol{\wh{h}(\eta)}|^2\Big)^{\frac{1}{2}}\\
	&\le C\|\<\eta_*\>^{-s_1}\<\eta\>^{s_2}\wh{G}(\eta_*,\eta)\|_{L^2_{\eta,\eta_*}}\|\<\eta\>^{s_3}\wh{h}\|_{L^2_\eta}
	\Big\|\frac{\1_{\<\eta-\eta_*\>\le\<\eta_*-\eta^-\>\lesssim \<\eta_*\>\approx\<\eta\>}}{\<\eta_*-\eta^-\>^{d+\ga}\<\eta-\eta_*\>^{s_2}\<\eta_*\>^{-s_1+s_3}}
	\Big\|_{L^1_{\si}L^\infty_\eta L^2_{\eta_*}},
\end{align*}
where
\begin{align*}\notag
	\Big\|\frac{\1_{\<\eta-\eta_*\>\le\<\eta_*-\eta^-\>\lesssim \<\eta_*\>\approx\<\eta\>}}{\<\eta_*-\eta^-\>^{d+\ga}\<\eta-\eta_*\>^{s_2}\<\eta_*\>^{-s_1+s_3}}
	\Big\|_{L^2_{\eta_*}}
	&\notag\le \Big\|
	\frac{1}{\<\eta-\eta_*\>^{d+\ga-s_1+s_2+s_3}}\Big\|_{L^2_{\eta_*}}
	\le C,
\end{align*}
provided that $d+\ga>0$, $-s_1+s_3\ge 0$ and $2d+2\ga-2s_1+2s_2+2s_3>d$. Thus, together with \eqref{esGS1}, we have 
\begin{align*}
	Q_{c3}^1\le C\|\<D_v\>^{-s_1}\<v\>^{-\ga}f\|_{L^2_v}\|\<D_v\>^{s_2}\<v\>^{\ga}g\|_{L^2_v}\|\<D_v\>^{s_3}h\|_{L^2_v}. 
\end{align*}
The estimate of $Q_{c3}^{2}$ is similar. On its support, we have $|\eta^-|\approx\<\eta_*\>\approx|\eta|$, $\<\eta_*-\eta^-\>\le C\<\eta_*\>$, $|\eta|\le2\<\eta_*\>$, $\<\eta-\eta_*\>>\<\eta_*-\eta^-\>$, and $b(\cos\th)\le C$. Using Cauchy-Schwarz inequality yields
\begin{align*}
	&Q_{c3}^{2}
	\le 
	C\Big(\int_{\eta,\eta_*,\si}
	|\<\eta_*\>^{-s_1}\<\eta-\eta_*\>^{s_2}\wh{G}(\eta_*,\eta-\eta_*)|^2
	\Big)^{\frac{1}{2}}\\
	&\quad\times 
	\Big(\int_{\eta,\eta_*,\si}
	\frac{\1_{|\eta^-|\approx|\eta|\approx\<\eta_*\>}
	\1_{\<\eta_*-\eta^-\>\lesssim\min\{\<\eta-\eta_*\>,\<\eta_*\>\}}
	\1_{|\eta|\le2\<\eta_*\>}
	}{\<\eta_*-\eta^-\>^{2d+2\ga}\<\eta_*\>^{-2s_1}\<\eta-\eta_*\>^{2s_2}\<\eta\>^{2s_3}}|\<\eta\>^{s_3}\ol{\wh{h}(\eta)}|^2\Big)^{\frac{1}{2}}\\
	&\le 
	C\|\<D_v\>^{-s_1}\<v\>^{-\ga}f\|_{L^2_v}\|\<D_v\>^{s_2}\<v\>^{\ga}g\|_{L^2_v}\|\<\eta\>^{s_3}\wh{h}\|_{L^2_\eta}
	\Big\|\frac{
	1}{\<\eta_*-\eta^-\>^{d+\ga-s_1+s_2+s_3}}\Big\|_{L^1_{\si}L^\infty_{\eta}L^2_{\eta_*}},
\end{align*}
provided that $s_2\ge 0$ and $-s_1+s_3\ge 0$, where the last term is bounded by a constant whenever $2d+2\ga-2s_1+2s_2+2s_3>d$. 

\smallskip
On the support of $Q_{c3}^{3}$, by \eqref{supportQc3}, we have $|\eta|\ge|\eta^-|\ge\frac{1}{2}\<\eta_*\>$, $\<\eta_*-\eta^-\>\le C\<\eta_*\>$, $|\eta^-|\le2|\eta_*|$, $\<\eta_*-\eta^-\>\le2\<\eta\>$, and $\<\eta_*-\eta\>\gtrsim\<\eta\>$. Thus, by H\"older's inequality and angular integral \eqref{angulareta}, 
\begin{align*}
	&Q_{c3}^{3}
	\le C\Big(\int_{\eta,\eta_*,\si}
	b(\cos\th)\frac{\<\eta_*\>^{2s}}{|\eta|^{2s}}\1_{|\eta^-|\ge\frac{1}{2}\<\eta_*\>}
	|\<\eta_*\>^{-s_1}\<\eta-\eta_*\>^{s_2}\wh{G}(\eta_*,\eta-\eta_*)|^2
	\Big)^{\frac{1}{2}}\\
	&\ \ \times 
	\Big(\int_{\eta,\eta_*,\si}
	b\frac{|\eta|^{2s}}{\<\eta_*\>^{2s}}
	\frac{\1_{2|\eta_*|\ge|\eta^-|\ge\frac{1}{2}\<\eta_*\>}
	\1_{\<\eta_*-\eta^-\>\le C\min\{\<\eta\>,\<\eta_*\>\}}
	\1_{\<\eta_*-\eta\>\gtrsim\<\eta\>}
	}{\<\eta_*-\eta^-\>^{2d+2\ga}\<\eta_*\>^{-2s_1}\<\eta-\eta_*\>^{2s_2}\<\eta\>^{2s_3}}|\<\eta\>^{s_3}\ol{\wh{h}(\eta)}|^2\Big)^{\frac{1}{2}}\\
	&\le C\|\<\eta_*\>^{-s_1}\<\eta\>^{s_2}\wh{G}(\eta_*,\eta)\|_{L^2_{\eta,\eta_*}}
	\Big(\int_{\eta,\eta_*,\si}
	b
	\frac{|\eta|^{2s}\1_{|\eta^-|\ge\frac{1}{C}\<\eta_*-\eta^-\>}
	|\<\eta\>^{s_3}\ol{\wh{h}(\eta)}|^2}{\<\eta_*-\eta^-\>^{2d+2\ga}\<\eta_*\>^{2s-2s_1}\<\eta\>^{2s_2+2s_3}}\Big)^{\frac{1}{2}},
\end{align*}
provided that $s_2\ge 0$. The last factor, if $s_2+s_3\ge 2s$ and $2s-2s_1\ge 0$, by translation $\eta_*\mapsto\eta_*-\eta^-$ and angular integral \eqref{angulareta}, can be estimated as  
\begin{align*}
	&\le \Big(\int_{\eta,\eta_*,\si}
	b(\cos\th)
	\frac{\1_{|\eta^-|\ge\frac{1}{2}\<\eta_*-\eta^-\>}
	}{\<\eta_*-\eta^-\>^{2d+2\ga+2s-2s_1}\<\eta\>^{2s}}|\<\eta\>^{s_3}\ol{\wh{h}(\eta)}|^2\Big)^{\frac{1}{2}}
	\\
	&\le \Big(\int_{\R^{2d}_{\eta,\eta_*}}
	\frac{1}{\<\eta_*\>^{2d+2\ga+4s-2s_1}}|\<\eta\>^{s_3}\ol{\wh{h}(\eta)}|^2\,d\eta_*d\eta\Big)^{\frac{1}{2}}\le \|\<\eta\>^{s_3}\ol{\wh{h}(\eta)}\|_{L^2_\eta},
\end{align*}
provided that $2d+2\ga+4s-2s_1>d$. 

\smallskip 
For the term $Q_{c3}^{4}$, by \eqref{supportQc3}, on the support of $\1_{|\eta^-|^2>2|\eta_*\cdot\eta^-|}$, we have $\<\eta_*-\eta^-\>\ge\<\eta_*\>$ and hence, by using \eqref{nakThc}, we have
\begin{align*}
	|\wh{\Th_c}(\eta_*-\eta^-)|\1_{|\eta^-|^2>2|\eta_*\cdot\eta^-|}+|\wh{\Th_c}(\eta_*)|\le C\<\eta_*\>^{-d-\ga},
\end{align*}
whenever $-d-\ga<0$. Therefore, by angular integral \eqref{angulareta} and $|\eta^-|\ge\frac{1}{2}\<\eta_*\>$ on its support, we can estimate $Q_{c3}^{4}$ as 
\begin{align*}
	|Q_{c3}^{4}|
	&\le 
	C\int_{\eta,\eta_*,\si}b(\cos\th)\frac{\1_{|\eta^-|\ge\frac{1}{2}\<\eta_*\>}}{\<\eta_*\>^{d+\ga}}\frac{|\<\eta_*\>^{-s_1}\<\eta-\eta_*\>^{s_2}\wh{G}(\eta_*,\eta-\eta_*)\<\eta\>^{s_3}\ol{\wh{h}(\eta)}|}{\<\eta_*\>^{-s_1}\<\eta-\eta_*\>^{s_2}\<\eta\>^{s_3}}\\
	&\le 
	C\int_{{\eta,\eta_*}}\frac{|\eta|^{2s}\1_{|\eta|\ge\frac{1}{2}\<\eta_*\>}}{\<\eta_*\>^{d+\ga+2s}}\frac{|\<\eta_*\>^{-s_1}\<\eta-\eta_*\>^{s_2}\wh{G}(\eta_*,\eta-\eta_*)\<\eta\>^{s_3}\ol{\wh{h}(\eta)}|}{\<\eta_*\>^{-s_1}\<\eta-\eta_*\>^{s_2}\<\eta\>^{s_3}}\\
	&\le C\Big(\int_{{\eta,\eta_*}}
	|\<\eta_*\>^{-s_1}\<\eta-\eta_*\>^{s_2}\wh{G}(\eta_*,\eta-\eta_*)|^2
	\Big)^{\frac{1}{2}}\\
	&\quad\times \Big(\int_{{\eta,\eta_*}}
	\frac{\1_{|\eta|\ge\frac{1}{2}\<\eta_*\>}|\eta|^{4s-2s_3}}{\<\eta_*\>^{2d+2\ga+4s-2s_1}\<\eta-\eta_*\>^{2s_2}}
	|\<\eta\>^{s_3}\ol{\wh{h}(\eta)}|^2\Big)^{\frac{1}{2}}.
\end{align*}
Here, by \eqref{equivetastar}, we have $\<\eta\>\approx\<\eta_*\>\gtrsim\<\eta-\eta_*\>$ when $\frac{1}{2}|\eta|\le\<\eta_*\>\le2|\eta|$, and $\<\eta\>\approx\<\eta-\eta_*\>\gtrsim\<\eta_*\>$ when $\<\eta_*\>\le\frac{1}{2}|\eta|.$ Thus, 
\begin{align*}
	&\int_{\R^d_{\eta_*}}\frac{\1_{|\eta|\ge\frac{1}{2}\<\eta_*\>}\<\eta\>^{4s-2s_3}}{\<\eta_*\>^{2d+2\ga+4s-2s_1}\<\eta-\eta_*\>^{2s_2}}\,d\eta_*\\
	&=
	\int_{\R^d_{\eta_*}}\frac{
		\1_{\<\eta\>\approx\<\eta_*\>\gtrsim\<\eta-\eta_*\>}}{\<\eta_*\>^{2d+2\ga-2s_1+2s_3}\<\eta-\eta_*\>^{2s_2}}\,d\eta_*
	+\int_{\R^d_{\eta_*}}\frac{
		\1_{\<\eta\>\approx\<\eta-\eta_*\>\gtrsim\<\eta_*\>}\<\eta\>^{4s-2s_3}}{\<\eta_*\>^{2d+2\ga+4s-2s_1}\<\eta-\eta_*\>^{2s_2}}\,d\eta_*\\
	&\le 
	\int_{\R^d_{\eta_*}}\frac{C}{\<\eta-\eta_*\>^{2d+2\ga-2s_1+2s_2+2s_3}}\,d\eta_*
	+\int_{\R^d_{\eta_*}}\frac{C}{\<\eta_*\>^{2d+2\ga-2s_1+2s_2+2s_3}}\,d\eta_*
	\le C,
\end{align*}
provided that $2d+2\ga-2s_1+2s_3\ge 0$, $s_2+s_3\ge 2s$, and $2d+2\ga-2s_1+2s_2+2s_3>d$. Collecting the above estimates, we have found that 
\begin{align*}
	|Q_{c3}|\le C\|\<D_v\>^{-s_1}\<v\>^{-\ga}f\|_{L^2_v}\|\<D_v\>^{s_2}\<v\>^{\ga}g\|_{L^2_v}\|\<D_v\>^{s_3}h\|_{L^2_v}, 
\end{align*}
provided the above conditions with respect to $s_1,s_2,s_3$ are valid.  
Together with the estimates of $Q_{c1},Q_{c2}$ in \eqref{Qc12}, \eqref{Qc2es1}, and \eqref{Qc2es2}, substituting into \eqref{Qc123} yields
\begin{align}\label{QcTemp}
	|(Q_c(f,g),h)_{L^2_v}|
	\le C\|\<D_v\>^{-s_1}\<v\>^{-\ga}f\|_{L^2_v}\|\<D_v\>^{s_2}\<v\>^{\ga}g\|_{L^2_v}\|\<D_v\>^{s_3}h\|_{L^2_v}, 
\end{align}
whenever 
$\ga>-d$, 
$s_2\ge 0$, 
$s_1\le s_3$, 
$s_2+s_3\ge 2s$, 
$\ga+2s-s_1>-\frac{d}{2}$,
$\ga-s_1+s_2+1>-\frac{d}{2}$,
and 
$\ga-s_1+s_2+s_3>-\frac{d}{2}$. This proves \eqref{Qcesfinal}. 

\section{Commutator estimate of the collision operator}\label{SecComm}
In this section, we present commutator estimates for the non-cutoff Boltzmann equation, which will be used in the Littlewood-Paley analysis carried out in Section \ref{SecLPkinetic}. We focus on collision operators acting only on velocity variables. Throughout this section, the Fourier transform with respect to $v$ is denoted by $\wh{f}=\F_vf=\int_{\R^d_v}f(v)e^{-2\pi iv\cdot\eta}\,dv$.

\subsection{Main commutator estimate}
We consider the self-adjoint pseudo-differential operator $P_a$ corresponding to symbol $\wh{\Xi}(a^{-1}\zeta)$ $(a\ge 1)$ given by \eqref{XiDef111}, 
\begin{align}\label{DeADef}
	P_{a}f(v)
	&=\int_{\R^d_\eta\times\R^d_{u}}e^{2\pi i\eta\cdot(v-u)}\wh{\Xi}(a^{-1}\eta)f(u)\,dud\eta
	=\int_{\R^d_{u}}a^d\Xi(a(v-u))f(u)\,du.
\end{align}

\smallskip \noindent
\textbf{Main commutator.}
By using the generalized Bobylev formula (c.f., \cite{Alexandre2012,Alexandre2011,Alexandre2010,Morimoto2016}), we intend to rewrite the collision operator in the frequency variable. 
For simplicity, since we mollify $Q(f,g)$ and $h$ with respect to $v$, we assume that $f,g,h$ are Schwartz functions and extend to $L^2_v$ (or $L^2_D$) later by standard density arguments. 
Our main target is to calculate the commutators 
\begin{align*}
	\big(P_{a}Q(f,g)-Q(f,P_{a}g),h\big)_{L^2_v},\quad\text{ and }\quad
	\big(R_rQ(f,g)-Q(f,R_rg),h\big)_{L^2_v}.
\end{align*}
By applying \eqref{CollBoltzDef}, expanding the pseudo-differential operator $P_{a}$ as in \eqref{DeADef}, writing $g=(\wh{g})^\vee,h=(\wh{h})^\vee$ whenever necessary, and splitting the cross-section $B(v-v_*,\si)=B_c(v-v_*,\si)+B_{\bar c}(v-v_*,\si)$ as in \eqref{Bbarc}, we have 
\begin{align}\label{decomMain2}\notag
	&\big(P_{a}Q(f,g)-Q(f,P_{a}g),h\big)_{L^2_v}\\
	&\notag=\int_{\R^{2d}_{v,v_*}\times\S^{d-1}_\si}\,d\si dv_*dv
	\int_{\R^d_\zeta\times\R^d_{u}}\,dud\zeta
	\int_{\R^d_\eta\times\R^d_{w}}\,dwd\eta~
	f(v_*)g(u)\ol{h(w)}e^{2\pi i\zeta\cdot(v-u)}
	\\
	&\notag\quad\times \Big[
	B_c(v-v_*,\si)\Big(\wh{\Xi}(a^{-1}\eta)-\wh{\Xi}(a^{-1}\zeta)\Big)\Big(e^{2\pi i\eta\cdot(w-v')}-e^{2\pi i\eta\cdot(w-v)}\Big)\\
	&\notag\qquad+B_{\bar{c}}(v-v_*,\si)\Big(\wh{\Xi}(a^{-1}\eta)-\wh{\Xi}(a^{-1}\zeta)\Big)\Big(e^{2\pi i\eta\cdot(w-v')}-e^{2\pi i\eta\cdot(w-v)}\Big)\Big]\\
	&=:I^{\text{large}}+I^{\text{small}}. 
\end{align}

We begin by stating the main result of this section.
\begin{Thm}[{Commutator estimate}]\label{ThmComm}
	Let $\ga+s>-\frac{d}{2}$, $\wt{l_0}\in\R$, and $\om=\om_02^{\al j+rl_1}$ with $\om_0,\al>0$ and $l_1\ge 2s-1$. Denote $P_k,R_r$ as in \eqref{DejQkDef} and $\wt{P_k},\wt{R_r}$ as in \eqref{SlightlyLarger}. 
	Then for any suitable functions $f,g,h$, any $\wt{l_0}\in\R$ and $s'\le s$, we have
\begin{align}\label{ThmCommeswt}\notag
	\big|\big(P_kR_r\Ga(f,g)-\Ga(f,P_kR_rg),\wt{R_r}\wt{P_k}h\big)_{L^2_v}\big|
	&\le C
	2^{\wt{l_0}r}\|f\|_{L^2_v}\|\<v\>^{-\wt{l_0}+\frac{\ga+2s}{2}}g\|_{L^2_v}\|\wt{R_r}\wt{P_k}h\|_{L^2_D}
	\\
	&\quad
	+C\|\<v\>^{-N}f\|_{L^2_v}\|\<v\>^{\ga}R_rg\|_{L^2_v}\|\<D_v\>^{s}\wt{R_r}\wt{P_k}h\|_{L^2_v},
\end{align}
and similarly, 
\begin{align}\label{ThmCommes}
	\notag
	&
	\big|\big(P_kR_r\Ga(f,g)-\Ga(f,P_kR_rg),P_kR_rh\big)_{L^2_v}\big|
	\le 
	C2^{r\wt{l_0}}\|f\|_{L^2_v}\|\<v\>^{-\wt{l_0}+\frac{\ga+2s}{2}}g\|_{L^2_v}\|P_kR_rh\|_{L^2_D}\\
	&\qquad\qquad\qquad\qquad\qquad+C2^{r\wt{l_0}}(\om2^k)^{s-s'-1}\|f\|_{L^2_v}\|\<v\>^{-\wt{l_0}+\frac{\ga}{2}}g\|_{L^2_v}\|\<v\>^{\frac{\ga}{2}-1}\<D_v\>^{s'}R_rh\|_{L^2_v}.
\end{align}
Here we include the factor $2^{-r\wt{l_0}}$ to compensate for the loss of the weight $\<v\>^n$ on $g$.
\end{Thm}

In the following subsections (see estimates \eqref{Ilargees}, \eqref{Ismalles}, and \eqref{CommRr}), we compute the commutator estimates for velocity and its frequency:
\begin{Lem}\label{LemmaCommColli}
Let $M\in\R$ and $\ga,s,s_1,s_2,s_3$ be such that $\ga>-d$ and
\begin{align*}
	\begin{cases}
s_2\ge 0,\ \ 
s_1\le s_3, \ \ 
-s_1+s_2+s_3\ge s,\ \  
s_2+s_3\ge 2s-1,
\\
\ga-s_1+s_2+s_3>-\frac{d}{2},\quad 
\ga-s_1+s_2+1>-\frac{d}{2}.
	\end{cases}
\end{align*}
Then for any small $\ve>0$, we have 
\begin{align}\label{Ilargees1}
	I^{\text{large}}
	&\le C\|\<v\>^{2|\ga|+7+\frac{d}{2}}f\|_{L^2_v}
	\big(a^{-s}\|\<v\>^{\ga}g\|_{L^2_v}\|\<D_v\>^sh\|_{L^2_v}
	+a^{-1}\|\<v\>^{\ga+2s-1}g\|_{L^2_v}\|\<D_v\>^{s+\ve}h\|_{L^2_v}
	\big),
\end{align}
and
\begin{align}
	\label{PaCommIsmall}
	I^{\text{small}}
	&\le C\|\<D_{v}\>^{-s_1}\<v\>^{-M}f\|_{L^2_{v}}\|\<D_{v}\>^{s_2}\<v\>^{M}g(v)\|_{L^2_{v}}\|\<D_v\>^{s_3}h\|_{L^2_v}. 
\end{align}
Consequently, we define $\wt{P_a}$ as in \eqref{SlightlyLarger}, 
which implies that $\|\<D_v\>^{s+\ve}\wt{P_a} h\|_{L^2_v}\le C a^{\ve}\|\<D_v\>^{s}\wt{P_a} h\|_{L^2_v}$.
If $s_1\in[0,s]$, $s_2=0$, $s_3=s_1+s$, and $\ga+\min\{s,1-s_1\}>-\tfrac{d}{2}$, then:
\begin{align}
	\label{PaComm}\notag
	&\big|\big(P_{a}Q(f,g)-Q(f,P_{a}g),\wt{P_a}h\big)_{L^2_v}\big|\\
	\notag
	&\ \le Ca^{-s}\|\<v\>^{4s+\frac{d+1}{2}}f\|_{L^2_v}\|(\<v\>^{\ga}g,\<v\>^{\ga+2s-1}g)\|_{L^2_v}\|\<D_v\>^s\wt{P_a}h\|_{L^2_v}\\
	&\ \ 
	+C\|\<D_{v}\>^{-s_1}\<v\>^{-M}f\|_{L^2_{v}}\|\<v\>^{M}g\|_{L^2_v}\|\<D_v\>^{s+s_1}\wt{P_a}h\|_{L^2_v}.
\end{align}
For the commutator with $R_r$, we have 
\begin{align}
	\label{RrComm}
	&\big|\big(R_rQ(f,g)-Q(f,R_rg),h\big)_{L^2_v}\big|
	\le 
	C2^{-sr}\|\<v\>^{3|\ga|+3+\frac{d+1}{2}}f\|_{L^2_v}\|\<v\>^{n+\frac{\ga+s}{2}}g\|_{L^2_v}\|\<v\>^{-n}h\|_{L^2_D}.
\end{align}
\end{Lem}
In \eqref{PaComm}, we made a rough choice of $s_1,s_2,s_3$. 
We first prove Theorem \ref{ThmComm} based on Lemma \ref{LemmaCommColli}.

\begin{proof}[Proof of Theorem \ref{ThmComm}]

To obtain the collisional estimate with a velocity exponential tail, we compute the commutator 
\begin{align}\label{commes1}
	P_{a}R_r\Ga(f,g)-\Ga(f,P_{a}R_rg)
	&\notag=P_{a}R_r\Ga(f,g)-P_{a}R_rQ(\mu^{\frac{1}{2}}f,g)
	+P_{a}R_rQ(\mu^{\frac{1}{2}}f,g)-P_{a}Q(\mu^{\frac{1}{2}}f,R_rg)\\
	&\ \ 
	+P_{a}Q(\mu^{\frac{1}{2}}f,R_rg)-Q(\mu^{\frac{1}{2}}f,P_{a}R_rg)
	+Q(\mu^{\frac{1}{2}}f,P_{a}R_rg)-\Ga(f,P_{a}R_rg).
\end{align}
For the first term, by \eqref{esD}, we simply take adjoint of $P_{a}$ and apply commutator estimate \eqref{CommGaAndQ} (for $\Ga-Q$) to deduce 
\begin{align*}\notag
	\big|\big(P_{a}R_r\Ga(f,g)-P_{a}R_rQ(\mu^{\frac{1}{2}}f,g),h\big)_{L^2_v}\big|
	&\notag= \big|\big(\Ga(f,g)-Q(\mu^{\frac{1}{2}}f,g),R_rP_{a}h\big)_{L^2_v}\big|\\
	&\notag \le C\|f\|_{L^2_v}\|\<v\>^{\frac{\ga+2s}{2}+n}g\|_{L^2_v}\|\<v\>^{-n}R_rP_{a}h\|_{L^2_D},
\end{align*}
and similarly, 
\begin{align*}
	\big|\big(Q(\mu^{\frac{1}{2}}f,P_{a}R_rg)-\Ga(f,P_{a}R_rg),h\big)_{L^2_v}\big|
	&\le C\|f\|_{L^2_v}\|\<v\>^{\frac{\ga+2s}{2}+n}P_{a}R_rg\|_{L^2_v}\|\<v\>^{-n}h\|_{L^2_D},
\end{align*}
for any $n\in\R$. 
By using \eqref{PaComm} and \eqref{RrComm}, whenever $\ga+s>-\frac{d}{2}$, we have 
\begin{align*}\notag
	\big|\big(P_{a}Q(\mu^{\frac{1}{2}}f,R_rg)-Q(\mu^{\frac{1}{2}}f,P_{a}R_rg),h\big)_{L^2_v}\big|
	&\notag
	\le Ca^{-s}\|\<v\>^{-N}f\|_{L^2_v}\|\<v\>^{\max\{\ga,\ga+2s-1\}}R_rg\|_{L^2_v}\|\<D_v\>^{s}\wt{P_a}h\|_{L^2_v}
	\\\notag
	&\quad+C\|\<v\>^{-N}f\|_{L^2_{v}}\|\<v\>^{\ga}R_rg\|_{L^2_v}\|\<D_v\>^{s}\wt{P_a}h\|_{L^2_v},
\end{align*}
and 
\begin{align*}\notag
	&\big|\big(P_{a}R_rQ(\mu^{\frac{1}{2}}f,g)-P_{a}Q(\mu^{\frac{1}{2}}f,R_rg),h\big)_{L^2_v}\big|
	\le C2^{-sr}\|\<v\>^{-N}f\|_{L^2_v}\|\<v\>^{n+\frac{\ga}{2}}g\|_{L^2_v}\|\<v\>^{-n}P_ah\|_{L^2_D}, 
\end{align*}
for any $N\ge 0$.  
Now, let $P_a=P_k$, where $a=\om_02^{rl_1+k}$ and $P_k$ is given in \eqref{DejQkDef}.
Substituting the above estimates into \eqref{commes1} and taking suitable $n\in\R$, we have 
\begin{align*}
	&\big|\big(P_kR_r\Ga(f,g)-\Ga(f,P_kR_rg),\wt{R_r}\wt{P_k}h\big)_{L^2_v}\big|\\
	&\ \le C
	\|f\|_{L^2_v}\|\<v\>^{\frac{\ga+2s}{2}+n}P_kR_rg\|_{L^2_v}\|\<v\>^{-n}\wt{R_r}\wt{P_k}h\|_{L^2_D}\\
	&\quad+Ca^{-s}\|\<v\>^{-N}f\|_{L^2_v}\|\<v\>^{\max\{\ga,\ga+2s-1\}}R_rg\|_{L^2_v}\|\<D_v\>^{s}\wt{R_r}\wt{P_k}h\|_{L^2_v}
	\\\notag
	&\quad+C\|\<v\>^{-N}f\|_{L^2_{v}}\|\<v\>^{\ga}R_rg\|_{L^2_v}\|\<D_v\>^{s}\wt{R_r}\wt{P_k}h\|_{L^2_v}\\
	&\quad+C2^{-sr}\|\<v\>^{-N}f\|_{L^2_v}\|\<v\>^{n+\frac{\ga}{2}}g\|_{L^2_v}\|\<v\>^{-n}\wt{R_r}\wt{P_k}h\|_{L^2_D}\\
	&\ \le C
	2^{-nr}\|f\|_{L^2_v}\|\<v\>^{\frac{\ga+2s}{2}+n}g\|_{L^2_v}\|\wt{R_r}\wt{P_k}h\|_{L^2_D}\\
	&\quad+C\|\<v\>^{-N}f\|_{L^2_v}\|\<v\>^{\ga}R_rg\|_{L^2_v}\|\<D_v\>^{s}\wt{R_r}\wt{P_k}h\|_{L^2_v},
\end{align*}
for any $n\in\R$ and $N\ge 0$, where we let $l_1\ge 2s-1$. 
Simiarly, for $P_kR_rh$, by suitably choosing $n$, utilizing the commutator estimate \eqref{dissUpper1b} with $\|\<v\>^{\fr\ga2}\<D_v\>^s(\cdot)\|_{L^2_v}\le C\|\cdot\|_{L^2_D}$, we could extract the support properties of $\wt{R_r}$ and $\wt{P_k}$ to deduce \eqref{ThmCommes}. 
This completes the proof of Theorem \ref{ThmComm}. 
\end{proof}

\subsection{Commutator estimate of \texorpdfstring{$[P_a,Q]$}{[Pa,Q]} with large relative velocity}
\label{SecLargeRelative}
In this part, we deal with $I^{\text{large}}$, defined in \eqref{decomMain2}, corresponding to the case where the relative velocity is away from the singularity at $0$. We decompose $I^{\text{large}}$ as 
\begin{align}\label{Ilarge}\notag
	I^{\text{large}}&=\int_{v,v_*,\si,\zeta,u,\eta,w}
	\Th_{\bar{c}}(v-v_*)b(\cos\th)f(v_*)g(u)\ol{h(w)}e^{2\pi i\zeta\cdot(v-u)}
	\\
	&\quad\notag\qquad\times\Big(\wh{\Xi}(a^{-1}\eta)-\wh{\Xi}(a^{-1}\zeta)\Big)\Big(e^{2\pi i\eta\cdot(w-v')}-e^{2\pi i\eta\cdot(w-v)}\Big)\\
	&\notag\quad=\int_{v,v_*,\si,\zeta,u,\eta,w}
	b(\cos\th)f(v_*)g(u)\ol{h(w)}e^{2\pi i\zeta\cdot(v-u)}\\
	&\notag\qquad\times
	\Big\{\Big(\Th_{\bar{c}}(v-v_*)\wh{\Xi}(a^{-1}\eta)-\Th_{\bar{c}}(u-v_*)\wh{\Xi}(a^{-1}\zeta)\Big)e^{2\pi i\eta\cdot(w-v')}\\
	&\notag\qquad\quad
	+
	\Big(\Th_{\bar{c}}(u-v_*)\wh{\Xi}(a^{-1}\zeta)-\Th_{\bar{c}}(v-v_*)\wh{\Xi}(a^{-1}\eta)\Big)e^{2\pi i\eta\cdot(w-v)}\\
	&\notag\qquad\quad
	+\Big(\Th_{\bar{c}}(u-v_*)-\Th_{\bar{c}}(v-v_*)\Big)\wh{\Xi}(a^{-1}\zeta)\Big(e^{2\pi i\eta\cdot(w-v')}-e^{2\pi i\eta\cdot(w-v)}\Big)\Big\}\\
	&\quad=:I^{l}_1+I^{l}_2+I^{l}_3,   
\end{align}
where for each term, we will integrate the corresponding ``good'' variable. 
In $I^{l}_1$, note that for any $v$, (since $g(u)\Th_{\bar{c}}(u-v_*)$ is integrable over $u$)
\begin{align}\label{ThetavThetau}
	&\int_{\zeta,u}g(u)\Th_{\bar{c}}(v-v_*)e^{2\pi i\zeta\cdot(v-u)}\,d\zeta du
	=g(v)\Th_{\bar{c}}(v-v_*)
	=\int_{\zeta,u}g(u)\Th_{\bar{c}}(u-v_*)e^{2\pi i\zeta\cdot(v-u)}\,d\zeta du.
\end{align}
Then we integrate with respect to $v$ and subsequently in $\zeta$, utilizing \eqref{deltafunction}, to obtain 
\begin{align*}
	I^{l}_1&=\int_{v,v_*,\si,\zeta,u,\eta,w}
	b(\cos\th)f(v_*)\Th_{\bar{c}}(u-v_*)g(u)\ol{h(w)}e^{2\pi i\zeta\cdot(v-u)}
	\Big(\wh{\Xi}(a^{-1}\eta)-\wh{\Xi}(a^{-1}\zeta)\Big)e^{2\pi i\eta\cdot(w-v')}\\
	&=\int_{v_*,\si,u,\eta}b\big(\frac{\eta}{|\eta|}\cdot\si\big)f(v_*)\Th_{\bar{c}}(u-v_*)g(u)\ol{\wh{h}(\eta)}
	\Big(\wh{\Xi}(a^{-1}\eta)-\wh{\Xi}(a^{-1}\eta^+)\Big)e^{-2\pi i\eta^+\cdot u-2\pi i\eta^-\cdot v_*}
\end{align*} 
Rewriting $f(v_*)=\<v_*\>^{M}f(v_*)\times\<v_*\>^{-M}$ where $M\ge\frac{d+1}{2}$, and integrating over $u,v_*$, we have 
\begin{align*}
	I^l_1
	&=\int_{\eta_*,\si,\eta}b\big(\frac{\eta}{|\eta|}\cdot\si\big)
	\ol{\wh{h}(\eta)}\Big(\wh{\Xi}(a^{-1}\eta)-\wh{\Xi}(a^{-1}\eta^+)\Big)\\
	&\qquad\qquad\times\F_{v_*}\big(\<v_*\>^{M}f(v_*)\big)(\eta_*)\F_{v_*,u}\big(\<v_*\>^{-M}\Th_{\bar{c}}(u-v_*)g(u)\big)(\eta_*-\eta^-,\eta^+). 
\end{align*}
By Taylor expansion \eqref{expanfetaplue} (note that, by \eqref{bobylevangle}, $|\eta^++\tau(\eta-\eta^+)|\approx|\eta|$) and noting that $\wh{\Xi}(a^{-1}\zeta)$ is a radial function whose derivative has compact support $\{\frac{6}{7}a\le |\zeta|\le 2a\}$, we have 
\begin{align*}
	|I^l_1|
	&\le \int_{\si,\eta,\eta_*}b\big(\frac{\eta}{|\eta|}\cdot\si\big)
	\big|\ol{\wh{h}(\eta)}\big|\min\big\{1,\,\sin^2\frac{\th}{2},\,a^{-2s}|\zeta|^{2s}\sin^2\frac{\th}{2}\big\}\\
	&\quad \times
	\big|\F_{v_*}\big(\<v_*\>^{M}f(v_*)\big)(\eta_*)\big|\big|\F_{v_*,u}\big(\<v_*\>^{-M}\Th_{\bar{c}}(u-v_*)g(u)\big)(\eta_*-\eta^-,\eta^+)\big|\\
	&\le C\Big(\int_{\si,\eta,\eta_*}b\big(\frac{\eta}{|\eta|}\cdot\si\big)\big|\F_{v_*}\big(\<v_*\>^{M}f(v_*)\big)(\eta_*)\big|^2\big|\ol{\wh{h}(\eta)}\big|^2
	a^{-2s}|\zeta|^{2s}\sin^2\frac{\th}{2}\Big)^{\frac{1}{2}}\\
	&\ \ \times\Big(\int_{\si,\eta,\eta_*}b\big(\frac{\eta}{|\eta|}\cdot\si\big)\big|\F_{v_*,u}\big(\<v_*\>^{-M}\Th_{\bar{c}}(u-v_*)g(u)\big)(\eta_*-\eta^-,\eta^+)\big|^2
	\sin^2\frac{\th}{2}\Big)^{\frac{1}{2}}.
\end{align*}
Note that $\F_{v_*,u}\big(\<v_*\>^{-M}\Th_{\bar{c}}(u-v_*)g(u)\big)(\eta_*-\eta^-,\eta^+)$ has a good decay $\<\eta_*-\eta^-\>^{-N}$. Thus, by the 
trivial angular integral $\int_{\si}b\th^2<\infty$, whenever $M\ge\frac{d+1}{2}$, we have 
\begin{align}\label{esLargeI1}\notag
	|I^l_1|
	&\le C\Big(\int_{v_*,v}|
	\<v_*\>^{M}f(v_*)|^2a^{-2s}|\<D_v\>^sh(v)|^2\Big)^{\frac{1}{2}}\Big(\int_{v_*,u}\big|\<v_*\>^{-M}\Th_{\bar{c}}(u-v_*)g(u)\big|^2\Big)^{\frac{1}{2}}\\
	&\le Ca^{-s}\|
	\<v\>^{\frac{d+1}{2}}f\|_{L^2_v}\|\<v\>^{\ga}g\|_{L^2_v}\|\<D_v\>^sh\|_{L^2_v}.
\end{align}
For the term $I^{l}_2$ in \eqref{Ilarge}, integrating over $v$, and utilizing \eqref{ThetavThetau} and \eqref{deltafunction}, we have 
{\small\begin{align}\label{esLargeI2}\notag
	I^{l}_2
	&\notag=\int_{v_*,\si,\zeta,u,\eta,w}
	b(\si_1)f(v_*)g(u)\ol{h(w)}\Th_{\bar{c}}(u-v_*)\Big(\wh{\Xi}(a^{-1}\zeta)-\wh{\Xi}(a^{-1}\eta)\Big)\\
	&\qquad\qquad\qquad\times e^{-2\pi i\zeta\cdot u+2\pi i\eta\cdot w}\de(\eta-\zeta)
	=0. 
\end{align}}
For the last term $I^l_{3}$ in \eqref{Ilarge}, applying the Taylor expansion to $\Th_{\bar{c}}$, we have
\begin{align*}
	I^{l}_{3}
	&=\int_{v,v_*,\si,u}
	b(\cos\th)f(v_*)g(u)\Big(\Th_{\bar{c}}(u-v_*)-\Th_{\bar{c}}(v-v_*)\Big)
	\\
	&\notag\qquad\qquad\times\Big(\ol{h(v')}a^d\Xi(a(v-u))-\ol{h(v)}a^d\Xi(a(v-u))\Big)\\
	&=\int_{v,v_*,\si,u}
	b(\cos\th)f(v_*)g(u)\int^1_0\wt{\Th}\,d\de~
	\big(\ol{H(v')}-\ol{H(v)}\big)\psi_1(v-u),
\end{align*}
where, for brevity, we denote  
\begin{align*}
	&H(z):=\ol{h(z)}a^{-1},\quad
	\psi_1(v-u)=a^{d+1}(v-u)\Xi(a(v-u))\\
	&\notag\wt{\Th}:=\na\Th_{\bar{c}}(\de(v-v_*)+(1-\de)(u-v_*)),, 
\end{align*}
which satisfies
\begin{align}\label{psi1es22}
	|\psi_1(v-u)|\le C_Na^{d}\<a(v-u)\>^{-N},\quad |\wt{\Th}|\le C_\ga\<u\>^{\ga-1}\<v-u\>^{|\ga-1|}\<v_*\>^{|\ga-1|}.
\end{align}
To reveal the fractional-derivative structure, we perform a non-dilated dyadic decomposition of $H$ in the velocity-frequency variable, i.e.,
\begin{align}\label{DecomHk}
	H(v)=\sum_{k\ge 0}P_kH(v)=:\sum_{k\ge 0}H_k,\ \text{ where }\ P_kH=\int_{\R^d_\zeta\times\R^d_u}e^{2\pi i\zeta\cdot(v-u)}\wh{\Psi_{k}}(\zeta)H(u)\,dud\zeta.  
\end{align}
For each $k\ge0$, we split the integration sphere $\S^{d-1}_\si$ into two parts with corresponding to $2^k$: 
\begin{align}\label{Il32Def}
	&\{\th\le 2^{-k}|v-v_*|^{-1}\}\cup\{\th> 2^{-k}|v-v_*|^{-1}\},
	\ 
	\text{ and split $I^{l}_{3}=\sum_{k\ge 0}(I^{l,s}_{3,k}+I^{l,l}_{3,k})$ accordingly}.
\end{align}
For the term $I_{3,k}^{l,s}$, applying Taylor's expansion to $H_k$ yields 
\begin{align*}
	I_{3,k}^{l,s}&:=\int_{v,v_*,\si,u,\de}\1_{\th\le 2^{-k}|v-v_*|^{-1}}bf(v_*)g(u)
	\psi_1(v-u)\wt{\Th}
	\cdot\Big(H_k(v')-H_k(v)\Big)\\
	&=\int_{v,v_*,\si,u,\de}\Big(\cdots\Big)
	\Big(\na_vH_k(v)\cdot(v'-v)+\int^1_0(1-\tau)\na_v^2H_k(v+\tau(v'-v)):(v'-v)^{\otimes 2}\,d\tau\Big)\\
	&
	=I_{3,k,1}^{l,s}+I_{3,k,2}^{l,s}.
\end{align*}
For $I_{3,k,1}^{l,s}$, by \eqref{vprivth}, i.e., $v'-v =\sin^2\frac{\th}{2}(v_*-v)+\frac{1}{2}|v-v_*|\sin\th\omega$, \eqref{angularintegral}, and the symmetry with respect to $\omega$, one has 
\begin{align*}
	I_{3,k,1}^{l,s}
	&\le C\int_{v,v_*,\si,u,\de}\1_{\th\le 2^{-k}|v-v_*|^{-1}}b\sin^2\frac{\th}{2}|v-v_*||f(v_*)g(u)\psi_1(v-u)\wt{\Th}\na_vH_k(v)| + 0\\
	&\le C(2^k)^{2s-2}\int_{v,v_*,u}|v-v_*|^{-1+2s}\<u\>^{\ga-1}\<v_*\>^{|\ga-1|}a^{d}\<a(v-u)\>^{-N}|f(v_*)g(u)||\na_vH_k(v)|, 
\end{align*}
where, in the last inequality, we used \eqref{psi1es22}.
Applying \eqref{gafg}, i.e., $\int_{\R^{2d}_{v,v_*}}|v-v_*|^{\ga}|f(v_*)g(v)|\,dv_*dv\le C\|\<v_*\>^{|\ga|+\frac{d+1}{2}}f\|_{L^2_v}\|\<v\>^{\ga}g\|_{L^1_v}$, to the terms involving singular $|v-v_*|^{-1+2s}$, and using $\<v\>^{k}\le C_k\<u\>^{k}\<a(v-u)\>^{|k|}$ ($k\in\R$) whenever $a\ge 1$, we have 
\begin{align*}
	|I_{3,k,1}^{l,s}|
	&\le C(2^k)^{2s-2}\|\<v_*\>^{|\ga-1|+|-1+2s|+\frac{d+1}{2}}f(v_*)\|_{L^2_{v_*}}\\
	&\qquad\times
	\Big\|\int_{u}\<v\>^{-1+2s}a^{d}\<a(v-u)\>^{-N}|\<u\>^{\ga-1}g(u)||\na_vH_k(v)|\,du\Big\|_{L^1_v}\\
	&\le C(2^k)^{2s-2}\|\<v\>^{|\ga-1|+|-1+2s|+\frac{d+1}{2}}f(v_*)\|_{L^2_{v}}\|\<v\>^{\ga+2s-2}g\|_{L^2_v}\|\na_vH_k\|_{L^2_v},
\end{align*}
where we used convolution-type estimate $\|G(v)(H*\<\cdot\>^{-N})(v)\|_{L^1_v}\le\|G(v)\|_{L^2_v}\|H\|_{L^2_v}\|\<\cdot\>^{-N}\|_{L^1_v}$. 
The second-order term $I_{3,k,2}^{l,s}$ is similar. Using $|v'-v|=|v-v_*|\sin\frac{\th}{2}$ and H\"older's inequality, we have 
\begin{align*}
	|I_{3,k,2}^{l,s}|
	&\le C\Big(\int_{v,v_*,\si,u,\de,\tau}\1_{\th\le 2^{-k}|v-v_*|^{-1}}b\sin^2\frac{\th}{2}|v-v_*|^{2}\frac{|f(v_*)||g(u)|^2}{\<v+\tau(v'-v)\>^{-2s}}
	a^{d}\<a(v-u)\>^{-N}|\wt{\Th}|^2
	\Big)^{\frac{1}{2}}\\
	&\quad\times\Big(\int_{v,v_*,\si,u,\de,\tau}\1_{\th\le 2^{-k}|v-v_*|^{-1}}b\sin^2\frac{\th}{2}|v-v_*|^{2}|f(v_*)|
	a^{d}\<a(v-u)\>^{-N}
	\frac{|\na_v^2H_k(v+\tau(v'-v))|^2}{\<v+\tau(v'-v)\>^{2s}}\Big)^{\frac{1}{2}}\\
	&=:(I^{l,s}_{3,2,g})^{\frac{1}{2}}\times(I^{l,s}_{3,2,h})^{\frac{1}{2}}.
\end{align*} 
For the $g$ term $I^{l,s}_{3,2,g}$, applying \eqref{psi1es22}, $\<v\>^{2s}\le C\<u\>^{2s}\<v-u\>^{2s}$, and $\<v+\tau(v'-v)\>^k\le C\<u\>^k\<v-u\>^{|k|}\<v_*\>^{|k|}$ via \eqref{distancevpriv} and \eqref{equivlam}, and integrating over $\si$ via \eqref{angularintegral}, we have 
\begin{align*}
	I^{l,s}_{3,2,g}
	&\le C(2^k)^{2s-2}\int_{v,v_*,u}|v-v_*|^{2s}\<u\>^{2\ga-2}\<v_*\>^{2|\ga-1|+2s}\frac{|f(v_*)||g(u)|^2}{\<u\>^{-2s}}
	a^{d}\<a(v-u)\>^{-N},\\
	&\le C(2^k)^{2s-2}\|\<v_*\>^{2|\ga-1|+4s+\frac{d+1}{2}}f(v_*)\|_{L^2_{v_*}}\int_{v,u}\<v\>^{2s}\<u\>^{2\ga-2}\frac{|g(u)|^2}{\<u\>^{-2s}}a^{d}\<a(v-u)\>^{-N}\\
	&\le C(2^k)^{2s-2}\|\<v\>^{2|\ga-1|+4s+\frac{d+1}{2}}f\|_{L^2_v}\|\<v\>^{\ga+2s-1}g\|_{L^2_v}^2. 
\end{align*}
For the $h$ term $I^{l,s}_{3,2,h}$, we write $z=v+\tau(v'-v)$ for brevity. Integrating over $u$, using \eqref{distancevpriv}, applying the change of variables $v\mapsto z=v+\tau(v'-v)$ as in \eqref{regularz}, and finally integrating $\si$ by \eqref{angularintegral}, we have 
\begin{align*}
	I^{l,s}_{3,2,h}
	&\le C(2^k)^{2s-2}\int_{v,v_*,\si,\tau}\1_{\th\le 2^{-k}|v-v_*|^{-1}}b\sin^2\frac{\th}{2}|z-v_*|^{2}|f(v_*)|
	\frac{|\na_v^2H_k(z)|^2}{\<z\>^{2s}}\\
	&\le 
	C(2^k)^{2s-2}\int_{z,v_*}|z-v_*|^{2s}|f(v_*)|
	\frac{|\na_v^2H_k(z)|^2}{\<z\>^{2s}}\\
	&\le C(2^k)^{2s-2}\|\<v_*\>^{2s+\frac{d+1}{2}}f(v_*)\|_{L^2_{v_*}}
	\|\na_v^2H_k\|_{L^2_v}^2. 
\end{align*}

For $I^{l,l}_{3}$, on the support $\{\th> 2^{-k}|v-v_*|^{-1}\}$, by \eqref{psi1es22} and H\"older's inequality, we compute
\begin{align*}
	|I^{l,l}_{3}|
	&:=\Big|\int_{v,v_*,\si,u,\de}\1_{\th> 2^{-k}|v-v_*|^{-1}}bf(v_*)g(u)
	\psi_1(v-u)\wt{\Th}
	\cdot(H_k(v')-H_k(v))\Big|\\
	&\notag\le C_N\Big(\int_{v,v_*,\si,u}\1_{\th> 2^{-k}|v-v_*|^{-1}}b\frac{|f(v_*)||g(u)|^2}{\<v\>^{-2s}}a^{d}\<a(v-u)\>^{-N}\<u\>^{2\ga-2}\<v_*\>^{2|\ga-1|}\Big)^{\frac{1}{2}}\\
	&\notag\qquad\times\Big(\int_{v,v_*,\si,u}\1_{\th> 2^{-k}|v-v_*|^{-1}}b|f(v_*)|a^{d}\<a(v-u)\>^{-N}\frac{|H_k(v)|^2+|H_k(v')|^2}{\<v\>^{2s}}\Big)^{\frac{1}{2}}\\
	&=:C_N\big(I_{3,k}^{l,g}\big)^{\frac{1}{2}}\times\big(I_{3,k}^{l,h}\big)^{\frac{1}{2}}.
\end{align*}
For the $g$ term, as before, 
applying angular integral \eqref{angularintegral} deduces  
\begin{align*}
	I_{3,k}^{l,g}
	&\le 
	C\int_{v,v_*,\si,u}\1_{\th> 2^{-k}|v-v_*|^{-1}}
	b\frac{|f(v_*)||g(u)|^2}{\<u\>^{-2s}}a^d\<a(v-u)\>^{-N}\<u\>^{2\ga-2}\<v_*\>^{2|\ga-1|}\\
	&\le C(2^k)^{2s}\int_{v,v_*,u}|v-v_*|^{2s}\<v_*\>^{2|\ga-1|}|f(v_*)|\<u\>^{2\ga+2s-2}|g(u)|^2a^d\<a(v-u)\>^{-N}\\
	&\le C(2^k)^{2s}\|\<v_*\>^{2|\ga-1|+2s+\frac{d+1}{2}}f(v_*)\|_{L^2_v}\|\<u\>^{\ga+2s-1}g(u)\|_{L^2_u}^2. 
\end{align*}
For the $h$ term, integrating over $u$, applying \eqref{distancevpriv}, using regular change of variables \eqref{regular} for $v'$ term, and then utilizing angular integral \eqref{angularintegral}, we have 
\begin{align*}
	I_{3,k}^{l,h}
	&\le C_N\int_{v,v_*,\si}b|f(v_*)|\Big(\1_{\th> 2^{-k}|v-v_*|^{-1}}\frac{|H_k(v)|^2}{\<v\>^{2s}}+\1_{\th> 2^{-k}|v-v_*|^{-1}}\frac{\<v_*\>^{2s}|H_k(v')|^2}{\<v'\>^{2s}}\Big)\\
	&\le C_N(2^k)^{2s}\int_{v,v_*}|v-v_*|^{2s}\<v_*\>^{2s}\<v\>^{-2s}|f(v_*)||H_k(v)|^2\\
	&\le C_N(2^k)^{2s}\|\<v_*\>^{4s+\frac{d+1}{2}}f(v_*)\|_{L^2_v}\|H_k(v)\|_{L^2_v}^2. 
\end{align*}

Lastly, we collect the above estimates for $I^l_{3}$ into \eqref{Il32Def} and sum over $k\ge0$ using \eqref{DecomHk}. For any $\ve>0$, thanks for Littlewood-Paley theorem ($\sum_{k\ge 0}(2^k)^s\|H_k\|_{L^2_v}\le C\|\<D_v\>^{s+\ve}H\|_{L^2_v}$), we have 
\begin{align}\label{esLargeI3}
	|I^l_{3}|
	&\le Ca^{-1}\|\<v\>^{2|\ga-1|+3s+\frac{d+3}{2}}
	f\|_{L^2_v}\|\<v\>^{\ga+2s-1}g\|_{L^2_v}\|\<D_v\>^{s+\ve}h\|_{L^2_v}.
\end{align}
Hence, substituting the estimates $I^l_1$ \eqref{esLargeI1}, $I^l_2$ \eqref{esLargeI2}, and $I^l_{3}$ \eqref{esLargeI3} into \eqref{Ilarge}, we obtain  
\begin{align}\label{Ilargees}\notag
	I^{\text{large}}
	&\le C\|\<v\>^{2|\ga-1|+3s+\frac{d+3}{2}}f\|_{L^2_v}
	\big(a^{-s}\|\<v\>^{\ga}g\|_{L^2_v}\|\<D_v\>^sh\|_{L^2_v}\\
	&\qquad\qquad\qquad\qquad\quad
	+a^{-1}\|\<v\>^{\ga+2s-1}g\|_{L^2_v}\|\<D_v\>^{s+\ve}h\|_{L^2_v}
	\big),
\end{align}
for any $\ve>0$. 
This implies the estimate of the large relative velocity $I^{\text{large}}$ part in \eqref{Ilargees1}.

\subsection{Commutator estimate of \texorpdfstring{$[P_a,Q]$}{[Pa,Q]} with small relative velocity}\label{SecCOmmSmall}
In this part, we will deal with the term $I^{\text{small}}$ in \eqref{decomMain2} corresponding to cross-section $\Th_{c}(v-v_*)=|v-v_*|^\ga\phi(v-v_*)\chi_{|v-v_*|\le 1}$ that has a singularity near $0$. As in \eqref{Gc}, we utilize 
\begin{align*}
	G(v_*,v)&:=\chi_{|v-v_*|\le 1}f(v_*)g(v). 
\end{align*}
Since $\Th_c$ is integrable, we can rewrite $\Th_{c}(v-v_*)=\int_{\zeta_*,u_*}\Th_{c}(u_*-v_*)e^{2\pi i\zeta_*\cdot (v-u_*)}$, and as in \eqref{QcBobylev} and \eqref{FourGvstar}, using the identities from \eqref{deltafunction} and integrating over $v$, we have 
\begin{align*}
	I^{\text{small}}
	&\notag=\int_{v,v_*,\si,\eta,u,\zeta,u_*,\zeta_*}
	b(\cos\th)\Th_c(u_*-v_*)G(v_*,u)\ol{\wh{h}(\eta)}
	e^{2\pi i\zeta\cdot(v-u)}e^{2\pi i\zeta_*\cdot (v-u_*)}\\
	&\quad\notag\qquad\times\Big(\wh{\Xi}(a^{-1}\eta)-\wh{\Xi}(a^{-1}\zeta)\Big)\Big(e^{-2\pi i\eta\cdot v'}-e^{-2\pi i\eta\cdot v}\Big)\\
	&\notag=\int_{v_*,\si,\eta,u,\zeta,u_*,\zeta_*}
	b\big(\frac{\eta}{|\eta|}\cdot\si\big)\Th_c(u_*-v_*)G(v_*,u)\ol{\wh{h}(\eta)}
	e^{-2\pi i\zeta\cdot u}e^{-2\pi i\zeta_*\cdot u_*}\\
	&\quad\notag\qquad\times\Big(\wh{\Xi}(a^{-1}\eta)-\wh{\Xi}(a^{-1}\zeta)\Big)\Big(e^{-2\pi i\eta^-\cdot v_*}\de(\zeta+\zeta_*-\eta^+)-\de(\zeta+\zeta_*-\eta)\Big).
\end{align*}
Integrating over $\zeta_*$ and then $u_*,v_*,u$, by change of variables $\zeta\mapsto \eta_*=\eta-\zeta$ (with fixed $\eta$), we have 
\begin{align}\notag
	I^{\text{small}}
	\label{Ismall}\notag
	&\notag=\int_{\eta_*,\eta,\si}
	b
	\wh{G}(\eta_*,\eta-\eta_*)
	\ol{\wh{h}(\eta)}
	\Big(\wh{\Th_c}(\eta_*-\eta^-)
	-\wh{\Th_c}(\eta_*)\Big)\Big(\wh{\Xi}(a^{-1}\eta)-\wh{\Xi}(a^{-1}(\eta-\eta_*))\Big)\\
	&=:J_1+J_2+J_3,
\end{align}
where we denote $\wh{G}(\eta_*,\eta)=\F_{v_*,v}G(\eta_*,\eta)$ while $I^{\text{small}}$ is further decomposed into three parts: using 
\begin{align*}
	\wh{\Th_c}(\eta_*-\eta^-)-\wh{\Th_c}(\eta_*)
	&=-\eta^-\cdot\na_\eta\wh{\Th_c}(\eta_*)+\int^1_0(1-\de)\na^2_\eta\wh{\Th_c}(\eta_*-\de\eta^-):\eta^-\otimes\eta^-\,d\de, 
\end{align*}
and rewriting $\Big(\cdots\Big)=b\wh{G}(\eta_*,\eta-\eta_*)\ol{\wh{h}(\eta)}\big\{\wh{\Xi}(a^{-1}\eta)-\wh{\Xi}(a^{-1}(\eta-\eta_*))\big\}$ temporarily, 
\begin{align}\notag\label{J123aaa}
	J_{1}&=\int_{\eta,\eta_*,\si}\1_{|\eta^-|\le\frac{1}{2}\<\eta_*\>}
	(-\eta^-)\cdot\na_\eta\wh{\Th_c}(\eta_*)\Big(\cdots\Big)
	\\
	J_{2}\notag&=\int_{\eta,\eta_*,\si}\1_{|\eta^-|\le\frac{1}{2}\<\eta_*\>}
	\int^1_0(1-\de)\na^2_\eta\wh{\Th_c}(\eta_*-\de\eta^-):\eta^-\otimes\eta^-\,d\de~\Big(\cdots\Big),\\
	J_{3}&=\int_{\eta,\eta_*,\si}\1_{|\eta^-|\ge\frac{1}{2}\<\eta_*\>}
	\big(\wh{\Th_c}(\eta_*-\eta^-)-\wh{\Th_c}(\eta_*)\big)\Big(\cdots\Big),
\end{align}
Using \eqref{nakThc}, we have $|\na^k\wh{\Th_c}(\eta)|\le C_{N,d,\ga}\<\eta\>^{-d-\ga-k}$, and 
on the support of $\1_{|\eta^-|\le\frac{1}{2}\<\eta_*\>}$, we have 
\begin{align*}
	|\na^2_\eta\wh{\Th_c}(\eta_*-\de\eta^-)|
	&\le C\<\eta_*-\de\eta^-\>^{-d-\ga-2}
	\le C\<|\eta_*|-\de|\eta^-|\>^{-d-\ga-2}\le C\<\eta_*\>^{-d-\ga-2}. 
\end{align*}
For the term $J_{1}$, we divide $\eta^-$ as in \eqref{decompositionofetaminus}: 
	$\eta^-
	=\frac{|\eta|}{2}\Big(\big(\frac{\eta}{|\eta|}\cdot\si\big)\frac{\eta}{|\eta|}-\si\Big)+\Big(1-\big(\frac{\eta}{|\eta|}\cdot\si\big)\Big)\frac{\eta}{2}$. 
Using decomposition $\si=\cos\th\,\mathbf{j}+\sin\th\,\omega$ in \eqref{decomsitoj} 
and the symmetry of $\omega$ in \eqref{symmetaomega}, the integral corresponding the first term $\frac{|\eta|}{2}\big((\frac{\eta}{|\eta|}\cdot\si)\frac{\eta}{|\eta|}-\si\big)$ vanishes (note that, by \eqref{bobylevangle}, $|\eta^-|=|\eta|\sin\frac{\th}{2}$ is independent of $\om$). For the second term, we have 
	$\big(1-(\frac{\eta}{|\eta|}\cdot\si)\big)\frac{\eta}{2}
	\equiv(1-\cos\th)\frac{\eta}{2}
	=\eta\sin^2\frac{\th}{2}$. 
For $J_{1}$ and $J_{2}$, we apply the Taylor expansion
\begin{align*}
    \wh{\Xi}(a^{-1}\eta)-\wh{\Xi}(a^{-1}(\eta-\eta_*))
    =\int^1_0 \nabla \wh{\Xi} (a^{-1} (\eta-\tau \eta_*))\,d\tau \cdot a^{-1}\eta_*
    \quad \text{when}\ |\eta|>2\langle \eta_*\rangle.
\end{align*}
To calculate the difference $\wh{\Xi}(a^{-1}\eta)-\wh{\Xi}(a^{-1}(\eta-\eta_*))$, we split $\1=\1_{|\eta|\le 2\langle \eta_*\rangle}+\1_{|\eta|> 2\langle \eta_*\rangle}$, and on their supports, we will utilize 
\begin{align}\label{Xitoetastar}
	\begin{aligned}
	|\wh{\Xi}(a^{-1}\eta)-\wh{\Xi}(a^{-1}(\eta-\eta_*))|\1_{|\eta|\le 2\langle \eta_*\rangle}&\le 2\1_{|\eta|\le 2\langle \eta_*\rangle},\\
	\Big|\int^1_0 \nabla \wh{\Xi} (a^{-1} (\eta-\tau \eta_*))\,d\tau \cdot a^{-1}\eta_*\Big|\1_{|\eta|> 2\langle \eta_*\rangle}
	&\le C\frac{\<\eta_*\>}{\<\eta\>}\1_{|\eta|>2\langle \eta_*\rangle},
	\end{aligned}
\end{align}
where we used the fact that, on the support of $\1_{|\eta|> 2\langle \eta_*\rangle}$, 
	$|\eta-\tau\eta_*|\ge |\eta|-|\eta_*|>\frac{|\eta|}{2}>1$, for $\tau\in(0,1)$. 
Substituting the above estimates of $|\wh{\Xi}(a^{-1}\eta)-\wh{\Xi}(a^{-1}(\eta-\eta_*))|$ and $|\eta^-|=|\eta|\sin\frac{\th}{2}$ into $J_1$ and $J_2$, 
\begin{align*}\notag
    J_{1}+J_2&\le 
		C\int_{\eta,\eta_*,\si,\tau}\
		\1_{|\eta^-|\le\frac{1}{2}\<\eta_*\>} b\sin^2\frac{\theta}{2}\Big(|\eta||\na_\eta\wh{\Th_c}(\eta_*)|
		+|\eta|^2|\na_\eta^2\wh{\Th_c}(\eta_*-\tau\eta^-)|\Big)
		\\&\qquad\qquad\qquad\times|\wh{G}(\eta_*,\eta-\eta_*)\ol{\wh{h}(\eta)}|\Big(\1_{|\eta|\le 2\langle \eta_*\rangle}
		+\frac{\<\eta_*\>}{\<\eta\>}\1_{|\eta|>2\langle \eta_*\rangle}\Big).
\end{align*}
Thus, together with angular integral \eqref{angulareta} and estimate of $\wh{\Th_c}$ from \eqref{nakThc}, we can obtain 
\begin{align*}
	|J_{1}|+|J_{2}|
	&\le C\int_{{\eta,\eta_*}}\min\{1,|\eta|^{-2+2s}\<\eta_*\>^{2-2s}\}\Big(\frac{|\eta|}{\<\eta_*\>^{d+\ga+1}}+\frac{|\eta|^2}{\<\eta_*\>^{d+\ga+2}}\Big)
		\\&\qquad\times\Big|\wh{G}(\eta_*,\eta-\eta_*)\wh{h}(\eta)\Big|\Big(\1_{|\eta|\le 2\langle \eta_*\rangle}
		+\frac{\<\eta_*\>}{\<\eta\>}\1_{|\eta|>2\langle \eta_*\rangle}\Big)
	=:J'_{11}+J'_{12},
\end{align*}
where we split $J'_{11}+J'_{12}$ according to $\1_{|\eta|\le 2\langle \eta_*\rangle}
		+\frac{\<\eta_*\>}{\<\eta\>}\1_{|\eta|>2\langle \eta_*\rangle}$. 
For $J'_{11}$, applying H\"older's inequality and \eqref{esGS1} to $\wh{G}$, for any $M\in\R$, we have 
\begin{align*}
	J'_{11}
	&\le 
	C\int_{\eta,\eta_*}\frac{\1_{|\eta|\le 2\<\eta_*\>}\<\eta\>}{\<\eta_*\>^{d+\ga-s_1+1}\<\eta-\eta_*\>^{s_2}\<\eta\>^{s_3}}\Big|\<\eta_*\>^{-s_1}\<\eta-\eta_*\>^{s_2}\wh{G}(\eta_*,\eta-\eta_*)\<\eta\>^{s_3}\wh{h}(\eta)\Big|\\
	&\le 
	C\|\<\eta_*\>^{-s_1}\<\eta-\eta_*\>^{s_2}\wh{G}(\eta_*,\eta-\eta_*)\|_{L^2_{\eta,\eta_*}}
	\Big\|\frac{\1_{|\eta|\le 2\<\eta_*\>}\<\eta\>^{s_3-1}\wh{h}(\eta)}{\<\eta_*\>^{d+\ga-s_1+1}\<\eta-\eta_*\>^{s_2}\<\eta\>^{s_3}}\Big\|_{L^2_{\eta,\eta_*}}
	\\
	&\le C\|\<D_v\>^{-s_1}\<v\>^{-M}f\|_{L^2_v}\|\<D_v\>^{s_2}\<v\>^{M}g\|_{L^2_v}\|\<D_v\>^{s_3}h\|_{L^2_v},
\end{align*}
for any $M\in\R$. Here, in the last inequality, we can split the support $\1_{|\eta|\le 2\<\eta_*\>}=\1_{|\eta|\le\frac{1}{2}\<\eta_*\>}+\1_{\frac{1}{2}\<\eta_*\>\le|\eta|\le 2\<\eta_*\>}$ and use \eqref{equivetastar} to deduce 
\begin{align*}
	\int_{\R^d_{\eta_*}}
	\frac{\1_{|\eta|\le\frac{1}{2}\<\eta_*\>}\1_{\<\eta-\eta_*\>\approx \<\eta_*\>\gtrsim\<\eta\>}}{\<\eta_*\>^{2d+2\ga-2s_1+2}\<\eta-\eta_*\>^{2s_2}\<\eta\>^{2s_3-2}}\,d\eta_*
	&\le \frac{1}{\<\eta\>^{d+2\ga-2s_1+2s_2+2s_3}}\le C,\\
	\int_{\R^d_{\eta_*}}\frac{\1_{\frac{1}{2}\<\eta_*\>\le|\eta|\le 2\<\eta_*\>}\1_{\<\eta\>\approx\<\eta_*\>\gtrsim\<\eta-\eta_*\>}}{\<\eta_*\>^{2d+2\ga-2s_1+2}\<\eta-\eta_*\>^{2s_2}\<\eta\>^{2s_3-2}}\,d\eta_*
	&\le \int_{\R^d_{\eta_*}}\frac{1}{\<\eta-\eta_*\>^{2d+2\ga-2s_1+2s_2+2s_3}}\,d\eta_*\le C,
\end{align*}
provided that $2d+2\ga-2s_1+2+2s_2>d$, $2d+2\ga-2s_1+2s_3\ge0$, and $2d+2\ga-2s_1+2s_2+2s_3>d$, where, in the first inequality, we have used \eqref{integrallargeR}. 

\smallskip 
For $J'_{12}$, similarly, by H\"older's inequality and \eqref{esGS1}, for any $M\in\R$, we have 
\begin{align*}
	J'_{12}&\le C\int_{\R^{2d}_{\eta,\eta_*}}\frac{\<\eta_*\>^{2-2s}}{|\eta|^{2-2s}}\frac{|\eta|^2}{\<\eta_*\>^{d+\ga+2}}\frac{\<\eta_*\>}{\<\eta\>}\Big|\wh{G}(\eta_*,\eta-\eta_*)\wh{h}(\eta)\Big|
    \1_{|\eta|> 2\<\eta_*\>}\,d\eta_*d\eta\\
	&\le C\int_{\R^{2d}_{\eta,\eta_*}}\frac{
	\<\eta\>^{-1+2s}\1_{|\eta|> 2\<\eta_*\>}|\<\eta_*\>^{-s_1}\<\eta-\eta_*\>^{s_2}\wh{G}(\eta_*,\eta-\eta_*)\<\eta\>^{s_3}\wh{h}(\eta)|
	}{\<\eta_*\>^{d+\ga+2s-1-s_1}\<\eta-\eta_*\>^{s_2}\<\eta\>^{s_3}}\\
	&\le C\|\<D_{v}\>^{-s_1}\<v\>^{-M}f\|_{L^2_{v}}\|\<D_{v}\>^{s_2}\<v\>^{M}g(v)\|_{L^2_{v}}\|\<D_v\>^{s_3}h\|_{L^2_v},
\end{align*}
where, in the last inequality, by \eqref{equivetastar} ($\<\eta\>\approx\<\eta-\eta_*\>\gtrsim\<\eta_*\>$ on the support of $\1_{|\eta|>2\<\eta_*\>}$), we can calculate 
\begin{align*}
	&\int_{\R^{d}_{\eta_*}}\frac{\<\eta\>^{-2+4s}\1_{|\eta|> \frac{3}{2}\<\eta_*\>}\1_{\<\eta\>\approx\<\eta-\eta_*\>\gtrsim\<\eta_*\>}}{\<\eta_*\>^{2d+2\ga+4s-2-2s_1}\<\eta-\eta_*\>^{2s_2}\<\eta\>^{2s_3}}\,d\eta_*\\
	&\quad\le C\int_{\R^{d}_{\eta_*}}\frac{\<\eta-\eta_*\>^{-2+4s-2s_2-2s_3}\1_{|\eta|> \frac{3}{2}\<\eta_*\>}\1_{\<\eta\>\approx\<\eta-\eta_*\>\gtrsim\<\eta_*\>}}{\<\eta_*\>^{2d+2\ga+4s-2-2s_1}}\,d\eta_*\\
	&\quad\le C\int_{\R^{d}_{\eta_*}}\frac{1}{\<\eta_*\>^{2d+2\ga-2s_1+2s_2+2s_3}}\,d\eta_*\le C,
\end{align*}
provided that $2d+2\ga-2s_1+2s_2+2s_3>d$ and $s_2+s_3\ge 2s-1$.

\medskip 
For $J_{3}$ in \eqref{J123aaa}, rewriting 
 $\big(\cdots\big)=b\big(\frac{\eta}{|\eta|}\cdot\si\big)\wh{G}(\eta_*,\eta-\eta_*)\ol{\wh{h}(\eta)}\big\{\wh{\Xi}(a^{-1}\eta)-\wh{\Xi}(a^{-1}(\eta-\eta_*))\big\}$,
we split
\begin{align*}
	J_{3}&=\int_{\eta,\eta_*,\si}\1_{|\eta^-|\ge\frac{1}{2}\<\eta_*\>}\big(\wh{\Th_c}(\eta_*-\eta^-)-\wh{\Th_c}(\eta_*)\big)\Big(\cdots\Big)\\
	&=\int_{\eta,\eta_*,\si}\1_{|\eta^-|\ge\frac{1}{2}\<\eta_*\>}\wh{\Th_c}(\eta_*-\eta^-)\1_{|\eta^-|^2\le2|\eta_*\cdot\eta^-|}
	\1_{\frac{1}{2}|\eta|\le\<\eta_*\>}\1_{\<\eta-\eta_*\>\le\<\eta_*-\eta^-\>}\Big(\cdots\Big)\\
	&\ +\int_{\eta,\eta_*,\si}\1_{|\eta^-|\ge\frac{1}{2}\<\eta_*\>}\wh{\Th_c}(\eta_*-\eta^-)\1_{|\eta^-|^2\le2|\eta_*\cdot\eta^-|}
	\1_{\frac{1}{2}|\eta|\le\<\eta_*\>}\1_{\<\eta-\eta_*\>>\<\eta_*-\eta^-\>}\Big(\cdots\Big)\\
	&\ +\int_{\eta,\eta_*,\si}\1_{|\eta^-|\ge\frac{1}{2}\<\eta_*\>}\wh{\Th_c}(\eta_*-\eta^-)\1_{|\eta^-|^2\le2|\eta_*\cdot\eta^-|}
	\1_{\frac{1}{2}|\eta|>\<\eta_*\>}\Big(\cdots\Big)\\
	&\ +\int_{\eta,\eta_*,\si}\1_{|\eta^-|\ge\frac{1}{2}\<\eta_*\>}\Big(\wh{\Th_c}(\eta_*-\eta^-)\1_{|\eta^-|^2>2|\eta_*\cdot\eta^-|}-\wh{\Th_c}(\eta_*)\Big)\Big(\cdots\Big)\\
	&=:(J_{31}+J_{32}+J_{33}+J_{34}).
\end{align*}
Here, we list the properties of these supports as in \eqref{supportQc3}: 
\begin{align}\label{supportJ3}
	\begin{cases}
		\text{in general,}~
		&b(\cos\th)\lesssim \th^{-d+1-2s}\lesssim\frac{|\eta|^{d-1+2s}}{|\eta^-|^{d-1+2s}},\\
		\text{when}~\frac{1}{2}|\eta|>\<\eta_*\>,
		&\<\eta_*\>\lesssim\<\eta\>\approx\<\eta-\eta_*\>,\\
		\text{when}~|\eta^-|^2>2|\eta_*\cdot\eta^-|,&\<\eta_*-\eta^-\>\ge \<\eta_*\>,\\
		\text{when}~|\eta^-|^2\le2|\eta_*\cdot\eta^-|,
		&|\eta^-|\lesssim |\eta_*|\ \text{ and }\ \<\eta_*-\eta^-\>\lesssim \<\eta_*\>,\\
		\text{when}~ |\eta^-|\ge\frac{1}{2}\<\eta_*\>
		&\frac{1}{2}\<\eta_*\>\le \frac{1}{\sqrt{2}}|\eta|\ \text{ and }\ \<\eta_*-\eta^-\>\lesssim|\eta|\approx\<\eta_*\>,\\
		\qquad\ \ \text{ and }\frac{1}{2}|\eta|\le\<\eta_*\>,
		&\text{ and }b(\cos\th)\lesssim \th^{-d+1-2s}\lesssim\frac{|\eta|^{d-1+2s}}{|\eta^-|^{d-1+2s}}\lesssim 1.
	\end{cases}
\end{align}
Moreover, by \eqref{nakThc}, we know that $|\wh{\Th_c}(\eta_*-\eta^-)|\le C\<\eta_*-\eta^-\>^{-d-\ga}$. 
For $J_{31}$, by the estimates in \eqref{supportJ3} related to its support, i.e., $\<\eta\>\approx\<\eta_*\>\gtrsim\<\eta_*-\eta^-\>\ge\<\eta-\eta_*\>$ and $b(\cos\th)\le C$, H\"older's inequality, and estimate \eqref{esGS1} for $\wh{G}$, for any $M\in\R$, we can estimate 
\begin{align*}
	J_{31}
	&\le C\Big(\int_{\eta,\eta_*,\si}
	\frac{\1_{|\eta^-|\ge\frac{1}{2}\<\eta_*\>}\1_{|\eta^-|^2\le2|\eta_*\cdot\eta^-|}
	\1_{\frac{1}{2}|\eta|\le\<\eta_*\>}\1_{\<\eta-\eta_*\>\le\<\eta_*-\eta^-\>}}{\<\eta_*-\eta^-\>^{2d+2\ga}\<\eta_*\>^{-2s_1}\<\eta-\eta_*\>^{2s_2}\<\eta\>^{2s_3}}
	|\<\eta\>^{s_3}\wh{h}(\eta)|^2\Big)^{\frac{1}{2}}\\&\quad\times 
	\Big(\int_{\eta,\eta_*,\si}\<\eta\>^{-2s_1}\<\eta-\eta_*\>^{2s_2}|\wh{G}(\eta_*,\eta-\eta_*)|^2\Big)^{\frac{1}{2}}\\
	&\le C\Big(\int_{\eta,\eta_*}
	\frac{|\<\eta\>^{s_3}\wh{h}(\eta)|^2}{\<\eta-\eta_*\>^{2d+2\ga-2s_1+2s_2+2s_3}}
	\,d\eta_*d\eta\Big)^{\frac{1}{2}}\|\<D_v\>^{-s_1}f\|_{L^2_v}\|\<D_v\>^{s_2}g\|_{L^2_v}\\
	&\le C\|\<D_{v}\>^{-s_1}\<v\>^{-M}f\|_{L^2_{v}}\|\<D_{v}\>^{s_2}\<v\>^{M}g(v)\|_{L^2_{v}}\|\<D_v\>^{s_3}h\|_{L^2_v},
\end{align*}
provided that $2d+2\ga>0$, $-s_1+s_3\ge 0$ and $2d+2\ga-2s_1+2s_2+2s_3>d$. 

\smallskip
The estimate of $J_{32}$ is similar. For any $M\in\R$, as in \eqref{supportJ3}, using $\<\eta\>\approx\<\eta_*\>\gtrsim\<\eta_*-\eta^-\>$ and $\<\eta-\eta_*\>>\<\eta_*-\eta^-\>$ on its support yields
\begin{align*}
	J_{32}
	&\le 
	C\int_{\eta,\eta_*,\si}\frac{\1_{|\eta^-|\ge\frac{1}{2}\<\eta_*\>}
	\1_{\frac{1}{2}|\eta|\le\<\eta_*\>}\1_{\<\eta-\eta_*\>>\<\eta_*-\eta^-\>}}{\<\eta_*-\eta^-\>^{d+\ga}\<\eta_*\>^{-s_1}\<\eta-\eta_*\>^{s_2}\<\eta\>^{s_3}}|\<\eta_*\>^{-s_1}\<\eta-\eta_*\>^{s_2}\wh{G}(\eta_*,\eta-\eta_*)\<\eta\>^{s_3}\ol{\wh{h}}(\eta)|\\
	&\le 
	C\int_{\eta,\eta_*,\si}\frac{1}{\<\eta_*-\eta^-\>^{d+\ga-s_1+s_2+s_3}}|\<\eta_*\>^{-s_1}\<\eta-\eta_*\>^{s_2}\wh{G}(\eta_*,\eta-\eta_*)\<\eta\>^{s_3}\ol{\wh{h}}(\eta)|\\
	&\le 
	C\|\<D_{v}\>^{-s_1}\<v\>^{-M}f\|_{L^2_{v}}\|\<D_{v}\>^{s_2}\<v\>^{M}g(v)\|_{L^2_{v}}\|\<D_v\>^{s_3}h\|_{L^2_v}, 
\end{align*}
provided that $-s_1+s_3\ge 0$, $s_2\ge 0$, and $2d+2\ga-2s_1+2s_2+2s_3>d$.

\smallskip
For the term $J_{33}$, we further use $\wh{\Xi}(a^{-1}\eta)-\wh{\Xi}(a^{-1}(\eta-\eta_*))$ (Taylor expansion and \eqref{Xitoetastar}) to introduce $\frac{|\eta_*|}{|\eta|}$ on the support of $\1_{|\eta|>2\<\eta_*\>}$, which provides the desired frequency in $\eta$.
Therefore, by using H\"older's inequality, 
$\<\eta_*-\eta^-\>\lesssim\<\eta_*\>\approx |\eta^-|\le |\eta|$ and, $\<\eta_*-\eta^-\>\lesssim\<\eta_*\>\lesssim\<\eta\>\approx\<\eta-\eta_*\>$ on the support of $J_{33}$, we have 
\begin{align*}
	J_{33}
	&\le C\int_{\eta,\eta_*,\si}
	b\big(\frac{\eta}{|\eta|}\cdot\si\big)
	\frac{\1_{|\eta^-|\ge\frac{1}{2}\<\eta_*\>}\1_{|\eta^-|^2\le2|\eta_*\cdot\eta^-|}\1_{\frac{1}{2}|\eta|>\<\eta_*\>}}{\<\eta_*-\eta^-\>^{d+\ga}}\Big|\wh{G}(\eta_*,\eta-\eta_*)\wh{h}(\eta)\Big|\frac{\<\eta_*\>}{\<\eta\>}\\
	&\le C\Big(\int_{\eta,\eta_*,\si}
	b\1_{|\eta^-|\ge\frac{1}{2}\<\eta_*\>}|\eta|^{-2s}\<\eta_*\>^{2s}|\<\eta_*\>^{-s_1}\<\eta-\eta_*\>^{s_2}\wh{G}(\eta_*,\eta-\eta_*)|^2\Big)^{\frac{1}{2}}\\
	&\quad\times\Big(\int_{\eta,\eta_*,\si}
	\frac{b\1_{|\eta^-|\ge\frac{1}{2}\<\eta_*\>}\1_{\<\eta_*-\eta^-\>\lesssim\<\eta_*\>\lesssim\<\eta\>}\<\eta_*\>^{2+2s_1-2s}}{\<\eta_*-\eta^-\>^{2d+2\ga}\<\eta\>^{2-2s+2s_2+2s_3}}|\<\eta\>^{s_3}{\wh{h}(\eta)}|^2\Big)^{\frac{1}{2}},
\end{align*}
where $\cos\th=\frac{\eta}{|\eta|}\cdot\si$ and $\th\approx\frac{|\eta|}{|\eta^-|}$.
The first right-hand factor can be estimated by angular integral \eqref{angulareta} and estimate of $\wh{G}$ \eqref{esGS1}. 
For the last factor, we can use the support of $\1_{\<\eta_*-\eta^-\>\lesssim\<\eta_*\>\lesssim\<\eta\>}$, translation $\eta_*\mapsto\eta_*-\eta^-$, \eqref{equivetastar} and angular integral \eqref{angulareta} to deduce that it is 
\begin{align*}
	&\le \int_{\eta,\eta_*,\si}
	\frac{b(\frac{\eta}{|\eta|}\cdot\si)\1_{|\eta^-|\ge\frac{1}{2}\<\eta_*\>}\1_{\<\eta_*-\eta^-\>\lesssim\<\eta_*\>\lesssim\<\eta\>}\<\eta_*\>^{2+2s_1-2s}}{\<\eta_*-\eta^-\>^{2d+2\ga}\<\eta\>^{2-4s+2s_2+2s_3}\<\eta\>^{2s}}|\<\eta\>^{s_3}{\wh{h}(\eta)}|^2\\
	&\le \int_{\eta,\eta_*,\si}
	\frac{b(\frac{\eta}{|\eta|}\cdot\si)\1_{|\eta^-|\ge\frac{1}{C}\<\eta_*\>}}{\<\eta_*\>^{2d+2\ga
	-2s-2s_1+2s_2+2s_3}\<\eta\>^{2s}}|\<\eta\>^{s_3}{\wh{h}(\eta)}|^2\\
	&\le C\int_{\eta,\eta_*}
	\frac{\<\eta\>^{2s}\<\eta_*\>^{-2s}}{\<\eta_*\>^{2d+2\ga
	-2s-2s_1+2s_2+2s_3}\<\eta\>^{2s}}|\<\eta\>^{s_3}{\wh{h}(\eta)}|^2
	\le C\|\<D_v\>^{s_3}h\|_{L^2_v},
\end{align*}
provided that $2-4s+2s_2+2s_3\ge 0$, $2s+2s_1-2s_2-2s_3\le 0$, and $2d+2\ga
	-2s_1+2s_2+2s_3>d$.
Note that $|\eta^-|\ge\frac{1}{2}\<\eta_*+\eta^-\>$ implies $|\eta^-|\gtrsim \<\eta_*\>$. 
Therefore, for any $M\in\R$,
\begin{align*}
	J_{33}
	&\le 
	C\|\<D_{v}\>^{-s_1}\<v\>^{-M}f\|_{L^2_{v}}\|\<D_{v}\>^{s_2}\<v\>^{M}g(v)\|_{L^2_{v}}\|\<D_v\>^{s_3}h\|_{L^2_v},
\end{align*}
provided that $s_2+s_3\ge 2s-1$, $-s_1+s_2+s_3\ge s$, and $\ga
	-s_1+s_2+s_3>-\frac{d}{2}$. 

\smallskip 
For the term $J_{34}$, on the support of $\1_{|\eta^-|^2>2|\eta_*\cdot\eta^-|}$, we have 
	$|\eta_*-\eta^-|^2=|\eta_*|^2-2\eta_*\cdot\eta^-+|\eta^-|^2>|\eta_*|^2$,
and hence, by using \eqref{nakThc}, we have
\begin{align*}
	|\wh{\Th_c}(\eta_*-\eta^-)|\1_{|\eta^-|^2>2|\eta_*\cdot\eta^-|}+|\wh{\Th_c}(\eta_*)|\le C\<\eta_*\>^{-d-\ga},
\end{align*}
whenever $-d-\ga<0$. Therefore, for any $M\in\R$, by angular integral \eqref{angulareta}, estimate \eqref{Xitoetastar} for $\wh{\Xi}(a^{-1}\eta)-\wh{\Xi}(a^{-1}(\eta-\eta_*))$, and H\"older's inequality, we can estimate $J_{34}$ as 
\begin{align*}
	J_{34}
	&\le 
	C\int_{\eta,\eta_*,\si}\frac{b\1_{|\eta^-|\ge\frac{1}{2}\<\eta_*\>}}{\<\eta_*\>^{d+\ga}}|\wh{G}(\eta_*,\eta-\eta_*){\wh{h}(\eta)}|\\
	&\le 
	C\int_{\R^{2d}_{\eta,\eta_*}}
	\frac{\<\eta\>^{2s}}{\<\eta_*\>^{d+\ga+2s-s_1}\<\eta-\eta_*\>^{s_2}\<\eta\>^{s_3}}|\<\eta_*\>^{-s_1}\<\eta-\eta_*\>^{s_2}\wh{G}(\eta_*,\eta-\eta_*)|\\
	&\qquad\times\Big(|\<\eta\>^{s_3}\wh{h}(\eta)|\1_{|\eta|\le2\<\eta_*\>}+|\<\eta_*\>\<\eta\>^{s_3-1}\wh{h}(\eta)|\1_{|\eta|>2\<\eta_*\>}\Big)\\
	&\le \|\<\eta_*\>^{-s_1}\<\eta-\eta_*\>^{s_2}\wh{G}(\eta_*,\eta-\eta_*)\|_{L^2_{\eta,\eta_*}}
	\Big\{\Big\|\frac{\<\eta\>^{2s-s_3}\1_{|\eta|\le2\<\eta_*\>}}{\<\eta_*\>^{d+\ga+2s-s_1}\<\eta-\eta_*\>^{s_2}}\Big\|_{L^\infty_{\eta}L^2_{\eta_*}}\|\<\eta\>^{s_3}\wh{h}(\eta)\|_{L^2_{\eta}}\\
	&
	\qquad\qquad
	+\Big\|\frac{\<\eta\>^{2s-s_3-1}\1_{|\eta|>2\<\eta_*\>}}{\<\eta_*\>^{d+\ga+2s-s_1-1}\<\eta-\eta_*\>^{s_2}}\Big\|_{L^\infty_{\eta}L^2_{\eta_*}}\|\<\eta\>^{s_3}\wh{h}(\eta)\|_{L^2_{\eta}}
	\Big\}
	\\
	&\le 
	C\|\<D_{v}\>^{-s_1}\<v\>^{-M}f\|_{L^2_{v}}\|\<D_{v}\>^{s_2}\<v\>^{M}g(v)\|_{L^2_{v}}\|\<D_v\>^{s_3}h\|_{L^2_v},
\end{align*}
where, in the last inequality, we used the fact that, by \eqref{equivetastar}, 
\begin{align*}
	&\Big\|\frac{\<\eta\>^{2s-s_3}\1_{|\eta|\le 2\<\eta_*\>}}{\<\eta_*\>^{d+\ga+2s-s_1}\<\eta-\eta_*\>^{s_2}}\Big\|_{L^\infty_{\eta}L^2_{\eta_*}}\\
	& \le C\Big\|\frac{\<\eta\>^{2s-s_3}}{\<\eta_*\>^{d+\ga+2s-s_1}\<\eta-\eta_*\>^{s_2}}\Big(\1_{\<\eta_*\>>2|\eta|}
	+\1_{\frac{1}{2}|\eta|\le\<\eta_*\>\le2|\eta|}
	\Big)\Big\|_{L^\infty_{\eta}L^2_{\eta_*}}\\
	& \le 
	C\Big\|\frac{\1_{\<\eta_*\>\approx\<\eta-\eta_*\>}}{\<\eta_*\>^{d+\ga+2s-s_1}\<\eta-\eta_*\>^{-2s+s_2+s_3}}\Big\|_{L^\infty_{\eta}L^2_{\eta_*}}
	+ C\Big\|\frac{\1_{\<\eta_*\>\gtrsim\<\eta-\eta_*\>}}{\<\eta_*\>^{d+\ga-s_1+s_3}\<\eta-\eta_*\>^{s_2}}\Big\|_{L^\infty_{\eta}L^2_{\eta_*}}\\
	& \le 
	C\Big\|\frac{1}{\<\eta_*\>^{d+\ga-s_1+s_2+s_3}}\Big\|_{L^2_{\eta_*}}+C\Big\|\frac{1}{\<\eta-\eta_*\>^{d+\ga-s_1+s_2+s_3}}\Big\|_{L^\infty_{\eta}L^2_{\eta_*}}\le C,
\end{align*}
provided that $2s-s_3\le 0$, $d+\ga-s_1+s_3\ge 0$, and $\ga-s_1+s_2+s_3>-\frac{d}{2}$, as well as  
\begin{align*}
	\Big\|\frac{\<\eta\>^{2s-s_3-1}\1_{\<\eta_*\><\frac{1}{2}|\eta|}}{\<\eta_*\>^{d+\ga+2s-s_1+1}\<\eta-\eta_*\>^{s_2}}\Big\|_{L^\infty_{\eta}L^2_{\eta_*}}
	&\le C\Big\|\frac{\<\eta\>^{2s-s_2-s_3-1}}{\<\eta_*\>^{d+\ga+2s-s_1+1}}\1_{\<\eta\>\approx\<\eta-\eta_*\>\gtrsim\<\eta_*\>}\Big\|_{L^\infty_{\eta}L^2_{\eta_*}}\\
	&\le C\Big\|\frac{1}{\<\eta_*\>^{d+\ga-s_1+s_2+s_3}}\Big\|_{L^\infty_{\eta}L^2_{\eta_*}}
	\le C,
\end{align*}
provided that $s_2+s_3\ge2s-1$, $d+\ga-s_1+s_2+s_3>\frac{d}{2}$. 
Substituting all the above estimates into \eqref{Ismall}, for any $M\in\R$, we deduce 
\begin{align}\label{Ismalles}
	I^{\text{small}}
	&\le C\|\<D_{v}\>^{-s_1}\<v\>^{-M}f\|_{L^2_{v}}\|\<D_{v}\>^{s_2}\<v\>^{M}g(v)\|_{L^2_{v}}\|\<D_v\>^{s_3}h\|_{L^2_v},
\end{align}
provided that 
$\ga>-d$,
$s_1\le s_3$, 
$s_2\ge 0$, 
$\ga-s_1+s_2+s_3>-\frac{d}{2}$, 
$\ga-s_1+s_2+1>-\frac{d}{2}$, 
$-s_1+s_2+s_3\ge s$, and 
$s_2+s_3\ge 2s-1$. 
This concludes the estimate of $I^{\text{small}}$ in \eqref{PaCommIsmall}.

\subsection{Commutator estimate of $[R_r,Q]$}
In this Subsection, we consider the commutator estimate for the dyadic decomposition about velocity, i.e., $\sum_{r=0}^\infty R_r(v)=1$, where 
	$R_rf(v)=R_r(v)f(v)=\wh{\Psi_r}(\<v\>)f(v)$ is defined in \eqref{DejQkDef}.
Note that $\<v\>$ has a natural lower bound of $1$.
For any $n\in\R$, we can split 
\begin{align}\label{RrQIlrIls2}\notag
	\big(R_rQ(f,g)-Q(f,R_rg),h\big)_{L^2_v}
	&\notag=\int_{\R^{2d}_{v,v_*}\times\S^{d-1}_\si}B(v-v_*,\si)f(v_*)g(v)\ol{h(v')}\big(R_r(v')-R_r(v)\big)\\
	&\notag=\int_{v,v_*,\si}Bf(v_*)g(v)\big(R_r(v')-R_r(v)\big)
	\Big(\big(\<v'\>^{-n}\ol{h(v')}-\<v\>^{-n}\ol{h(v)}\big)\<v'\>^n\\
	&\qquad\qquad\notag
	+\<v\>^{-n}\ol{h(v)}\big(\<v'\>^n-\<v\>^n\big)
	+\ol{h(v)}\<v\>^n\Big)\\
	&=:I^{r}_{1}+I^{r}_{2}+I^{r}_{3}. 
\end{align}
For the term $I^r_1$, by H\"older's inequality and \eqref{ColliUpperTriple}, we have 
\begin{align*}
	|I^r_1|&\le C\|\<v\>^{|n|+3|\ga|+2+\frac{d+1}{2}}f\|_{L^2_v}^{\frac{1}{2}}\|\<v\>^{-n}h\|_{L^2_D}\\
	&\qquad\times
	\Big(\int_{v,v_*,\si}b(\cos\th)|v-v_*|^{\ga}\<v_*\>^{-|n|}|f(v_*)||g(v)|^2\<v'\>^{2n}|R_r(v')-R_r(v)|^2\Big)^{\frac{1}{2}}.
\end{align*}
Moreover, by Taylor expansion, $\na_zR_r(z)\le C2^{-\frac{1}{2}r}\<z\>^{-\frac{1}{2}}$, and \eqref{distancevpriv} (i.e., $\<z\>^{-1}\le\<v\>^{-1}\<v_*\>$), we can estimate 
\begin{align}\label{esRrTaylor}\notag
	|R_r(v')-R_r(v)|
	&\le\min\Big\{2,\,\Big|(v'-v)\cdot\int^1_0\na_zR_r(z)\Big|_{z=v'+\tau(v'-v)}\,d\tau\Big|\Big\}\\
	&\le C\min\{1,\,2^{-\frac{1}{2}r}\<v_*\>^{\frac{1}{2}}\<v\>^{-\frac{1}{2}}|v'-v|,\,|v'-v|\}.
\end{align}
Moreover, by angular integral estimate \eqref{angularintegral}, one has 
\begin{align}\label{es647}\notag
	&\int_{\S^{d-1}_\si}b(\cos\th)\min\{1,\,2^{-r}\<v_*\>\<v\>^{-1}|v'-v|^2\}\,d\si
	\notag\le 
	\int_{\si}b(\cos\th)\1_{\th>2^{\frac{1}{2}r}\<v_*\>^{-\frac{1}{2}}\<v\>^{\frac{1}{2}}|v-v_*|^{-1}}\\
	&\qquad\qquad\notag\quad
	+\int_{\si}b(\cos\th)\sin^2\frac{\th}{2}\1_{\th\le 2^{\frac{1}{2}r}\<v_*\>^{-\frac{1}{2}}\<v\>^{\frac{1}{2}}|v-v_*|^{-1}}2^{-r}\<v_*\>\<v\>^{-1}|v-v_*|^2\\
	&\qquad\le C
	2^{-sr}\<v_*\>^{s}\<v\>^{-s}|v-v_*|^{2s}.
\end{align}
Thus, by estimates \eqref{es647} and \eqref{gafg}, and \eqref{distancevpriv} (i.e., $\<v'\>^{n}\le C_n\<v_*\>^{|n|}\<v\>^{n}$), we have 
\begin{align}\label{es648}\notag
	|I^r_1|&\le C\|\<v\>^{|n|+3|\ga|+2+\frac{d+1}{2}}f\|_{L^2_v}^{\frac{1}{2}}\|\<v\>^{-n}h\|_{L^2_D}\\
	&\notag\quad 
	\times 2^{-sr}\|\<v\>^{|n|+|\ga+s|+\frac{d+1}{2}}f\|_{L^2_v}^{\frac{1}{2}}
	\|\<v\>^{n+\frac{\ga+s}{2}}g\|_{L^2_v}\\
	&\le C2^{-sr}\|\<v\>^{|n|+3|\ga|+2+\frac{d+1}{2}}f\|_{L^2_v}\|\<v\>^{n+\frac{\ga+s}{2}}g\|_{L^2_v}\|\<v\>^{-n}h\|_{L^2_D}, 
\end{align}
provided that $\ga+2s>-\frac{d}{2}$. 
The term $I^r_2$ can be treated similarly. Applying \eqref{esRrTaylor} and a similar Taylor expansion to $\<v\>^{n}$, we have 
\begin{align*}
	|R_r(v')-R_r(v)||\<v'\>^n-\<v\>^n|
	&\le C\min\{1,\,2^{-\frac{1}{2}r}\<v_*\>^{|n|+\frac{3}{2}}\<v\>^{n-\frac{3}{2}}|v'-v|^2\}.
\end{align*}
Using the angular integral estimate \eqref{angularintegral}, similar to \eqref{es647}, we have 
\begin{align*}
	&\int_{\S^{d-1}_\si}b(\cos\th\min\{1,\,2^{-\frac{1}{2}r}\<v_*\>^{|n|+\frac{1}{2}}\<v\>^{n-\frac{1}{2}}|v'-v|^2\}\,d\si\\
	&\quad\le C
	2^{-sr}\<v_*\>^{|n|+s}\<v\>^{-n-s}|v-v_*|^{2s}.
\end{align*} 
Therefore, by \eqref{gafg}, $I^r_2$ can be estimated as 
\begin{align*}
	|I^r_2|
	&\le C2^{-sr}\int_{v,v_*,\si}b|v-v_*|^{\ga+2s}|f(v_*)g(v)
	\<v\>^{-n}\ol{h(v)}|\\
	&\qquad\times\big(\<v'\>^n-\<v\>^n\big)\<v_*\>^{|n|+1+s}\<v\>^{-n-1-s}\\
	&\le C2^{-sr}\|\<v\>^{|n|+|\ga+2s|+s+\frac{d+1}{2}}f\|_{L^2_v}\|\<v\>^{n+\frac{\ga+s}{2}}g\|_{L^2_v}\|\<v\>^{-n}h\|_{L^2_v},
\end{align*}
provided that $\ga+2s>-\frac{d}{2}$. 
The term $I^r_3$ is ``regular''. Applying the second-order Taylor expansion to $R_r(v)=\wh{\Psi}(2^{-r}\<v\>)$, we have 
\begin{align*}
	R_r(v')-R_r(v)
	&=(v'-v)\cdot \na_vR_r(v)
	+(v'-v)^{\otimes2}:\int_{0}^{1}\na^2_z\big(\wh{\Psi}(2^{-r}\<z\>)\big)\big|_{z=v+\tau(v'-v)}(1-\tau)\,d\tau,
\end{align*}
where $|\na_vR_r(v)|\le C2^{-\frac{1}{2}r}\<v\>^{-\frac{1}{2}}$ and $|\na^2_zR_r(z)|\le C2^{-r}\<z\>^{-1}\le C2^{-r}\<v\>^{-1}\<v_*\>$. 
For the first-order term, we apply \eqref{vprivth}, i.e., $v'-v=\sin^2\frac{\th}{2}(v_*-v)+\frac{1}{2}|v-v_*|\sin\th\omega$, angular integral \eqref{angularintegral}, and the symmetry with respect to $\omega$ (which makes the term with $\om$ vanish) to deduce 
\begin{align*}
	I^r_2
	&\le \int_{v,v_*,\si}Bf(v_*)g(v)\ol{h(v)}\sin^2\frac{\th}{2}\Big((v_*-v)\cdot \na_vR_r(v)
	+0\\
	&\qquad\quad+(v'-v)\otimes(v'-v):\int_{0}^{1}\na^2_z\big(\wh{\Psi}(2^{-r}\<z\>)\big)(1-\tau)\,d\tau\Big).
\end{align*}
Together with the trivial upper bound $|R_r(v')-R_r(v)|\le 2$ and \eqref{gafg}, we have 
\begin{align}\label{es648a}\notag
	|I^r_2|&\notag\le 
	C\int_{v,v_*,\si}B
	|f(v_*)g(v)\ol{h(v)}\min\{1,\,2^{-2r}|v-v_*|\sin^2\frac{\th}{2},\,2^{-r}\<v_*\>\<v\>^{-1}|v-v_*|^2\sin^2\frac{\th}{2}\}\\
	&\notag\le 
	C\int_{v,v_*}
	2^{-sr}\big(|v-v_*|^{\ga+s}+\<v_*\>^{s}\<v\>^{-s}|v-v_*|^{\ga+2s}\big)
	|f(v_*)g(v)h(v)|\\
	&\le C2^{-sr}\|\<v\>^{|\ga|+3s+\frac{d+1}{2}}f\|_{L^2_v}\|\<v\>^{n+\ga+s}g\|_{L^2_v}\|\<v\>^{-n}h\|_{L^2_v},
\end{align}
provided that $\ga+s>-\frac{d}{2}$. Collecting \eqref{es648} and \eqref{es648a} into \eqref{RrQIlrIls2}, we have, for any $n\in\R$, 
\begin{align}
	\label{CommRr}
	&|\big(R_rQ(f,g)-Q(f,R_rg),h\big)_{L^2_v}|
	\le C2^{-sr}\|\<v\>^{|n|+3|\ga|+3+\frac{d+1}{2}}f\|_{L^2_v}\|\<v\>^{n+\frac{\ga+s}{2}}g\|_{L^2_v}\|\<v\>^{-n}h\|_{L^2_D}.
\end{align}
Collecting the estimates \eqref{Ilargees}, \eqref{CommRr}, and \eqref{Ismalles}, one can conclude Lemma \ref{LemmaCommColli}.
Using the same arguments, but replacing $R_r(v)=\wh{\Psi_r}(\<v\>)$ with $\<v\>^l$ (which shares similar properties, but without needing the factor $2^{-sr}\approx\<v\>^{-s}$ on the support of $R_r$), we obtain, for any $l,n\in\R$,   
\begin{align}
	\label{CommvlTemp}
	|\big(\<v\>^lQ(f,g)-Q(f,\<v\>^lg),h\big)_{L^2_v}|
	\le C\|\<v\>^{|n|+3|\ga|+3+\frac{d+1}{2}}f\|_{L^2_v}\|\<v\>^{l+n+\frac{\ga}{2}}g\|_{L^2_v}\|\<v\>^{-n}h\|_{L^2_D}.
\end{align}


\subsection{A Key Lemma, the commutator $Q-\Ga$, $L^p$ estimates, and a coercive estimate}
In this part, we calculate the commutator estimate of $Q(\mu^{\fr12}f,g)-\Ga(f,g)$, and the $L^p$ estimate for $\Ga(G,f)$ with any $p\ge2$ and $G=\mu^{\fr12}+g\ge 0$ (with the convention that $|\cdot|^{\fr p2}=(\cdot)$ when $p=2$). 

\subsubsection{A key lemma}
To calculate the $L^p$ estimate, we need the following key lemma. 

\begin{Lem}\label{Lem21}
	For any suitable $f,g$, we have the following. 
	\begin{itemize}[leftmargin=1.5em]
		\item 
	Let $1<p'\le 2\le p<\infty$ satisfy $\frac{1}{p}+\fr1{p'}=1$. If $\ga>-\fr{d}{p'}$, then  
	\begin{align}
		\label{gafg}
		\int_{\R^{2d}_{v,v_*}}|v-v_*|^{\ga}|f(v_*)g(v)|\,dv_*dv
		&\le C\|\<v\>^{|\ga|+\frac{d+1}{p'}}f\|_{L^p_v}\|\<v\>^{\ga}g\|_{L^1_v}. 
	\end{align}
	\item Let $s\in(0,1)$ and $1<p'\le 2\le p<\infty$ satisfy $\frac{1}{p}+\fr1{p'}=1$. If $\ga-s>-\fr{d}{p'}$ and $G\ge 0$, then we have the following rough estimate:
	\begin{align}
		\label{gafgDs}
		\int_{\R^{2d}_{v,v_*}}|v-v_*|^{\ga}G(v_*)|f(v)|\,dv_*dv&\le C\|\<v\>^{|\ga|+d+1}\<D_v\>^{-s}G\|_{L^p_v}\|\<v\>^{\ga}f\|_{L^1_v}. 
	\end{align}
	\end{itemize}
\end{Lem}
\begin{proof}
To obtain the optimal weight function of $g$, we will frequently apply 
\begin{align*}
	\begin{cases}
		|v-v_*|^{\ga}\le C_\ga\<v\>^\ga\<v_*\>^{\ga},&\text{ if }\ga\ge 0,\\
		|v-v_*|^{\ga}\le C_\ga\<v-v_*\>^{\ga}\le C_\ga\<v_*\>^{|\ga|}\<v\>^{\ga},&\text{ if }\ga< 0\text{ and }|v-v_*|\ge 1,\\
		\int_{\R^d_{v_*}}\1_{|v-v_*|<1}|v-v_*|^{p'\ga}\,dv_*<\infty,&\text{ if }-\frac{d}{p'}<\ga< 0,\\
\<v\>\approx\<v_*\>,&\text{ if }|v-v_*|<1.
	\end{cases}
\end{align*}
Therefore, when $\ga\ge 0$, it's direct to calculate $|v-v_*|^\ga\le \<v\>^{\ga}\<v_*\>^{\ga}$ and hence, 
	\begin{align*}
		\int_{\R^{2d}_{v,v_*}}|v-v_*|^{\ga}|f(v_*)g(v)|\,dv_*dv\le C_\ga\|\<v\>^{\ga}f\|_{L^1_v}\|\<v\>^{\ga}g\|_{L^1_v}.
	\end{align*}
When $-\frac{d}{p'}<\ga<0$, we split $1=\1_{|v-v_*|< 1}+\1_{|v-v_*|\ge 1}$ and calculate
	\begin{align*}
		\int_{|v-v_*|>1}|v-v_*|^{\ga}|f(v_*)|\,dv_*&\le C_\ga\|\<v_*\>^{|\ga|}\<v\>^{\ga}f(v_*)\|_{L^1_{v_*}},\\
		\int_{|v-v_*|\le 1}|v-v_*|^{\ga}|f(v_*)|\,dv_*&\le \Big(\int_{|v-v_*|\le 1}|v-v_*|^{p'\ga}\,dv_*\Big)^{\frac{1}{p'}}\|\1_{|v-v_*|\le1}f(v_*)\|_{L^p_{v_*}}\\
		&\le C_\ga\|\<v_*\>^{|\ga|}\<v\>^{\ga}f(v_*)\|_{L^p_{v_*}}.
	\end{align*}
Integrating this over $v$ against $g(v)$ yields \eqref{gafg}.
Finally, if $G\ge 0$, for any $s\in(0,1)$ and $p\ge 2$, using the equivalence in \eqref{equiv}, and the embedding $W^{s,p'}_{v_*}\hookrightarrow H^{s,p'}$ as in \cite[Chap. 1]{Triebel1992}, we have  
\begin{align*}
	\int_{v,v_*}|v-v_*|^{\ga}G(v_*)|f(v)|
	&\le\int_{v,v_*}\<D_{v_*}\>^{-s}G(v_*)\<D_{v_*}\>^{s}|v-v_*|^{\ga}|f(v)|\\
	&\le 
	C\|\<v_*\>^{M}\<D_{v_*}\>^{-s}G(v_*)\|_{L^p_{v_*}}\|\<v_*\>^{-M}\<D_{v_*}\>^{s}|v-v_*|^{\ga}f(v)\|_{L^{p'}_{v_*}L^1_v}\\
	&\le 
	C\|\<v_*\>^{M}\<D_{v_*}\>^{-s}G(v_*)\|_{L^p_{v_*}}\|\<v_*\>^{-M}|v-v_*|^{\ga}f(v)\|_{L^1_vW^{s,p'}_{v_*}}\\
	&\le 
	C\|\<v\>^{M}\<D_{v}\>^{-s}G(v)\|_{L^p_{v}}
	\|\<v\>^{\ga}f(v)\|_{L^1_v},
\end{align*}
provided that $\ga-s>-\frac{d}{p'}$ and $M>\frac{d}{p'}$, where we claimed that the fractional Sobolev-Slobodeckij norm $\|\<v\>^{-M}|u-v|^{\ga}\|_{W^{s,p'}_{v}}$ satisfies 
\begin{align}\label{esFract1}\notag
	\|\<v\>^{-M}|u-v|^{\ga}\|_{W^{s,p'}_{v}}
	&:=\notag\|\<v\>^{-M}|u-v|^{\ga}\|_{L^{p'}_{v}}
	+\Big(\int_{v,v_*}\fr{|\<v\>^{-M}|u-v|^{\ga}-\<v_*\>^{-M}|u-v_*|^{\ga}|^{p'}}{|v-v_*|^{sp'+d}}\Big)^{\fr1{p'}}\\
	&\le C\<u\>^{\ga},\ \ \text{ provided that $\ga-s>-\fr{d}{p'}$.}
\end{align}
For simplicity, we can choose $M\ge |\ga|+d+1$. 
The estimate $\|\<v\>^{-M}|u-v|^{\ga}\|_{L^{p'}_{v}}\le C\<u\>^{\ga}$ follows from standard methods whenever $\ga>-\frac{d}{p'}$. The Gagliardo seminorm can be handled by splitting $1=\1_{|v-v_*|\ge 1}+\1_{|v-v_*|<1}$. The part corresponding to $\1_{|v-v_*|\ge 1}$ has no singularity and can be estimated as in the first term; details are omitted for brevity. 
For the part $\1_{|v-v_*|<1}$, we split 
\begin{align}\label{split1a}\notag
	\1_{|v-v_*|<1}&\le\1_{|v-v_*|<1}\big(\1_{|v-v_*|\le \fr{1}{2}|u-v|}+\1_{|v-v_*|\le \fr{1}{2}|u-v_*|}\\
	&\qquad\qquad\qquad+\1_{|v-v_*|>\fr{1}{2}|u-v|}\1_{|v-v_*|>\fr{1}{2}|u-v_*|}\big),
\end{align}
and, due to symmetry, it suffices to deal with the first and last terms in \eqref{split1a}. 


For the part involving $\1_{|v-v_*|<1}\1_{|v-v_*|\le \fr{1}{2}|u-v|}$, we have $|u-v_*|\ge |u-v|-|v-v_*|\ge\frac{1}{2}|u-v|\ge \fr12|v-v_*|$ and $|u-v_*|\le |u-v|+|v-v_*|\le \fr32|u-v|$, which imply $|u-v_*|\approx|u-v|\approx|u-v_*+\tau(v-v_*)|$ for any $\tau\in(0,1)$. Moreover, when $|v-v_*|\le1$, we have $\<v\>\approx\<v_*\>\approx\<v_*+\tau(v-v_*)\>$ for any $\tau\in(0,1)$. Thus, the Taylor expansion yields
\begin{align*}
	|\<v\>^{-M}|u-v|^{\ga}-\<v_*\>^{-M}|u-v_*|^{\ga}|
	&\ =\Big|\int^1_0(v-v_*)\cdot\na_z\big(\<z\>^{-M}|u-z|^{\ga}\big)\big|_{z=v_*+\tau(v-v_*)}\,d\tau\Big|\\
	&\ \le C|v-v_*|\<v\>^{-M}\big(|u-v|^{\ga}+|u-v|^{\ga-1}\big).
\end{align*}
Thus, under the cutoff $\1_{|v-v_*|<1}\1_{|v-v_*|\le \fr{1}{2}|u-v|}$, using $\int_{v_*}\1_{|v-v_*|\le\min\{1,\fr12|u-v|\}}|v-v_*|^{-(s-1)p'-d}\le C\min\{1,|u-v|\}^{-(s-1)p'}$ whenever $s<1$, we can calculate the corresponding Gagliardo seminorm, for $Mp'>|\ga|p'+d$ and $\ga-s>-\fr{d}{p'}$,  as 
\begin{align}\label{es226a}\notag
	&\Big(\int_{v,v_*}\1_{|v-v_*|<1}\1_{|v-v_*|\le \fr{1}{2}|u-v|}\fr{|\<v\>^{-M}|u-v|^{\ga}-\<v_*\>^{-M}|u-v_*|^{\ga}|^{p'}}{|v-v_*|^{sp'+d}}\Big)^{\fr1{p'}}\\
	&\notag\le C\Big(\int_{v}\<v\>^{-Mp'}\min\{1,|u-v|\}^{-(s-1)p'}\big(|u-v|^{\ga p'}+|u-v|^{(\ga-1)p'}\big)\Big)^{\fr1{p'}}\\
	&\le C\<u\>^{\ga-s}. 
\end{align}
For the part involving $\1_{|v-v_*|<1}\1_{|v-v_*|>\fr{1}{2}|u-v|}\1_{|v-v_*|>\fr{1}{2}|u-v_*|}$, the denominator $|v-v_*|$ is less singular than $|u-v|$ and $|u-v_*$, and hence, we can simply estimate $|\<v\>^{-M}|u-v|^{\ga}-\<v_*\>^{-M}|u-v_*|^{\ga}|^{p'}\le C\<v\>^{-Mp'}|u-v|^{\ga p'}+C\<v_*\>^{-Mp'}|u-v_*|^{\ga p'}$ and then integrate over corresponding $v$ or $v_*$ as 
\begin{align}\label{es226b}\notag
	&\Big(\int_{v,v_*}\1_{|v-v_*|<1}\1_{|v-v_*|>\fr{1}{2}|u-v|}\1_{|v-v_*|>\fr{1}{2}|u-v_*|}\fr{|\<v\>^{-M}|u-v|^{\ga}-\<v_*\>^{-M}|u-v_*|^{\ga}|^{p'}}{|v-v_*|^{sp'+d}}\Big)^{\fr1{p'}}\\
	&\notag\le C\Big(\int_{v,v_*}\1_{|v-v_*|<1}\fr{\<v\>^{-Mp'}|u-v|^{(\ga-s-\ve)p'}+\<v_*\>^{-Mp'}|u-v_*|^{(\ga-s-\ve)p'}}{|v-v_*|^{-\ve p'+d}}\Big)^{\fr1{p'}}\\
	&\le C\Big(\int_{v}\<v\>^{-Mp'}|u-v|^{(\ga-s-\ve)p'}\Big)^{\fr1{p'}}
	\le C\<u\>^{\ga-s-\ve},
\end{align}
provided that $\ve>0$, $Mp'>|\ga|p'+d$ and $\ga-s-\ve>-\fr{d}{p'}$. 
Combining decomposition \eqref{split1a}, and estimates \eqref{es226a} and \eqref{es226b}, we deduce the claim \eqref{esFract1} for $\ga-s-\ve>-\fr{d}{p'}$ and $M\ge |\ga|+d+1$. Finally, one can choose $\ve>0$ sufficiently small so that $\ga-s-\ve>-\fr{d}{p'}$ holds, provided that $\ga-s>-\fr{d}{p'}$. 
This concludes Lemma \ref{Lem21}.
\end{proof}

\subsubsection{Commutator estimate of \texorpdfstring{$Q-\Gamma$}{Q-Gamma}}
To obtain the dissipation effect, we need a more delicate estimate of the commutator $\Ga(\cdot,\cdot)-Q(\mu^{\frac{1}{2}}(\cdot),\cdot)$. Since we are considering non-negative solutions, for simplicity we restrict our attention to non-negative functions $f\ge 0$. We then rewrite
\begin{align}\label{GminusQDef}\notag
	&\big(\Ga(f,g)-Q(\mu^{\frac{1}{2}}f,g),h\big)_{L^2_v}
	=\int_{v,v_*,\si}B\big(\mu(v'_*)^{\frac{1}{2}}-\mu(v_*)^{\frac{1}{2}}\big)f(v_*)g(v)h(v')\\
	&\notag=\int_{v,v_*,\si}B\big(\mu(v'_*)^{\frac{1}{4}}-\mu(v_*)^{\frac{1}{4}}\big)^2f(v_*)g(v)h(v')\\
	&\notag\quad+2\int_{v,v_*,\si}B\big(\mu(v'_*)^{\frac{1}{4}}-\mu(v_*)^{\frac{1}{4}}\big)\mu(v_*)^{\frac{1}{4}}f(v_*)g(v)h(v)\\
	&\notag\quad+2\int_{v,v_*,\si}B\big(\mu(v'_*)^{\frac{1}{4}}-\mu(v_*)^{\frac{1}{4}}\big)\mu(v_*)^{\frac{1}{4}}f(v_*)g(v)\big(h(v')-h(v)\big)\\
	&=:I_1+I_2+I_3.
\end{align}
For the term $I_1$, the factor $(\cdots)^2$ provides sufficient angular regularity $\th^2$ via Taylor expansion. Thus, by H\"older's inequality, the regular change of variables $v\mapsto v'$ as in \eqref{regular}, angular integral \eqref{angularintegral}, and the $\<D_v\>^{-s}$ estimate \eqref{gafgDs}, we obtain
\begin{align*}
	|I_1|
	&\le C\int_{v,v_*,\si}b\min\big\{1,|v-v_*|^2\sin^2\fr{\th}{2}\big\}|v-v_*|^{\ga}|f(v_*)||g(v)||h(v')|\\
	&\le C\Big(\int_{v,v_*}|v-v_*|^{\ga+2s}|f(v_*)||g(v)|^2\<v'\>^{-2n}\<v_*\>^{-|n|}\Big)^{\fr12}\\
	&\quad\times \Big(\int_{v,v_*}|v-v_*|^{\ga+2s}|f(v_*)||h(v')|^2\<v'\>^{2n}\<v_*\>^{|n|}\Big)^{\fr12}\\
	&\le C\|\<v\>^{|n|+|\ga|+d+1}\<D_v\>^{-s}f\|_{L^r_v}\|\<v\>^{-n+\fr{\ga+2s}{2}}g\|_{L^2_v}\|\<v\>^{n+\fr{\ga+2s}{2}}h\|_{L^2_v},
\end{align*}
for any $n\in\R$, provided that $\ga+s>-\fr{d}{r'}$ and $f\ge 0$. Here, we used \eqref{distancevpriv} to control $\<v'\>^{-2s}$.

For the term $I_2$, we note that the first-order Taylor expansion of $\mu(v'*)^{\frac{1}{4}}-\mu(v*)^{\frac{1}{4}}$ yields $v'_*-v_*=v'_*  =\sin^2\frac{\th}{2}(v-v_*)-\frac{1}{2}|v-v_*|\sin\th\omega$ where the integral involving $\omega$ vanishes due to symmetry. Thus, using the Taylor expansion up to second order and the $\<D_v\>^{-s}$ estimate \eqref{gafgDs},
\begin{align*}
	|I_2|
	&\le \int_{v,v_*,\si}B\min\big\{1,\sin^2\frac{\th}{2}|v-v_*|+\sin^2\frac{\th}{2}|v-v_*|^2\big\}\mu(v_*)^{\frac{1}{4}}|f(v_*)g(v)h(v)|\\
	&\le C\|\<v\>^{-N}\<D_v\>^{-s}f\|_{L^r_v}\|\<v\>^{-n+\fr{\ga+2s}{2}}g\|_{L^2_v}\|\<v\>^{n+\fr{\ga+2s}{2}}h\|_{L^2_v},
\end{align*}
for any $n\in\R$ and $N\ge 0$, provided that $\ga>-\fr{d}{r'}$ and $f\ge 0$. 

For $I_3$, by H\"older's inequality, Taylor expansion and the $\<D_v\>^{-s}$ estimate \eqref{gafgDs} (with $\fr{d-2s}{2}$ therein), we have 
\begin{align*}
	|I_3|
	&\le 
	C\Big(\int_{v,v_*,\si}B\big(\mu(v'_*)^{\frac{1}{4}}-\mu(v_*)^{\frac{1}{4}}\big)^2\mu(v_*)^{\frac{1}{4}}f(v_*)|g(v)|^2\<v\>^{-2n}\Big)^{\fr12}\\
	&\ \ \times \Big(\int_{v,v_*,\si}B\mu(v_*)^{\frac{1}{4}}f(v_*)\Big\{\big(\<v'\>^{n}h(v')-\<v\>^{n}h(v)\big)^2+\big(\<v'\>^{n}-\<v\>^n\big)^2(h(v'))^2\Big\}\Big)^{\fr12}\\
	&\le C\Big(\int_{v,v_*,\si}b\min\big\{1,|v-v_*|^2\sin^2\fr{\th}{2}\}\mu(v_*)^{\frac{1}{4}}f(v_*)|g(v)|^2\Big)^{\fr12}\\
	&\ \ \times 
	\Big(\int_{v,v_*,\si}B\mu(v_*)^{\frac{1}{4}}f(v_*)
	\Big\{-2\<v\>^nh(v)\big(\<v'\>^nh(v')-\<v\>^nh(v)\big)\\
	&\qquad\qquad\qquad\qquad+(\<v'\>^nh(v'))^2-(\<v\>^nh(v))^2+\big(\<v'\>^{n}-\<v\>^n\big)^2(h(v'))^2\Big\}\Big)^{\fr12}\\
	&\le C\|\<v\>^{-N}\<D_v\>^{-\fr{d-2s}{2}}f\|^{\fr12}_{L^r_v}\|\<v\>^{-n+\fr{\ga+2s}{2}}g\|_{L^r_v}\\
	&\quad\times \Big(|(Q(\mu^{\fr14}f,\<v\>^nh),\<v\>^nh)_{L^2_v}|+\|\<v\>^{-N}\<D_v\>^{-\fr{d-2s}{2}}f\|^{\fr12}_{L^r_v}\|\<v\>^{n+\fr{\ga}{2}}h\|_{L^2_v}\Big), 
\end{align*}
for any $N\ge 0$ and $n\in\R$, 
provided that $\ga-\fr{d-2s}{2}>-\fr{d}{r'}$. 
Here, the term $(\<v'\>^nh(v'))^2-(\<v\>^nh(v))^2$ is of cancellation form and can be handled using the regular change of variables \eqref{regular}. The term $\big(\<v'\>^{n}-\<v\>^n\big)^2$ provides sufficient angular regularity $\th^2$ and can be treated via Taylor expansion and \eqref{distancevpriv}. Lastly, for the collision term, we apply the upper bound \eqref{QfghUpper} to deduce 
\begin{align*}
	|I_3|
	&\le C\|\<v\>^{-N}\<D_v\>^{-\fr{d-2s}{2}}f\|_{L^r_v}\|\<v\>^{-n+\fr{\ga+2s}{2}}g\|_{L^2_v}\|\<v\>^{n}h\|_{L^2_D}.
\end{align*}
Collecting the above estimates of $I_i$'s into \eqref{GminusQDef}, we obtain for $N\ge 0$ and $n\in\R$ that  
\begin{align}\label{esGaminQTemp}\notag
	\big|\big(\Ga(f,g)-Q(\mu^{\frac{1}{2}}f,g),h\big)_{L^2_v}\big|
	\notag
	&\le C\|\<v\>^{|n|+|\ga|+d+1}\<D_v\>^{-s}f\|_{L^r_v}\|\<v\>^{-n+\fr{\ga+2s}{2}}g\|_{L^2_v}\|\<v\>^{n+\fr{\ga+2s}{2}}h\|_{L^2_v}\\
	&\quad+C\|\<v\>^{-N}\<D_v\>^{-\fr{d-2s}{2}}f\|_{L^r_v}\|\<v\>^{-n+\fr{\ga+2s}{2}}g\|_{L^2_v}\|\<v\>^{n}h\|_{L^2_D},
\end{align}
provided that $\ga-\fr{d-2s}{2}>-\fr{d}{r'}$ and $f\ge 0$.

\subsubsection{$L^p$ estimate for the Boltzmann collision}
In this part, we derive a simple $L^p$ estimate for the Boltzmann collision operator. By the pre-post change of variables, 
\begin{align*}
	&\big(\Ga(\mu^{\fr12},f),f|f|^{p-2}\big)_{L^2_v}
	=\int_{\R^{2d}_{v,v_*}\times\S^{d-1}_\si}B\mu^{\fr12}(v_*)f(v)\big(\mu^{\frac{1}{2}}(v'_*)(f|f|^{p-2})(v')-\mu^{\frac{1}{2}}(v_*)(f|f|^{p-2})(v)\big). 
\end{align*}
Since $\mu^{\fr12}\ge 0$, for any $p\ge 2$, one has 
\begin{align*}\notag
	&f(v)\big(\mu^{\frac{1}{2}}(v'_*)(f|f|^{p-2})(v')-\mu^{\frac{1}{2}}(v_*)(f|f|^{p-2})(v)\big)
	\le \mu^{\frac{1}{2}}(v'_*)|f(v)||f(v')|^{p-1}-\mu^{\frac{1}{2}}(v_*)|f(v)|^{p}.
\end{align*}
Moreover, we apply the basic inequality $\th^{\frac{2}{p'}}-1\le\frac{1}{p'}(\th^2-1)-\frac{1}{\max\{p,p'\}}(\th-1)^2$ ($\th\ge 0$, $p\in(1,\infty]$ and $\frac{1}{p}+\frac{1}{p'}=1$) from \cite[Lemma 1]{Alonso2019}. Thus, for any $p\ge 2$ and non-negative $f\ge 0$, we have 
\begin{align*}\notag
	&\mu^{\fr12}(v_*)f(v)\big(\mu^{\frac{1}{2}}(v'_*)f^{p-1}(v')-\mu^{\frac{1}{2}}(v_*)f^{p-1}(v)\big)
	\\
	&\le 
	\mu(v_*)f^p(v)\Big(\frac{(\mu^{\frac{p'}{4}}(v'_*)f^{\frac{p}{2}}(v'))^{\frac{2}{p'}}}{(\mu^{\frac{p'}{4}}(v_*)f^{\frac{p}{2}}(v))^{\frac{2}{p'}}}-1\Big)\\
	&\le\frac{\mu(v_*)f^p(v)}{p'}\Big(\frac{\mu^{\frac{p'}{2}}(v'_*)|f|^{p}(v')}{\mu^{\frac{p'}{2}}(v_*)|f|^{p}(v)}-1\Big)
	-\frac{\mu(v_*)f^p(v)}{\max\{p,p'\}}\Big(\frac{\mu^{\frac{p'}{4}}(v'_*)f^{\frac{p}{2}}(v')}{\mu^{\frac{p'}{4}}(v_*)f^{\frac{p}{2}}(v)}-1\Big)^2,
\end{align*}
where $p'$ is given by $\frac{1}{p}+\frac{1}{p'}=1$, and hence, we can decompose $\int_{v}\Ga(\mu^{\fr12},f)f|f|^{p-2}$ correspondingly:
\begin{align}\label{esLpT1}
	\big(\Ga(\mu^{\fr12},f),f|f|^{p-2}\big)_{L^2_v}\notag
	&\le \int_{v,v_*,\si}B\mu^{1-\frac{p'}{2}}(v_*)\Big\{\frac{1}{p'}\Big(\mu^{\frac{p'}{2}}(v'_*)|f|^{p}(v')-\mu^{\frac{p'}{2}}(v_*)|f|^{p}(v)\Big)
	\\
	&\quad-\frac{1}{p}\Big(\mu^{\frac{p'}{4}}(v'_*)|f|^{\frac{p}{2}}(v')-\mu^{\frac{p'}{4}}(v_*)|f|^{\frac{p}{2}}(v)\Big)^2\Big\}
	:=I^p_1+I^p_2. 
\end{align} 
The term $I^p_1$ can be computed by pre-post and regular change of variables \eqref{regular}, as it is of the cancellation form. That is, 
\begin{align}\label{esLpT2}\notag
	I^p_1&\notag=\frac{1}{p'}\int_{v,v_*,\si}
	B\Big(\mu^{1-\frac{p'}{2}}(v'_*)-\mu^{1-\frac{p'}{2}}(v_*)\Big)\mu^{\frac{p'}{2}}(v_*)|f|^{p}(v)\\
	&=\frac{1}{p'}\int_{v,v_*,\si}
	B\Big(\mu^{1-\frac{p'}{2}}(v'_*)\big(\mu^{\frac{p'}{2}}(v_*)-\mu^{\frac{p'}{2}}(v'_*)\big)
	+\big(\mu(v'_*)-\mu(v_*)\big)\Big)|f|^{p}(v)\notag\\
	&=:I^p_{1,1}+I^p_{1,2}.
\end{align}
For the term $I^p_{1,1}$, we use Taylor expansion to $\mu^{\frac{p'}{2}}(v_*)-\mu^{\frac{p'}{2}}(v'_*)$ for the case $\th\le|v-v_*|^{-1}$, apply the vanishing property \eqref{regularvprime} to the first-order expansion (replacing $(v,v')$ with $(v_*,v'_*)$ therein), and utilize the vanishing regular change of variables $v_*\mapsto v'_*$ \eqref{regularvprime} and the angular integral \eqref{angularintegral} to the second-order expansion, together with the estimate \eqref{gafg}: 
for any $N\ge 0$, 
\begin{align}\label{esLpT3}\notag
	|I^p_{1,1}|
	&= 
    \frac{C}{p'}\int_{v,v_*,\si}\1_{\th>|v-v_*|^{-1}}B\mu^{1-\frac{p'}{2}}(v'_*)|f|^{p}(v)\\
    &\notag\ \ +
    0 + \frac{1}{p'}\Big|\int_{v,v_*,\si}B\1_{\th\le|v-v_*|^{-1}}\mu^{1-\frac{p'}{2}}(v'_*)|f|^{p}(v)
	\int^1_0\na^2\mu^{\frac{p'}{2}}(v_*'+\tau(v_*-v_*')):(v_*-v_*')^{\otimes 2}\,d\tau\Big|\\
	&\notag\le \frac{C}{p'}\int_{v,v_*,\si}b(\cos\th)\sin^2\fr\th2|v-v_*|^{\ga+2}\mu^{1-\frac{p'}{2}}(v'_*)|f|^{p}(v)\\
	&\le C_p\int_{v,v_*}|v-v_*|^{\ga+2s}\mu^{1-\frac{p'}{2}}(v_*)|f|^{p}(v)
	\le C_p
	\|\<v\>^{\fr{\ga+2s}{2}}|f|^{\fr{p}{2}}\|_{L^2_v}^2.  
\end{align}
For the term $I^p_{1,2}$, we use the regular change of variables \eqref{regular} and the estimate \eqref{gafg} to deduce 
\begin{align}\label{esLpT4}
	I^p_{1,2}
	&= \frac{1}{p'}\int_{v,v_*,\si}b|v-v_*|^{\ga}\big(\frac{1}{\cos^{d+\ga}\fr\th2}-1\big)\mu(v_*)|f|^{p}(v)
	\le C_p
	\|\<v\>^{\frac{\ga}{2}}|f|^{\fr{p}{2}}\|_{L^2_v}.
(2.23)\end{align}
For the term $I^p_2$ in \eqref{esLpT1}, as in \cite[Prop. 2.16]{Alexandre2012}, by $(a+b)^2\ge\frac{1}{2}a^2-b^2$, we have 
\begin{align*}
	&\mu^{1-\frac{p'}{2}}(v_*)\Big(\mu^{\frac{p'}{4}}(v'_*)|f|^{\frac{p}{2}}(v')-\mu^{\frac{p'}{4}}(v_*)|f|^{\frac{p}{2}}(v)\Big)^2\\
	&\notag= \mu^{1-\frac{p'}{2}}(v_*)\Big(\big(\mu^{\frac{p'}{4}}(v'_*)-\mu^{\frac{p'}{4}}(v_*)\big)|f|^{\frac{p}{2}}(v')-\mu^{\frac{p'}{4}}(v_*)\big(|f|^{\frac{p}{2}}(v')-|f|^{\frac{p}{2}}(v)\big)\Big)^2
\\
&\ge \frac{\mu(v_*)}{2}\big(|f|^{\frac{p}{2}}(v')-|f|^{\frac{p}{2}}(v)\big)^2-\mu^{1-\frac{p'}{2}}(v_*)\big(\mu^{\frac{p'}{4}}(v'_*)-\mu^{\frac{p'}{4}}(v_*)\big)^2|f|^{p}(v').
\end{align*}
Moreover, the second factor of the first right-hand term can be further rewritten as 
\begin{align*}
	\big(|f|^{\frac{p}{2}}(v')-|f|^{\frac{p}{2}}(v)\big)^2
	&=-2|f|^{\frac{p}{2}}(v)\big(|f|^{\frac{p}{2}}(v')-|f|^{\frac{p}{2}}(v)\big)
	+|f|^{p}(v')-|f|^{p}(v). 
\end{align*}
Thus, by \eqref{CollBoltzDef} and \eqref{esLpT2}, we can rewrite $I^p_2$ as 
\begin{align}\label{esLpT5}\notag
	I^p_2
	&\le \frac{1}{p}\int_{v,v_*,\si}B\Big\{\mu(v_*)|f|^{\frac{p}{2}}(v)\big(|f|^{\frac{p}{2}}(v')-|f|^{\frac{p}{2}}(v)\big)
	+\frac{\mu(v_*)}{2}\big(|f|^{p}(v')-|f|^{p}(v)\big)\\
	&\notag\qquad\qquad\qquad+\mu^{1-\frac{p'}{2}}(v_*)\big(\mu^{\frac{p'}{4}}(v'_*)-\mu^{\frac{p'}{4}}(v_*)\big)^2|f|^{p}(v')\Big\}\\
	&=:\fr1p\big(\Ga(\mu^{\fr12},|f|^{\frac{p}{2}}),|f|^{\frac{p}{2}}\big)_{L^2_v}
	+I^p_{2,1}+I^p_{2,2}.
\end{align}
The term $I^p_{2,1}$ is of the cancellation form while $I^p_{2,2}$ contains the regular angular regularity $\th^2$ factor arising from $\big(\mu^{\frac{p'}{4}}(v'_*)-\mu^{\frac{p'}{4}}(v_*)\big)^2$. Thus, similar to \eqref{esLpT3} and \eqref{esLpT4}, by regular change of variable \eqref{regular}, angular integral \eqref{angularintegral}, and estimate \eqref{gafg}, we have 
\begin{align}\label{esLpT6}
	I^p_{2,1}\le C_p\|\<v\>^{\frac{\ga}{2}}|f|^{\fr{p}{2}}\|_{L^2_v},\ \ \
	I^p_{2,2}\le C_p\|\<v\>^{\frac{\ga+2s}{2}}|f|^{\fr{p}{2}}\|_{L^2_v}. 
\end{align}
Collecting estimates \eqref{esLpT3}, \eqref{esLpT4}, and \eqref{esLpT6} into \eqref{esLpT2} and \eqref{esLpT5}, and then into \eqref{esLpT1}, we obtain that, for any $d\ge 2$, $p\ge 2$ and $N\ge 0$,
\begin{align*}
	\big(\Ga(\mu^{\fr12},f),f|f|^{p-2}\big)_{L^2_v}
	&\le p^{-1}\big(\Ga(\mu^{\fr12},|f|^{\frac{p}{2}}),|f|^{\frac{p}{2}}\big)_{L^2_v}
	+C_{p,N}\|\<v\>^{\fr{\ga+2s}{2}}|f|^{\fr{p}{2}}\|_{L^2_v}^2. 
\end{align*}
provided that $\ga+2s>-d$.
Therefore, by \eqref{esAf}, we obtain for $d\ge 2$ and $p\ge 2$ that 
\begin{align}\label{GaLpes}
	\big(\Ga(\mu^{\frac{1}{2}},f),f|f|^{p-2}\big)_{L^2_v}
	&\le -c_{0}p^{-1}\||f|^{\frac{p}{2}}\|^2_{L^2_D}+C_p\|\<v\>^{\fr{\ga+2s}{2}}|f|^{\fr{p}{2}}\|_{L^2_v}^2, 
\end{align}
for some $c_0,C_p>0$, provided that $\ga+2s>-d$. 

\subsubsection{A coercive estimate}
Assume that $\Phi_1=\mu+\mu^{\fr12}\phi_1$ satisfies \eqref{lowerphi}. We calculate the coercive estimate of $(Q(\Phi_1,f),f)_{L^2_v}$. By the pre-post change of variables, 
\begin{align*}
	&(Q(\Phi_1,f),f)_{L^2_v}
	=\int_{v,v_*,\si}B(v-v_*,\si)\Phi_1(v_*)f(v)\big(f(v')-f(v)\big)\\
	&=-\fr12\int_{v,v_*,\si}B\Phi_1(v_*)\big(f(v')-f(v)\big)^2
	+\fr12\int_{v,v_*,\si}B\Phi_1(v_*)(f(v')^2-f(v)^2)\\
	&\le -\fr1C\int_{v,v_*,\si}Be^{-\fr{L_0|v_*|^2}{C}}\big(f(v')-f(v)\big)^2+C\int_{v,v_*}|v-v_*|^{\ga}\Phi_1(v_*)f(v)^2
\end{align*} 
where the first right-hand term can be rewritten as the first right-hand term of \eqref{vertiii} by dilation (the second right-hand term of \eqref{vertiii} satisfies \cite[Lemma 2.6]{Alexandre2012}), and the second term is calculated via the regular change of variables \eqref{regular}, the angular integral \eqref{angularintegral}. Thus, applying estimate \eqref{gafgDs} to the last term, we have 
\begin{align*}
	(Q(\Phi_1,f),f)_{L^2_v}
	&\le -c_0\|f\|_{L^2_D}^2+C\|\<v\>^{\fr{\ga+2s}{2}}f\|_{L^2_v}^2
	+C\|\<v\>^{|\ga|+d+1}\<D_v\>^{-\fr{d-2s}{2}}\Phi_1\|_{L^r_v}\|\<v\>^{\ga}f\|_{L^2_v}^2, 
\end{align*}
for some generic constant $c_0=c_0(\ga,s,C_{\mathrm{low}},L_0)>0$, provided that $\ga-\fr{d-2s}{2}>-\fr{d}{r'}$. 
Together with the commutator estimate $\Ga-Q$ in \eqref{esGaminQTemp}, we have 
\begin{align}\label{GaPhi1ff}\notag
	&(\Ga(\mu+\phi_1,f),f)_{L^2_v}
	\le -c_0\|f\|_{L^2_D}^2+C\|\<v\>^{\fr{\ga+2s}{2}}f\|_{L^2_v}^2\\
	&\notag
	\qquad
	+C\|\<v\>^{|\ga|+d+1}\<D_v\>^{-\fr{d-2s}{2}}\mu^{\fr12}\phi_1\|_{L^r_v}\|\<v\>^{\ga}f\|_{L^2_v}^2\\
	&\notag\qquad+C\|\<v\>^{|n|+|\ga|+d+1}\<D_v\>^{-s}\phi_1\|_{L^r_v}\|\<v\>^{-n+\fr{\ga+2s}{2}}f\|_{L^2_v}\|\<v\>^{n+\fr{\ga+2s}{2}}f\|_{L^2_v}\\
	&\qquad+C\|\<v\>^{-N}\<D_v\>^{-\fr{d-2s}{2}}\phi_1\|_{L^r_v}\|\<v\>^{-n+\fr{\ga+2s}{2}}f\|_{L^2_v}\|\<v\>^{n}f\|_{L^2_D}, 
\end{align}
provided that $\ga-\fr{d-2s}{2}>-\fr{d}{r'}$ and $\phi_1$ satisfies \eqref{lowerphi}, for any $n\in\R$ and $N\ge 0$.

\section{Spatial upper bounds and commutator estimates}
To apply the pseudo-differential operator $\De_j=\int_{\R^{d}_{\xi}}e^{2\pi i\xi\cdot x}\wh{\Psi_j}(\varpi^{-1}\xi)\F_xf(\xi)\,d\xi$, defined in \eqref{DejQkDef}, to the collision terms, we will consider the decomposition $f=f_\de+\wt{f}_\de$ in \eqref{splitphi} later, where $f_\de$ is regular while $\wt{f}_\de$ is irregular but suitably small. We deal with these two parts in the following subsections, while the detailed decomposition will be given by \eqref{splitphi} later. 

\subsection{Commutator estimate of $[\texorpdfstring{\De_jP_kR_r}{DeltajPkRr},\texorpdfstring{\Ga}{Gamma}(f,\cdot)]$ for regular $f$}
In this part, we calculate the commutator $\De_jP_kR_r\Ga(f,g)-\Ga(f,\De_jP_kR_rg)$ by assuming that $f$ is sufficiently smooth with respect to $x$. Firstly, noting $\F_x\De_jf(\xi)=\wh{\Psi_j}(\varpi^{-1}\xi)\wh{f}(\xi)$, we have  
\begin{align*}
	&\De_j(fg)-f\De_jg
	=\int_{\xi,\xi_*}e^{2\pi i\xi\cdot x}
	\Big(\wh{\Psi_j}(\varpi^{-1}\xi)-\wh{\Psi_j}(\varpi^{-1}\xi_*)\Big)\wh{f}(\xi-\xi_*)\wh{g}(\xi_*)\\
	&=\int_{\xi,\xi_*}e^{2\pi i\xi\cdot x}\int^1_0\,d\tau~
	\varpi^{-1}(\xi-\xi_*)\na\wh{\Psi_j}(\varpi^{-1}(\xi_*+\tau(\xi-\xi_*)))\wh{f}(\xi-\xi_*)\wh{g}(\xi_*)\\
	&=\int_{\xi,\xi_*,\tau}\fr{e^{2\pi i\xi\cdot x}}{2\pi i}
	\varpi^{-1}\na\wh{\Psi_j}(\varpi^{-1}(\xi_*+\tau(\xi-\xi_*)))\wh{\na_xf}(\xi-\xi_*)\wh{g}(\xi_*).
\end{align*}
Applying this estimate to $\De_j\Ga(f,g)-\Ga(f,\De_jg)$ and using the collision estimate \eqref{UpperGaNoSum} (with $s_1=s$ and $n=0$ therein), we have 
\begin{align}\label{CommDejfTemp}\notag
	&\big|\big(\De_j\Ga(f,g)-\Ga(f,\De_jg),h\big)_{L^2_{x,v}}\big|\\
	\notag
	&\le \int_{\xi,\xi_*,\tau}
	\varpi^{-1}\Big|\na\wh{\Psi_j}(\varpi^{-1}(\xi_*+\tau(\xi-\xi_*)))\Big(\Ga\big(\wh{\na_xf}(\xi-\xi_*),\wh{g}(\xi_*)\big),\wh{h}(\xi)\Big)_{L^2_v}\Big|\\
	&\notag\le C\int_{\xi,\xi_*}
	\varpi^{-1}2^{-j}
	\Big(
	\|\<D_v\>^{-s}\<v\>^{-N}\wh{\na_xf}(\xi-\xi_*)\|_{L^2_v}\|\wh{g}(\xi_*)\|_{L^2_D}\|\wh{h}(\xi)\|_{L^2_D}
	\\
	&\notag\qquad\qquad\qquad +\|\wh{\na_xf}(\xi-\xi_*)\|_{L^2_v}\|\<v\>^{\frac{\ga+2s}{2}}\wh{g}(\xi_*)\|_{L^2_v}\|\wh{h}(\xi)\|_{L^2_D}\Big)
	\\
	&\le 
	C\varpi^{-1}2^{-j}
	\|\<D_x\>^{\frac{d+3}{2}}f\|_{L^2_{x,v}}\|g\|_{L^2_xL^2_D}\|h\|_{L^2_xL^2_D}, 
\end{align}
provided that $\ga+s>-\frac{d}{2}$, 
where we simply used $\|\na\wh{\Psi_j}\|_{L^\infty}\le C2^{-j}$ and $\|\wh{f}\|_{L^1_{\xi}}\le C\|\<D_x\>^{\frac{d+1}{2}}f\|_{L^2_x}$. Thus, using $\|\cdot\|_{L^2_D}\le C\|\<v\>^{\fr{\ga+2s}{2}}\<D_v\>^{s}(\cdot)\|_{L^2_v}$ with $\ga+2s<0$, applying upper bound \eqref{dissUpper} to the $g,h$ terms and noticing $\varpi=\varpi_02^{rl_0}$, we have for any $\ka>0$ that 
\begin{align}\label{CommDejf}\notag
	&\sum_{r,j,k\ge 0}(\om2^k)^{-2s}2^{2rl}\big|\big(\De_j\Ga(f,P_kR_rg)-\Ga(f,\De_jP_kR_rg),\De_jP_kR_rh\big)_{L^2_{x,v}}\big|
	\\
	&\notag\le C_\ka\varpi_0^{-2}
	\|\<D_x\>^{\frac{d+3}{2}}f\|_{L^2_{x,v}}^2
	\|\<v\>^{-l_0+l+\fr{\ga+2s}{2}}g\|_{L^2_{x,v}}^2\\
	&\quad
	+\ka\sum_{r,j,k\ge 0}(\om2^k)^{-2s}2^{2rl}\|\De_jP_kR_rh\|_{L^2_xL^2_D}^2. 
\end{align}

\subsection{Bony-type upper bounds of \texorpdfstring{$\De_j$}{Delta j}\texorpdfstring{$\Ga(f,g)$}{Gamma(f,g)} for irregular $f$}\label{SecSpatialNonlinear}
We deal with the nonlinear collision term $\De_j\Ga(f,g)$ that appears for $t\in[0,T]$ in equation \eqref{eqkinetic12}. 
In this Subsection, we write $\om=\om(j)$ and $P^k=P^j_k$ to emphasize the {\bf dependence} of $\om$ on $j$, and use the spatial and velocity pseudo-differential operators $\De_j,P^j_k$ defined in \eqref{DejQkDef} with dilation $\varpi$ and $\om=\om(j)$, respectively, i.e., 
\begin{align}
	\label{Def712}\notag
	\De_jf(x)&=\int_{\R^{d}_{\xi}}e^{2\pi i\xi\cdot x}\wh{\Psi_j}(\varpi^{-1}\xi)\F_{x}f(\xi)\,d\xi,\qquad\qquad \varpi=\varpi_02^{rl_0},\\
	P^j_kf(v)&=\int_{\R^d_\eta}e^{2\pi i\eta\cdot v}\wh{\Psi_{k}}((\om(j))^{-1}\eta)\F_vf(\eta)\,d\eta,\quad \om(j)=\om_02^{\al j+rl_1},
\end{align}
Moreover, we assume the partition-of-unity property of $\Psi_j$ (see \eqref{sum1}) in this subsection. Hence,
	$f(x)=\sum_{j=0}^\infty\De_jf(x)$ with any fixed $r\ge 0$. 
For the collision term $\Ga(f,P^j_kR_rg)$, we have the upper bound from Theorem \ref{ThmUpper}: for any $N>0$, 
\begin{align}\label{UpperGa}\notag
	|(\Ga(f,P^j_kR_rg),P^j_kR_rh)_{L^2_v}|
	&\notag\le C\|\<D_v\>^{-s}\<v\>^{-N}f\|_{L^2_v}\|P^j_kR_rg\|_{L^2_D}\|P^j_kR_rh\|_{L^2_D}\\
  &\ \ +C\|f\|_{L^2_v}\|\<v\>^{\frac{\ga+2s}{2}}P^j_kR_rg\|_{L^2_v}\|P^j_kR_rh\|_{L^2_D}. 
\end{align}

\subsubsection{The Bony decomposition}\label{SubsecBony}
The Bony decomposition was also used in \cite{Duan2015a,Morimoto2016} for kinetic equations. 
Splitting $f=\sum_{j}\De_jf,g=\sum_{j}\De_jg$, we have the Bony paraproduct decomposition 
\begin{align*}
	f(x)g(x)=\sum_{q=0}^\infty S_{q-1}f\De_qg+
	\sum_{p=0}^\infty \De_pfS_{p-1}g+\sum_{|p-q|\le 1}\De_pf\De_qg,
\end{align*}
where $S_{p-1}=\sum_{q\le p-2}\De_q$. 
Applying $\De_j$ to the Bony decomposition and applying their corresponding support on the time-spatial-frequency variable, we have 
\begin{align}\label{eqBony}\notag
	\De_j(fg)
	&=\De_j\sum_{|q-j|\le 4}S_{q-1}f\De_qg
	+\De_j\sum_{|p-j|\le 4}\De_pfS_{p-1}g
  \\
	&\quad+\De_j\sum_{\max\{p,q\}\ge j-2}\sum_{|p-q|\le 1}\De_pf\De_qg.
\end{align}
We will multiply by $(\om(j)2^k)^{-2s}$ to eliminate the $s$-order derivative in our main energy estimate. 
That is, taking $\sum_{j,k\ge 0}(\om(j)2^k)^{-2s}\times\big(\De_j\Ga(f,P^j_kR_rg),P^j_kR_rh\big)_{L^2_{x,v}}$, and using the estimate \eqref{UpperGa}, we have, for any $N>0$, 
\begin{align}\label{UpperTT}\notag
	&\sum_{j,k\ge 0}(\om(j)2^k)^{-2s}\big(\De_j\Ga(f,P^j_kR_rg),P^j_kR_rh\big)_{L^2_{x,v}}\\
	&\notag\le C\sum_{j,k\ge 0}(\om(j)2^k)^{-2s}
	\bigg\{\sum_{|q-j|\le 4}\int_{\R^d_x}
	\|S_{q-1}\<D_v\>^{-s}\<v\>^{-N}f\|_{L^2_v}\|\De_{q}P^j_kR_rg\|_{L^2_D}\|\De_jP^j_kR_rh\|_{L^2_D}\\
	&\quad\notag+
	\sum_{|p-j|\le 4}\int_{\R^d_x}
	\|\De_q\<D_v\>^{-s}\<v\>^{-N}f\|_{L^2_v}\|S_{q-1}P^j_kR_rg\|_{L^2_D}\|\De_jP^j_kR_rh\|_{L^2_D}\\
	&\quad\notag+
	\sum_{p,q\ge 0}\1_{\substack{|p-q|\le 1,\qquad\\\max\{p,q\}\ge j-2}}
	\int_{\R^d_x}\|\De_p\<D_v\>^{-s}\<v\>^{-N}f\|_{L^2_v}\|\De_{q}P^j_kR_rg\|_{L^2_D}\|\De_jP^j_kR_rh\|_{L^2_D}\\
	&\quad\notag+
	\sum_{|q-j|\le 4}\int_{\R^d_x}
	\|S_{q-1}f\|_{L^2_v}\|\<v\>^{\frac{\ga+2s}{2}}\De_{q}P^j_kR_rg\|_{L^2_v}\|\De_jP^j_kR_rh\|_{L^2_D}\\
	&\quad\notag+
	\sum_{|p-j|\le 4}\int_{\R^d_x}
	\|\De_pf\|_{L^2_v}\|\<v\>^{\frac{\ga+2s}{2}}S_{p-1}P^j_kR_rg\|_{L^2_v}\|\De_jP^j_kR_rh\|_{L^2_D}\\
	&\quad\notag+
	\sum_{p,q\ge 0}
	\1_{\substack{|p-q|\le 1,\qquad\\\max\{p,q\}\ge j-2}}\int_{\R^d_x}
	\|\De_pf\|_{L^2_v}\|\<v\>^{\frac{\ga+2s}{2}}\De_{q}P^j_kR_rg\|_{L^2_v}\|\De_jP^j_kR_rh\|_{L^2_D}\bigg\}\\
	&=:\sum_{i=1}^6\Ga_i(f,g,h).
\end{align}

\subsubsection{Basic observations and properties}\label{SubsecBasicBony}
For the operator $S_{q-1}=\sum_{j\le q-2}\De_j$, since its kernel can be rewritten as $\sum_{j\le q-2}\wh{\Psi_j}(\varpi^{-1}\xi)=\wh\vp(2^{-(q-2)}\varpi^{-1}\xi)$ with some Schwartz function $\wh\vp$, we have \emph{Young's convolution inequality} for any $p\in[1,\infty]$, given by 
\begin{align}\notag
	S_{q-1}f(x)&=\int_{\R^{2d}_{y,\xi}}\sum_{j\le q-2}\wh{\Psi_j}(\varpi^{-1}\xi)f(y)e^{2\pi i\xi\cdot(x-y)}\,dyd\xi\\
	\notag&=\int_{\R^d_x}(2^{q-2}\varpi)^d\vp(2^{q-2}\varpi(x-y))f(y)\,dy,\\
	\label{esSq}
	\|S_{q-1}f(x)\|_{L^p_xL^2_v}&\le \Big\|\int_{\R^d_x}(2^{q-2}\varpi)^d\vp(2^{q-2}\varpi(x-y))\|f(y)\|_{L^2_v}\,dy\Big\|_{L^p_x}\le \|f\|_{L^p_xL^2_v}. 
\end{align}
The same estimate holds for the $L^2_D$ norm replacing the $L^2_v$ norm, since $\|f\|_{L^2_D}=\|\ti{a}^wf\|_{L^2_v}$. 
Moreover, we use the following Bernstein-type inequality. Let $r_0,r_0^*\in(1,\infty]$ satisfying $1+\frac{1}{r_0}=\frac{1}{2}+\frac{1}{(1-\ve/d)^{-1}}$ with any $\ve\in(0,d)$, by Young's convolution inequality as in \eqref{esSq}, we have 
\begin{align}\label{SobolevEmbeddLp}\notag
	\|\De_jf\|_{L^{r_0}_xL^2_D}&=\|\wt{\De}_j\De_jf\|_{L^{r_0}_xL^2_D}
	\le \|\varpi^d\wt{\Psi_j}(\varpi(\cdot))\|_{L^{(1-\ve/d)^{-1}}_x}\|\De_jf\|_{L^2_xL^2_D}\\
	&
	\notag\le C(\varpi2^j)^{d(1-(1-\ve/d))}\|\De_jf\|_{L^2_xL^2_D}
	\\
	&
	\le C(\varpi_02^{rl_0}2^{j})^{\ve}\|\De_jf\|_{L^2_xL^2_D}. 
\end{align}
This will be utilized when $r_0>2$ is close to $2$. 
Similar estimates hold for the $L^2_v$ norm replacing the $L^2_D$ norm.
Here and below, we use the technique of $\De_j=\De_j\wt{\De}_j$ to preserve the term $\De_j$, where $\wt{\De_j}$ is defined as in \eqref{SlightlyLarger}. 
Furthermore, the equivalence of weight $2^{r}\approx\<v\>$ will be used frequently due to the support of $R_r(v)$. 

\smallskip 
Finally, we need to deal with the term $\|\De_qP^j_kR_rg\|_{L^2_D}$. Recall that our dilated pseudo-differential operator $P^j_k$ with respect to the velocity depends on spatial-frequency index $j$, while the term of $g$ in \eqref{UpperTT} has a different spatial-frequency index. Our goal is to apply translation in $k$ to $\De_qP^j_k$ to obtain $\De_jP^j_k$, which has the {\bf same} index $j$ in both the spatial dyadic decomposition and the dilation $\om(j)$ in $P^j_k$. 
Notice that the kernel $\wh{\Psi_0}(\om(j)^{-1}\eta)$ of $P^j_0$ and the kernel $\wh{\Psi_k}(\om(j)^{-1}\eta)$ of $P^j_k$ $(k\in\R)$ are supported in (recall that $\om(j)=\om_02^{\al j+rl_1}$)
\begin{align*}
	\{|\eta|\le 2\om(j)\}&=\{|\eta|\le 2\om(q)2^{\al(j-q)}\},\\
	\big\{\frac{6}{7}\om(j)2^k\le |\eta|\le 2\om(j)2^k\big\}&=\big\{\frac{6}{7}\om(q)2^{\al(j-q)+k}\le |\eta|\le 2\om(q)2^{\al(j-q)+k}\big\},\quad k\ge 1,
\end{align*}
respectively. We denote the floor function of $z\in\R$ by $\lfloor z \rfloor$, the greatest integer less than or equal to $z$. Using the partition-of-unity property in \eqref{sum1} and their supports, we have 
\begin{align}\label{Pj0ksum}
	\begin{aligned}
	P^j_0&=P^j_0\big(P^q_0+P^q_1+\cdots+P^q_{\max\{0,\,\lfloor\al(j-q)\rfloor+2\}}\big),\\
	P^j_k&=P^j_k\big(P^q_{\max\{0,\,k+\lfloor\al(j-q)\rfloor-1\}}+\cdots+P^q_{\max\{0,\,k+\lfloor\al(j-q)\rfloor+2\}}\big).
	\end{aligned}
\end{align}
Note that $P^j_k=P^j_kP^q_0$ whenever $k+\lfloor\al(j-q)\rfloor\le -2$. 
For the upper bound of $\|P^j_kg\|_{L^2_D}=\|\ti{a}^wP^j_kg\|_{L^2_v}$, note that $\ti{a}\in S(\ti{a})$, as given by \eqref{tiaa}, admits an invertible Weyl quantization $\ti{a}^w$ (as shown in \cite{Global2019,Deng2020a}), and $\ti{a}^wP^j_k\in \operatorname{Op}(\ti{a}^w)$). Using \eqref{boundDea} and Lemma \ref{basicCommuLem}, we obtain
\begin{align*}
	\|P^j_kg\|_{L^2_D}=\|\underbrace{\ti{a}^wP^j_k}_{\in\operatorname{Op}(\ti{a})}g\|_{L^2_{v}}\le C\|g\|_{L^2_D},
\end{align*} 
for $j,k\ge 0$ and any suitable $g$, 
where $C>0$ is independent of $j,k,q$. 

\smallskip For any $j\ge 0$, on the support of $\1_{q\le j+4}$, using \eqref{Pj0ksum} and taking translation $k\mapsto k+\lfloor \al(j-q)\rfloor$ in the last inequality (which is valid due to the indicator $\1_{q\le j+4}$; otherwise, if $q-j\to \infty$, there are infinitely many $P_0$ in the summation $\sum_{k\ge 0}$ which is not summable), we have 
\begin{align}
	\label{esg719}\notag
	&\sum_{k,q\ge 0}\1_{q\le j+4}(\om(j)2^k)^{-2s}\|\De_{q}P^j_kg\|^2_{L^2_D}\\
	&\notag\le C
	\sum_{q\ge 0}
	\1_{q\le j+4}
	(\om(q)2^{\al(j-q)})^{-2s}
	\|\De_{q}P^j_0\big(P^q_0+P^q_1+\cdots+P^q_{\max\{0,\,\lfloor\al(j-q)\rfloor+2\}}\big)g\|^2_{L^2_D}
	\\
	&\ \notag +
	C\sum_{q\ge 0}\sum_{k\ge 1}
	\1_{q\le j+4}(\om(q)2^{k+\al(j-q)})^{-2s}\sum_{-1\le m\le 2}
	\|\De_{q}P^j_kP^q_{\max\{0,\,k+\lfloor\al(j-q)\rfloor+m\}}g\|^2_{L^2_D}
	\\
	&\le C\sum_{q,k\ge 0}(\om(q)2^{k})^{-2s}\|\De_{q}P^q_{k}g\|^2_{L^2_D}, 
\end{align}
and similarly, on the support of $\1_{|q-j|\le 4}$, we have 
\begin{align}
	\label{esg719s}\notag
	&\sum_{j,k,q\ge 0}\1_{|q-j|\le 4}(\om(j)2^k)^{-2s}\|\De_{q}P^j_kg\|^2_{L^2_D}\\
	&\notag\le 
	\sum_{j,q\ge 0}\1_{|q-j|\le 4}
	(\om(q)2^{\al(j-q)})^{-2s}
	\|\De_{q}P^j_0\big(P^q_0+P^q_1+\cdots+P^q_{\max\{0,\,\lfloor\al(j-q)\rfloor+2\}}\big)g\|^2_{L^2_D}
	\\
	&\notag\ +
	\sum_{j,q\ge 0}
	\sum_{k\ge 1}
	\1_{|q-j|\le 4}(\om(q)2^{k+\al(j-q)})^{-2s}\sum_{-1\le m\le 2}
	\|\De_{q}P^j_kP^q_{\max\{0,\,k+\lfloor\al(j-q)\rfloor+m\}}g\|^2_{L^2_D}\\
	&\le C\sum_{j,q,k\ge 0}\1_{|q-j|\le 4}(\om(q)2^{k})^{-2s}\|\De_{q}P^q_{k}g\|^2_{L^2_D}
	\le C\sum_{q,k\ge 0}(\om(q)2^{k})^{-2s}\|\De_{q}P^q_{k}g\|^2_{L^2_D}. 
\end{align}
Next, we deal with the case of $\max\{p,q\}\ge j-2$ and $|p-q|\le 1$ in the Bony decomposition \eqref{eqBony}. 
The case of $q\le j+4$ is calculated in \eqref{esg719}. 
For the case of $q\ge j+5$ and ${|p-q|\le 1}$, using the property of supports in \eqref{Pj0ksum}, splitting the summation $k\ge 0$ into $0\le k\le -\lfloor\al(j-q)\rfloor$ and $k\ge -\lfloor\al(j-q)\rfloor+1$, applying $\|\cdot\|_{L^2_D}\le C\|\<v\>^{\frac{\ga+2s}{2}}\<D_v\>^s(\cdot)\|_{L^2_v}$ to the case of small $k$ and translation $k\mapsto k+\lfloor\al(j-q)\rfloor$ to the large $k$, and used inequality $
(q+rl_0+\log_2\vpi_0+1)^{-1}
\le C_{\om_0}(2^{\al q+rl_1})^{-2s}+\om_0^{-2s}
\le C_{\om_0,l_0,l_1}\om^{-2s}+\om_0^{-2s}$, we have 
\begin{align}\notag\label{ges1a}\notag
	&\sum_{k,q\ge 0}\1_{|p-q|\le 1}\1_{q\ge j+5}(\om(j)2^k)^{-2s}\|\De_{q}P^j_kg\|^2_{L^2_D}\\
	&\notag\le C
	\sum_{q\ge 0}\sum_{k=0}^{-\lfloor\al(j-q)\rfloor}
	\1_{|p-q|\le 1}\1_{q\ge j+5}
	(\om(j)2^k)^{-2s}
	\|\De_{q}\<v\>^{\fr{\ga+2s}{2}}\underbrace{\<D_v\>^sP^j_k}_{\in\operatorname{Op}((\om2^k)^{s})}P^q_0g\|^2_{L^2_v}
	\\
	&\notag\ \ \ +
	C\sum_{q\ge 0}\sum_{k=-\lfloor\al(j-q)\rfloor+1}^\infty
	\1_{|p-q|\le 1}\1_{q\ge j+5}(\om(q)2^{k+\al(j-q)})^{-2s}
	\sum_{m=-1}^2\|\De_{q}P^j_kP^q_{\max\{0,\,k+\lfloor\al(j-q)\rfloor+m\}}g\|^2_{L^2_D}
	\\
	&\notag\le C
	\sum_{q\ge 0}\al(q-j)
	\1_{|p-q|\le 1}\1_{q\ge j+5}
	\|\De_{q}\<v\>^{\fr{\ga+2s}{2}}P^q_0g\|^2_{L^2_v}
	+
	C\sum_{q\ge 0}\sum_{k=1}^\infty
	\1_{|p-q|\le 1}\1_{q\ge j+5}(\om(q)2^{k})^{-2s}
	\|\De_{q}P^q_{k}g\|^2_{L^2_D}
	\\
	&\notag\le C_\al\1_{p\ge 1}(q+1)(q+rl_0+\log_2\vpi_0+1)
	\sum_{q\ge 1}(q+rl_0+\log_2\vpi_0+1)^{-1}
	\|\<v\>^{\fr{\ga+2s}{2}}\De_{q}P^q_0g\|^2_{L^2_v}
	\\
	&\notag\ \ 
	+
	C\sum_{q,k\ge 0}(\om(q)2^{k})^{-2s}
	\|\De_{q}P^q_{k}g\|^2_{L^2_D}
	\\
	&\notag\le C_\al\1_{p\ge 1}(q+rl_0+\log_2\vpi_0+1)^2
	\Big(\sum_{q\ge 1}
		C_{\om_0,l_0,l_1}(\om(q))^{-2s}\|\<v\>^{\fr{\ga+2s}{2}}\De_{q}P^q_0g\|^2_{L^2_v}
	+\om_0^{-2s{}}\|g\|_{L^2_v}^2\Big)\\
	&\ \ \ +
	C\sum_{q,k\ge 0}(\om(q)2^{k})^{-2s}
	\|\De_{q}P^q_{k}g\|^2_{L^2_D},
\end{align}
where we used $\sum_{q\ge 1}\|\<v\>^{\fr{\ga+2s}{2}}\De_{q}P^q_0g\|^2_{L^2_v}\le \|g\|_{L^2_v}^2$ and simply estimate $q+1\le q+rl_0+\log_2\vpi_0+1$. 
	In the first sum, we use that $q\ge j+5\ge1$ and $p\ge q-1\ge1$, so that the cases $p=0$ and $q=0$ are excluded, which is a \textbf{crucial} point in our estimate.

\subsubsection{Estimate of $\Ga_i(\phi_1,f,f)$ and $\Ga_i(f,\phi_2,f)$}\label{SecGammai}
For the difference $f=\phi_1-\phi_2$, using the trilinear form $\Ga_i(f,g,h)$ given by \eqref{UpperTT}, two nonlinear terms $\Ga(\phi_1,f)$ and $\Ga(f,\phi_2)$ shall be analyzed as 
\begin{align}\label{DecomGafg}
	\begin{aligned}
	\sum_{j,k\ge 0}(\om(j)2^k)^{-2s}\big(\De_j\Ga(\phi_1,P^j_kR_rf),P^j_kR_rf\big)_{L^2_{x,v}}
	\le \sum_{i=1}^6\Ga_i(\phi_1,f,f),\\
	\sum_{j,k\ge 0}(\om(j)2^k)^{-2s}\big(\De_j\Ga(f,P^j_kR_r\phi_2),P^j_kR_rf\big)_{L^2_{x,v}}
	\le \sum_{i=1}^6\Ga_i(f,\phi_2,f).
	\end{aligned}
\end{align}
Then we aim at estimating $\Ga_i(\phi_1,f,f)$ and $\Ga_i(f,\phi_2,f)$ for $1\le i\le 6$. In what follows, we let $r_0,r_0^*\in(2,\infty)$ be such that $\frac{1}{2}=\frac{1}{r_0}+\frac{1}{r_0^*}$. 

\smallskip For the $\Ga_i(\phi_1,f,f)$ $(1\le i\le 3)$ terms, we will use the  $l^2_{j,k}-l^2_{j,k}$ and the  $L^\infty_{x}-L^2_{x}-L^2_{x}$ H\"older's inequality to $\phi_i$-$f$-$f$.
For the term $\Ga_1(\phi_1,f,f)$, by estimate \eqref{esSq}, H\"older's inequality, $\sum_{q\ge 0}\1_{|q-j|\le 4}|v_q|\le C\big(\sum_{q\ge 0}\1_{|q-j|\le 4}|v_q|^2\big)^{\fr12}$, and estimate \eqref{esg719s}, we have 
\begin{align*}
	\Ga_1(\phi_1,f,f)
	&\notag\le C\|\<D_v\>^{-s}\<v\>^{-N}\phi_1\|_{L^\infty_{x}L^2_v}
	\Big(\sum_{j,k,q\ge 0}\1_{|q-j|\le 4}(\om(j)2^k)^{-2s}\|\De_{q}P^j_kR_rf\|^2_{L^2_{x}L^2_D}\Big)^{\frac{1}{2}}\\
	&\quad\notag\times\Big(\sum_{j,k,q\ge 0}\1_{|q-j|\le 4}(\om(j)2^k)^{-2s}\|\De_jP^j_kR_rf\|^2_{L^2_{x}L^2_D}\Big)^{\frac{1}{2}}\\
	&\le 
		C\|\<D_v\>^{-s}\<v\>^{-N}\phi_1\|_{L^\infty_{x}L^2_v}\sum_{j,k\ge 0}(\om(j)2^k)^{-2s}\|\De_jP^j_kR_rf\|^2_{L^2_{x}L^2_D}.
\end{align*}
Similarly, for the term $\Ga_2$, by splitting $S_{p-1}=\sum_{q\le p-2}\De_q$, using $\sum_{j,p,q\ge 0}\1_{|p-j|\le 4}\1_{q\le p-2}|u_p|^2\le C\sum_{p\ge 0}(p+1)|u_p|^2$, and applying \eqref{esg719}, we have 
\begin{align*}
	&\Ga_2(\phi_1,f,f)
	\notag\le\sum_{j,k,p,q\ge 0}(\om(j)2^k)^{-2s}\1_{|p-j|\le 4}\1_{q\le p-2}\|\De_{p}\<D_v\>^{-s}\<v\>^{-N}\phi_1\|_{L^\infty_{x}L^2_v}\\
	&\qquad\qquad\qquad\qquad\qquad\qquad\notag\times\|\De_{q}P^j_kR_rf\|_{L^2_{x}L^2_D}\|\De_jP^j_kR_rf\|_{L^2_{x}L^2_D}\\
	&\notag\le 
	\Big(\sum_{j,k\ge 0}\Big\{\sum_{p,q\ge 0}\1_{|p-j|\le 4}\1_{q\le p-2}\|\De_{p}\<D_v\>^{-s}\<v\>^{-N}\phi_1\|_{L^\infty_{x}L^2_v}^2\Big\}\\
	&\notag\quad\times \Big\{\sum_{p,q\ge 0}\1_{|p-j|\le 4}\1_{q\le j+2}(\om(j)2^k)^{-2s}\|\De_{q}P^j_kR_rf\|^2_{L^2_{x}L^2_D}\Big\}\Big)^{\frac{1}{2}}
	\Big(\sum_{j,k\ge 0}(\om(j)2^k)^{-2s}\|\De_jP^j_kR_rf\|^2_{L^2_{x}L^2_D}\Big)^{\frac{1}{2}}\\
	&\le 
	C\Big(\sum_{p\ge 0}(p+1)\|\De_{p}\<D_v\>^{-s}\<v\>^{-N}\phi_1\|_{L^\infty_{x}L^2_v}^2\Big)^{\frac{1}{2}}
	\sum_{j,k\ge 0}(\om(j)2^k)^{-2s}\|\De_jP^j_kR_rf\|_{L^2_{x}L^2_D}^2. 
\end{align*}

For the term $\Ga_3$, using $l^2-l^2$ H\"older's inequality, we have from \eqref{esg719s} (for $j-3\le q\le j+4$) and \eqref{ges1a} 
(for $q\ge j+5$) that 
\begin{align*}
	&\Ga_3(\phi_1,f,f)\\
	&\notag\le\Big(\sum_{j,k\ge 0}\Big\{\sum_{p,q\ge 0}
	(q+rl_0+\log_2\vpi_0+1)^2\1_{q\ge j+5}
	\1_{\substack{|p-q|\le 1,\qquad\\\max\{p,q\}\ge j-2}}\|\De_p\<D_v\>^{-s}\<v\>^{-N}\phi_1\|^2_{L^\infty_{x}L^2_v}
	\\
	&\notag\qquad\times
		\sum_{p,q\ge 0}
	(q+rl_0+\log_2\vpi_0+1)^{-2}\1_{q\ge j+5}
	\1_{\substack{|p-q|\le 1,\qquad\\\max\{p,q\}\ge j-2}}(\om(j)2^k)^{-2s}\|\De_qP^j_kR_rf\|_{L^2_{x}L^2_D}^2
		\Big\}\\
	&\quad
	+\Big\{\sum_{p,q\ge 0}
	\1_{|q-j|\le 4}
	\1_{\substack{|p-q|\le 1,\qquad\\\max\{p,q\}\ge j-2}}\|\De_p\<D_v\>^{-s}\<v\>^{-N}\phi_1\|^2_{L^\infty_{x}L^2_v}
	\\
	&\notag\qquad\times\sum_{p,q\ge 0}
	\1_{|q-j|\le 4}
	\1_{\substack{|p-q|\le 1,\qquad\\\max\{p,q\}\ge j-2}}(\om(j)2^k)^{-2s}\|\De_qP^j_kR_rf\|_{L^2_{x}L^2_D}^2\Big\}
	\Big)^{\frac{1}{2}}
	\\
	&\quad\times\notag
	\Big(\sum_{j,k\ge 0}(\om(j)2^k)^{-2s}\|\De_jP^j_kR_rf\|^2_{L^2_{x}L^2_D}\Big)^{\frac{1}{2}}\\
	&\le\notag
	C_\al \Big(\sum_{p\ge 0}(q+rl_0+\log_2\vpi_0+1)^2\|\De_p\<D_v\>^{-s}\<v\>^{-N}\phi_1\|^2_{L^\infty_{x}L^2_v}
	\Big)^{\frac{1}{2}}\\
	&\ \ \times
	\Big(
		C_{\om_0,l_0,l_1}\sum_{q\ge 1}\om^{-2s}\|\De_{q}P^q_0R_rf\|^2_{L^2_{x,v}}
	+\om_0^{-2s{}}\|R_rf\|_{L^2_{x,v}}^2
	+
	\sum_{q,k\ge 0}(\om(q)2^{k})^{-2s}
	\|\De_{q}P^q_{k}R_rf\|^2_{L^2_{x}L^2_D}\Big)^{\frac{1}{2}}\\
	&\ \ \times
	\Big(\sum_{j,k\ge 0}(\om(j)2^k)^{-2s}\|\De_jP^j_kR_rf\|^2_{L^2_{x}L^2_D}\Big)^{\frac{1}{2}}. 
\end{align*}
For the terms $\Ga_i(\phi_1,f,f)$ $(4\le i\le 6)$ in \eqref{UpperTT}, we use the  $l^2_k-l^2_k$, $l^2_{p,q}-l^2_{p,q}$ H\"older's inequality for sequences, followed by the  $L^{r_0^*}_{x}-L^2_{x}-L^{r_0}_{x}$ H\"older's inequality with $\frac{1}{2}=\frac{1}{r_0}+\frac{1}{r_0^*}$ and $r_0,r_0^*\in(2,\infty)$.
By Littlewood-Paley Theorem \ref{LPThm} (e.g., $\sum_{k\ge 0}\|P^j_kf\|^2_{L^2_{v}}\le C\|f\|_{L^2_{v}}^2$), noting that $\sum_{q\ge 0}\1_{|q-j|\le 4}|v_q|\le C\big(\sum_{q\ge 0}|v_q|^2\big)^{\fr12}$ and $\sum_{p,q\ge 0}\1_{|p-q|\le 1}|u_p||v_q|\le C\big(\sum_{p,q\ge 0}|u_p|^2|v_q|^2\big)^{\fr12}$, we have 
\begin{align*}
	&\sum_{4\le i\le 6}\Ga_i(\phi_1,f,f)
	\le
	C\int_{x}\sum_{j\ge 0}\Big(\sum_{k\ge 0}(\om(j)2^k)^{-4s}\|\De_jP^j_kR_rf\|_{L^2_D}^2\Big)^{\frac{1}{2}}
	\\
	&\qquad \times 
	\Big\{
	\sum_{q\ge 0}\1_{|q-j|\le 4}
	\|S_{q-1}\phi_1\|_{L^2_v}\Big\{\sum_{k\ge 0}\|\<v\>^{\frac{\ga+2s}{2}}\De_{q}P^j_kR_rf\|_{L^2_v}^2\Big\}^{\frac{1}{2}}\\
	&\qquad\quad\notag
	+
	\sum_{p\ge 0}\1_{|p-j|\le 4}
	\|\De_p\phi_1\|_{L^2_v}\Big\{\sum_{k\ge 0}\|\<v\>^{\frac{\ga+2s}{2}}S_{p-1}P^j_kR_rf\|_{L^2_v}^2\Big\}^{\frac{1}{2}}\\
	&\qquad\quad\notag+
	\sum_{p,q\ge 0}\1_{|p-q|\le 1}
	\1_{\substack{|p-q|\le 1,\qquad\\\max\{p,q\}\ge j-2}}
	\|\De_p\phi_1\|_{L^2_v}\Big\{\sum_{k\ge 0}\|\<v\>^{\frac{\ga+2s}{2}}\De_{q}P^j_kR_rf\|_{L^2_v}^2\Big\}^{\frac{1}{2}}\Big\}\\
	&\le
	C\Big\|\sum_{j\ge 0}\Big(\sum_{k\ge 0}(\om(j)2^k)^{-4s}\|\De_jP^j_kR_rf\|_{L^2_D}^2\Big)^{\frac{1}{2}}\Big\|_{L^{r_0}_{x}}
	\\
	&\qquad \times 
	\Big\{
	\Big\|\Big(\sum_{q\ge 0}\1_{|q-j|\le 4}\|S_{q-1}\phi_1\|_{L^2_v}^2\Big)^{\frac{1}{2}}\Big\|_{L^{r_0^*}_{x}}\Big\|\Big(\sum_{q\ge 0}\|\<v\>^{\frac{\ga+2s}{2}}\De_{q}R_rf\|_{L^2_v}^2\Big)^{\frac{1}{2}}\Big\|_{L^2_{x}}\\
	&\qquad\quad\notag
	+
	\Big\|\Big(\sum_{p\ge 0}\|\De_p\phi_1\|_{L^2_v}^2\Big)^{\frac{1}{2}}\Big\|_{L^{r_0^*}_{x}}
	\Big\|\Big(\sum_{p\ge 0}\1_{|p-j|\le 4}\|\<v\>^{\frac{\ga+2s}{2}}S_{p-1}R_rf\|_{L^2_v}^2\Big)^{\frac{1}{2}}\Big\|_{L^2_{x}}\\
	&\qquad\quad\notag+
	\Big\|\Big(\sum_{p\ge 0}\|\De_p\phi_1\|_{L^2_v}^2\Big)^{\frac{1}{2}}\Big\|_{L^{r_0^*}_{x}}
	\Big\|\Big(\sum_{q\ge 0}\|\<v\>^{\frac{\ga+2s}{2}}\De_{q}R_rf\|_{L^2_v}^2\Big)^{\frac{1}{2}}\Big\|_{L^2_{x}}\Big\}\\
	&\le C2^{r(\frac{\ga+2s}{2})}\Big\|\sum_{j\ge 0}\Big(\sum_{k\ge 0}(\om(j)2^k)^{-4s}\|\De_jP^j_kR_rf\|_{L^2_D}^2\Big)^{\frac{1}{2}}\Big\|_{L^{r_0}_{x}}
	\|\phi_1\|_{L^{r_0^*}_{x}L^2_v}\|R_rf\|_{L^2_{x,v}}, 
\end{align*}
where we have applied \eqref{esSq} to deduce that, for any $r\in[2,\infty]$ and suitable $u$, 
\begin{align}\label{esSump}
	\Big\|\Big(\sum_{p\ge 0}\1_{|p-j|\le 4}\|S_{p-1}u\|_{L^2_v}^2\Big)^{\frac{1}{2}}\Big\|_{L^r_{x}}
	&\le \Big(\sum_{p\ge 0}\1_{|p-j|\le 4}\|S_{p-1}u\|_{L^r_{x}L^2_v}^2\Big)^{\frac{1}{2}}\le 9^{\frac{1}{2}}\|u\|_{L^r_{x}L^2_v}, 
\end{align}
and applied the equivalence of Triebel-Lizorkin spaces (see, for instance \cite[Remark 2.2.2]{Grafakos2014a}), i.e., 
\begin{align}
	\label{equivTriebelLizonkin}
	\Big\|\Big(\sum_{p\ge 0}\|\De_pu\|_{L^2_v}^2\Big)^{\frac{1}{2}}\Big\|_{L^{r}_{x}}\approx \|u\|_{L^{r}_{x}L^2_v}, \qquad 1<r<\infty.
\end{align}

Next, we estimate $\Ga_i(f,\phi_2,f)$ $(1\le i\le 6)$ from \eqref{UpperTT}. The $\phi_2$ term can be directly calculated using $\|\cdot\|_{L^2_D}\le C\|\<v\>^{\frac{\ga+2s}{2}}\<D_v\>^{s}(\cdot)\|_{L^2_v}$, since any excess weight function can be allocated to $\phi_2$ as needed. Moreover, \eqref{equivTriebelLizonkin} will be used from time to time.
By estimate \eqref{dissUpper} and the support $\{|\eta|\le 2\om(j)2^k\}$ of the kernel of $P^j_k$, we have 
\begin{align}\label{esphi2}
	\sum_{k\ge 0}(\om(j)2^k)^{-2s}\|\De_{q}P^j_kR_r\phi_2\|^2_{L^2_D}
	&\le C\|\<v\>^{\frac{\ga+2s}{2}}\De_{q}R_r\phi_2\|_{L^{2}_v}^2. 
\end{align}
For $\Ga_1(f,\phi_2,f)$, applying the $l^2_{j,k,q}-l^2_{j,k,q}$ H\"older's inequality, splitting $S_{q-1}=\sum_{p\le q-2}\De_p$ with $\sum_{p\ge 0}|v_p|\le \big(\sum_{p\ge 0}(p+1)^2|v_p|^2\big)^{\frac{1}{2}}$, using estimates \eqref{esphi2} and \eqref{equivTriebelLizonkin}, and the $L^{r_0}_{x}-L^{r_0^*}_{x}-L^{2}_{x}$ H\"older's inequality to $f$-$\phi_2$-$f$, we have 
\begin{align*}
	&\Ga_1(f,\phi_2,f)\le C\int_{x}
	\Big(\sum_{j,q\ge 0}\1_{|q-j|\le 4}
	\|S_{q-1}\<D_v\>^{-s}\<v\>^{-N}f\|_{L^2_v}^2
	\|\<v\>^{\frac{\ga+2s}{2}}\De_{q}R_r\phi_2\|_{L^2_v}^2\Big)^{\frac{1}{2}}
	\\
	&\qquad\qquad\qquad\quad\times 
	\Big(\sum_{j,k,q\ge 0}\1_{|q-j|\le 4}(\om(j)2^k)^{-2s}\|\De_jP^j_kR_rf\|_{L^2_D}^2\Big)^{\frac{1}{2}}\\
	&\le C\Big\|
	\Big(\sum_{p\ge 0}(p+1)^2\|\De_p\<D_v\>^{-s}\<v\>^{-N}f\|_{L^2_v}^2\Big)^{\frac{1}{2}}\Big\|_{L^{r_0}_{x}}
	\Big\|\Big(\sum_{j,q\ge 0}\1_{|q-j|\le 4}\|\<v\>^{\frac{\ga+2s}{2}}\De_{q}R_r\phi_2\|_{L^2_v}^2\Big)^{\frac{1}{2}}\Big\|_{L^{r_0^*}_{x}}
	\\
	&\quad\times 
	\Big\|\Big(\sum_{j,k\ge 0}(\om(j)2^k)^{-2s}\|\De_jP^j_kR_rf\|_{L^2_D}^2\Big)^{\frac{1}{2}}\Big\|_{L^2_{x}}\\
	&\le C
	\Big(\sum_{p\ge 0}(p+1)^2\|\De_p\<D_v\>^{-s}\<v\>^{-N}f\|_{L^{r_0}_{x}L^2_v}^2\Big)^{\frac{1}{2}}
	\|\<v\>^{\frac{\ga+2s}{2}}R_r\phi_2\|_{L^{r_0^*}_{x}L^2_v}
	\Big(\sum_{j,k\ge 0}(\om(j)2^k)^{-2s}\|\De_jP^j_kR_rf\|_{L^2_{x}L^2_D}^2\Big)^{\frac{1}{2}}.
\end{align*}
For $\Ga_2(f,\phi_2,f)$, similarly, applying estimate \eqref{esphi2}, the $l^2_p-l^2_p$ H\"older's inequality, splitting $S_{q-1}=\sum_{p\le q-2}\De_p$ with $\sum_{q\ge 0}\1_{q\le p-2}|v_q|\le (p+1)^{\frac{1}{2}}\big(\sum_{q\ge 0}|v_q|^2\big)^{\frac{1}{2}}$ and estimate \eqref{esSq}, and the $L^{r_0}_{x}-L^{r_0^*}_{x}-L^{2}_{x}$ H\"older's inequality to $f$-$\phi_2$-$f$, we have 
\begin{align*}
	\Ga_2(f,\phi_2,f)
	&\le C
	\sum_{p\ge 0}
	\|\De_{p}\<D_v\>^{-s}\<v\>^{-N}f\|_{L^{r_0}_{x}L^2_v}\|\<v\>^{\frac{\ga+2s}{2}}S_{p-1}R_r\phi_2\|_{L^{r_0^*}_{x}L^2_v}\\
	&\qquad\times 
	\Big(\sum_{j,k\ge 0}(\om(j)2^k)^{-2s}\|\De_jP^j_kR_rf\|_{L^2_{x}L^2_D}^2\Big)^{\frac{1}{2}}\\
	&\le C
	\Big(\sum_{p\ge 0}
	(p+1)^2\|\De_{p}\<D_v\>^{-s}\<v\>^{-N}f\|_{L^{r_0}_{x}L^2_v}^2\Big)^{\frac{1}{2}}\|\<v\>^{\frac{\ga+2s}{2}}R_r\phi_2\|_{L^{r_0^*}_{x}L^2_v}
	\\
	&\qquad\times 
	\Big(\sum_{j,k\ge 0}(\om(j)2^k)^{-2s}\|\De_jP^j_kR_rf\|_{L^2_{x}L^2_D}^2\Big)^{\frac{1}{2}}.
\end{align*}
For $\Ga_3(f,\phi_2,f)$ in \eqref{UpperTT}, by the $l^2_{j,k}-l^2_{j,k}$ H\"older's inequality, $\sum_{j\ge 0}\1_{\max\{p,q\}\ge j-2}\1_{|p-q|\le 1}\le C(p+1)$, $\sum_{p,q\ge 0}\1_{|p-q|\le 1}|u_p||v_q|\le C\big(\sum_{p,q\ge 0}|u_p|^2|v_q|^2\big)^{\fr12}$, and $L^{r_0}_{x}-L^{r_0^*}_{x}-L^{2}_{x}$ H\"older's inequality, we have 
\begin{align*}
	\Ga_3(f,\phi_2,f)
	&\le C
	\Big\|\Big(\sum_{p\ge 0}(p+1)
	\|\De_p\<D_v\>^{-s}\<v\>^{-N}f\|_{L^2_v}^2\Big)^{\frac{1}{2}}\Big\|_{L^{r_0}_{x}}
	\Big\|\Big(\sum_{q\ge 0}\|\<v\>^{\frac{\ga+2s}{2}}\De_{q}R_r\phi_2\|_{L^2_v}^2\Big)^{\frac{1}{2}}\Big\|_{L^{r_0^*}_{x}}\\
	&\quad\times\Big\|\Big(\sum_{j,k\ge 0}(\om(j)2^k)^{-2s}\|\De_jP^j_kR_rf\|_{L^2_D}^2\Big)^{\frac{1}{2}}\Big\|_{L^{2}_{x}}\\
	&\le C
	\Big(\sum_{p\ge 0}(p+1)
	\|\De_p\<D_v\>^{-s}\<v\>^{-N}f\|_{L^{r_0}_{x}L^2_v}^2\Big)^{\frac{1}{2}}
	\|\<v\>^{\frac{\ga+2s}{2}}\De_{q}R_r\phi_2\|_{L^{r_0^*}_{x}L^2_v}\\
	&\quad\times\Big(\sum_{j,k\ge 0}(\om(j)2^k)^{-2s}\|\De_jP^j_kR_rf\|_{L^{2}_{x}L^2_D}^2\Big)^{\frac{1}{2}}, 
\end{align*}
where we also used \eqref{equivTriebelLizonkin} in the last inequality. 
For $\Ga_i(f,\phi_2,f)$ $(4\le i\le 6)$, similar to the estimate of $\Ga_i(\phi,f,f)$, we allocate the weight $(\om(j)2^k)^{-2s}$ to the second $f$, and apply the $l^2_k-l^2_k$, $l^2_{p,q}-l^2_{p,q}$, $L^{r_0^*}_{x}-L^2_{x}-L^{r_0}_{x}$ H\"older's inequality. 
Thus, using \eqref{esSump} and \eqref{equivTriebelLizonkin}, we have 
\begin{align*}
	&\sum_{4\le i\le 6}\Ga_i(f,\phi_2,f)
	\le
	C\sum_{j\ge 0}\int_{x}\Big(\sum_{k\ge 0}(\om(j)2^k)^{-4s}\|\De_jP^j_kR_rf\|_{L^2_D}^2\Big)^{\frac{1}{2}}
	\\
	&\qquad \times 
	\Big\{
	\Big(\sum_{q\ge 0}\1_{|q-j|\le 4}\|S_{q-1}f\|_{L^2_v}^2\Big)^{\frac{1}{2}}\Big(\sum_{q\ge 0}\1_{|q-j|\le 4}\|\<v\>^{\frac{\ga+2s}{2}}\De_{q}R_r\phi_2\|_{L^2_v}^2\Big)^{\frac{1}{2}}\\
	&\qquad\quad\notag
	+
	\Big(\sum_{p\ge 0}\|\De_pf\|_{L^2_v}^2\Big)^{\frac{1}{2}}
	\Big(\sum_{p\ge 0}\1_{|p-j|\le 4}\|\<v\>^{\frac{\ga+2s}{2}}S_{p-1}R_r\phi_2\|_{L^2_v}^2\Big)^{\frac{1}{2}}\\
	&\qquad\quad\notag+
	\Big(\sum_{p\ge 0}\|\De_pf\|_{L^2_v}^2\Big)^{\frac{1}{2}}
	\Big(\sum_{q\ge 0}\|\<v\>^{\frac{\ga+2s}{2}}\De_{q}R_r\phi_2\|_{L^2_v}^2\Big)^{\frac{1}{2}}\Big\}\\
	&\le C\Big\|\sum_{j\ge 0}\Big(\sum_{k\ge 0}(\om(j)2^k)^{-4s}\|\De_jP^j_kR_rf\|_{L^2_D}^2\Big)^{\frac{1}{2}}\Big\|_{L^{r_0}_{x}}
	\|f\|_{L^{2}_{x,v}}\|\<v\>^{\frac{\ga+2s}{2}}R_r\phi_2\|_{L^{r_0^*}_{x}L^2_v}. 
\end{align*}

\subsubsection{Combining the estimates}
For the $\|\cdot\|_{L^2_t L^{r_0}_x L^2_D}$ term involving the coefficient $(\om(j)2^k)^{-4s}$, we apply the Bernstein-type estimate \eqref{SobolevEmbeddLp} and use $\om(j)=\om_0 2^{\al j+rl_1}$ with $\alpha>0$ to absorb $(j+1)^22^{2\ve j}$ (for small $\varepsilon \in (0,\min\{\frac{s\alpha}{4},d\})$). This yields that, for any $r_0>2$ satisfying $1+\frac{1}{r_0}=\frac{1}{2}+\frac{1}{(1-\ve/d)^{-1}}$,
\begin{align}\label{esinter1}\notag
	&2^{r(\frac{\ga+2s}{2})}\Big\|\sum_{j\ge 0}\Big(\sum_{k\ge 0}(\om(j)2^k)^{-4s}2^{4rl}\|\De_jP^j_kR_rf\|_{L^2_D}^2\Big)^{\frac{1}{2}}\Big\|_{L^{r_0}_{x}}\\
	&\notag \le C\Big(\sum_{j,k\ge 0}\om_0^{-6s}(\om(j)2^k)^{-2s}2^{2rl-r\ve l_0}\\
	&\notag\quad\ \times\Big(\om_0^{4s}(j+1)^2(\varpi_02^{rl_0}2^{j})^{2\ve}(2^{\al j+rl_1+k})^{-2s}2^{r(\ve l_0+\ga+2s+2l)}\Big)\|\De_jP^j_kR_rf\|_{L^{2}_{x}L^2_D}^2\Big)^{\frac{1}{2}}\\
	&\notag \le C\Big(\sum_{j,k\ge 0}\om_0^{-6s}(\om(j)2^k)^{-2s}2^{2rl-r\ve l_0}\\
	&\notag\quad\ \times\Big(1+\Big\{\om_0^{4s}\varpi_0^{2\ve}2^{r(\ga+2s+3\ve l_0-2sl_1+2l)}2^{(-\frac{3}{2}s\al+2\ve)j-2sk}\Big\}^M\Big)\|\De_jP^j_kR_rf\|_{L^{2}_{x}L^2_D}^2\Big)^{\frac{1}{2}}\\
	& \le C\Big(\om_0^{-6s}\sum_{j,k\ge 0}(\om(j)2^k)^{-2s}2^{2rl-r\ve l_0}\|\De_jP^j_kR_rf\|_{L^{2}_{x}L^2_D}^2
	\notag\\
	&\quad\ 
	+C_{\om_0,\vpi_0}
	\sum_{j,k\ge 0}(\om(j)2^k)^{-2s}2^{2rl-r\ve l_0}\|\De_jP^j_kR_rf\|_{L^{2}_{x,v}}^2\Big)^{\frac{1}{2}},
\end{align}
where we used $(j+1)^2\le C2^{\frac{s\al j}{2}}$, Young's inequality with sufficiently large $M>0$ such that 
\begin{align}
	\label{choiceM2temp}
	M(\ga+2s+3\ve l_0-2sl_1+2l)\le -2sl_1,\ \ 
	M(-\frac{3}{2}s\al+2\ve)j\le -2s\al j, 
\end{align}
and $\|P^j_kf\|_{L^2_D} \le C(\om(j)2^k)^s\|\<v\>^{\fr{\ga+2s}{2}}P^j_kf\|_{L^2_v}$ as in \eqref{supportPk} for the second term.  
This constraint can be fulfilled by choosing $\ga+2s<0$, $sl_1=l$, and $\ve l_0\ll 1$. 
The term $2^{-r\ve l_0}$ is designed for the convergence of summation $\sum_{r\ge 0}2^{-r\ve l_0}<C_\ve$. Although the above constant depends on $\ve$, we omit this dependence since $\ve$ will be chosen to depend only on $l_0, l_1, l, \ga, s$, and $l_0, l_1, l$ depend only on $\ga, s$ (see \eqref{choice10}). 
The term $3\ve l_0+\ga+2s$ is one of the reasons for considering \textbf{soft potentials} $\ga+2s<0$.
Substituting all the estimates of $\Ga_i$'s from Subsection \ref{SecGammai} into \eqref{DecomGafg}, using estimate \eqref{esinter1}, multiplying the resulting expression by $2^{2rl}$, and summing over $r\ge 0$, we proceed as follows for any $r_0>2$ (with $\frac{1}{r_0}+\frac{1}{r_0^*}=\frac{1}{2}$). 
Integrating \eqref{DecomGafg} over $t\in[0,T]$, and applying \eqref{dissUpper1} to establish the equivalence $\<v\>\approx2^r$ from the decomposition $R_r(v)$ (where we simply absorb the unfavorable weight $2^{rl}\approx \<v\>^l$ into $\phi_2$), we obtain: 
\begin{align}
	\label{DejGaMaines}\notag
	&\sum_{r,j,k\ge 0}(\om(j)2^k)^{-2s}2^{2rl}\big|\big(\De_j\Ga(\phi_1,P^j_kR_rf)+\De_j\Ga(f,P^j_kR_r\phi_2),P^j_kR_rf\big)_{L^2_t([0,T])L^2_{x,v}}\big|\\
	&\notag\le C
	\Big(\|\<D_v\>^{-s}\<v\>^{-N}\phi_1\|_{L^\infty_t([0,T])L^\infty_{x}L^2_v}
	+\|\<v\>^{2l}\phi_2\|_{L^{r_0^*}_t([0,T])L^{r_0^*}_xL^2_v}\\
	&\qquad\notag
	+\sup_{r\ge 0}\Big(\sum_{j\ge 0}(j+rl_0+\log_2\vpi_0+1)^2\|\De_j\<D_v\>^{-s}\<v\>^{-N}\phi_1\|_{L^\infty_t([0,T])L^\infty_{x}L^2_v}^2\Big)^{\frac{1}{2}}\Big)
	\\
	&\notag\ \times\Big(\sum_{r,j,k\ge 0}(\om(j)2^k)^{-2s}2^{2rl}\|\De_jP^j_kR_rf\|^2_{L^2_t([0,T])L^2_{x}L^2_D}
	+\om_0^{-2s{}}\|
	f\|_{L^2_t([0,T])L^2_{x,v}}^2
	\\
	&\qquad\notag
	+C_{\om_0}\sum_{r,j,k\ge 0}
	(\om(j)2^k)^{-2s}2^{2rl}\|\De_jP^j_kR_rf\|_{L^2_t([0,T])L^{2}_{x,v}}^2\Big)\\
	&\notag\ \ +\|(\phi_1,\,\<v\>^{2l}\phi_2)\|_{L^\infty_t([0,T])L^{r_0^*}_xL^2_v}\Big(
	\om_0^{-4s}\|f\|_{L^2_t([0,T])L^2_{x,v}}^2\\
	&\qquad \notag
	+\om_0^{-2s}\sum_{r,j,k\ge 0}(\om(j)2^k)^{-2s}2^{2rl}\|\De_jP^j_kR_rf\|_{L^2_t([0,T])L^{2}_{x}L^2_D}^2\\
	&\qquad \notag
	+C_{\om_0,\varpi_0,l_0,l_1}\sum_{r,j,k\ge 0}(\om(j)2^k)^{-2s}2^{2rl}\|\<v\>^{\fr{\ga+2s}{2}}\De_jP^j_kR_rf\|_{L^2_t([0,T])L^{2}_{x,v}}^2\Big)
	\\
	&\ \ 
	+
	\|\<v\>^{2l}\phi_2\|_{L^{r_0^*}_t([0,T])L^{r_0^*}_xL^2_v}\sup_{r\ge 0}\sum_{p\ge 0}(p+1)^2\|\De_p\<D_v\>^{-s}\<v\>^{-N}f\|_{L^{r_0}_t([0,T])L^{r_0}_{x}L^2_v}^2,
\end{align}
for any large $N>0$ (coming from the decay of $\mu^{\frac{1}{2}}$), 
provided that
\begin{align}
	\label{choiceM2}
	\ga+2s+3\ve l_0-2sl_1+2l<0,\qquad
	-\frac{3}{2}s\al+2\ve<0.
\end{align}
Under these conditions, we can choose $M=M(\ga,s,l_0,l_1,l,\al,\ve)>0$ sufficiently large so that \eqref{choiceM2temp} holds.
The $(\om2^k)^{-2s}\|\cdot\|_{L^2_t([0,T])L^2_{x,v}}^2$ terms in \eqref{DejGaMaines} have large coefficients and will be controlled by Gr\"onwall's inequality in the end.

\subsection{Bony-type commutator estimate with irregular \texorpdfstring{$\phi$}{phi}}
As before, we let $r_0,r_0^*\in(2,\infty)$ such that $\frac{1}{2}=\frac{1}{r_0}+\frac{1}{r_0^*}$, and consider the commutator terms 
$\De_jP_kR_r\Ga(\phi_1,f)-\De_j\Ga(\phi_1,P_kR_rf)$ and $\De_jP_kR_r\Ga(f,\phi_2)-\De_j\Ga(f,P_kR_r\phi_2)$. They can be handled similarly, as they are the lower-order terms.
As in Subsection \ref{SecGammai}, we will allocate the $L^{r_0^*}_{x}$ norm to $\phi_1$ and $\phi_2$, the $L^{r_0^*}_{x}$ norm to $(\om2^k)^{-2s}\|f\|_{L^2_D}$, where $\frac{1}{2}=\frac{1}{r_0}+\frac{1}{r_0^*}$. 
Applying commutator Theorem \ref{ThmComm} (with $\wt{l_0}\in\R$ and $s'=-s$ therein) and using commutator estimate \eqref{dissUpper1b} (applied to the first right-hand term in \eqref{ThmCommes} and to convert the $H^s_v$ into $L^2_D$ norm; some redundant and good coefficients can be dropped), we have (letting $[0,T]$ be the underlying time interval of the norms)
\begin{align}\label{es713aaphi1}\notag
	&\notag\sum_{r,j,k\ge 0}(\om2^k)^{-2s}2^{2rl}\big(\De_jP_kR_r\Ga(\phi_1,f)-\De_j\Ga(\phi_1,P_kR_rf),\De_jP_kR_rf\big)_{L^2_t([0,T])L^2_{x,v}}\\
	&\notag+\sum_{r,j,k\ge 0}(\om2^k)^{-2s}2^{2rl}\big(\De_jP_kR_r\Ga(f,\phi_2)-\De_j\Ga(f,P_kR_r\phi_2),\De_jP_kR_rf\big)_{L^2_t([0,T])L^2_{x,v}}\\
	&\notag\le 
	C\int_{t,x}\sum_{r,j,k\ge 0}(\om2^k)^{-2s}2^{2rl}\Big(2^{r(\wt{l_1}+\frac{\ga+2s}{2})}\|\phi_1\|_{L^2_v}\|\<v\>^{-\wt{l_1}}f\|_{L^2_v}\|\De_jP_kR_rf\|_{L^2_D}\\
	&\notag\qquad\qquad\quad
	+2^{r(\wt{l_2}+\frac{\ga+2s}{2})}(\om2^k)^{s-1}\|\phi_1\|_{L^2_v}\|\<v\>^{-\wt{l_2}}f\|_{L^2_v}\|\<v\>^{\frac{\ga+2s}{2}-1}\De_jR_rf\|_{L^2_v}\\
	&\notag\qquad\qquad\quad
	+2^{r(\wt{l_3}+\frac{\ga+2s}{2})}\|f\|_{L^2_v}\|\<v\>^{-\wt{l_3}}\phi_2\|_{L^2_v}\|\De_jP_kR_rf\|_{L^2_D}\\
	&\qquad\qquad\quad
	+2^{r(\wt{l_4}+\frac{\ga+2s}{2})}(\om2^k)^{s-1}\|f\|_{L^2_v}\|\<v\>^{-\wt{l_4}}\phi_2\|_{L^2_v}\|\<v\>^{\frac{\ga+2s}{2}-1}\De_jR_rf\|_{L^2_v}\Big)
	=:\sum_{i=1}^{4}\wt{\Ga}_i,
\end{align}
for any $\wt{l_j}\in\R$ ($1\le j\le 4$).
We focus on the right-hand term $\wt{\Ga}_1$, since it is the leading term. Similar to estimate \eqref{esinter1}, by Bernstein-type inequality \eqref{SobolevEmbeddLp}, for any $n\in\R$, we have 
\begin{align}
	\label{esinter2}\notag
	&2^{r(\fr{\ga+2s}{2})}\sum_{j,k\ge 0}(\om2^k)^{-2s}2^{2rl}\|\De_jP_kR_rf\|_{L^{r_0}_{x}L^2_D}\\
	&\notag\le C\Big(\sum_{j,k\ge 0}\om_0^{-6s}(\om2^k)^{-2s}2^{r(2l-\ve l_0)}\|\De_jP^j_kR_rf\|_{L^{2}_{x}L^2_D}^2\\
	&\notag\qquad\ \times\Big(\om_0^{4s}(j+1)^2(k+1)^2(\varpi_02^{rl_0}2^{j})^{2\ve}(2^{\al j+rl_1+k})^{-2s}2^{r(\ga+2s+\ve l_0+2l)}\Big)\Big)^{\frac{1}{2}}\\
	&\notag \le C2^{-\fr{r\ve l_0}{2}}\Big(\om_0^{-6s}\sum_{j,k\ge 0}(\om2^k)^{-2s}2^{2rl}\|\De_jP^j_kR_rf\|_{L^{2}_{x}L^2_D}^2\\
	&\qquad\qquad 
	+C_{\om_0,\varpi_0}\sum_{j,k\ge 0}(\om2^k)^{-2s}2^{2rl}\|\De_jP^j_kR_rf\|_{L^{2}_{x,v}}^2\Big)^{\frac{1}{2}},
\end{align}
provided that \eqref{choiceM2} holds. 
Thus, by $\sum_{r\ge0}2^{-\fr{r\ve l_0}{2}}|u_r|\le C_{\ve,l_0}(\sum_{r\ge0}|u_r|^2)^{\fr12}$ and choosing suitable $\wt{l_1}$, we estimate $\wt{\Ga}_1$ as 
\begin{align*}
	\wt{\Ga}_1
	&\le 
	\int_{t}\sum_{r,j,k\ge 0}2^{r(\wt{l_1}+\frac{\ga+2s}{2})}(\om2^k)^{-2s}2^{2rl}\|\phi_1\|_{L^{r_0^*}_{x}L^2_v}\|\<v\>^{-\wt{l_1}}f\|_{L^2_{x}L^2_v}\|\De_jP_kR_rf\|_{L^{r_0}_{x}L^2_D}\\
	&\le 
	C\int_{t}\sum_{r\ge 0}2^{r(\wt{l_1}-\fr{r\ve l_0}{2})}2^{2rl}\|\phi_1\|_{L^{r_0^*}_{x}L^2_v}\|\<v\>^{-\wt{l_1}}f\|_{L^2_{x}L^2_v}\\
	&\quad\times
	\Big(\om_0^{-6s}\sum_{j,k\ge 0}(\om2^k)^{-2s}
	\|\De_jP_kR_rf\|_{L^{2}_{x}L^2_D}^2\\
	&\qquad
	+C_{\om_0,\varpi_0}
	\sum_{j,k\ge 0}
	(\om2^k)^{-2s}2^{2rl}\|\De_jP_kR_rf\|_{L^2_{x,v}}^2\Big)^{\frac{1}{2}}\\
	&\le 
	C\|\phi_1\|_{L^\infty_tL^{r_0^*}_{x}L^2_v}
	\Big(
	\om_0^{-4s}\|f\|_{L^2_t([0,T])L^2_{x,v}}^2\\
	&\qquad \notag
	+\om_0^{-2s}\sum_{r,j,k\ge 0}(\om2^k)^{-2s}2^{2rl}\|\De_jP^j_kR_rf\|_{L^2_t([0,T])L^{2}_{x}L^2_D}^2\\
	& \notag
	+\om_0^{4sM}\varpi_0^{2\ve M}\sum_{r,j,k\ge 0}(\om2^k)^{-2s}2^{2rl}\|\<v\>^{\fr{\ga+2s}{2}}\De_jP^j_kR_rf\|_{L^2_t([0,T])L^{2}_{x,v}}^2\Big).
\end{align*}
The term $\wt{\Ga}_2$ can be estimated similarly as in \eqref{esinter2}. By Littlewood-Paley theorem \ref{LPThm}, we have 
\begin{align*}
	\sum_{j,k\ge 0}(\om2^k)^{-s-1}\|\<v\>^n\De_jR_rf\|_{L^{r_0}_{x}L^2_v}
	&\le C\om_0^{-s-1}\varpi_0^{\ve}2^{r(\ve l_0+(-s-1)l_1)}\Big(\sum_{j\ge 0}\|\<v\>^n\De_jR_rf\|_{L^2_{x,v}}^2\Big)^{\fr12}\\
	&\le C\om_0^{-s-1}\varpi_0^{\ve}\|\<v\>^{\ve l_0+(-s-1)l_1+n}R_rf\|_{L^2_{x,v}}.
\end{align*}
and hence, by $\sum_{r\ge0}2^{-\ve l_0}|u_r|\le C_{\ve,l_0}(\sum_{r\ge0}|u_r|^2)^{\fr12}$ and choosing suitable $\wt{l_2}$, we have 
\begin{align*}
	&\wt{\Ga}_2
	\le 
	\int_{t}\sum_{r,j,k\ge 0}2^{r\wt{l_2}}(\om2^k)^{-s-1}2^{2rl}\|\phi_1\|_{L^{r_0^*}_{x}L^2_v}\|\<v\>^{-\wt{l_2}+\frac{\ga+2s}{2}}f\|_{L^2_{x}L^2_v}\|\<v\>^{\frac{\ga+2s}{2}-1}\De_jR_rf\|_{L^{r_0}_{x}L^2_v}\\
	&\le 
	C\om_0^{-s-1}\varpi_0^{\ve}\int_{t}\sum_{r\ge0}2^{r\wt{l_2}}\|\phi_1\|_{L^{r_0^*}_{x}L^2_v}\|\<v\>^{-\wt{l_2}+\frac{\ga+2s}{2}}f\|_{L^2_{x}L^2_v}\|\<v\>^{\ve l_0+(-s-1)l_1+l+\frac{\ga+2s}{2}-1}R_rf\|_{L^2_{x,v}}\\
	&\le 
	C\om_0^{-s-1}\varpi_0^{\ve}\|\phi_1\|_{L^\infty_tL^{r_0^*}_{x}L^2_v}\|f\|_{L^2_{t,x}L^2_v}\|\<v\>^{2\ve l_0+(-s-1)l_1+l+\ga+2s-1}f\|_{L^{2}_{t,x}L^2_v}.
\end{align*}
The lower-order terms $\wt{\Ga_3}$ and $\wt{\Ga_4}$ can be estimated more easily, as we can allocate the redundant weight function to $\phi_2$. That is, 
\begin{align*}
	\wt{\Ga}_3+\wt{\Ga}_4
	&\le C\|f\|_{L^2_{t,x}L^2_v}\|\<v\>^{C_{l_0,l_1,l,\ga,s}}\phi_2\|_{L^\infty_tL^{r_0^*}_{x}L^2_v}
	\times \Big(\om_0^{-6s{}}\sum_{r,j,k\ge 0}(\om2^k)^{-2s}
	\|\De_jP_kR_rf\|_{L^{2}_{t,x}L^2_D}^2\\
	&\ \ 
	+\om_0^{4sM-4s}\varpi_0^{2\ve}
	\sum_{r,j,k\ge 0}
	(\om2^k)^{-2s}\|\De_jP_kR_rf\|_{L^2_{t,x,v}}^2
	+\om_0^{-2s-2}\varpi_0^{2\ve M}\|f\|_{L^2_{t,x,v}}^2\Big)^{\frac{1}{2}}. 
\end{align*}
Collecting the above estimates for $\wt{\Ga}_i$ ($1\le i\le 4$) into \eqref{es713aaphi1} and using the \emph{a priori} estimate \eqref{apriori2}, we obtain, for any $\ka>0$, that if \eqref{choiceM2} holds, then
\begin{align}\label{es713aaphi2}
	&\notag\sum_{r,j,k\ge 0}(\om2^k)^{-2s}2^{2rl}\big(\De_jP_kR_r\Ga(\phi_1,f)-\De_j\Ga(\phi_1,P_kR_rf),\De_jP_kR_rf\big)_{L^2_t([0,T])L^2_{x,v}}\\
	&\notag\ +\sum_{r,j,k\ge 0}(\om2^k)^{-2s}2^{2rl}\big(\De_jP_kR_r\Ga(f,\phi_2)-\De_j\Ga(f,P_kR_r\phi_2),\De_jP_kR_rf\big)_{L^2_t([0,T])L^2_{x,v}}\\
	&\notag\le 
	C_\ka \om_0^{-4s{}}\|(\phi_1\,\<v\>^{C_{l_0,l_1,l,\ga,s}}\phi_2)\|_{L^\infty_tL^{r_0^*}_{x}L^2_v}^2\|f\|_{L^2_{t,x}L^2_v}^2
	\notag\\
	&\ \notag
	+\ka\sum_{r,j,k\ge 0}(\om2^k)^{-2s}
	\|\De_jP_kR_rf\|_{L^{2}_{t,x}L^2_D}^2
	+\om_0^{4sM}\varpi_0^{2\ve M}
	\sum_{r,j,k\ge 0}
	(\om2^k)^{-2s}\|\De_jP_kR_rf\|_{L^2_{t,x,v}}^2\\
	&\ +C\om_0^{-s-1}\varpi_0^{\ve}\|(\phi_1,\,\<v\>^{C_{l_0,l_1,l,\ga,s}}\phi_2)\|_{L^\infty_tL^{r_0^*}_{x}L^2_v}
	\|(\<v\>^{2\ve l_0+(-s-1)l_1+l+\ga+2s-1}f,\,f)\|_{L^{2}_{t,x}L^2_v}^2. 
\end{align}

\section{Littlewood-Paley analysis on kinetic equations}\label{SecLPkinetic}
In this section, we apply the Littlewood-Paley operators $P_k$ and $\De_j$, defined in \eqref{DejQkDef}, to the Boltzmann equation. Let $T_*>0$ be the fixed time as in Theorem~\ref{MainThm1}, and let $T\in(0,T_*]$. Let $\phi_1(t),\phi_2(t)$ be two solutions of the Boltzmann equation \eqref{B1} on $[0,T]$ with initial data $\phi_{1,0},\phi_{2,0}$, respectively. Then the difference $f=\phi_1-\phi_2$ solves the kinetic equation on $[0,T]\times\R^d_x\times\R^d_v$ given by 
\begin{align*}
	\begin{cases}
	(\pa_t+v\cdot\na_x)f=g,&\text{in }\ [0,T]\times\R^d_x\times\R^d_v,\\
	f(0,x,v) = f_0:=\phi_{1,0}-\phi_{2,0} & \text{in }\  \R^d_x\times\R^d_v,
	\end{cases}
\end{align*}
in the weak sense, as in \eqref{weakf}, where $g$ denotes the difference of the collision operator, i.e., 
\begin{align}\label{eqfgxxgg}
	g=\Ga(\mu^{\frac{1}{2}}+\phi_1,f)+\Ga(f,\mu^{\frac{1}{2}}+\phi_2).
\end{align}
Thus, for any function $\vp\in C^2_c(\R_t\times\R^d_x\times\R^d_v)$ and a.e. $s\in[0,T]$, 
\begin{align}\label{5weaka}
	\big(f(t),\vp(t)\big)_{L^2_{x,v}}\big|_{t=0}^{t=s} = \big(f,(-\pa_t-v\cdot\na_x)\vp\big)_{L^2_t(0,s)L^2_{x,v}}+(g,\vp)_{L^2_t(0,s)L^2_{x,v}}.
\end{align}

\subsection{Extension to the whole time and decomposition of $\phi_1,\phi_2$}
To utilize the Fourier transform in time, we first extend the solution $\phi_1(t),\phi_2(t)$ (and hence, $f=\phi_1-\phi_2$) from $[0,T]$ to the whole time $t\in\R$. Specifically, we consider the following extension of the transport equation in $(-\infty,0)$ and $(T,\infty)$ with ``good'' collision $\Ga(\mu^{\fr12},f)$ and a sufficiently large constant $A_0=A_0(T)>0$:
\begin{align}\label{extensionTT}
	\begin{cases}
		\pa_tf+v\cdot\na_xf=A_0f-\Ga(\mu^{\fr12},f) &\text{ in }(-\infty,0)\times\R^d_x\times\R^d_v,\\
		\pa_tf+v\cdot\na_xf=-A_0f+\Ga(\mu^{\fr12},f) &\text{ in }(T,\infty)\times\R^d_x\times\R^d_v,
	\end{cases}
\end{align}
with given initial data $f(0),\,f(T)\in L^2_{x,v}$.
Their weak solutions can be obtained by the standard linear parabolic method and by using the weak forms: for any function $\Phi\in C^2_c(\R^{1+2d}_{t,x,v})$,
\begin{align}
	\begin{aligned}
		\label{5weakb}
	\big(f,\Phi\big)_{L^2_{x,v}}\Big|^{t=0}_{t=s}
	=  \big(f,(\pa_t+v\cdot\na_x+A_0)\Phi-\Ga(\mu^{\fr12},\Phi)\big)_{L^2_t(s,0)L^2_{x,v}},\  \text{ for }s\in(-\infty,0),&\\
	\big(f,\Phi\big)_{L^2_{x,v}}\Big|^{t=s}_{t=T}
	=  \big(f,(\pa_t+v\cdot\na_x-A_0)\Phi+\Ga(\mu^{\fr12},\Phi)\big)_{L^2_t(T,s)L^2_{x,v}},\  \text{ for }s\in(T,\infty),&
	\end{aligned}
\end{align}
where, as implied by \eqref{GaBolt}, $\Ga(\mu^{\fr12},\cdot)$ is self-adjoint.
By collision estimate \eqref{esAf}, the standard $L^2$ energy estimates in $(-\infty,0)$ and $(T,\infty)$ yield 
\begin{align}\label{esExtenL2Energy}
	\begin{aligned}
	\|f\|^2_{L^\infty_t(-\infty,0)L^2_{x,v}}+A_0\|f\|^2_{L^2_t(-\infty,0)L^2_{x,v}}+\|f\|^2_{L^2_t(-\infty,0)L^2_xL^2_D}&\le C\|f(0)\|^2_{L^2_{x,v}},\\
	\|f\|^2_{L^\infty_t(T,\infty)L^2_{x,v}}+A_0\|f\|^2_{L^2_t(T,\infty)L^2_{x,v}}+\|f\|^2_{L^2_t(T,\infty)L^2_xL^2_D}&\le C\|f(T)\|^2_{L^2_{x,v}},
	\end{aligned}
\end{align}
if $\fr{A_0}{2}$ is large than the constant in \eqref{esAf}. 
Moreover, $f$ has exponential decay as $t\to\pm\infty$, so that $\Delta_jf$ has sufficient decay as $t\to\pm\infty$.
Moreover, noting $\ga+2s<0$ and using the $L^p$ collision estimate \eqref{GaLpes}, the standard $L^p$ energy estimate for $t\in(T,\infty)$, and similarly for $t\in(-\infty,0)$, implies that, for any $p>2$, 
\begin{align}
	\label{LpTinftyEnergy}
	\begin{aligned}
	\|f\|^p_{L^\infty_t(T,\infty)L^p_{x,v}}+A_0\|f\|^p_{L^p_t(T,\infty)L^p_{x,v}}+\||f|^{\frac{p}{2}}\|^2_{L^2_t(T,\infty)L^2_xL^2_D}&\le C\|f(T)\|^p_{L^p_{x,v}},\\
	\|f\|^p_{L^\infty_t(-\infty,0)L^p_{x,v}}+A_0\|f\|^p_{L^p_t(-\infty,0)L^p_{x,v}}+\||f|^{\frac{p}{2}}\|^2_{L^2_t(-\infty,0)L^2_xL^2_D}&\le C\|f(0)\|^p_{L^p_{x,v}},
	\end{aligned}
\end{align}
having assumed that $A_0=A_0(p)>0$ is sufficiently large (this is one of the places where we need to choose $A_0$ to be large).

\smallskip 
\noindent
\textbf{Case of the solution $\phi_1(t),\phi_2(t)$.}
If $\phi$ solves the Boltzmann equation \eqref{B1} on $t\in[0,T]$, gluing the solution $\phi(t)$ in $(-\infty,0)$, $(0,T)$, and $(T,\infty)$ by summing their weak forms \eqref{5weaka} and \eqref{5weakb}, we obtain the solution $\phi(t)$ ($t\in\R$) that solves 
\begin{align}\label{eqkinetic1}
		\pa_t\phi+v\cdot\na_x\phi=
	\begin{cases}-A_0\phi+\Ga(\mu^{\frac{1}{2}},\phi) &\text{in }\ (T,\infty)\times\R^d_x\times\R^d_v,\\
		\Ga(\mu^{\frac{1}{2}}+\phi,\phi)+\Ga(\phi,\mu^{\frac{1}{2}}) &\text{in }\ [0,T]\times\R^d_x\times\R^d_v,\\
		A_0\phi-\Ga(\mu^{\frac{1}{2}},\phi) &\text{in }\ (-\infty,0)\times\R^d_x\times\R^d_v. 
	\end{cases}
\end{align}

\smallskip 
\noindent\textbf{Case of the difference $f=\phi_1-\phi_2$.}
Recalling \eqref{eqfgxxgg}, we let 
\begin{align}\label{gDefdifference}
G=\begin{cases}
-A_0f+\Ga(\mu^{\frac{1}{2}},f) &\text{in }\ (T,\infty),\\
\Ga(\mu^{\frac{1}{2}}+\phi_1,f)+\Ga(f,\mu^{\frac{1}{2}}+\phi_2)&\text{in }\ (0,T),\\
A_0f-\Ga(\mu^{\frac{1}{2}},f)&\text{in }\ (-\infty,0),
\end{cases}
\end{align}
and extend the equation for $f=\phi_1-\phi_2$, by summing their weak forms \eqref{5weaka} and \eqref{5weakb}, as 
\begin{align}
	\label{eqkinetic12}
	\pa_tf+v\cdot\na_xf = G, \quad& \text{ in } \R_t\times\R^d_x\times\R^d_v,
\end{align}
In summary, we can extend $\phi_1,\phi_2$ as in \eqref{eqkinetic1} and $f=\phi_1-\phi_2$ as in \eqref{eqkinetic12}. 

\smallskip\noindent{\bf Decomposition of $\phi_1,\phi_2$.} 
We split $\phi_1,\phi_2$ into a regular part and an irregular part with small energy in order to utilize the dissipative effect and obtain smallness. Let $\chi^x_R=\chi_R(x)$ and $\chi^{x,v}_R=\chi_R(x,v)$ be the $C^\infty_c$ functions in $x$ and $(x,v)$ (as in \eqref{smoothcutoff}), respectively, which satisfy 
\begin{align*}
	\1_{|x|\le R}\chi^x_R(x)\le \1_{|x|\le\frac{4}{3}R},\quad \1_{|(x,v)|\le R}\chi^{x,v}_R(x,v)\le \1_{|(x,v)|\le\frac{4}{3}R}. 
\end{align*}
Let $\vp^x\in C^\infty_c(\R^{d}_{x})$ and $\vp^{x,v}\in C^\infty_c(\R^{2d}_{x,v})$ satisfy $\vp^x,\vp^{x,v}\ge 0$ and $\int_{x}\vp^x=1$, $\int_{x,v}\vp^{x,v}=1$, and define the standard mollifiers by $\vp^x_\de(x)=\de^{-d}\vp^x(\de^{-1}x)$ and $\vp^{x,v}_\de(x,v)=\de^{-2d}\vp^{x,v}(\de^{-1}(x,v))$.
Then we let 
\begin{align}\label{es78}
	\phi_{1,\de}=\vp^x_\de*_x\phi_1,\qquad \phi_{2,\de}=\vp^{x,v}_\de*_{x,v}\phi_2. 
\end{align}
Then, by the properties of the approximation identity, $\phi_{i,\de}\to\phi_i$ in $L^p_t([0,T])L^p_{x,v}(\<v\>^C)$ as $\de\to\infty$ for $\phi_i\in L^p_t([0,T])L^p_{x,v}(\<v\>^C)$, $p\in[1,\infty)$, and $T,C>0$, where $L^p_{x,v}(\<v\>^C)$ is equipped with the norm $\|\phi\|_{L^p_{x,v}(\<v\>^C)}:=\|\<v\>^C\phi\|_{L^p_{x,v}}$. Moreover, if $\mu+\mu^{\fr12}\phi_1\ge C_{\mathrm{low}}^{-1}\mu^{L_0}$ for some $L_0>0$, then, since $\vp^x_\de*_x1=1$, 
\begin{align*}
	\mu+\mu^{\fr12}\phi_{1,\de}
	&\ge \mu+\mu^{\fr12}\vp^x_\de*_x\phi_1
	\ge \mu+\vp^x_\de*_x(C_{\mathrm{low}}^{-1}\mu^{L_0}-\mu)
	\ge C_{\mathrm{low}}^{-1}\mu^{L_0},
\end{align*}
Therefore, for any $p\in[1,\infty)$ and $\phi_1,\phi_2\in L^p_t([0,T])L^p_{x,v}$, there exist functions $\phi_{1,\de}\in L^p_tH^{\infty,p}(\R^d_x;L^p_v)$ and $\phi_{2,\de}\in L^p_tH^{\infty,p}(\R^{2d}_{x,v})$ (defined by \eqref{es78}) such that 
\begin{align}\label{Propphide}
	\begin{aligned}
	&\phi_{1,\de}\to \phi_1 \quad\text{and}\quad
	\phi_{2,\de}\to \phi_2 \ \ \text{ in }L^p_t([0,T])L^p_{x,v}(\<v\>^C), \ \ \text{for any $C>0$, as }\de\to 0,
	\\
	&\text{if $\phi_{1}$ satisfies the lower bound \eqref{lowerphi}, then $\phi_{1,\de}$ satisfies the same lower bound}.
	\end{aligned}
\end{align}
For convergence in $(t,x,v)$, see estimate \eqref{es718} below.
 Therefore, we split 
 \begin{align}\label{splitphi}
	\phi_{i}=\phi_{i,\de_i}+\wt\phi_{i,\de_i} \quad(i=1,2),
 \end{align}
 where $\phi_{i,\de_i}$ is the regular part and $\wt\phi_{i,\de_i}$ is the irregular but small part, i.e., $\|\<v\>^C\wt\phi_{i,\de_i}\|_{L^p_t([0,T])L^p_{x,v}}=o_{\de_i}(1)\to 0$ as $\de_i\to 0$. 
 We emphasize that we use \textbf{different convergence rate} $\de_1,\de_2$ for $\wt\phi_{1,\de_1},\wt\phi_{2,\de_2}$, respectively. 
Moreover, for the brevity of calculation, under such a decomposition, we assume the \emph{a priori} assumption: $\phi_1,\phi_2$ satisfy \eqref{assumption}, and for $i=1,2$, 
\begin{align}
	\label{apriori2}
	\begin{cases}
	\sup_{r\ge 0}\sum_{j\ge 0}(j+rl_0+\log_2\vpi_0+1)^2\|\De_j\<D_v\>^{-s}\<v\>^{-N}\wt\phi_{1,\de_1}\|^2_{L^\infty_t([0,T])L^\infty_{x}L^2_v}\le o_{\de_1}(1),\\
	\|\<D_v\>^{-s}\<v\>^{-N}\wt\phi_{1,\de_1}\|_{L^\infty_t([0,T])L^\infty_{x}L^2_v}\le o_{\de_1}(1),\\
	\|\<D_v\>^{-s}\<v\>^{-N}\phi_{1,\de_1}\|_{L^\infty_t([0,T])L^\infty_xL^2_v}\le C_{M_0},\\
	\|\<v\>^{|\ga|+d+1}\<D_v\>^{-\fr{d-2s}{2}}\phi_{1,\de_1}\|_{L^\infty_t([0,T])L^\infty_xL^r_v}\le C_{M_0},\\
	\|\<D_x\>^{\frac{d+3}{2}}\phi_{1,\de_1}\|_{L^\infty_t([0,T])L^2_{x,v}}^2\le C_{\de_1,M_0},\\
	\|\<v\>^{C_{l_0,l_1,l,\ga,s}}\<D_x\>^{\fr{d+1}{2}}\<D_v\>^{2s}\phi_{2,\de_2}\|_{L^\infty_t([0,T])L^2_{x,v}}\le C_{\de_2,M_0},\\
	\|\wt\phi_{1,\de_1}\|_{L^{r_0^*}_t([0,T])L^{r_0^*}_{x,v}}\le o_{\de_1}(1),\\
	\|\<v\>^{C_{l_0,l_1,l,\ga,s}}\wt\phi_{2,\de_2}\|_{L^{r_0^*}_t([0,T])L^{r_0^*}_xL^2_v}\le o_{\de_2}(1),\\
	\|(\phi_1,\,\<v\>^{C_{l_0,l_1,l,\ga,s}}\phi_2)\|_{L^\infty_t([0,T])L^{r_0^*}_xL^2_v}\le CM_0,
	\end{cases}
\end{align}
where $r_0^*=r_0^*(\ga,s,d)$ is determined by estimate \eqref{es77} and $\frac{1}{2}=\frac{1}{r_0}+\frac{1}{r_0^*}$.
The constant implicit in $o_{\de_i}(1)$ is independent of the non-fixed parameters $\om_0,\varpi_0$, but depends on the fixed constants $A_0,M_0,l_0,l_1,l,\ga,s$. Here $l_0,l_1,l>0$ are chosen sufficiently large depending only on $\ga,s$, and $A_0=A_0(M_0)>0$, as specified in \eqref{choice10} and \eqref{con5}.

\smallskip 
The first three lines of the \emph{a priori} estimate \eqref{apriori2} follow from the approximation property of Besov spaces. Indeed, using the definition of the dilated Littlewood-Paley decomposition \eqref{DejQkDef} and $\vpi=\vpi_02^{rl_0}$, we know that, for any $J>0$, 
\begin{align*}
	&\sum_{j\ge J+1}(j+rl_0+\log_2\vpi_0+1)^2\|\F^{-1}_x\wh{\Psi_j}(\varpi^{-1}\xi)\F_xf(\xi)\|^2_{L^\infty_t([0,T])L^\infty_{x}L^2_v}\\
	&\ =\sum_{j\ge J+1+\lfloor rl_0+\log_2\vpi_0\rfloor}(j+1)^2\|\F^{-1}_x\wh{\Psi}(2^{-j-rl_0-\log_2\vpi_0+\lfloor rl_0+\log_2\vpi_0\rfloor}\xi)\F_xf(\xi)\|^2_{L^\infty_t([0,T])L^\infty_{x}L^2_v}\\
	&\ \le C\sum_{j\ge J}(j+1)^2\|\F^{-1}_x\wh{\Psi_j}(\xi)\F_xf(\xi)\|^2_{L^\infty_t([0,T])L^\infty_{x}L^2_v}, 
\end{align*}
where $\lfloor \cdot\rfloor$ is the floor function and we used almost-orthogonality like $\Psi_j=\Psi_j(\Psi_{j-1}+\Psi_j+\Psi_{j+1})$ (by letting $\Psi_{-1}=0$) 
We can apply such a dilation to the first line of \eqref{apriori2} and rewrite it as a standard dyadic decomposition as follows:
\begin{align*}
	&\sup_{r\ge 0}\sum_{j\ge 0}(j+rl_0+\log_2\vpi_0+1)^2\|\De_j\<D_v\>^{-s}\<v\>^{-N}\wt\phi_{1,\de_1}\|^2_{L^\infty_t([0,T])L^\infty_{x}L^2_v}
	\\
	&\le C\sum_{j\ge 0}(j+1)^2\|\De^{\mathrm{std}}_j\<D_v\>^{-s}\<v\>^{-N}\wt\phi_{1,\de_1}\|^2_{L^\infty_t([0,T])L^\infty_{x}L^2_v}. 
\end{align*}
For the standard dyadic decomposition, by assumption \eqref{assumption}, for any $\ve>0$, there exists $J>0$ such that 
\begin{align*}
	&\sum_{j\ge 0}(j+1)^2\|\De^{\mathrm{std}}_j\<D_v\>^{-s}\<v\>^{-N}\wt\phi_{1,\de_1}\|^2_{L^\infty_t([0,T])L^\infty_{x}L^2_v}
	<\fr{\ve}{2}. 
\end{align*}
On the other hand, with such a fixed $J>0$, noting that 
\begin{align*}
	\|\wt\De_j\phi_{1,\de_1}\|_{L^\infty_{t,x}L^2_v}
	&\le\int_{\R^d_y}\vp^x(y)\|\De_j\phi_1(x-\de_1y)-\De_j\phi_1(x)\|_{L^\infty_{t,x}L^2_v}\,dy\\
	&\le \de_1\|\na\wt{\Delta}_j\De_j\phi_1\|_{L^\infty_{t,x}L^2_v}
	\le \de_1C_j\|\De_j\phi_1\|_{L^\infty_{t,x}L^2_{v}}, 
\end{align*}
where we used $\Delta_j=\wt{\Delta}_j \Delta_j$ as in \eqref{SlightlyLarger}, 
one can deduce 
\begin{align*}
	&\sum_{0\le j< J}(j+1)^2\|\De^{\mathrm{std}}_j\<D_v\>^{-s}\<v\>^{-N}\wt\phi_{1,\de_1}\|^2_{L^\infty_t([0,T])L^\infty_{x}L^2_v}\\
	&\ \le \de_1C\sum_{0\le j<J}C_j
	(j+1)^2\|\De^{\mathrm{std}}_j\<D_v\>^{-s}\<v\>^{-N}\phi_{1}\|^2_{L^\infty_t([0,T])L^\infty_{x}L^2_v}\\
	&\ \le \de_1C_JM_0<\fr{\ve}{2},  
\end{align*}
provided that $\de_1<\fr{\ve}{2C_JM_0}$. 
Combining the above two estimates, we obtain the first line of the \emph{a priori} estimate \eqref{apriori2}. The second line follows from the first by noting that 
$\wt\phi_{1,\de_1}=\sum_{j\ge0}\De_j\wt\phi_{1,\de_1}$ and 
$\sum_{j\ge 0}|u_j|\le \big(\sum_{j\ge 0}(j+1)^{2}|u_j|^2\big)^{\fr12}$.
The third line follows from the second line and the identity 
$\phi_{1,\de_1}=\phi_1-\wt\phi_{1,\de_1}$, together with the bound 
$o_{\de_1}(1)\le C_{M_0}$ appearing in the second line.
The proof of the fourth line is similar, but involves an additional weight $\<v\>^{|\ga|+d+1}$ and uses $L^r_v$ in place of $L^2_v$, together with the assumption in \eqref{assumption}.


\smallskip 
The fifth through sixth lines in the \emph{a priori} assumptions \eqref{apriori2} can be obtained straightforwardly from assumption \eqref{assumption} and their constructions in \eqref{es78} via mollification. 
For the seventh and eighth lines, noting that $\|\vp^x_\de*_x\phi_1\|_{L^{r_0^*}_x}\le \|\phi_1\|_{L^{r_0^*}_x} \in L^{r_0^*}_t([0,T])L^{r_0^*}_v$, dominated convergence theorem yields 
\begin{align}\label{es718}
\lim_{\de_1\to 0} \|\wt\phi_{1,\de_1}\|_{L^{r_0^*}_t([0,T])L^{r_0^*}_{x,v}}
\le \big\|\lim_{\de_1\to 0}\|\wt\phi_{1,\de_1}\|_{L^{r_0^*}_{x}}\big\|_{L^{r_0^*}_t([0,T])L^p_v} \to 0,
\end{align}
while the term $\wt\phi_{2,\de_2}$ follows by the same argument.
The last line of the \emph{a priori} estimate \eqref{apriori2} follows directly from the assumption \eqref{assumption}. 

\smallskip 
For clarity, we indicate the dependence of the positive parameters below:
\begin{align*}
	\begin{cases}
	A_0=A_0(M_0)\gg 1,\ \ \de_1\ll 1,\ \ \de_3\ll 1,\ \ \de_4=\de_4(A_0,M_0)\ll 1, \\ \om_0=\om_0(A_0,M_0,\de_1,\de_3,\de_4)\gg 1,\ \ \de_2=\de_2(\om_0)\ll 1.
	\end{cases}
\end{align*}


\subsection{Littlewood-Paley decomposition for kinetic equations}\label{SubsecLPkinetic}
For any $j,k,r\ge 0$, applying the Littlewood-Paley decomposition $\De_jP_kR_r$ from \eqref{DejQkDef} to equation \eqref{eqkinetic12}, we have 
\begin{align}\label{eqfgxx1}
	\pa_tP_k\De_jR_rf+v\cdot\na_xP_k\De_jR_rf+[P_k,v\cdot\na_x]\De_jR_rf=P_k\De_jR_rG.
\end{align}
Note that $\De_j$ and $R_r$ commute with $\pa_t$ and $v\cdot\na_x$. Although we haven't written the superscript $j$ in $P_k^j$ in this subsection, all the $P_k$ here are \textbf{indeed} $P_k^j$ as in \eqref{Def712}. 
However, all the estimates in this subsection use the \textbf{same} index $j$ for $\De_j$ and $P_k^j$, except for intermediate steps of the nonlinear collision in Subsection \ref{SecSpatialNonlinear}, which is the only place where the Bony decomposition is applied, i.e., in the estimate \eqref{DejGaMaines}.
By the convolution form of $P_k$, the commutator $[P_k,v\cdot\na_x]$ can be rewritten as   
\begin{align*}
	[P_k,v\cdot\na_x]\De_jf(v)&=
	\int_{\R^d_{v_*}}\om^d\Psi_{k}(\om(v-v_*))(v_*-v)\cdot\na_x\De_jf(v_*)\,dv_*\\
	&=-(\om2^k)^{-1}\na_x\cdot P_k^{(1)}\De_jf(v),
\end{align*}
where $P_k^{(1)}$ is a derivative variant of $P_k$ given by 
\begin{align}\label{wtPk}
	P_k^{(1)}f(v)&=\int_{\R^d_{v_*}}\om^d\Psi_k^{(1)}(\om(v-v_*))f(v_*)\,dv_*,\quad\text{ where }~\Psi_k^{(1)}(v)=2^kv\Psi_k(v),
\end{align}
and thus, $\wh{\Psi_k^{(1)}}(\eta)=\frac{-1}{2\pi i}\na\wh{\Psi_k}(\eta)$, i.e., derivative of $\wh{\Psi_k}$. 
Taking the $L^2$ inner product of \eqref{eqfgxx1} with $\De_jP_kR_rf$ over $\R^d_x\times\R^d_v$, we obtain the main Littlewood-Paley decomposition equation: 
\begin{align}\label{eqDeftemp}\notag
	\frac{1}{2}\pa_t\|\De_jP_kR_rf\|_{L^2_{x,v}}^2
	&=(\om2^k)^{-1}\big(\na_x\cdot P_k^{(1)}\De_jR_rf,\De_jP_kR_rf\big)_{L^2_{x,v}}\\
	&\quad+\big(\De_jP_kR_rG,\De_jP_kR_rf\big)_{L^2_{x,v}}.
\end{align}
We will multiply the energy estimate \eqref{eqDeftemp} by $(\om2^k)^{-2s}2^{2rl}$ and estimate the commutator term $\nabla_x\cdot P_k^{(1)}$ and the collision term $G$ in the following subsections.

\subsubsection{The estimate of the commutator term}\label{SecEsCommTerm}
For the commutator $[P_k, v\cdot \na_x]$, we utilize the lower bound of $\eta$ (arising from the derivative version $\nabla \widehat{\Psi_k}(\eta)$ in \eqref{wtPk}).
For any $k\ge 0$, applying the upper bound $|\xi|\le 2\varpi_02^j2^{rl_0}$ of the kernel of $\De_j$ to $|\na_x|$, we have 
\begin{align*}
	&\big|(\om2^k)^{-1-2s}\big(\na_x\cdot P_k^{(1)}\De_jR_rf,\De_jP_kR_rf\big)_{L^2_{x,v}}\big|\\
	&\notag=\big|(\om2^k)^{-1-2s}\big(\<D_v\>^{-s}\na_x\cdot P_k^{(1)}\De_jR_rf,\<D_v\>^{s}\De_jP_kR_rf\big)_{L^2_{x,v}}\big|\\
	&\le C(\om2^k)^{-1-2s}\varpi_02^j2^{rl_0}\|\<v\>^{-\frac{\ga}{2}}\<D_v\>^{-s}P_k^{(1)}\De_jR_rf\|_{L^2_{x,v}}\|\<v\>^{\frac{\ga}{2}}\<D_v\>^{s}\De_jP_kR_rf\|_{L^2_{x,v}}.
\end{align*}
Thus, summing over $k\ge 0$ and using $\|\<v\>^{\frac{\ga}{2}}\<D_v\>^{s}f\|_{L^2_v}\le C\|f\|_{L^2_D}$, we have 
\begin{align}\notag\label{es811}
	&\sum_{k\ge 0}\big|(\om2^k)^{-1-2s}\big(\na_x\cdot P_k^{(1)}\De_jR_rf,\De_jP_kR_rf\big)_{L^2_{x,v}}\big|\\
	&
	\le C\Big(\sum_{k\ge 0}(\om2^k)^{-2-2s}\varpi_0^22^{2j}2^{2rl_0}\|\<v\>^{-\frac{\ga}{2}}\<D_v\>^{-s}P_k^{(1)}\De_jR_rf\|_{L^2_{x,v}}^2\Big)^{\frac{1}{2}}
	\Big(\sum_{k\ge 0}(\om2^k)^{-2s}\|\De_jP_kR_rf\|_{L^2_{x}L^2_D}^2\Big)^{\frac{1}{2}}.
\end{align}
For the first right-hand factor, using the equivalence of weight $2^{r}\approx \<v\>$ induced by the cutoff $R_r$, we apply the commutator estimate \eqref{LowerSupportComm} (with $l=-(1+2s)l_1+l_0$, $n=-\frac{\gamma}{2}$, $m=-s$, and $s'=0$ therein) to deduce 
\begin{align*}\notag
	&2^{2rl_0}\|\<v\>^{-\frac{\ga}{2}}\<D_v\>^{-s}P_k^{(1)}\De_jR_rf\|_{L^2_v}^2\\
	&\le C2^{(2+4s)rl_1}
	\Big((\om2^k)^{-4s}\|\<v\>^{-(1+2s)l_1+l_0-\frac{\ga}{2}}\<D_v\>^{s}(P_{k-1}+P_k+P_{k+1})\De_jR_rf\|_{L^2_v}^2\\
	&\qquad\qquad\qquad\quad+(\om2^k)^{-2s-2}\|\<v\>^{-(1+2s)l_1+l_0-\frac{\ga}{2}-1}\De_jR_rf\|_{L^2_v}^2\Big),
\end{align*}
where we set $P_{-1}=0$. 
We would like to choose $l_0\ge 0$ to be as large as possible. Consequently, recalling our choice $\om=\om_02^{\al j+rl_1}$, we set 
\begin{align}\label{choice3}
	\al=\frac{1}{1+2s},
	\quad (-1-2s)l_1+l_0-\frac{\ga}{2}=\frac{\ga}{2}. 
\end{align}
For the term $P_{k-1}+P_k+P_{k+1}=P^j_{k-1}+P^j_k+P^j_{k+1}$, we can apply a simple translation in $k$ as in \eqref{esg719s}.
Thus, by \eqref{esD}, $\|\<v\>^{(-1-2s)l_1+l_0-\frac{\ga}{2}}\<D_v\>^{s}(\cdot)\|_{L^2_v}\le C\|\cdot\|_{L^2_D}$, and the first factor in \eqref{es811} is 
\begin{align*}
	&\sum_{k\ge 0}(\om2^k)^{-2-2s}\varpi_0^22^{2j}2^{2rl_0}\|\<v\>^{-\frac{\ga}{2}}\<D_v\>^{-s}P_k^{(1)}\De_jR_rf\|_{L^2_{x,v}}^2\\
	&\le C\sum_{k\ge 0}(\om_02^{\al j}2^k)^{-2-4s}\varpi_0^22^{2j}
	\Big((\om2^k)^{-2s}\|(P_{k-1}+P_k+P_{k+1})\De_jR_rf\|_{L^2_xL^2_D}^2
	\\
	&\qquad\qquad\qquad\qquad\quad
	+(\om_02^{\al j+rl_1}2^k)^{-2}\|\<v\>^{(-1-2s)l_1+l_0-\frac{\ga}{2}-1}\De_jR_rf\|_{L^2_{x,v}}^2
	\Big)\\
	&\le C\om_0^{-2-4s}\varpi_0^2
	\Big(
	\sum_{k\ge 0}(\om2^k)^{-2s}\|P_k\De_jR_rf\|_{L^2_{x}L^2_D}^2
	+
	(\om_02^{\al j})^{-2}\|\<v\>^{(-2-2s)l_1+l_0-\frac{\ga}{2}-1}\De_jR_rf\|_{L^2_{x}L^2_v}^2
	\Big), 
\end{align*}
The coefficient $2^{\al j}$ in the second right-hand term can also be thrown away since it's the lower-order term.
Thus, using $l_0=\ga+(1+2s)l_1$ and summing $r,j\ge 0$ in \eqref{es811} yields  
\begin{align}\notag\label{es729}
	&\sum_{r,j,k\ge 0}(\om2^k)^{-2s}2^{2rl}\big|(\om2^k)^{-1}\big(\na_x\cdot P_k^{(1)}\De_jR_rf,\De_jP_kR_rf\big)_{L^2_{x,v}}\big|\\
	&
	\ \ \le C\om_0^{-1-2s}\varpi_0\sum_{r,j,k\ge 0}(\om2^k)^{-2s}2^{2rl}\|\De_jP_kR_rf\|_{L^2_{x}L^2_D}^2
	+
	C\om_0^{-3-2s}\varpi_0\|\<v\>^{-l_1+l+\frac{\ga}{2}-1}f\|_{L^2_{x}L^2_v}^2,
\end{align}
where we will verify that $-l_1+l+\frac{\ga}{2}-1\le -sl_1+l$ later.

\subsubsection{The estimate of collision terms}
Here, we estimate the collision term $G$, defined in \eqref{gDefdifference}, restricted to $[0,\infty)$:
\begin{align*}
	G\1_{t\in[0,\infty)}&=\Ga(\mu^{\frac{1}{2}},f)\1_{t\in[T,\infty)}-A_0f\1_{t\in(T,\infty)}
 	+\big(\Ga(f,\mu^{\frac{1}{2}})+\Ga(\mu^{\fr12}+\phi_1,f)+\Ga(f,\phi_2)\big)\1_{t\in[0,T]}.
\end{align*}
In the following, we calculate the collision terms and insert the cutoffs $\1_{t\in(T,\infty)}$ and $\1_{t\in[0,T]}$ later, where necessary.
For the trivial part $A_0f$, we have 
 \begin{align}\label{esA00}
	\sum_{k\ge 0}(\om2^k)^{-2s}2^{2rl}\big(\De_jP_kR_rA_0f,\De_jP_kR_rf\big)_{L^2_t([T,\infty))L^2_{x,v}}
	&= A_0\sum_{k\ge 0}(\om2^k)^{-2s}2^{2rl}\|\De_jP_kR_rf\|_{L^2_t([T,\infty))L^2_{x,v}}^2.
 \end{align}
For any linearized operator $L$ with respect to $v$, it's straightforward to verify that $\De_jLf=L\De_jf$; hence, a Bony-type or commutator estimate is not needed. 
For the linear regular term $\Ga(f,\mu^{\frac{1}{2}})$ on $[0,T]$, by collision estimate \eqref{esKf}, commutator estimate \eqref{dissUpper1b}, and $\sum_{j,k\ge 0}(\om2^k)^{-s}\le C\om_0^{-s}$, we have for any $N>0$ that
\begin{align}\label{es713a}
	&\notag
	\sum_{r,j,k\ge 0}(\om2^k)^{-2s}2^{2rl}\big(\De_jP_kR_r\Ga(f,\mu^{\frac{1}{2}}),\De_jP_kR_rf\big)_{L^2_t([0,T])L^2_{x,v}}\\
	&\le \notag
	C\sum_{r,j,k\ge 0}(\om2^k)^{-2s}2^{2rl}\|\<v\>^{-N}f\|_{L^2_t([0,T])L^2_{x,v}}\|\<v\>^{-N}\De_jP_kR_rf\|_{L^2_t([0,T])L^2_{x,v}}\\
	&\le \notag
	C\sum_{r,j,k\ge 0}(\om2^k)^{-2s}2^{2rl}\|\<v\>^{-N}f\|_{L^2_t([0,T])L^2_{x,v}}
	2^{-rN}\big(\|\De_jP_kR_rf\|_{L^2_t([0,T])L^2_{x,v}}
	+\om_0^{-1}\|R_rf\|_{L^2_t([0,T])L^2_{x,v}}
	\big)\\
	&
	\le C\om_0^{-1-2s}\|\<v\>^{-N}f\|_{L^2_t([0,T])L^2_{x,v}}^2
	+C\om_0\sum_{r,j,k\ge 0}(\om2^k)^{-2s}2^{2rl}\|\De_jP_kR_rf\|^2_{L^2_t([0,T])L^2_{x,v}}. 
\end{align}
For the linear term $\Ga(\mu^{\frac{1}{2}},f)$ and its commutator in $P_kR_r$, applying the commutator estimate \eqref{ThmCommes} from Theorem \ref{ThmComm} (with any $\wt{l_0}\in\R$ and $s'=0$ therein) and \eqref{dissUpper1b}, and using $\<v\>^{l}\approx 2^{rl}$ and $|\eta|\le \om2^k$ on the support of $R_r$ and the kernel of $P_k$, respectively, we have 
\begin{align*}
	&\sum_{k\ge 0}(\om2^k)^{-2s}2^{2rl}\big(\De_jP_kR_r\Ga(\mu^{\frac{1}{2}},f)-\Ga(\mu^{\frac{1}{2}},\De_jP_kR_rf),\De_jP_kR_rf\big)_{L^2_v}\\
	&\le 
	C\sum_{k\ge 0}(\om2^k)^{-2s}2^{2rl}2^{r\wt{l_0}}\|\<v\>^{-\wt{l_0}+\frac{\ga+2s}{2}}\De_jf\|_{L^2_v}\big(\|\De_jP_kR_rf\|_{L^2_D}+(\om2^k)^{s-1}\|\<v\>^{\frac{\ga}{2}-1}\De_jR_rf\|_{L^2_v}\big).
\end{align*}
For the term $\|\De_jP_kR_rf\|_{L^2_D}$, similar to \eqref{esinter1}, we have 
\begin{align*}
	&\sum_{k\ge 0}(\om2^k)^{-2s}2^{2rl}\|\De_jP_kR_rf\|_{L^2_D}
	\le C\Big(\sum_{k\ge 0}(k+1)^2(\om2^k)^{-4s}2^{4rl}\|\De_jP_kR_rf\|^2_{L^2_D}\Big)^{\fr12}\\
	&\le C\Big(\sum_{k\ge 0}\big((\om2^k)^{-2s}2^{2rl}+(k+1)^2(\om2^k)^{-6s}2^{6rl}\big)\|\De_jP_kR_rf\|^2_{L^2_D}\Big)^{\fr12}\\
	&\le C\Big(\sum_{k\ge 0}(\om2^k)^{-2s}2^{2rl}\|\De_jP_kR_rf\|^2_{L^2_D}\Big)^{\fr12}
	+C\Big(\sum_{k\ge 0}2^{r(-4sl_1+4l)}(\om2^k)^{-2s}2^{2rl}\|\De_jP_kR_rf\|^2_{L^2_v}\Big)^{\fr12}
\end{align*}
Thus, letting $\wt{l_0}=\fr{\ga+2s}{2}+sl_1-l$ and using $\sum_{r\ge 0}2^{r(\ga+2s)}<\infty$, we integrate over $(t,x)\in[T,\infty)\times\R^d$, sum over $r,j$, and apply the support condition of $R_r$ and the kernel of $P_k$, together with the help of commutator estimate \eqref{dissUpper1b} and Littlewood-Paley theorem \ref{LPThm}. It follows that, for any $N\ge 0$, we have for any $\ka>0$ that 
\begin{align}\label{es713temp}
	&\sum_{r,j,k\ge 0}(\om2^k)^{-2s}2^{2rl}\big(\De_jP_kR_r\Ga(\mu^{\frac{1}{2}},f)-\Ga(\mu^{\frac{1}{2}},\De_jP_kR_rf),\De_jP_kR_rf\big)_{L^2_t([T,\infty))L^2_{x,v}}
	\notag\\
	& \le 
	\ka\sum_{r,j,k\ge 0}(\om2^k)^{-2s}2^{2rl}\|\De_jP_kR_rf\|_{L^2_{t}([T,\infty))L^{2}_{x}L^2_D}^2
	+C_\ka\om_0^{-2s}\|\<v\>^{-sl_1+l}f\|_{L^2_t([T,\infty))L^2_{x,v}}^2,
\end{align}
For the part $\Ga(\mu^{\frac{1}{2}},\De_jP_kR_rf)$, using the dissipation estimate \eqref{esAf}, 
\begin{align}\label{es713}\notag
	&\sum_{r,j,k\ge 0}(\om2^k)^{-2s}2^{2rl}\big(\Ga(\mu^{\frac{1}{2}},\De_jP_kR_rf),\De_jP_kR_rf\big)_{L^2_t([T,\infty))L^2_{x,v}}\\
	&\ \le 
	-c_0\sum_{r,j,k\ge 0}(\om2^k)^{-2s}2^{2rl}\|\De_jP_kR_rf\|^2_{L^2_t([T,\infty))L^2_{x}L^2_D}
	+C\om_0^{-2s}\|\<v\>^{-N}f\|_{L^2_t([T,\infty))L^2_{x}L^2_v}^2.
\end{align}
For the term $\Ga(\mu^{\fr12}+\phi_1,f)+\Ga(f,\phi_2)$ in $[0,T]$, together with the decomposition \eqref{splitphi}, we split  
\begin{align}\label{es711}\notag
	&\De_jP_kR_r\big(\Ga(\mu^{\fr12}+\phi_1,f)+\Ga(f,\phi_2)\big)\\
	&\notag=
	\big(\De_jP_kR_r\Ga(\mu^{\frac{1}{2}},f)-\Ga(\mu^{\frac{1}{2}},\De_jP_kR_rf)\big)\\
	&\notag\ 
	+
	\big(\De_jP_kR_r\Ga(\phi_1,f)-\De_j\Ga(\phi_1,P_kR_rf)
	+\De_jP_kR_r\Ga(f,\phi_2)-\De_j\Ga(f,P_kR_r\phi_2)\big)\\
	&\notag\ +\big(\De_j\Ga(\phi_{1,\de_1},P_kR_rf)-\Ga(\phi_{1,\de_1},\De_jP_kR_rf)\big)\\
	&\notag\ 
	+\big(\De_j\Ga(\wt\phi_{1,\de_1},P_kR_rf)+\De_j\Ga(f,P_kR_r\wt\phi_{2,\de_2})\big)\\
	&\ 
	+\Ga(\mu^{\fr12}+\phi_{1,\de_1},\De_jP_kR_rf)
	+\De_j\Ga(f,P_kR_r\phi_{2,\de_2})=:\sum_{m=1}^6\Ga_m. 
\end{align}
For the first to third right-hand terms, we apply the same arguments as those that deduced estimates \eqref{es713temp} (replacing $[T,\infty)$ with $[0,T]$), \eqref{es713aaphi2}, and \eqref{CommDejf}, respectively. In fact, together with the \emph{a priori} assumption \eqref{apriori2}, we have 
\begin{align}\label{es730z}\notag
	&\sum_{m=1}^3\sum_{r,j,k\ge 0}(\om2^k)^{-2s}2^{2rl}\big(\Ga_m,\De_jP_kR_rf\big)_{L^2_t([0,T])L^2_{x,v}}\\
	&\notag\le \ka\sum_{r,j,k\ge 0}(\om2^k)^{-2s}2^{2rl}\|\De_jP_kR_rf\|_{L^2_{t}L^{2}_{x}L^2_D}^2
	+C_\ka\om_0^{-2s}\|\<v\>^{-sl_1+l}f\|_{L^2_tL^2_{x}L^2_v}^2\\
	&\notag\ 
	+
	C_\ka \om_0^{-2s{}}\|(\phi_1\,\<v\>^{C_{l_0,l_1,l,\ga,s}}\phi_2)\|_{L^\infty_tL^{r_0^*}_{x}L^2_v}^2\|f\|_{L^2_{t,x}L^2_v}^2
	\notag\\
	&\ \notag
	+\ka\sum_{r,j,k\ge 0}(\om2^k)^{-2s}2^{2rl}
	\|\De_jP_kR_rf\|_{L^{2}_{t,x}L^2_D}^2\\
	&\notag\ 
	+\om_0^{4sM}\varpi_0^{2\ve M}
	\sum_{r,j,k\ge 0}
	(\om2^k)^{-2s}2^{2rl}\|\De_jP_kR_rf\|_{L^2_{t,x,v}}^2\\
	&\notag\ +C\om_0^{-s-1}\varpi_0^{\ve}\|(\phi_1,\,\<v\>^{C_{l_0,l_1,l,\ga,s}}\phi_2)\|_{L^\infty_tL^{r_0^*}_{x}L^2_v}
	\|(\<v\>^{2\ve l_0+(-s-1)l_1+l+\ga+2s-1}f,\,f)\|_{L^{2}_{t,x}L^2_v}^2
	\\
	&\notag\ +
	C_\ka\varpi_0^{-2}
	\|\<D_x\>^{\frac{d+3}{2}}\phi_{1,\de_1}\|_{L^2_tL^2_{x,v}}^2
	\|\<v\>^{-l_0+l+\fr{\ga+2s}{2}}f\|_{L^2_tL^2_{x,v}}^2\\
	&
	\notag
	\le 
	2\ka\sum_{r,j,k\ge 0}(\om2^k)^{-2s}2^{2rl}\|\De_jP_kR_rf\|_{L^2_{t}([0,T])L^{2}_{x}L^2_D}^2
	\\
	&\ 
	+C_{\ka} \big(M_0^2
		\om_0^{-2s{}}
		+M_0\om_0^{-s-1}\varpi_0^{\ve}
		+C_{\de_1,M_0}\varpi_0^{-2}
		\big)\|f\|_{L^2_{t}([0,T])L^2_{x}L^2_v}^2\notag\\
	&\ +C_{\om_0,\varpi_0}
	\sum_{r,j,k\ge 0}
	(\om2^k)^{-2s}2^{2rl}\|\De_jP_kR_rf\|_{L^2_{t,x,v}}^2, 
\end{align}
for any $\ka,\ve>0$,
provided that \eqref{choiceM2}, and the conditions in \eqref{choideM3} below hold.
\begin{align}
\label{choideM3}
5\ve l_0
+\frac{\ga+2s}{2}\le 0,\ \ 
-sl_1+l
\le 0,\ \ 
-l_0+l+\fr{\ga+2s}{2}\le 0.
\end{align}
For $\Ga_4$, by Bony-type estimate \eqref{DejGaMaines} and the \emph{a priori} assumption \eqref{apriori2}, 
\begin{align}\label{es730}\notag
	&\sum_{r,j,k\ge 0}(\om2^k)^{-2s}2^{2rl}\big(\Ga_4,\De_jP_kR_rf\big)_{L^2_t([0,T])L^2_{x,v}}\\
	&\notag\le C
	\Big(\|\<D_v\>^{-s}\<v\>^{-N}\phi_1\|_{L^\infty_t([0,T])L^\infty_{x}L^2_v}
	+\|\<v\>^{2l}\phi_2\|_{L^{r_0^*}_t([0,T])L^{r_0^*}_xL^2_v}\\
	&\quad\notag
	+\sup_{r\ge 0}\Big(\sum_{j\ge 0}(j+rl_0+\log_2\vpi_0+1)^2\|\De_j\<D_v\>^{-s}\<v\>^{-N}\phi_1\|_{L^\infty_t([0,T])L^\infty_{x}L^2_v}^2\Big)^{\frac{1}{2}}\Big)
	\\
	&\notag\ \times\Big(\sum_{r,j,k\ge 0}(\om(j)2^k)^{-2s}2^{2rl}\|\De_jP^j_kR_rf\|^2_{L^2_t([0,T])L^2_{x}L^2_D}
	\\
	&\quad\notag
	+\om_0^{-2s{}}\|
	f\|_{L^2_t([0,T])L^2_{x,v}}^2
	+C_{\om_0}\sum_{r,j,k\ge 0}
	(\om(j)2^k)^{-2s}2^{2rl}\|\De_jP^j_kR_rf\|_{L^2_t([0,T])L^{2}_{x,v}}^2\Big)\\
	&\notag +\|(\phi_1,\,\<v\>^{2l}\phi_2)\|_{L^\infty_t([0,T])L^{r_0^*}_xL^2_v}\Big(
	\om_0^{-4s}\|f\|_{L^2_t([0,T])L^2_{x,v}}^2\\
	&\quad \notag
	+\om_0^{-2s}\sum_{r,j,k\ge 0}(\om(j)2^k)^{-2s}2^{2rl}\|\De_jP^j_kR_rf\|_{L^2_t([0,T])L^{2}_{x}L^2_D}^2\\
	&\quad \notag
	+C_{\om_0,\varpi_0}\sum_{r,j,k\ge 0}(\om(j)2^k)^{-2s}2^{2rl}\|\<v\>^{\fr{\ga+2s}{2}}\De_jP^j_kR_rf\|_{L^2_t([0,T])L^{2}_{x,v}}^2\Big)
	\\
	&
	+
	\|\<v\>^{2l}\phi_2\|_{L^{r_0^*}_t([0,T])L^{r_0^*}_xL^2_v}\sup_{r\ge 0}\sum_{p\ge 0}(p+1)^2\|\De_p\<D_v\>^{-s}\<v\>^{-N}f\|_{L^{r_0}_t([0,T])L^{r_0}_{x}L^2_v}^2\notag
	\\
	&
	\notag
	\le 
	C\big(o_{\de_1}(1)+o_{\de_2}(1)+
	M_0\om_0^{-2s}
	\big)\\
	&\notag\ \ \times \Big(\sum_{r,j,k\ge 0}(\om2^k)^{-2s}2^{2rl}\|\De_jP_kR_rf\|^2_{L^2_t([0,T])L^2_{x}L^2_D}
	+\om_0^{-2s{}}\|f\|_{L^2_t([0,T])L^2_{x,v}}^2\Big)
	\\
	&\notag\ \ 
	+C_{\om_0,\varpi_0,M_0}
	\sum_{r,j,k\ge 0}
	(\om2^k)^{-2s}2^{2rl}\|\De_jP_kR_rf\|_{L^2_t([0,T])L^{2}_{x,v}}^2
	\notag\\
	&\ \ 
	+
	o_{\de_2}(1)\sup_{r\ge 0}\sum_{j\ge 0}(j+1)^2\|\De_j\<D_v\>^{-s}\<v\>^{-N}f\|_{L^{r_0}_t([0,T])L^{r_0}_{x}L^2_v}^2, 
\end{align}
for any $\ve,\ka>0$, provided that \eqref{choiceM2} hold. For the term $\Ga_5=\Ga(\mu^{\fr12}+\phi_{1,\de_1},\De_jP_kR_rf)$ in \eqref{es711}, since $\phi_1$ satisfies the lower bound \eqref{lowerphi}, by \eqref{Propphide}, $\phi_{1,\de_1}$ satisfies the same lower bound \eqref{lowerphi}. Thus, using \eqref{GaPhi1ff}, we have 
\begin{align*}\notag
	&(\Ga(\mu^{\fr12}+\phi_{1,\de_1},\De_jP_kR_rf),\De_jP_kR_rf)_{L^2_{v}}
	\le -\fr{c_0}{2}\|\De_jP_kR_rf\|_{L^2_D}^2
	\\
	&+C\|\<v\>^{\fr{\ga+2s}{2}}\De_jP_kR_rf\|_{L^2_v}^2\\
	\notag
	&+C\|(\<v\>^{|\ga|+d+1}\<D_v\>^{-s}\phi_{1,\de_1},\,\<v\>^{-N}\<D_v\>^{-\fr{d-2s}{2}}\phi_{1,\de_1})\|_{L^r_v}^2\|\<v\>^{\fr{\ga+2s}{2}}\De_jP_kR_rf\|_{L^2_v}^2,
\end{align*}
provided that $r=r(\ga,s,d)\in(2,\infty)$ is sufficiently large. Indeed, our assumption $\ga+s>-\fr{d}{2}$ implies $\ga-\fr{d-2s}{2}>-d$, so that $\ga-\fr{d-2s}{2}>-\fr{d}{r'}$ for some $r'>1$ sufficiently close to $1$, and hence a large $r\in(2,\infty)$ exists with $\fr1r+\fr{1}{r'}=1$.
This, together with $\|\<v\>^{\cdots}\<D_v\>^{-(\cdots)}\phi_{1,\de_1}\|_{L^\infty_t([0,T])L^\infty_xL^r_v}\le C_{M_0}$ in the \emph{a priori} assumption \eqref{apriori2} and Cauchy-Schwarz inquality, implies that 
\begin{align}\notag\label{esGa5}
	&\sum_{r,j,k\ge 0}(\om2^k)^{-2s}2^{2rl}(\Ga(\mu^{\fr12}+\phi_{1,\de_1},\De_jP_kR_rf),\De_jP_kR_rf)_{L^2_t([0,T])L^2_{x,v}}\\
	&
	\le -\fr{c_0}{2}\sum_{r,j,k\ge 0}(\om2^k)^{-2s}2^{2rl}\|\De_jP_kR_rf\|_{L^2_t([0,T])L^2_xL^2_D}^2
	+C_{M_0}\sum_{r,j,k\ge 0}(\om2^k)^{-2s}2^{2rl}\|\De_jP_kR_rf\|_{L^2_t([0,T])L^2_{x,v}}^2.
\end{align}
For the last term $\Ga_6=\De_j\Ga(f,P_kR_r\phi_{2,\de_2})$ in \eqref{es711}, applying \eqref{64eq} and placing the $2s$ derivative and the bad weight on $\phi_{2,\de_2}$ yields
\begin{align}\label{esGa6}\notag
	&\sum_{r,j,k\ge 0}(\om2^k)^{-2s}2^{2rl}\big(\De_j\Ga(f,P_kR_r\phi_{2,\de_2}),\De_jP_kR_rf\big)_{L^2_t([0,T])L^2_{x,v}}\\
	&\notag\ \le 
	C\sum_{r,j,k\ge 0}(\om2^k)^{-2s}2^{2rl}\|f\|_{L^2_{t,x,v}}\|\<v\>^{\ga+2s}\<D_v\>^{2s}P_kR_r\phi_{2,\de_2}\|_{L^\infty_{t,x}L^2_v}\|\De_jP_kR_rf\|_{L^2_{t,x,v}}\\
	&\notag\ \le 
	C\om_0^{-s}\|\<v\>^{-sl_1+l+\ga+2s}\<D_x\>^{\fr{d+1}{2}}\<D_v\>^{2s}\phi_{2,\de_2}\|_{L^\infty_t([0,T])L^2_{x,v}}\|f\|_{L^2_t([0,T])L^2_{x,v}}\\
	&\notag\quad\times\Big(\sum_{r,j,k\ge 0}(\om2^k)^{-2s}2^{2rl}\|\De_jP_kR_rf\|_{L^2_t([0,T])L^2_xL^2_v}^2\Big)^{\fr12}\\
	&\ \le \om_0^{-2s}\|f\|_{L^2_t([0,T])L^2_{x,v}}^2+C_{\de_2,M_0}\sum_{r,j,k\ge 0}(\om2^k)^{-2s}2^{2rl}\|\De_jP_kR_rf\|_{L^2_t([0,T])L^2_xL^2_v}^2. 
\end{align}
Therefore, using the decomposition \eqref{es711}, collecting the estimates \eqref{esA00}, \eqref{es713a}, \eqref{es713temp}, \eqref{es713}, \eqref{es730z} \eqref{es730}, \eqref{esGa5}, and \eqref{esGa6} into the estimate of $G$ with sufficiently small $\ka,\, o_{\de_1}(1),o_{\de_2}(1)>0$ (to absorb the $L^2_D$ terms; the smallness depends only on $c_0>0$, and therefore on $\ga, s$, and the size of the lower bound \eqref{lowerphi}), and applying the \emph{a priori} assumption \eqref{apriori2}, we obtain 
\begin{align}\notag\label{es733}
	&\sum_{r,j,k\ge 0}(\om2^k)^{-2s}2^{2rl}\Big(\De_jP_kR_rG\1_{t\in[0,\infty)},\De_jP_kR_rf\Big)_{L^2_t([0,\infty))L^2_{x,v}}\\
	&\notag\ \le -A_0\sum_{k\ge 0}(\om2^k)^{-2s}2^{2rl}\|\De_jP_kR_rf\|_{L^2_t([T,\infty))L^2_{x,v}}^2
	\\
	&\notag\ \ 
	-\fr{c_0}{2}\sum_{r,j,k\ge 0}(\om2^k)^{-2s}2^{2rl}\|\De_jP_kR_rf\|^2_{L^2_t([T,\infty))L^2_{x}L^2_D}\\
	&\notag\ \ 
	-\fr{c_0}{4}\sum_{r,j,k\ge 0}(\om2^k)^{-2s}2^{2rl}\|\De_jP_kR_rf\|_{L^2_t([0,T])L^2_xL^2_D}^2+\om_0^{-2s}\|f\|_{L^2_t([T,\infty))L^2_{x,v}}^2\\
	&\notag\ \ 
	+C\big(
		M_0^2\om_0^{-2s{}}
	+M_0\om_0^{-s-1}\varpi_0^{\ve}
	+C_{\de_1,M_0}\varpi_0^{-2}
	\big)\|f\|_{L^2_{t}([0,T])L^2_{x}L^2_v}^2\\
	&\notag\ \ 
	+
	o_{\de_2}(1)\sup_{r\ge 0}\sum_{j\ge 0}(j+1)^2\|\De_j\<D_v\>^{-s}\<v\>^{-N}f\|_{L^{r_0}_t([0,T])L^{r_0}_{x}L^2_v}^2\\
	&\ \ 
	+C_{\de_2,\om_0,\varpi_0,M_0}\sum_{r,j,k\ge 0}(\om2^k)^{-2s}2^{2rl}\|\De_jP_kR_rf\|_{L^2_t([0,T])L^2_xL^2_v}^2,
\end{align}
for any $\ve,\de>0$ and a generic constant $c_0=c_0(\ga,s,C_{\mathrm{low}},L_0)>0$, provided that \eqref{choiceM2} and \eqref{choideM3} hold. 
Here, we also simply verify $\big(o_{\de_1}(1)+o_{\de_2}(1)\big)\om_0^{-2s}\le \om_0^{-2s}$. 

\subsubsection{The main $L^2$ energy estimates}
Assume that $\al,l_0,l_1\ge 0$ satisfy \eqref{choice3}. 
We take the summation $\sum_{r,j,k\ge 0}\int_{[0,\tau)}\eqref{eqDeftemp}\,dt$ with any $\tau>0$, utilize the \emph{a priori} assumption \eqref{apriori2}, use the commutator estimate \eqref{es729} and the collision estimate \eqref{es733}, 
and finally take $\sup_{\tau\ge 0}$. 
For the energy of the initial data, we simply apply Littlewood-Paley theorem \ref{LPThm} and $\om=\om_02^{\al j+rl_1}\ge \om_02^{rl_1}$ to obtain 
\begin{align*}
\sum_{r,j,k\ge 0}(\om2^k)^{-2s}2^{2rl}\|\De_jP_kR_rf_0\|_{L^2_{x,v}}^2\le
C\om_0^{-2s}\|\<v\>^{-sl_1+l}f_0\|_{L^2_{x,v}}^2. 
\end{align*}
Together with \eqref{choice3}, we choose $l_0,l_1,l,\om_0,\varpi_0,\kappa\ge0$ such that \eqref{choiceM2},and \eqref{choideM3} hold, and
\begin{align}
	\label{choice6}
	\begin{cases}
	(-1-2s)l_1+l_0=\ga,\\
	\om_0^{-1-2s}\varpi_0+\om_0^{-2s} +o_{\de_1}(1)+o_{\de_2}(1) \le \frac{c_0}{8C},&\text{(for \eqref{es729} and \eqref{es733})},
	\end{cases}
\end{align}
where we used the assumption of soft potentials $\ga+2s<0$. 
Then we obtain the energy estimate 
%
\begin{align}\notag
	\label{es21}
	&\Big\|\sum_{r,j,k\ge 0}(\om2^k)^{-2s}2^{2rl}\|\De_jP_kR_rf\|_{L^2_{x,v}}^2\Big\|_{L^\infty_t([0,\infty))}
	\notag
	+
	\frac{c_0}{4}\sum_{r,j,k\ge 0}(\om2^k)^{-2s}2^{2rl}\|\De_jP_kR_rf\|_{L^{2}_{t}([0,\infty))L^{2}_{x}L^2_D}^2\\
	&\ \ \notag
	+A_0\sum_{r,j,k\ge 0}(\om2^k)^{-2s}2^{2rl}\|\De_jP_kR_rf\|_{L^{2}_{t}([0,\infty))L^2_{x,v}}^2
	\le C\om_0^{-2s}\|f_0\|_{L^2_{x,v}}^2
	+\om_0^{-2s}\|f\|_{L^2_t([T,\infty))L^2_{x,v}}^2
	\\
	&\notag\ \
	+
	C\big(
		M_0^2\om_0^{-2s{}}
	+M_0\om_0^{-s-1}\varpi_0^{\ve}
	+C_{\de_1,M_0}\varpi_0^{-2}
	\big)\|f\|_{L^2_{t}([0,T])L^2_{x}L^2_v}^2\\
	&\notag\ \ 
	+C_{\de_2,\om_0,\varpi_0,M_0}\sum_{r,j,k\ge 0}(\om2^k)^{-2s}2^{2rl}\|\De_jP_kR_rf\|_{L^2_t([0,T])L^2_xL^2_v}^2\\
	&\ \ 
	+
	o_{\de_2}(1)\sup_{r\ge 0}\sum_{j\ge 0}(j+1)^2\|\De_j\<D_v\>^{-s}\<v\>^{-N}f\|_{L^{r_0}_t([0,T])L^{r_0}_{x}L^2_v}^2,
\end{align}
for some $c_0=c_0(\ga,s,C_{\mathrm{low}},L_0)>0$, any $N>0$, and any small $\ve,\de>0$. Note $-sl_1+l\le 0$ as in \eqref{choideM3}.  

\subsubsection{The high velocity-frequency $P_k\De_j$ with $k\ge 1$ and $j\ge 0$}
For the part $k\ge 1$, using the commutator \eqref{dissUpper1b} and the lower bound of the dissipation in \eqref{dissipationLower}, we have 
\begin{align}\label{esRHS1}\notag
	&\sum_{r,j\ge0,\,k\ge 1}2^{2rl+r\ga}\|\De_jP_kR_rf\|^2_{L^2_t([0,\infty))L^2_{x,v}}\\
	&\notag\le C\sum_{r,j\ge0,\,k\ge 1}2^{2rl}\Big(\|\<v\>^{\fr\ga2}\De_jP_kR_rf\|^2_{L^2_t([0,\infty))L^2_{x,v}}
	+(\om2^k)^{-2}2^{2rl+r\ga-2r}\|\De_jR_rf\|^2_{L^2_t([0,\infty))L^2_{x,v}}\Big)\\
	&
	\le C\sum_{r,j\ge0,\,k\ge 1}(\om2^k)^{-2s}2^{2rl}\|\De_jP_kR_rf\|^2_{L^2_t([0,\infty))L^2_xL^2_D}
	+C\om_0^{-2}\|\<v\>^{-l_1+l+\frac{\ga}{2}-1}f\|^2_{L^2_t([0,\infty))L^2_{x,v}}. 
\end{align}
where we used $\|\<D_v\>^{s-1}\<v\>^{\frac{\ga}{2}-1}(\cdot)\|_{L^2}\le C\|\cdot\|_{L^2_D}$. 
We will also verify $\<v\>^{-l_1+l+\fr\ga2-1}\le\<v\>^{-sl_1+l}$ later.

\subsubsection{The low velocity-frequency $\De_jP_0$ with $k=0$ and $j\ge 1$}
For the $P_0$ part, we apply Theorem \ref{Thmlow} to equation \eqref{eqkinetic12}, i.e., $\pa_tf+v\cdot\na_xf = G$ in $\R^{1+2d}_{t,x,v}$ where $G$ is given by \eqref{gDefdifference}. This yields, for $j \ge 1$ and $k = 0$ (note that the time domain should be taken as $\mathbb{R}$),
\begin{align}\label{hypoes11}\notag
	\|\De_jP_0R_rf\|_{L^2_t(\R)L^2_{x,v}}^2
	&\le C_0^{-1}\|\De_jQ_0R_rf\|_{L^2_t(\R)L^2_{x,v}}^2
	+CC_0\om^{-2}\|\<v\>^{-1}\De_jR_rf\|_{L^2_t(\R)L^2_{x,v}}^2\notag\\
	&\notag\quad
	+CC_0\om^{-2s}2^{-r\ga}\|\<v\>^{\frac{\ga}{2}}\<D_v\>^{s}(P_0+P_1)\De_jR_rf\|_{L^2_t(\R)L^2_{x,v}}^2\\
	&\quad
	+CC_0\om^2(\varpi2^j)^{-2}\|\De_jQ_0R_rG\|_{L^2_t(\R)L^2_{x,v}}^2.
\end{align}
For the first right-hand term of \eqref{hypoes11}, applying $Q_0=Q_0(P_0+P_1)$ and \eqref{LowerSupportComm} (to $P_1$) yields
\begin{align}\label{es741}\notag
	C_0^{-1}\|\De_jQ_0R_rf\|_{L^2_t(\R)L^2_{x,v}}^2
	&\le C_0^{-1}\|\De_jP_0R_rf\|_{L^2_t(\R)L^2_{x,v}}^2
	\\
	&\notag\ \ 
	+CC_0^{-1}\om^{-2s}2^{-r\ga}\|\<v\>^{\frac{\ga}{2}}\<D_v\>^{s}(P_0+P_1)\De_jR_rf\|_{L^2_t(\R)L^2_{x,v}}^2\\
	&\ \ 
	+CC_0^{-1}\om^{-2}\|\<v\>^{-1}\De_jR_rf\|_{L^2_t(\R)L^2_{x,v}}^2,
\end{align}
where the first right-hand term can be absorbed into the left-hand side of \eqref{hypoes11} if $C_0>0$ is sufficiently large. The remaining $f$-terms in \eqref{es741} and \eqref{hypoes11} are lower-order and require no further treatment.

\smallskip 
To estimate the collision term $\|\De_jQ_0R_r\Ga(f,g)\|_{L^2_tL^2_xL^2_v}$ appearing in $G$ in \eqref{hypoes11}, we prepare the following. Applying $Q_0=Q_0(P_0+P_1)$, using $R_r=\wt{R_r}R_r,P_k=\wt{P_k} P_k$ as in \eqref{SlightlyLarger}, we split 
\begin{align}\label{splitTheWtf}\notag
	Q_0R_r\Ga(f,g)
	&\notag=\sum_{k=0,1}Q_0\wt{P_k}[P_k,\wt{R_r}]R_r\Ga(f,g)
	+\sum_{k=0,1}Q_0\wt{P_k}\wt{R_r}\Big(P_kR_r\Ga(f,g)-\Ga(f,P_kR_rg)\\
	&\qquad\qquad+\Ga(f,P_kR_rg)-Q(\mu^{\fr12}f,P_kR_rg)+Q(\mu^{\fr12}f,P_kR_rg)\Big),
\end{align}
and apply the duality arguments. 
For the commutator $[P_k,\wt{R_r}]$, using \eqref{symbolic} and collisional estimate \eqref{64eq}, we have 
\begin{align}\label{es751comm}\notag
	\|Q_0\wt{P_k}[P_k,\wt{R_r}]R_r\Ga(f,g)\|_{L^2_v}
	&\notag=\|\underbrace{Q_0\wt{P_k}}_{\in\operatorname{Op}(\<D_v\>^{-2s}(\om2^k)^{2s})}\underbrace{[P_k,\wt{R_r}]}_{\in\operatorname{Op}(2^{-r}(\om2^k)^{-1})}\underbrace{R_r}_{\in\operatorname{Op}(\<v\>^{\wt{l_0}}2^{-r\wt{l_0}})}\Ga(f,g)\|_{L^2_v}\\
	&\notag\le C(\om2^k)^{2s-1}2^{-r-r\wt{l_0}}\|\<v\>^{\wt{l_0}}\<D_v\>^{-2s}\Ga(f,g)\|_{L^2_v}\\
	&\le C(\om2^k)^{2s-1}2^{-r-r\wt{l_0}}\|f\|_{L^2_v}\|\<v\>^{\wt{l_0}+\ga+2s}g\|_{L^2_v}, 
\end{align}
for any $\wt{l_0}\in\R$. 
For the commutator part $P_kR_r\Ga(f,g)-\Ga(f,P_kR_rg)$, by the commutator estimate \eqref{ThmCommeswt}, for $k=0,1$ and any $\wt{l_1}\in\R$, we have	
\begin{align}\notag\label{es751commColli}
	&\|Q_0\wt{P_k}\wt{R_r}(P_kR_r\Ga(f,g)-\Ga(f,P_kR_rg))\|_{L^2_v}
	\\
	&\notag=\sup_{\|h\|_{L^2_v}\le 1}|(P_kR_r\Ga(f,g)-\Ga(f,P_kR_rg),\wt{R_r}\wt{P_k}Q_0h)_{L^2_v}|\\
	&\notag\le C\sup_{\|h\|_{L^2_v}\le 1}\Big(2^{r\wt{l_1}}\|f\|_{L^2_v}\|\<v\>^{-\wt{l_1}+\ga+2s}g\|_{L^2_v}\|\<D_v\>^s\wt{R_r}\wt{P_k}Q_0h\|_{L^2_v}\\
	&\notag\qquad\qquad\qquad+C(\om2^k)^{s-1}\|f\|_{L^2_{v}}\|\<v\>^{-\wt{l_1}+\ga+2s-1}g\|_{L^2_v}\|\wt{R_r}Q_0h\|_{L^2_v}\Big)\\
	&\le C2^{r\wt{l_1}}\om^{s}\|f\|_{L^2_v}\|\<v\>^{-\wt{l_1}+\ga+2s}g\|_{L^2_v}.
\end{align}
For the commutator $\Ga(f,P_kR_rg)-Q(\mu^{\fr12}f,P_kR_rg)$, applying \eqref{CommGaAndQ}, we have, for any $\wt{l_2}\in\R$,
\begin{align}\label{es751commGaQ}
	\|Q_0\wt{P_k}\wt{R_r}(\Ga(f,P_kR_rg)-Q(\mu^{\fr12}f,P_kR_rg))\|_{L^2_v}
	&\le C2^{r\wt{l_2}}\om^{s}\|f\|_{L^2_v}\|\<v\>^{-\wt{l_2}+\ga+2s}g\|_{L^2_v}.
\end{align}
The commutator estimates \eqref{es751commColli} and \eqref{es751commGaQ} admit the same upper bound, which contributes only to the lower-order terms. For the leading term $Q(\mu^{\frac{1}{2}}f, P_k R_r g)$, using \eqref{QfghUpper}, when $\ga+s>-\frac{d}{2}$, for $k=0,1$ and any $N \ge 0$,
\begin{align}\label{es751leading}\notag
	&\|Q_0\wt{P_k}\wt{R_r}Q(\mu^{\fr12}f,P_kR_rg)\|_{L^2_v}
	=\sup_{\|h\|_{L^2_v}\le 1}|(Q(\mu^{\fr12}f,P_kR_rg),\wt{R_r}\wt{P_k}Q_0h)_{L^2_v}|\\
	&\notag\le C\sup_{\|h\|_{L^2_v}\le 1}\|\<D_v\>^{-s}\<v\>^{-N}f\|_{L^2_v}
		\|P_kR_rg\|_{L^2_D}\|\wt{R_r}\wt{P_k}Q_0h\|_{L^2_D}\\
	&\le C2^{r(\frac{\ga+2s}{2})}\om^s\|\<D_v\>^{-s}\<v\>^{-N}f\|_{L^2_v}\|P_kR_rg\|_{L^2_D}. 
\end{align}
We now handle each term in $G$ appearing in the last term $\|\De_jQ_0R_rG\|_{L^2}$ in \eqref{hypoes11}, with $G$ defined in \eqref{gDefdifference}.

\smallskip 
{\bf The linear term $\Ga(\mu^{\frac{1}{2}},f)$.}
Applying $\De_j\Ga(\mu^{\fr12},f)=\Ga(\mu^{\fr12},\De_jf)$, the estiamtes \eqref{es751comm}--\eqref{es751leading} (with suitable constants $\wt{l_i}$), using commutator estimate \eqref{dissUpper1b}, and recalling $\om=\om_02^{\al j+rl_1}$ and $\varpi=\varpi_02^{rl_0}$ with $\al=\frac{1}{1+2s}$, we have 
\begin{align}\label{es748}\notag
	&\sum_{j\ge 1}\om^2(\varpi2^j)^{-2}\|\De_jQ_0R_r\Ga(\mu^{\frac{1}{2}},f)\|_{L^2_t(\R)L^2_{x,v}}^2\\
	&\notag\le C\sum_{j\ge 1,\,
	k=0,1}\om^{2+2s}(\varpi2^j)^{-2}
	2^{2r(\frac{\ga+2s}{2})}\|\De_jP_kR_rf\|_{L^2_t(\R)L^2_xL^2_D}^2\\
	&\notag\ \ 
	+C\sum_{j\ge 1}\om^{2+2s}(\varpi2^j)^{-2}2^{2r(\frac{\ga+2s}{2})}2^{-2r\wt{l_0}}
		\|\<v\>^{\wt{l_0}}\De_jf\|_{L^2_t(\R)L^2_{x,v}}^2
	\\
	&\notag\le C\om_0^{2+4s}\varpi_0^{-2}\sum_{j,k\ge0}(\om2^k)^{-2s}2^{r(-\ga+2s)}\|\De_jP_kR_rf\|_{L^2_t(\R)L^2_xL^2_D}^2\\
	&\quad
	+C\om_0^{2+2s}\varpi_0^{-2}2^{-2rsl_1
		+r(-\ga+2s)-2r\wt{l_0}}\|
	\<v\>^{\wt{l_0}}
	f\|_{L^2_t(\R)L^2_{x,v}}^2.
\end{align} 
for any $\wt{l_0}\in\R$, 
where we applied $l_0=\ga+(1+2s)l_1$ in \eqref{choice6}.

\smallskip{\bf The linear term $A_0f$ on $t\in(-\infty,0)\cup(T,\infty)$.}
Using $l_0=\ga+(1+2s)l_1$, it's direct to obtain 
\begin{align}\label{es748a}\notag
	&\sum_{j\ge 1}\om^2(\varpi2^j)^{-2}\|\De_jQ_0R_rA_0f\|_{L^2_t((-\infty,0)\cup(T,\infty))L^2_{x,v}}^2\\
	&\notag\le 
	\om_0^{2+2s}\varpi_0^{-2}\sum_{j\ge 1}\om^{-2s}2^{(2+2s)\al j-2j}2^{r(2+2s)l_1-2rl_0}\|\De_jQ_0R_rA_0f\|_{L^2_t((-\infty,0)\cup(T,\infty))L^2_{x,v}}^2\\
	&\le A_0^2\om_0^{2+2s}\varpi_0^{-2}2^{-2rsl_1-2r\ga}
	\sum_{j\ge 1}\om^{-2s}\|\De_j(P_0+P_1)R_rf\|^2_{L^2_t((-\infty,0)\cup(T,\infty))L^2_{x,v}}.
\end{align}

{\bf The linear (regular) term $\Ga(f,\mu^{\frac{1}{2}})$ on $t\in[0,T]$}. By \eqref{esKf} and $l_0=\ga+(1+2s)l_1$, we have 
\begin{align}\label{es748b}
	\notag\sum_{j\ge 1}\om^2(\varpi2^j)^{-2}\|\De_jQ_0R_r\Ga(f,\mu^{\frac{1}{2}})\|_{L^2_t([0,T])L^2_{x,v}}^2
	&\le \om_0^2\varpi_0^{-2}
	2^{2rl_1-2rl_0}\|\mu^{\frac{1}{200}}f\|_{L^2_t([0,T])L^2_{x,v}}^2
	\notag\\
	&\le \om_0^2\varpi_0^{-2}
	2^{-4rsl_1-2r\ga}\|\mu^{\frac{1}{200}}f\|_{L^2_t([0,T])L^2_{x,v}}^2. 
\end{align}

\noindent{\bf The nonlinear term $\Ga(\phi_1,f)$}. We take $[0,T]$ as the underlying time interval and split it as in \eqref{splitTheWtf}.
For the commutator part in \eqref{splitTheWtf}, i.e., $\De_jQ_0R_r\Ga(\wt\phi_{1,\de_1},f)-\De_jQ_0\wt{P_k}\wt{R_r}Q(\mu^{\fr12}\wt\phi_{1,\de_1},(P_0+P_1)R_rf)$, we first apply the Bernstein-type inequality as in \eqref{SobolevEmbeddLp} to obtain  
\begin{align*}
	\|\De_ju\|_{L^2_{x}L^2_v}^2
	&\le C(\varpi2^{j})^{2\ve}\|\De_ju\|_{L^{r_1}_{x}L^2_v}^2,\quad 
	1+\frac{1}{2}=\frac{1}{r_1}+\frac{1}{(1-\ve/d)^{-1}},
\end{align*}
for some $r_1\in[1,2)$ and $\ve\in(0,d)$, and any suitable function $u$. 
Moreover, since $\al=\frac{1}{1+2s}$, the corresponding coefficient satisfies $\om^{2+2s}(\varpi2^j)^{-2+2\ve}=\om_0^{2+2s}\varpi_0^{-2+2\ve}2^{r(2+2s)l_1+r(-2+2\ve)l_0}2^{\frac{-2sj}{1+2s}+2\ve j}$, which leads to 
\begin{align*}
	\Big(\sum_{j\ge 0}2^{\frac{-2sj}{1+2s}+2\ve j}\|\De_ju\|_{L^{r_1}_{x}L^2_v}^2\Big)^{\frac{1}{2}}
	&\le \Big\|\Big(\sum_{j\ge 0}2^{\frac{-2sj}{1+2s}+2\ve j}\|\De_ju\|_{L^2_v}^2\Big)^{\frac{1}{2}}\Big\|_{L^{r_1}_{x}}\le C\|u\|_{L^{r_1}_{x}L^2_v}. 
\end{align*}
provided that $\ve<\frac{s}{1+2s}$, where the last inequality is simply $F^{0}_2[L^{r_1}]=H^{0,r_1}$ as in Subsection \ref{SubsecTL_Besov}. 
Thus, utilizing the decomposition \eqref{splitTheWtf}, applying the commutator estimates \eqref{es751comm}--\eqref{es751commGaQ} (with suitable $\wt{l_0},\wt{l_1},\wt{l_2}$ therein), using H\"older's inequality in $x$ with index $\frac{1}{r_1}=\frac{1}{2}+\frac{1}{r_1^*}$ where $r_1^*=\frac{d}{\ve}$, and utilizing $l_0=\ga+(1+2s)l_1$, we have
\begin{align}\label{es779}\notag
	&\sum_{j\ge 1}\om^2(\varpi2^j)^{-2}\|\De_jQ_0R_r\Ga(\phi_{1},f)-\sum_{k=0,1}\De_jQ_0\wt{P_k}\wt{R_r}Q(\mu^{\fr12}\phi_{1},(P_0+P_1)R_rf)\|_{L^2_t([0,T])L^2_{x,v}}^2\\
	&\notag\le C\om_0^{2+2s}\varpi_0^{-2+2\ve}2^{r(2+2s)l_1+r(-2+2\ve)l_0+r(\ga+2s)-r\wt{l_0}}
	\Big\|\|\phi_{1}\|_{L^2_v}
	\|\<v\>^{\wt{l_0}}f\|_{L^2_v}\Big\|_{L^2_tL^{r_1}_x}^2\\
	&\notag\le C\om_0^{2+2s}\varpi_0^{-2+2\ve}2^{r(2+2s)l_1+r(-2+2\ve)l_0+r(\ga+2s)-r\wt{l_0}}
	\|\phi_{1}\|^2_{L^\infty_tL^{r_1^*}_xL^2_v}\|\<v\>^{\wt{l_0}}f\|^2_{L^2_{t,x,v}}\\
	&\le CM_0^2\om_0^{2+2s}\varpi_0^{-2+2\ve}
	2^{r(-2s)l_1+r(2\ve)l_0+r(-\ga+2s)
	-r\wt{l_0}
	}
	\|
	\<v\>^{\wt{l_0}}
	f\|^2_{L^2_{t}([0,T])L^2_{x,v}},
\end{align}
for any $\wt{l_0}\in\R$, 
where we used $\|\phi_{1}\|^2_{L^\infty_t([0,T])L^{r_1^*}_xL^2_v}\le M_0^2$ as in the \emph{a priori} assumption \eqref{apriori2}.
{On the other hand}, for the leading term
\begin{align*}
	\sum_{j\ge 1,\,k=0,1}\om^2(\varpi2^j)^{-2}\|\De_jQ_0\wt{P_k}\wt{R_r}Q(\mu^{\fr12}\phi_{1},(P_0+P_1)R_rf)\|_{L^2_t([0,T])L^2_{x,v}}^2, 
\end{align*}
we split $\phi_1=\phi_{1,\de_3}+\wt\phi_{1,\de_3}$ as in \eqref{splitphi}. We emphasize that we use a \textbf{different convergence rate} $\de_3$ from $\de_1$ in \eqref{es21}. 
For the regular part $\phi_{1,\de_3}$ of the leading term, we apply the commutator estimate \eqref{CommDejfTemp} and upper bound \eqref{QfghUpper} to deduce 
\begin{align*}\notag
	&\|\De_jQ_0\wt{P_k}\wt{R_r}Q(\mu^{\fr12}\phi_{1,\de_3},(P_0+P_1)R_rf)\|_{L^2_{x,v}}\\
	&\notag\le \sum_{k=0,1}\sup_{\|h\|_{L^2_{x,v}}\le 1}
	\Big(\De_jQ(\mu^{\fr12}\phi_{1,\de_3},P_kR_rf)
	-Q(\mu^{\fr12}\phi_{1,\de_3},\De_jP_kR_rf)
	+Q(\mu^{\fr12}\phi_{1,\de_3},\De_jP_kR_rf)
	,Q_0\wt{R_r}h\Big)_{L^2_{x,v}}\\
	&\notag\le C\sum_{k=0,1}\sup_{\|h\|_{L^2_{x,v}}\le 1}\Big(\varpi^{-1}2^{-j}
	\|\<D_x\>^{\frac{d+3}{2}}\mu^{\fr12}\phi_{1,\de_3}\|_{L^2_{x,v}}\|\De_jP_kR_rf\|_{L^2_xL^2_D}\|Q_0\wt{R_r}h\|_{L^2_xL^2_D}\\
	&\notag\qquad\qquad\qquad\qquad+\|\<D_v\>^{-s}\<v\>^{-N}\phi_{1,\de_3}\|_{L^\infty_xL^2_v}\|\De_jP_kR_rf\|_{L^2_xL^2_D}\|Q_0\wt{R_r}h\|_{L^2_xL^2_D}\Big)\\
	&\le \Big(C_{\de_3,M_0}\varpi^{-1}2^{-j}+C_{M_0}\Big)\om^{s}2^{r(\fr{\ga+2s}{2})}\sum_{k=0,1}\|\De_jP_kR_rf\|_{L^2_xL^2_D}, 
\end{align*}
for any $N>0$, provided that $\ga+s>-\fr{d}{2}$, 
where we used $\|\<D_x\>^{\frac{d+3}{2}}\mu^{\fr12}\phi_{1,\de_3}\|_{L^\infty_tL^2_{x,v}}\le C_{\de_3,M_0}$ and $\|\<D_v\>^{-s}\<v\>^{-N}\phi_{1,\de_3}\|_{L^\infty_{t,x}L^2_v}\le C_{M_0}$ as in the \emph{a priori} assumption \eqref{apriori2}. Therefore, by $\varpi=\varpi_02^{rl_0}$ and $l_0=\ga+(1+2s)l_1\ge 0$, we have 
\begin{align}\label{es777regular}\notag
	&\sum_{j\ge 1}\om^2(\varpi2^j)^{-2}2^{2rl}\|\De_jQ_0\wt{P_k}\wt{R_r}Q(\mu^{\fr12}\phi_{1,\de_3},(P_0+P_1)R_rf)\|_{L^2_{x,v}}^2\\
	&\notag \le \Big(C_{\de_3,M_0}\varpi_0^{-1}+C_{M_0}\Big)\sum_{j\ge 1,\,k=0,1}\om^{2+2s}\varpi^{-2}2^{-2j+r(\ga+2s)+2rl}\|\De_jP_kR_rf\|_{L^2_xL^2_D}^2\\
	& \le \Big(C_{\de_3,M_0}\varpi_0^{-1}+C_{M_0}\Big)\om_0^{2+4s}\varpi_0^{-2}\sum_{j\ge 1,\,k=0,1}(\om2^k)^{-2s}2^{2r(-\ga+2s)+2rl}\|\De_jP_kR_rf\|_{L^2_xL^2_D}^2, 
\end{align}
For the irregular small part $\wt\phi_{1,\de_3}$ of the leading term,
we apply the Bony decomposition \eqref{eqBony} and the estimate \eqref{es751leading} to obtain that, for $k=0,1$, 
\begin{align}\label{es751}\notag
	&\sum_{j\ge 1,\,k=0,1}\om(j)^2(\varpi2^j)^{-2}\|\De_jQ_0(\wt{P_0}+\wt{P_1})\wt{R_r}Q(\mu^{\fr12}\wt\phi_{1,\de_3},P^j_kR_rf)\|_{L^2_t([0,T])L^2_{x,v}}^2
	\\
	&\notag\le \om_0^{2+4s}\varpi_0^{-2}\sum_{j\ge 1,\,k=0,1}(\om(j)2^k)^{-2s}2^{r(-\ga+2s)}\\
	&\ \notag\ \times 
	\bigg\{\sum_{|q-j|\le 4}\int_{[0,T]\times\R^d_x}
	\|S_{q-1}\<D_v\>^{-s}\<v\>^{-N}\wt\phi_{1,\de_3}\|_{L^2_v}^2\|\De_{q}P^j_kR_rf\|_{L^2_D}^2\\
	&\notag\quad\notag+
	\sum_{|p-j|\le 4}\int_{[0,T]\times\R^d_x}
	\|\De_q\<D_v\>^{-s}\<v\>^{-N}\wt\phi_{1,\de_3}\|_{L^2_v}^2\|S_{q-1}P^j_kR_rf\|_{L^2_D}^2\\
	&\quad+
	\sum_{p,q\ge 0}\1_{\substack{|p-q|\le 1,\qquad\\\max\{p,q\}\ge j-2}}
	\int_{[0,T]\times\R^d_x}\|\De_p\<D_v\>^{-s}\<v\>^{-N}\wt\phi_{1,\de_3}\|_{L^2_v}^2\|\De_{q}P^j_kR_rf\|_{L^2_D}^2\bigg\}
	\notag
\\
	&
	=:\om_0^{2+4s}\varpi_0^{-2}\sum_{i=1}^3\ol{\Ga}_i(\wt\phi_{1,\de_3},f),
\end{align}
where we compute the coefficient by using $\om=\om_02^{\al j+rl_1}$, $\varpi=\varpi_02^{rl_0}$, and $\al=\fr{1}{1+2s}$ with $l_0=\ga+(1+2s)l_1$ in \eqref{choice3}: $(\om(j))^{2+4s}(\varpi2^j)^{-2}=\om_0^{2+4s}\varpi_0^{-2}2^{-2r\ga}$. 
Here, we again rewrite $\om=\om(j)$ and $P_k=P^j_k$ to emphasize the dependence on $j$ in the Bony decomposition. 
The right-hand terms can be calculated as in the Bony-decomposition-type estimate of $\Ga_i(\wt\phi_{1,\de_3},f,f)$ $(1\le i\le 3)$ in Subsections \ref{SubsecBasicBony} and \ref{SecGammai}.

\smallskip 
For $\ol{\Ga}_1(\phi_1,f)$, applying the H\"older's inequality, \eqref{esSq}, $\sum_{q\ge 0}\1_{|q-j|\le 4}|v_q|\le C\big(\sum_{q\ge 0}\1_{|q-j|\le 4}|v_q|^2\big)^{\fr12}$, and estimate \eqref{esg719s}, we have 
\begin{align*}
	&\ol\Ga_1(\phi_1,f)
	\notag\le C\|\<D_v\>^{-s}\<v\>^{-N}\phi_1\|_{L^\infty_{t,x}L^2_v}^2\sum_{j,k,q\ge 0}\1_{|q-j|\le 4}(\om(j)2^k)^{-2s}\|\De_{q}P^j_kR_rf\|^2_{L^2_{t,x}L^2_D}\\
	&\le C\|\<D_v\>^{-s}\<v\>^{-N}\phi_1\|_{L^\infty_{t}([0,T])L^\infty_{x}L^2_v}^2\sum_{j,k\ge 0}(\om(j)2^k)^{-2s}\|\De_jP^j_kR_rf\|^2_{L^2_{t}([0,T])L^2_{x}L^2_D}.
\end{align*}
For $\ol{\Ga}_2(\phi_1,f)$, by splitting $S_{p-1}=\sum_{q\le p-2}\De_q$, using $\sum_{j,p,q\ge 0}\1_{|p-j|\le 4}\1_{q\le p-2}|u_p|^2\le C\sum_{p\ge 0}(p+1)|u_p|^2$, and applying \eqref{esg719}, we have 
\begin{align*}
	&\ol\Ga_2(\phi_1,f)
	\notag\le\sum_{j,k,p,q\ge 0}(\om(j)2^k)^{-2s}\1_{|p-j|\le 4}\1_{q\le p-2}\|\De_{p}\<D_v\>^{-s}\<v\>^{-N}\phi_1\|_{L^\infty_{t,x}L^2_v}^2\|\De_{q}P^j_kR_rf\|_{L^2_{t,x}L^2_D}^2\\
	&\notag\le 
	\sum_{j,k\ge 0}\Big\{\sum_{p,q\ge 0}\1_{|p-j|\le 4}\1_{q\le p-2}\|\De_{p}\<D_v\>^{-s}\<v\>^{-N}\phi_1\|_{L^\infty_{t,x}L^2_v}^2\Big\}\\
	&\notag\quad\quad\times \Big\{\sum_{p,q\ge 0}\1_{|p-j|\le 4}\1_{q\le j+2}(\om(j)2^k)^{-2s}\|\De_{q}P^j_kR_rf\|^2_{L^2_{t,x}L^2_D}\Big\}\\
	&\le 
	C\sum_{p\ge 0}(p+1)\|\De_{p}\<D_v\>^{-s}\<v\>^{-N}\phi_1\|_{L^\infty_{t}([0,T])L^\infty_{x}L^2_v}^2
	\sum_{j,k\ge 0}(\om(j)2^k)^{-2s}\|\De_jP^j_kR_rf\|_{L^2_{t}([0,T])L^2_{x}L^2_D}^2. 
\end{align*}
For the term $\ol\Ga_3(\phi_1,f)$, using \eqref{esg719s} (for $j-3\le q\le j+4$), and \eqref{ges1a} (for $q\ge j+5$) that 
\begin{align*}
	&\ol\Ga_3(\phi_1,f)
	\le\int_{[0,T]}\sum_{j,k\ge 0}\Big\{\sum_{p,q\ge 0}
	(p+rl_0+\log_2\vpi_0+1)^2\1_{q\ge j+5}
	\1_{\substack{|p-q|\le 1,\qquad\\\max\{p,q\}\ge j-2}}\|\De_p\<D_v\>^{-s}\<v\>^{-N}\phi_1\|^2_{L^\infty_{x}L^2_v}
	\\
	&\notag\qquad\qquad\qquad\times
		\sum_{p,q\ge 0}
	(p+rl_0+\log_2\vpi_0+1)^{-2}\1_{q\ge j+5}
	\1_{\substack{|p-q|\le 1,\qquad\\\max\{p,q\}\ge j-2}}(\om(j)2^k)^{-2s}\|\De_qP^j_kR_rf\|_{L^2_{x}L^2_D}^2
		\\
	&\qquad\qquad\quad
	+
		\sum_{p,q\ge 0}
	\1_{|q-j|\le 4}
	\1_{\substack{|p-q|\le 1,\qquad\\\max\{p,q\}\ge j-2}}\|\De_p\<D_v\>^{-s}\<v\>^{-N}\phi_1\|^2_{L^\infty_{x}L^2_v}
	\\
	&\notag\qquad\qquad\qquad\times\sum_{p,q\ge 0}
	\1_{|q-j|\le 4}
	\1_{\substack{|p-q|\le 1,\qquad\\\max\{p,q\}\ge j-2}}(\om(j)2^k)^{-2s}\|\De_qP^j_kR_rf\|_{L^2_{x}L^2_D}^2\Big\}
	\\
	&\le\notag
	C_\al 
		\sum_{p\ge 0}(p+rl_0+\log_2\vpi_0+1)^2\|\De_p\<D_v\>^{-s}\<v\>^{-N}\phi_1\|^2_{L^\infty_{x}L^2_v}
	\\
	&\ \ \times
	\Big(
		C_{\om_0,l_0,l_1}\sum_{q\ge 1}\om^{-2s}\|\De_{q}P^q_0R_rf\|^2_{L^2_{x,v}}
	+\om_0^{-2s{}}\|R_rf\|_{L^2_{x,v}}^2
	+
	\sum_{q,k\ge 0}(\om(q)2^{k})^{-2s}
	\|\De_{q}P^q_{k}R_rf\|^2_{L^2_{x}L^2_D}\Big). 
\end{align*}
Summing the above cases into \eqref{es751} and substituting $\phi_1=\wt\phi_{1,\de_3}$, since noting that $\om^{2+4s}(\varpi2^j)^{-2}=\om_0^{2+4s}\varpi_0^{-2}2^{r(2+4s)l_1-2rl_0}$, 
and using the commutator estimate \eqref{dissUpper1b}, 
\begin{align}\label{es777}\notag
	&\sum_{j\ge 1}\om^2(\varpi2^j)^{-2}\|\De_jQ_0\Ga(\wt\phi_{1,\de_3},(P^j_0+P^j_1)R_rf)\|_{L^2_t([0,T])L^2_{x,v}}^2\\
	&\notag\le \om_0^{2+4s}\varpi_0^{-2}2^{r(2+4s)l_1-2rl_0+r(\ga+2s)}\\
	&\notag\ \times 
	\bigg\{
		C\|\<D_v\>^{-s}\<v\>^{-N}\wt\phi_{1,\de_3}\|_{L^\infty_{t,x}L^2_v}^2\sum_{j,k\ge 0}(\om(j)2^k)^{-2s}\|\De_jP^j_kR_rf\|^2_{L^2_{t,x}L^2_D}\\
		&\notag\quad+C\sum_{p\ge 0}(p+1)\|\De_{p}\<D_v\>^{-s}\<v\>^{-N}\wt\phi_{1,\de_3}\|_{L^\infty_{t,x}L^2_v}^2
	\sum_{j,k\ge 0}(\om(j)2^k)^{-2s}\|\De_jP^j_kR_rf\|_{L^2_{t,x}L^2_D}^2\\
	&\notag\quad+C\sum_{j\ge 0}(j+rl_0+\log_2\vpi_0+1)^2\|\De_j\<D_v\>^{-s}\<v\>^{-N}\wt\phi_{1,\de_3}\|^2_{L^\infty_{t,x}L^2_v}\\
	&\notag\qquad\times
	\Big(
	C_{\om_0,l_0,l_1}\sum_{q\ge 1}\om^{-2s}\|\De_{q}P^q_0R_rf\|^2_{L^2_tL^2_{x,v}}
	+\om_0^{-2s{}}\|R_rf\|_{L^2_tL^2_{x,v}}^2\\
	&\notag\qquad\quad
	+\sum_{q,k\ge 0}(\om(q)2^{k})^{-2s}\|\De_{q}P^q_{k}R_rf\|^2_{L^2_tL^2_{x}L^2_D}
	\Big)
	\bigg\}\\
	&\notag\le Co_{\de_3}(1)\om_0^{2+4s}\varpi_0^{-2}
	2^{r(-\ga+2s)}
	\Big\{\sum_{j,k\ge 0}(\om(j)2^k)^{-2s}\|\De_jP^j_kR_rf\|^2_{L^2_{t}([0,T])L^2_{x}L^2_D}
	\\
	&\qquad
	+\om_0^{-2s{}}\|R_rf\|_{L^2_t([0,T])L^2_{x,v}}^2
	+
	C_{\om_0,l_0,l_1}\sum_{j\ge 1}\om^{-2s}\|\De_jP^j_0R_rf\|^2_{L^2_t([0,T])L^2_{x,v}}\Big\}.
\end{align}
where we have used $l_0=\ga+(1+2s)l_1$ and the \emph{a priori} assumption \eqref{apriori2} to control $\wt\phi_{1,\de_3}$.  

\smallskip 
{\bf The nonlinear term $\Ga(f,\phi_2)$ for $t\in[0,T]$.} 
Similarly, using \eqref{splitTheWtf}, we can split 
\begin{align*}
	\De_jQ_0R_r\Ga(f,\phi_2)
	&=\big(\De_jQ_0R_r\Ga(f,\phi_2)
	-\De_jQ_0\wt{P_k}\wt{R_r}Q(\mu^{\fr12}f,(P_0+P_1)R_r\phi_2)\big)\\
	&\ \ +\De_jQ_0\wt{P_k}\wt{R_r}Q(\mu^{\fr12}f,(P_0+P_1)R_r\phi_2).
\end{align*}
The \emph{commutator part} can be estimated as in \eqref{es779} (with suitable $\wt{l_0}$ and H\"older's inequality therein):
\begin{align}
	\label{es780zgg}\notag
	&\sum_{r\ge 0,\,j\ge 1}\om^2(\varpi2^j)^{-2}2^{2rl}\|\De_jQ_0R_r\Ga(f,\phi_2)-\De_jQ_0\wt{P_k}\wt{R_r}Q(\mu^{\fr12}f,(P_0+P_1)R_r\phi_2)\|_{L^2_t([0,T])L^2_{x,v}}^2\\
	&\notag\le C\om_0^{2+2s}\varpi_0^{-2+2\ve}
	\|f\|_{L^2_{t}([0,T])L^2_{x}L^2_v}^2
	\|\<v\>^{l+(-2s+2\ve(1+2s))l_1+(-\ga+2\ve\ga+2s)}\phi_2\|^2_{L^\infty_{t}([0,T])L^{r_1^*}_xL^2_{v}}\\
	&\le CM_0^2\om_0^{2+2s}\varpi_0^{-2+2\ve}
	\|f\|_{L^2_{t}([0,T])L^2_{x}L^2_v}^2. 
\end{align}
For the \emph{leading part}, by splitting $\phi_2=\phi_{2,\de_4}+\wt\phi_{2,\de_4}$ as in \eqref{splitphi} (again we use a \textbf{different convergence rate} $\de_4$ from $\de_2$ in \eqref{es21}), the regular part $\phi_{2,\de_4}$ can be estimated in the $H^{2s}$ norm as in \eqref{esGa6}, while the irregular small part $\Ga_i(f,\wt\phi_{2,\de_4},f)$ can be estimated as in Subsection \ref{SecGammai}. 
That is, for the regular part $\phi_{2,\de_4}$, applying \eqref{64eq} and placing the $2s$-derivative and the redundant weight $\<v\>^{C}$ on $\wt\phi_{2,\de_4}$ yields
\begin{align}\label{es780zz}\notag
	&\sum_{r\ge0,\,j\ge 1}(\om2^k)^{-2s}2^{2rl}\|\De_jQ_0\wt{P_k}\wt{R_r}Q(\mu^{\fr12}f,(P_0+P_1)R_r\phi_2)\|_{L^2_t([0,T])L^2_{x,v}}\\
	&\notag\ \le 
	C\sum_{r\ge0,\,j\ge 1}(\om2^k)^{-2s}2^{2rl}\|f\|_{L^2_{t,x,v}}^2\|\<v\>^{\ga+2s}(P_0+P_1)R_r\phi_{2,\de_4}\|_{L^\infty_{t,x}H^{2s}_v}^2\\
	&\ \le C_{\de_4,M_0}\om_0^{-2s}\|f\|_{L^2_t([0,T])L^2_{x,v}}^2. 
\end{align}
For the \emph{leading part} of irregular small part $\wt\phi_{2,\de_4}$, similar to the Bony-type estimate in \eqref{es751}, we can apply \eqref{es751leading} and the Bony-decomposition-type estimate of $\Ga_i(f,\wt\phi_{2,\de_4},f)$ $(1\le i\le 3)$ in Subsection \ref{SecGammai}. That is, 
\begin{align*}
	&\sum_{j\ge 1}\om^2(\varpi2^j)^{-2}2^{2rl}\|\De_jQ_0\wt{P_k}\wt{R_r}Q(\mu^{\fr12}f,(P_0+P_1)R_r\wt\phi_{2,\de_4})\|_{L^2_t([0,T])L^2_{x,v}}^2\\
	&\notag\le \om_0^{2+4s}\varpi_0^{-2}2^{r(-\ga+2s)+2rl}\sum_{j\ge 1,\,k=0,1}(\om(j)2^k)^{-2s}\\
	&\ \notag\ \times 
	\bigg\{\sum_{|q-j|\le 4}\int_{[0,T]\times\R^d_x}
	\|S_{q-1}\<D_v\>^{-s}\<v\>^{-N}f\|_{L^2_v}^2\|\De_{q}P^j_kR_r\wt\phi_{2,\de_4}\|_{L^2_D}^2\\
	&\notag\quad+
	\sum_{|p-j|\le 4}\int_{[0,T]\times\R^d_x}
	\|\De_q\<D_v\>^{-s}\<v\>^{-N}f\|_{L^2_v}^2\|S_{q-1}P^j_kR_r\wt\phi_{2,\de_4}\|_{L^2_D}^2\\
	&\quad+
	\sum_{p,q\ge 0}\1_{\substack{|p-q|\le 1,\qquad\\\max\{p,q\}\ge j-2}}
	\int_{[0,T]\times\R^d_x}\|\De_p\<D_v\>^{-s}\<v\>^{-N}f\|_{L^2_v}^2\|\De_{q}P^j_kR_r\wt\phi_{2,\de_4}\|_{L^2_D}^2\bigg\}\\
	&
	=:\om_0^{2+4s}\varpi_0^{-2}2^{r(-\ga+2s)+2rl}\sum_{i=1}^3\mathring\Ga_i(f,\wt\phi_{2,\de_4}).
\end{align*}
where we compute the coefficient as in \eqref{es751}, i.e., $(\om(j))^{2+4s}(\varpi2^j)^{-2}=\om_0^{2+4s}\varpi_0^{-2}2^{-2r\ga}$.
For $\mathring\Ga_1(f,\phi_2)$, applying H\"older's inequality with $\fr{1}{r_0}+\fr{1}{r_0^*}=\fr12$, splitting $S_{q-1}=\sum_{p\le q-2}\De_p$ with $\sum_{p\ge 0}|v_p|\le \big(\sum_{p\ge 0}(p+1)^2|v_p|^2\big)^{\frac{1}{2}}$, using estimates \eqref{esphi2}, \eqref{equivTriebelLizonkin}, we have 
\begin{align*}
	&\mathring\Ga_1(f,\phi_2)\le C\int_{t,x}
		\sum_{j,q\ge 0}\1_{|q-j|\le 4}
	\|S_{q-1}\<D_v\>^{-s}\<v\>^{-N}f\|_{L^2_v}^2
	\|\<v\>^{\frac{\ga+2s}{2}}\De_{q}R_r\phi_2\|_{L^2_v}^2
	\\
	&\le C\Big\|
	\Big(\sum_{p\ge 0}(p+1)^2\|\De_p\<D_v\>^{-s}\<v\>^{-N}f\|_{L^2_v}^2\Big)^{\frac{1}{2}}\Big\|_{L^{r_0}_{t,x}}^2
	\Big\|\Big(\sum_{j,q\ge 0}\1_{|q-j|\le 4}\|\<v\>^{\frac{\ga+2s}{2}}\De_{q}R_r\phi_2\|_{L^2_v}^2\Big)^{\frac{1}{2}}\Big\|_{L^{r_0^*}_{t,x}}^2
	\\
	&\le C
	\sum_{p\ge 0}(p+1)^2\|\De_p\<D_v\>^{-s}\<v\>^{-N}f\|_{L^{r_0}_{t,x}L^2_v}^2
	\|\<v\>^{\frac{\ga+2s}{2}}R_r\phi_2\|_{L^{r_0^*}_{t,x}L^2_v}^2.
\end{align*}
For $\mathring\Ga_2(f,\phi_2)$, applying estimate \eqref{esphi2}, the H\"older's inequality, and splitting $S_{q-1}=\sum_{p\le q-2}\De_p$ with $\sum_{q\ge 0}\1_{q\le p-2}|v_q|\le (p+1)^{\frac{1}{2}}\big(\sum_{q\ge 0}|v_q|^2\big)^{\frac{1}{2}}$ and estimate \eqref{esSq}, we have 
\begin{align*}
	\mathring\Ga_2(f,\phi_2)
	&\le C
	\sum_{p\ge 0}
	\|\De_{p}\<D_v\>^{-s}\<v\>^{-N}f\|_{L^{r_0}_{t,x}L^2_v}^2\|\<v\>^{\frac{\ga+2s}{2}}S_{p-1}R_r\phi_2\|_{L^{r_0^*}_{t,x}L^2_v}^2\\
	&\le C
	\sum_{p\ge 0}
	(p+1)^2\|\De_{p}\<D_v\>^{-s}\<v\>^{-N}f\|_{L^{r_0}_{t,x}L^2_v}^2\|\<v\>^{\frac{\ga+2s}{2}}R_r\phi_2\|_{L^{r_0^*}_{t,x}L^2_v}^2.
\end{align*}
For $\mathring\Ga_3(f,\phi_2)$, by H\"older's inequality, $\sum_{j\ge 0}\1_{\max\{p,q\}\ge j-2}\1_{|p-q|\le 1}\le C(p+1)$, $\sum_{p,q\ge 0}\1_{|p-q|\le 1}|u_p||v_q|\le C\big(\sum_{p,q\ge 0}|u_p|^2|v_q|^2\big)^{\fr12}$, and estimate \eqref{equivTriebelLizonkin}, we have 
\begin{align*}
	\mathring\Ga_3(f,\phi_2)
	&\le C
	\Big\|\Big(\sum_{p\ge 0}(p+1)
	\|\De_p\<D_v\>^{-s}\<v\>^{-N}f\|_{L^2_v}^2\Big)^{\frac{1}{2}}\Big\|_{L^{r_0}_{t,x}}^2
	\Big\|\Big(\sum_{q\ge 0}\|\<v\>^{\frac{\ga+2s}{2}}\De_{q}R_r\phi_2\|_{L^2_v}^2\Big)^{\frac{1}{2}}\Big\|_{L^{r_0^*}_{t,x}}^2\\
	&\le C
	\sum_{p\ge 0}(p+1)
	\|\De_p\<D_v\>^{-s}\<v\>^{-N}f\|_{L^{r_0}_{t,x}L^2_v}^2
	\|\<v\>^{\frac{\ga+2s}{2}}\De_{q}R_r\phi_2\|_{L^{r_0^*}_{t,x}L^2_v}^2. 
\end{align*}
Substituting $\phi_2=\wt{\phi}_{2,\de_4}$ and collecting the above estimates into \eqref{es780zz}, 
\begin{align}\label{es780New}\notag
	&\sum_{r\ge 0,\,j\ge 1}\om^2(\varpi2^j)^{-2}2^{2rl}\|\De_jQ_0\wt{P_k}\wt{R_r}Q(\mu^{\fr12}f,(P_0+P_1)R_r\wt\phi_{2,\de_4})\|_{L^2_t([0,T])L^2_{x,v}}^2\\
	&\le Co_{\de_4}(1)\om_0^{2+4s}\varpi_0^{-2}\sum_{j\ge 0}(j+1)^2\|\De_j\<D_v\>^{-s}\<v\>^{-N}f\|_{L^{r_0}_{t}([0,T])L^{r_0}_{x}L^2_v}^2,
\end{align}
for any $r_0>2$ with $\fr{1}{r_0}+\fr{1}{r_0^*}=\fr{1}{2}$, and sufficiently large $N\ge 0$, where $\|\<v\>^{C_{l_1,l_0,l,\ga,s,d}}\wt\phi_{2,\de_4}\|_{L^{r_0^*}_{t}([0,T])L^{r_0^*}_{x,v}}^2 \le o_{\de_4}(1)$ is used as in the \emph{a priori} assumption \eqref{apriori2}.

\smallskip Finally, observing that the largest velocity weight in the above estimates is $2^{r(-\ga+2s)+r(2\ve l_0)}\approx\<v\>^{-\ga+2s+2\ve l_0}$, we substitute the estimates \eqref{es741}, \eqref{es748}, \eqref{es748a}, \eqref{es748b}, \eqref{es779}, \eqref{es777regular}, \eqref{es777}, \eqref{es780zgg}, \eqref{es780zz} and \eqref{es780New} into the linear combination $\sum_{r\ge0,\,j\ge1}2^{2rl+r(\ga-2s)+r(\ga+2s)}\times\eqref{hypoes11}$, by choosing suitable $\wt{l_0}$, where we choose the constant $\ga+2s<0$ for simplicity (so that $\sum_{r\ge0}2^{r(\ga+2s)}\le C$; in fact, any negative exponent suffices), and take $C_0>1$ sufficiently large.
Moreover, we choose $\om_0,\varpi_0\ge 1$, and
\begin{align}\label{choice7}
	\begin{cases}
	\ga+2s<0, \ \ l_0,l_1,l\ge 0,
	&l_0=(1+2s)l_1+\ga,\\
	-2l_1+l+2\ga-1\le 0,&\text{(for last term in \eqref{es741}),}\\
	-4sl_1+2l\le 0,&\text{(for \eqref{es748b}),}\\
	2\ve l_0+\ga+2s\le 0, &\text{(for \eqref{es779})}.
	\end{cases}
\end{align}
Thus, dropping the redundant coefficients (noting $\ga-2s< 0$ and $\ga+2s<0$), applying the Littlewood-Paley theorem \ref{LPThm} and $L^2$ energy estimate \eqref{esExtenL2Energy} to control the energy on $t\in(-\infty,0)$, we deduce 
\begin{align}\label{MainEnergyP0Dej}\notag
	&\sum_{r\ge 0,\,j\ge 1}2^{2rl+2r\ga}\|\De_jP_0R_rf\|_{L^2_t(\R)L^2_{x,v}}^2\\
	&\ \le \notag
	C_{C_0,M_0,\om_0,\varpi_0,A_0}\|(\<v\>^{-sl_1+l}f_0,\, f_0)\|_{L^2_{x,v}}^2
	+C\big(C_0\om_0^{-2}+\om_0^{2+2s}\varpi_0^{-2}\big)\|(\<v\>^{-sl_1+l}f,\,f)\|_{L^2_t([0,\infty))L^2_{x,v}}^2\\
	&\notag\quad
	+C\big(\om_0^2\varpi_0^{-2}
	+M_0^2\om_0^{2+2s}\varpi_0^{-2+2\ve}
	+C_{\de_4,M_0}\om_0^{-2s}
	\big)\|(\<v\>^{-sl_1+l}f,\,f)\|_{L^2_t([0,T])L^2_{x,v}}^2\\
	&\notag\quad
	+C\big(C_0
	+C_{M_0}\om_0^{2+4s}\varpi_0^{-2}
	+C_{\de_3,M_0}\om_0^{2+4s}\varpi_0^{-3}
	\big)
	\sum_{r,j,k\ge0}(\om2^k)^{-2s}2^{2rl}\|\De_jP_kR_rf\|_{L^2_t([0,\infty))L^2_xL^2_D}^2\\
	&\notag\quad+ 
	A_0^2\om_0^{2+2s}\varpi_0^{-2}\sum_{r,k,j\ge 1}(\om2^k)^{-2s}2^{2rl}\|\De_jP_kR_rf\|^2_{L^2_t((T,\infty))L^2_{x,v}}\\
	&\notag\quad
	+
	o_{\de_3}(1)C_{\om_0,\varpi_0,l_0,l_1}\sum_{r\ge 0,\,j\ge 1}\om^{-2s}2^{2rl}\|\De_jP^j_0R_rf\|^2_{L^2_t([0,T])L^2_{x,v}}
	\\
	&\quad+ Co_{\de_4}(1)\om_0^{2+4s}\varpi_0^{-2}\sup_{r\ge 0}\sum_{j\ge 0}(j+1)^2\|\De_j\<D_v\>^{-s}\<v\>^{-N}f\|_{L^{r_0}_t([0,T])L^{r_0}_{x}L^2_v}^2,
\end{align}
for any $N>0$, where we used $o_{\de_3}(1)\om_0^{2+2s}\varpi_0^{-2}\le M_0^2\om_0^{2+2s}\varpi_0^{-2+2\ve}$ for small $\de_3>0$ and $\varpi_0\ge 1$, $2^r\approx\<v\>$ on the support of $R_r$, and $C_0^{-1}<C_0$. 


\subsubsection{The low-low frequency case $P_0\De_0$}
For the case $k=j=0$, we simply use the instant energy $L^\infty_t([0,\infty))$ on $[0,T]$ and strong dissipation $A_0\|\cdot\|_{L^2_t((T,\infty))}^2$ on $(T,\infty)$. For the part $(T,\infty)$, picking out the case $k=j=0$ from the summation and noting that $\om=\om_02^{\al j+rl_1}$, we have 
\begin{align}\label{jk0}
	&A_0\sum_{r,j,k\ge 0}(\om2^k)^{-2s}2^{2rl}\|\De_jP_kR_rf\|_{L^{2}_{t}((T,\infty))L^2_{x,v}}^2
	\ge 
	A_0\om_0^{-2s}\sum_{r\ge 0}2^{-2srl_1}2^{2rl}
	\|\De_0P_0R_rf\|_{L^{2}_{t}((T,\infty))L^2_{x,v}}^2.
\end{align}

\subsubsection{The negative-order Bessel potential term $\<D_v\>^{-s}f$}\label{SubsecNegaBessel}
For the negative Bessel potential term $\<D_v\>^{-s}$ in \eqref{es21} and \eqref{MainEnergyP0Dej}, we would like to pull the $L^{r_0}_{x}$ ($r_0>0$ is close to $2$) norm back to an $L^2_{x}$ norm.
Splitting the cases $j=0$ and $j\ge 1$, using the fact that $(\varpi2^{j})^{\ve}\De_j\<D_x\>^{-\ve}$ and $\<D_x\>^{\ve}\<D_{t,x}\>^{-\ve}$ are $L^p$ multipliers for $j\ge 1$ and $p\in(1,\infty)$, together with $\sum_{j\ge 1}2^{-\fr{r_0\ve j}{2}}\le C_\ve$, 
we can control the last terms in \eqref{es21} and \eqref{MainEnergyP0Dej} by 
\begin{align}\label{es765}\notag
	&
	\Big(\sup_{r\ge 0}\sum_{j\ge 0}(j+1)^2\|\De_j\<D_v\>^{-s}\<v\>^{-N}f\|_{L^{r_0}_{t}([0,T])L^{r_0}_{x}L^2_v}^2\Big)^{\frac{1}{2}}\\
	&\notag \le 
	C_\ve\Big(\sup_{r\ge 0}\|\De_0\<D_v\>^{-s}\<v\>^{-N}f\|_{L^{r_0}_{t}(\R)L^{r_0}_{x}L^2_v}^{r_0}
	+\sup_{r\ge 0}\sum_{j\ge 1}2^{-\fr{r_0\ve j}{2}}(\varpi2^{j})^{r_0\ve}\|\De_j\<D_v\>^{-s}\<v\>^{-N}f\|_{L^{r_0}_{t}(\R)L^{r_0}_{x}L^2_v}^{r_0}\Big)^{\frac{1}{r_0}}\\
	&
	 \le C_\ve
		\|\<D_{t,x}\>^{\ve}\<D_v\>^{-s}\<v\>^{-N}f\|_{L^{r_0}_{t}(\R)L^{r_0}_{x}L^2_v},
\end{align}
for any $\ve>0$.
On the other hand, by applying the Embedding Theorem \ref{BesovembeddThm} and the upper bound estimate in Corollary \ref{CoroNegativeHypo} (with $m=s$ and $p=2$ therein), there exist $r_0 \in (2,\infty)$ (close to $2$), $r^*\in (2,\infty)$ (relatively large), and $\beta>0$ such that, by simply choosing $N>d+2+\max\{\ga+2s,0\}$, 
\begin{align}\label{es77}\notag
	&\quad\notag
	\|\<D_{t,x}\>^{\ve}\<D_v\>^{-s}\<v\>^{-N}f\|_{L^{r_0}_{t}(\R)L^{r_0}_{x}L^{2}_v}
	\le \|\<D_{t,x}\>^{\beta}\<D_v\>^{-s}\<v\>^{-N}f\|_{L^{2}_{t}(\R)L^{2}_{x}L^{2}_v}\\
	&\quad\notag\le C\big(1+\|\phi_1\|_{L^{\infty}_{t}([0,T])L^{\infty}_{x}L^2_v}+\|\<v\>^{1+\ga+2s-N}\phi_2\|_{L^{r^*}_{t}([0,T])L^{r^*}_{x}L^2_v}\big)\|\<v\>^{1-N}f\|_{L^2_{t}([0,T])L^2_{x}L^2_v}\\
	&\qquad\notag
	+C(1+A_0)\|\<v\>^{1-N}f\|_{L^2_{t}((-\infty,0)\cup(T,\infty))L^2_{x}L^2_v}\\
	&\quad\le C(1+M_0)\|f\|_{L^2_t([0,T])L^2_{x}L^2_v}+C(1+A_0)(\|f\|_{L^2_{t}((T,\infty))L^2_{x,v}}+A_0^{-1/2}\|f_0\|_{L^2_{x,v}}),
\end{align}
where we used $L^2$ estimate \eqref{esExtenL2Energy} for $t\in(-\infty,0)$ and applied assumption \eqref{assumption} to control $\phi_i$ (by choosing $r=r(d,\ga,s)>2$ in \eqref{assumption} sufficiently large).
Moreover, we will let $M_0,A_0\ge 1$ for simplicity. 
This is the \text{only} place where we determine $r_0>2$.
Note that $N>0$ can be taken arbitrarily large due to the decay of $\mu^{\frac{1}{2}}$, and that $\beta>0$ in Theorem \ref{ThmNegativeHypo} and Corollary \ref{CoroNegativeHypo} depends only on $s$ and $d$ when $p=2$.

\subsubsection{Summing the above cases}
When $\ga+2s<0$, for any $\om_0,C_0,\varpi_0>1$, we take combination $\eqref{es21}+c_0(8C)^{-1}\times\eqref{esRHS1}+c_0(8CC_0)^{-1}\times\eqref{MainEnergyP0Dej}$ and apply estimates \eqref{es765} and \eqref{es77} to control the $\|\<D_v\>^{-s}\<v\>^{-N}f\|$ term. The choice of the coefficients $c_0(8C)^{-1},c_0(8CC_0)^{-1}$ is made in order to absorb the $L^2_D$ dissipation terms. 
This yields the energy estimate 
 \begin{align}\notag\label{esSumEverything}
	&\Big\|\sum_{r,j,k\ge 0}(\om2^k)^{-2s}2^{2rl}\|\De_jP_kR_rf\|_{L^2_{x,v}}^2\Big\|_{L^\infty_t([0,\infty))}
	+
	\fr{3c_0}{4}\sum_{r,j,k\ge 0}(\om2^k)^{-2s}2^{2rl}\|\De_jP_kR_rf\|_{L^{2}_{t}([0,\infty))L^{2}_{x}L^2_D}^2
	\\
	& \notag
	+
	A_0\sum_{r,j,k\ge 0}(\om2^k)^{-2s}2^{2rl}\|\De_jP_kR_rf\|_{L^{2}_{t}((T,\infty))L^2_{x,v}}^2
	+\frac{c_0}{8C}\sum_{r,j\ge0,\,k\ge 1}2^{2rl+r\ga}\|\De_jP_kR_rf\|^2_{L^2_t([0,\infty))L^2_{x,v}}
	\\
	&\notag +\frac{c_0}{8CC_0}\sum_{r\ge 0,\,j\ge 1}
	2^{2rl+2r\ga}\|\De_jP_0R_rf\|_{L^2_t([0,\infty))L^2_{x,v}}^2\\
	&\le \notag
	C_{C_0,M_0,\om_0,\varpi_0,A_0}\|f_0\|_{L^2_{x,v}}^2
	+C\big(\om_0^{-2}+C_0^{-1}\om_0^{2+2s}\varpi_0^{-2}\big)\|f\|_{L^2_t([0,\infty))L^2_{x,v}}^2\\
	&\notag\ 
	+C\big(
	M_0^2\om_0^{-2s}
	+M_0\om_0^{-s-1}\varpi_0^{\ve}
	+C_{\de_1,M_0}\varpi_0^{-2}
	\\
	& \notag\quad
	+C_0^{-1}M_0^2\om_0^{2+2s}\varpi_0^{-2+2\ve}
	+C_{\de_4,M_0}C_0^{-1}\om_0^{-2s}
	\big)\|f\|_{L^2_t([0,T])L^2_{x,v}}^2\\
	&\notag\  
	+CC_0^{-1}\om_0^{2+4s}\varpi_0^{-2}\big(
	C_{M_0}
	+C_{\de_3,M_0}\varpi_0^{-1}
	\big)\sum_{r,j,k\ge0}(\om2^k)^{-2s}2^{2rl}\|\De_jP_kR_rf\|_{L^2_t([0,\infty))L^2_xL^2_D}^2\\
	&\notag\  + 
	C_0^{-1}A_0^2\om_0^{2+2s}\varpi_0^{-2}\sum_{r,k,j\ge 1}(\om2^k)^{-2s}2^{2rl}\|\De_jP_kR_rf\|^2_{L^2_t((T,\infty))L^2_{x,v}}\\
	&\notag\  + C\big(o_{\de_2}(1)+o_{\de_4}(1)C_0^{-1}\om_0^{2+4s}\varpi_0^{-2}\big)
	\big(CM_0^2\|f\|^2_{L^2_t([0,T])L^2_{x,v}}+CA_0^2\|f\|_{L^2_{t}((T,\infty))L^2_{x,v}}^2\big)
	\\
	&\  
	+C_{\de_2,\de_3,\om_0,\varpi_0,M_0,l_0,l_1}\sum_{r,j,k\ge 0}(\om2^k)^{-2s}2^{2rl}\|\De_jP_kR_rf\|_{L^2_t([0,T])L^2_xL^2_v}^2,
 \end{align}
 provided that \eqref{choiceM2}, \eqref{choideM3}, and \eqref{choice7} hold, 
 for any $\ve>0$, where we also used $\<v\>^{\fr{\ga+2s}{2}}\le 1$ and $-l_1+l+\fr{\ga}{2}-1\le -s l_1+l$, and dropped some redundant terms. Here, for simplicity, we let 
\begin{align}\label{choicerho0}
	sl_1=l,\quad\text{and}\quad \varpi_0=\om_0^{1+2s}\times\frac{c_0}{100C},
	\quad\text{and}\quad C_0^{-1}=A_0\om_0^{-2s}.
\end{align}
where the relation $\varpi_0=\om_0^{1+2s}\times\frac{c_0}{100C}$ is chosen to satisfy the third condition of \eqref{choice6}, with $\frac{c_0}{100C}$ being merely a fixed constant (depending only on $\ga,s$). 
Here, $A_0=\om_0^{2s}C_0^{-1}\ge 1$ should be chosen sufficiently large so that the $L^p$ estimate \eqref{LpTinftyEnergy} hold for the case $p=2$. 
Moreover, $C_0\ge 1$ should be chosen sufficiently large so that the first term on the right-hand side of \eqref{es741} is controlled, which will be achieved by taking $C_0=C_0(M_0)\gg M_0$ sufficiently large while $p=p(s,d)$ is fixed.
Furthermore, let $\varpi_0,\om_0,C_0>1$ be sufficiently large and $\de_i\in(0,1)$ be sufficiently small such that 
\begin{align}\label{choice9}
	\begin{cases}
	CC_0^{-1}\om_0^{2+4s}\varpi_0^{-2}\big(
		C_{M_0}+
	C_{\de_3,M_0}\varpi_0^{-1}
	\big)
	\le\frac{c_0}{8},&\text{($L^2_D$ dissipation absorbed),}\\
	C_0^{-1}A_0\om_0^{2+2s}\varpi_0^{-2}\le \frac{1}{4}, &\kern-3em\text{($A_0^2\sum_{r,j,k}\|\cdot\|_{L^2_t((T,\infty))}^2$ absorbed)}.
	\end{cases}
\end{align}
Applying \eqref{jk0} to capture the lower-frequency dissipation of $\De_0P_0$, utilizing the choice in \eqref{choicerho0}, dropping the second and third terms (after absorbing the dissipation $L^2_D$ and the large-time term $A_0\|\cdot\|_{L^2_t((T,\infty))}^2$), and splitting the $L^2_t([0,T])$ and $L^2_t([T,\infty))$ norms, the main energy estimate \eqref{esSumEverything} reduces to
 \begin{align}\notag
	\label{es21c}
	&\Big\|\sum_{r,j,k\ge 0}(\om2^k)^{-2s}2^{2rl}\|\De_jP_kR_rf\|_{L^2_{x,v}}^2\Big\|_{L^\infty_t([0,\infty))}
	+
	\notag
	\frac{c_0}{8C}\sum_{r,j\ge0,\,k\ge 1}2^{2rl+r\ga}\|\De_jP_kR_rf\|^2_{L^2_t([0,\infty))L^2_{x,v}}
	\\
	&\notag 
	+\frac{c_0}{16CC_0}\sum_{r\ge 0,\,j\ge 1}
	2^{2rl+2r\ga}\|\De_jP_0R_rf\|_{L^2_t([0,\infty))L^2_{x,v}}^2
	+\frac{1}{2C_0}\sum_{r\ge 0}
	\|\De_0P_0R_rf\|_{L^{2}_{t}((T,\infty))L^2_{x,v}}^2\\
&\le \notag
	C_{C_0,M_0,\om_0}\|f_0\|_{L^2_{x,v}}^2\\
	&\notag\ 
	+C\Big(\om_0^{-2}+C_0^{-1}\om_0^{2+2s}\varpi_0^{-2}
	+\big(o_{\de_2}(1)+o_{\de_4}(1)C_0^{-1}\om_0^{2+4s}\varpi_0^{-2}\big)A_0^2\Big)\|f\|_{L^2_t([T,\infty))L^2_{x,v}}^2\\
	&\notag\quad
	+C\Big(
	\om_0^{-2}+C_0^{-1}\om_0^{2+2s}\varpi_0^{-2}
	+
	M_0^2\om_0^{-2s}
	+M_0\om_0^{-s-1}\varpi_0^{\ve}
	+C_{\de_1,M_0}\varpi_0^{-2}
	+C_0^{-1}M_0^2\om_0^{2+2s}\varpi_0^{-2+2\ve}
	\\
	&\quad \qquad\notag
	+C_{\de_4,M_0}C_0^{-1}\om_0^{-2s}
	+\big(o_{\de_2}(1)+o_{\de_4}(1)C_0^{-1}\om_0^{2+4s}\varpi_0^{-2}\big)
	\Big)\|f\|_{L^2_t([0,T])L^2_{x,v}}^2
	\\
	&\  \notag
	+C_{\de_2,\de_3,M_0,\om_0}\sum_{r,j,k\ge 0}(\om2^k)^{-2s}2^{2rl}\|\De_jP_kR_rf\|_{L^2_t([0,T])L^2_xL^2_v}^2 \\
&\le \notag
	C_{C_0,M_0,\om_0}\|f_0\|_{L^2_{x,v}}^2
	+C\Big(\om_0^{-2}+C_0^{-1}\om_0^{-2s}
	+o_{\de_2}(1)A_0^{2}+o_{\de_4}(1)C_0^{-1}A_0^{2}\Big)\|f\|_{L^2_t([T,\infty))L^2_{x,v}}^2\\
	&\notag\ 
	+C\Big(
	\om_0^{-2}+C_0^{-1}\om_0^{-2s}
	+
	M_0^2\om_0^{-2s}
	+M_0\om_0^{-s-1+\ve(1+2s)}
	+C_{\de_1,M_0}\om_0^{-2-4s}\\
	&\quad\notag
	+C_0^{-1}M_0^2\om_0^{-2s+2\ve(1+2s)}
	+C_{\de_4,M_0}C_0^{-1}\om_0^{-2s}+\big(o_{\de_2}(1)+o_{\de_4}(1)C_0^{-1}\big)
	\Big)\|f\|_{L^2_t([0,T])L^2_{x,v}}^2
	\\
	&\  
	+C_{\de_2,\de_3,M_0,\om_0,l_0,l_1}\sum_{r,j,k\ge 0}(\om2^k)^{-2s}2^{2rl}\|\De_jP_kR_rf\|_{L^2_t([0,T])L^2_xL^2_v}^2,
 \end{align}
where the constants $C=C(p,\ga,s)>0$ is a generic constant. 
 The left-hand terms have weights of non-negative order if (note that $\ga<0$)
 \begin{align}\label{choiceN0}
	2l+2\ga\ge 0.
 \end{align}
 To control the right-hand $L^2$ norms with small constants, we expand using dyadic decomposition \eqref{fequivDejPk} and Littlewood-Paley theorem \ref{LPThm}:
\begin{align}\label{es769}
	\|f\|_{L^2_{x,v}}^2
	&\le C\sum_{r\ge 0}
	\|R_rf\|_{L^2_{x,v}}^2
	\le C\sum_{r,j,k\ge 0}
	\|\De_jP_kR_rf\|_{L^2_{x,v}}^2, 
\end{align}
Moreover, the case $\{k=j=0\}$ in \eqref{es769} on $[0,T]$ will be controlled by Gr\"onwall's inequality, while the $[T,\infty)$ part can be absorbed into the last term on the left-hand side of \eqref{es21c}, provided that
\begin{align}\label{choiceN2}
	\om_0^{-2}+C_0^{-1}\om_0^{-2s}
	+o_{\de_2}(1)A_0^{2}+o_{\de_4}(1)C_0^{-1}A_0^{2} \ll C_0^{-1}.
\end{align}
Similarly, the cases $\{k=0,\, j\ge 1\}$ and $\{k\ge 1,\,j\ge 0\}$ in \eqref{es769} on $[0,\infty)$ can be absorbed by the second and third left-hand terms of \eqref{es21c}, provided that $2l+2\ga\ge 0$, $C_0\gg 1$ and
\begin{align}\label{choiceN3}
	&\om_0^{-2}+C_0^{-1}\om_0^{-2s}
	+
	M_0^2\om_0^{-2s}
	+M_0\om_0^{-s-1+\ve(1+2s)}
	+C_{\de_1,M_0}\om_0^{-2-4s}\\
	&\qquad\notag
	+C_0^{-1}M_0^2\om_0^{-2s+2\ve(1+2s)}
	+C_{\de_4,M_0}C_0^{-1}\om_0^{-2s}+o_{\de_2}(1)+o_{\de_4}(1)C_0^{-1}\ll C_0^{-1}.
\end{align}

\smallskip 
Therefore, to absorb the right-hand side of \eqref{es21c} into the left-hand side, we choose the constants and coefficients so that, together with the parameter conditions from \eqref{choiceM2}, \eqref{choideM3}, \eqref{choice3}, \eqref{choice6}, \eqref{choice7}, \eqref{choice9}, \eqref{choiceN0}, \eqref{choiceN2}, and \eqref{choiceN3}, the following constraints hold: $\ga+2s<0$, $\om=\om_02^{\al j+rl_1}$, $\al=\frac{1}{1+2s}\in(0,1)$, $l_0,l_1,l\ge 0$, $\om_0,\varpi_0\ge 1$, and 
\begin{align}\label{choice10}
\begin{cases}
	l_0=(1+2s)l_1+\ga,\ \ 
	sl_1=l,\\
	-2l_1+l+2\ga-1\le 0,&\text{(for last term in \eqref{es741}),}\\
	-4sl_1+2l\le 0,&\text{(for \eqref{es748b}),}\\
	2l+2\ga\ge 0,&\text{(for left-hand weight)},\\
	-l_0+\frac{\ga+2s}{2}\le -sl_1,&\text{(for \eqref{es733})},\\
	5\ve l_0+\frac{\ga+2s}{2}\le 0,&\text{(for \eqref{es733})},\\
	2\ve l_0+\ga+2s\le 0, &\text{(for \eqref{es779})},\\
	\ga+2s+2\ve l_0-2sl_1+2l<0,&\text{(for \eqref{choiceM2})}, \\ 
	-\frac{3}{2}s\al+2\ve<0,&\text{(for \eqref{choiceM2})}, 
\end{cases}
\end{align}
with $\ve\in(0,\min\{\fr{s}{2(1+2s)},\fr{1-s}{2(1+2s)}\})$, and 
\begin{align}\label{condition1}
	\begin{cases}
	\frac{c_0}{100C}+\om_0^{-2s}
	+o_{\de_1}(1)+o_{\de_2}(1)
	\le \fr{c_0}{8C},\\
	C_0^{-1}C_{M_0}+C_{\de_3,M_0}C_0^{-1}\om_0^{-1-2s}
	\le\frac{c_0}{8C},
	\\
	\om_0^{-2}+C_0^{-1}\om_0^{-2s}
	+o_{\de_2}(1)A_0^{2}+o_{\de_4}(1)C_0^{-1}A_0^{2} \ll C_0^{-1},\\
	\om_0^{-2}+C_0^{-1}\om_0^{-2s}
	+
	M_0^2\om_0^{-2s}
	+M_0\om_0^{-s-1+\ve(1+2s)}
	+C_{\de_1,M_0}\om_0^{-2-4s}\\
	\qquad
	+\ C_0^{-1}M_0^2\om_0^{-2s+2\ve(1+2s)}
	+C_{\de_4,M_0}C_0^{-1}\om_0^{-2s}+o_{\de_2}(1)+o_{\de_4}(1)C_0^{-1}\ll C_0^{-1}.
	\end{cases}
\end{align}
To close the constraints \eqref{choice10} and \eqref{condition1}. Using the first line of \eqref{choice10} with $s\in(0,1)$, we then choose $l_1=l_1(\ga,s,d)>0$ sufficiently large (and hence $l_0>0$ sufficiently large) so that the second through fifth lines of \eqref{choice10} are satisfied. Lastly, we choose $\ve\in(0,\fr{s}{2(1+2s)})$ sufficiently small so that the sixth through ninth conditions in \eqref{choice10} are satisfied.
On the other hand, using $C_0^{-1}=A_0\om_0^{-2s}$, the constraint \eqref{condition1} can be reduced to 
\begin{align}\label{con5}
	\begin{cases}
	\om_0^{-2s}+o_{\de_1}(1)+o_{\de_2}(1)+o_{\de_3}(1)+o_{\de_4}(1)\ll 1,\\
	C_{M_0}A_0\om_0^{-2s}
	+C_{\de_3,M_0}A_0\om_0^{-1-4s}
	\ll 1,
	\\
	A_0^{-1}\om_0^{-2+2s}
	+o_{\de_2}(1)A_0\om_0^{2s}
	+o_{\de_4}(1)A_0^{2} \ll 1,\\
	A_0^{-1}\om_0^{-2+2s}+\om_0^{-2s}
	+M_0^2A_0^{-1}
	+M_0A_0^{-1}\om_0^{-1+s+\ve(1+2s)}
	+A_0^{-1}C_{\de_1,M_0}\om_0^{-2-2s}\\
	\qquad
	+\ M_0^2\om_0^{-2s+2\ve(1+2s)}
	+C_{\de_4,M_0}\om_0^{-2s}
	+o_{\de_2}(1)A_0^{-1}\om_0^{2s}
	+o_{\de_4}(1)\ll 1.
	\end{cases}
\end{align}
where the implicit constants depend only on $p,\ga,s>0$. 
 The worst term is $M_0^2A_0^{-1}$. 
Therefore, for any $\ve\in(0,\min\{\fr{s}{2(1+2s)},\fr{1-s}{2(1+2s)}\})$ (with the following choices uniform in $\ve$), we first choose $A_0=A_0(M_0)>0$ sufficiently large, and then take $\de_1,\de_3,\de_4=\de_4(A_0,M_0)>0$ sufficiently small. Then the constraint \eqref{con5} redueces to 
\begin{align}\label{con7}
	\begin{cases}
	\om_0^{-2s}+o_{\de_2}(1)\ll 1,\\
	C_{M_0}A_0\om_0^{-2s}
	+C_{\de_3,M_0}A_0\om_0^{-1-4s}
	\ll 1,
	\\
	A_0^{-1}\om_0^{-2+2s}
	+o_{\de_2}(1)A_0\om_0^{2s} \ll 1,\\
	A_0^{-1}M_0\om_0^{-1+s+\ve(1+2s)}
	+A_0^{-1}C_{\de_1,M_0}\om_0^{-2-2s}\\
	\qquad
	+ \ M_0^2\om_0^{-2s+2\ve(1+2s)}
	+C_{\de_4,M_0}\om_0^{-2s}
	+o_{\de_2}(1)A_0^{-1}\om_0^{2s}
	\ll 1.
	\end{cases}
\end{align}
This can be achieved by first choosing $\om_0=\om_0(A_0,M_0,\de_1,\de_3,\de_4)>0$ sufficiently large, and then $\de_2=\de_2(\om_0)>0$ sufficiently small, if $\ve\in(0,\min\{\fr{s}{2(1+2s)},\fr{1-s}{2(1+2s)}\})$. 


Therefore, the right-hand terms in \eqref{es21c} can be absorbed by the left-hand side, and we deduce 
 \begin{align*}
	&\notag\Big\|\sum_{r,j,k\ge 0}(\om2^k)^{-2s}2^{2rl}\|\De_jP_kR_rf\|_{L^2_{x,v}}^2\Big\|_{L^\infty_t([0,T])}
	+\|f\|_{L^2_t([0,T])L^2_{x,v}}^2\\
	&\ \ \le C_{M_0}\|f_0\|_{L^2_{x,v}}^2+C_{M_0}\sum_{r,j,k\ge 0}(\om2^k)^{-2s}2^{2rl}\|\De_jP_kR_rf\|_{L^2_t([0,T])L^2_xL^2_v}^2,
 \end{align*}
for any $T\in (0,T_*]$. Thus, by Gr\"onwall's inequality, we obtain 
 \begin{align}\label{EsContDepend}
	\Big\|\sum_{r,j,k\ge 0}(\om2^k)^{-2s}2^{2rl}\|\De_jP_kR_rf\|_{L^2_{x,v}}^2\Big\|_{L^\infty_t([0,T])}
	+\|f\|_{L^2_t([0,T])L^2_{x,v}}^2
	\le e^{C_{M_0}T}\|f_0\|_{L^2_{x,v}}^2,
 \end{align}
 for any $T\in(0,T_*]$, 
which implies the continuous dependence in \eqref{ThmContiDepend}. Finally if $f_0=0$, then $f=0$. 
This completes the proof of the Uniqueness Theorem \ref{MainThm1}. 


\section{Appendix: Toolbox}\label{SecToolbox}

\subsection{Tools for pseudo-differential calculus}
  We recall some notations and theorems of pseudo-differential calculus. For details, we refer to \cite[Chapter 2]{Lerner2010}, \cite[Prop. 1.1]{Bony1998-1999}, and \cite{Global2019,Deng2020a}. 
  We set $\Gamma=|dv|^2+|d\eta|^2$ as an admissible metric and let $M$ be a $\Gamma$-admissible weight function, i.e., $M:\R^{2d}\to (0,+\infty)$ satisfies: (a) there exists $\delta>0$ such that for any $X,Y\in\R^{2d}$ with $|X-Y|\le\delta$, $M(X)\approx M(Y)$; (b) there exists $C>0$, $N\in\R$, such that $\frac{M(X)}{M(Y)}\le C\<X-Y\>^N$.
	It's direct to check that $\ti{a}$ (defined in \eqref{tiaa}), $\<v\>^n$, $\<\eta\>^n$ $(n\in\R)$ are $\Gamma$-admissible weights.
	A symbol $a=a(v,\eta;\xi)\in S(M,\Gamma)$ (or simply $S(M)$) uniformly in $\xi$, if for any $\alpha,\beta\in \N^d$, $v,\eta\in\R^d$,
  \begin{align*}
    |\partial^\alpha_v\partial^\beta_\eta a(v,\eta,\xi)|\le C_{\alpha,\beta}M,
  \end{align*}
	with some constant $C_{\alpha,\beta}>0$ independent of $\xi$. The space $S(M)$ endowed with the seminorms
    $\|a\|_{k;S(M)} = \max_{0\le|\alpha|+|\beta|\le k}\sup_{(v,\eta)\in\R^{2d}}
    |M(v,\eta)^{-1}\partial^\alpha_v\partial^\beta_\eta a(v,\eta,\xi)|$
	becomes a Fr\'{e}chet space.
  We formally define the pseudo-differential operator by
  \begin{align*}
    (\operatorname{op}_ta)u(x)=\int_\Rd\int_\Rd e^{2\pi i (x-y)\cdot\xi}a((1-t)x+ty,\xi)u(y)\,dyd\xi
  \end{align*}
	for any $t\in\R$ and Schwartz function $f$.
  $a(v,D_v)=\operatorname{op}_0 a$ denotes the standard pseudo-differential operator, while $a^w = a^w(v,D_v) = \operatorname{op}_{1/2} a$ denotes the Weyl quantization of the symbol $a$. We write $A \in \operatorname{Op}(M)$ to indicate that $A$ is a Weyl quantization with symbol in the class $S(M)$. Moreover, by \cite[Theorem 2.3.18 and Proposition 1.1.10]{Lerner2010}, one can pass between different quantizations:
	$a(v,D_v)=(J^{-1/2}a)^w$,
where the mapping $J^t=\exp(2\pi i tD_v\cdot D_\eta)$ is an isomorphism of the Fr\'{e}chet space $S(M)$, with polynomial bounds in the real variable $t$, where $D=\partial/2\pi i$. 


\smallskip Next, we consider the commutator and composition of pseudo-differential operators; one may also refer to \cite[Prop. 1.1]{Bony1998-1999} and \cite[Theorem 2.3.7]{Lerner2010}. 
Let $a_1(v,\eta)\in S(M_1),a_2(v,\eta)\in S(M_2)$, then $a_1^wa_2^w=(a_1\#a_2)^w$, where $a_1\#a_2\in S(M_1M_2)$ satisfies 
\begin{align*}
  a_1\#a_2(v,\eta)&=a_1(v,\eta)a_2(v,\eta)
  +\int^1_0(\partial_{\eta}a_1\#_\theta \partial_{v} a_2-\partial_{v} a_1\#_\theta \partial_{\eta} a_2)\,d\theta,\\
  g\#_\theta h(Y)&:=\frac{1}{2i}\int_\Rd\int_\Rd e^{-2\pi i\sigma(Y-Y_1,Y-Y_2)}g(Y_1)\cdot h(Y_2)\,dY_1dY_2,\ \text{ where }Y=(v,\eta).
\end{align*}
For any $k\in\{1,2,\dots\}$, there exist constants $l,C$ independent of $\theta\in[0,1]$ such that
  $\|g\#_\theta h\|_{k;S(M_1M_2)}\le C\|g\|_{l,S(M_1)}\|h\|_{l,S(M_2)}$. 
Thus 
\begin{align}
	\label{Liebracket}
\text{if $\partial_{\eta}a_1,\partial_{\eta}a_2\in S(M'_1)$ and $\partial_{v}a_1,\partial_{v}a_2\in S(M'_2)$, then $[a_1,a_2]\in S(M'_1M'_2)$, }
\end{align}
where $[\cdot,\cdot]$ is the commutator defined by $[A,B]:=AB-BA$.
As in \cite[Definition 2.6.1]{Lerner2010}, we can define a Hilbert space $H(M):=\{u\in\S':\|u\|_{H(M)}<\infty\}$, where
\begin{align}\label{HmDef}
  \|u\|_{H(M)}:=\int M(Y)^2\|\varphi^w_Yu\|^2_{L^2}
	\,dY<\infty,
\end{align}
and $(\varphi_Y)_{Y\in\R^{2d}}$ is any uniformly confined family of symbols which forms a partition of unity. In this work, we simply use $H(\<v\>^n)=L^2_v(\<v\>^ndv)$ and $H(\<\eta\>^n)=H^n_v$ for any $n\in\R$. Let $a\in S(M)$; then $a^w:H(M_1)\to H(M_1/M)$ is linear and continuous, in the sense of unique bounded extension from Schwartz space $\mathscr{S}$ to $H(M_1)$. Moreover, the existence of a symbol $b\in S(M^{-1})$ satisfying $b\#a = a\#b = 1$ is equivalent to the invertibility of $a^w$ as an operator from $H(MM_1)$ onto $H(M_1)$ for some $\Gamma$-admissible weight function $M_1$.
The following simple upper-bound estimate will be useful.
\begin{Lem}[{\cite[Lemma 2.3 and Corollary 2.5]{Deng2020a}}]
	\label{basicCommuLem}
	Let $m,c$ be $\Ga$-admissible weights. Assume that $a\in S(m)$ and $a^w: H(mc)\rightarrow H(c)$ is an invertible operator. If $b \in S(m)$, then there exists $C>0$, depending only on the seminorms of $a$ and $b$, such that for any $k\in\R$ and suitable $f$,
\begin{align*}
	&\|b(v,D_v)f\|_{H(c)}+\|b^w(v,D_v)f\|_{H(c)}\le C\|a^w(v,D_v)f\|_{H(c)}.\\
	&\text{In particular,}\ \ \ \ 
	\|\<v\>^ka(v,D_v)f\|_{L^2_v}\le C\|\<v\>^ka^wf\|_{L^2_v}.
\end{align*}
\end{Lem}

\subsubsection{Embedding theorem}\label{SubsecTL_Besov}
Pick a function $\beta_0\in C^\infty(\mathbb{R}^d)$ such that $\beta_0(\xi)=1$ for $|\xi|\le 2-\tfrac{1}{7}$ and $\beta_0(\xi)=0$ for $|\xi|\ge 2$. 
For $k\ge 1$, define $\beta_k(\xi)=\beta_0(2^{-k}\xi)-\beta_0(2^{1-k}\xi)$ 
and $\widehat{\Lambda_k f}=\beta_k\,\widehat{f}$ for $k\ge 0$. 
Then the Besov and Triebel-Lizorkin spaces are defined by
\begin{align*}
	\begin{cases}
		\|f\|_{B^s_q[L^p]}=\Big(\sum_{k\ge 0}2^{sqk}\|\Lam_kf\|_{L^p}^q\Big)^{\frac{1}{q}},\\
		\|f\|_{F^s_q[L^p]}=\Big\|\Big(\sum_{k\ge 0}2^{sqk}|\Lam_kf|^q\Big)^{\frac{1}{q}}\Big\|_{L^p}.
	\end{cases}
\end{align*}
We use the embedding theorem for Triebel-Lizorkin and Besov spaces from \cite[Theorems 1.1 and 1.2]{Seeger2018}. Although \cite{Seeger2018} is formulated for Lorentz spaces, we work in the $L^p$ framework instead, using the fact that $L^{p,p}=L^p$ $(0< p\le \infty)$. 
\begin{Thm}
	\label{BesovembeddThm}
	Let $s_0,s_1\in\R$, $0<p_0,p_1<\infty$, $0<q_0,q_1\le\infty$. The embedding
	$B^{s_0}_{q_0}[L^{p_0}]\hookrightarrow F^{s_1}_{q_1}[L^{p_1}]$
holds if and only if one of the following conditions holds.
\begin{enumerate}[label=(\roman*)]
\item $s_0-s_1>d/p_0-d/p_1> 0$.
\item $s_0>s_1$, $p_0= p_1$.
\item $s_0-s_1=d/p_0-d/p_1> 0$, $q_0\le p_1$.
\item $s_0=s_1$, $p_0= p_1\ne q_1$, $q_0\le \min\{p_1,q_1\}$.
\item $s_0=s_1$, $p_0= p_1= q_1$, $q_0\le p_1$.
\end{enumerate}
Moreover, the embedding
	$F^{s_0}_{q_0}[L^{p_0}]\hookrightarrow B^{s_1}_{q_1}[L^{p_1}]$
holds if and only if one of the following conditions holds.
\begin{enumerate}[label=(\roman*)]
\item $s_0-s_1>d/p_0-d/p_1> 0$.
\item $s_0>s_1$, $p_0= p_1$.
\item $s_0-s_1=d/p_0-d/p_1> 0$, $p_0\le q_1$.
\item $s_0=s_1$, $p_0= p_1\ne q_0$, $q_1\ge \max\{p_0,q_0\}$.
\item $s_0=s_1$, $p_0= p_1= q_0$, $q_1\ge p_0$.
\end{enumerate}
Moreover, by \cite[Theorem 1.3.6]{Grafakos2014a}, we have $F^{s_1}_{2}[L^{p}]=H^{s_1,p}$ whenever $1<p<\infty$. 
\end{Thm}

\subsection{Tools for Boltzmann collision operator}
	\label{ToolboxNoncutoff}
For the non-cutoff Boltzmann collision operator, we used the following basic tools. As in the former analysis, we will split the collision into large and small relative-velocity regions. 
Let $\phi(r)\in[0,1]$ be a smooth radial function with value $1$ when $|r|\le \frac{1}{2}$, and $0$ for $|r|\ge 1$. 
Then we split (for simplicity $\phi(v):=\phi(|v|)$) 
\begin{align}\label{Bbarc}\notag
	B(v-v_*,\si)&=b(\cos\th)\Big(|v-v_*|^\ga\phi(v-v_*)+|v-v_*|^\ga(1-\phi(v-v_*))\Big)\\
	&
	\equiv b(\cos\th)\big(\Th_c(v-v_*)+\Th_{\bar c}(v-v_*)\big)
	\equiv B_c(v-v_*,\si)+B_{\bar c}(v-v_*,\si), 
\end{align}
where $\Th_c$ has a well-behaved Fourier transform, while $\Th_{\bar c}$ has no singular point.
Then we denote the corresponding collision operators $Q$ by 
\begin{align}\label{QcDef}
	Q_c(f,g)+Q_{\bar{c}}(f,g):=\int_{\R^{d}_{v_*}\times\S^{d-1}_\si}\big(B_c+B_{\bar c}\big)(v-v_*,\si)\big(f(v'_*)g(v')-f(v_*)g(v)\big).
\end{align}
Moreover, by the so-called pre-post velocity change of variable $(v,v_*)\mapsto (v',v'_*)$, we have 
\begin{align}\label{CollBoltzDef}
	\big(Q(f,g),h\big)_{L^2_v}
	=\int_{\R^{2d}_{v,v_*}\times\S^{d-1}_\si}B(v-v_*,\si)f(v_*)g(v)\big(h(v')-h(v)\big). 
\end{align}
Under the exponential form of collision \eqref{GaBolt}, the linearized collision terms $\Ga(\mu^{\frac{1}{2}},f)$ and $\Ga(f,\mu^{\frac{1}{2}})$ satisfy the following. 
By \cite[Prop. 4.8, p. 983]{Alexandre2012} and \eqref{esD}, for $\ga>-d$ and $s\in(0,1)$, one has
\begin{align}
	\label{esAf}
	\big(\Gamma(\mu^{\frac{1}{2}},f),f\big)_{L^2_v} & \le -2c_0\|f\|_{L^2_D}^2+C\|\<v\>^{\frac{\ga}{2}}f\|_{L^2_v}^2 
	 \le -c_0\|f\|_{L^2_D}^2+C\|\1_{|v|\le R_0}f\|_{L^2_v}^2,
\end{align}
for some constant $c_0>0$ and large $R_0>0$, where we used interpolation in $v$ and $\1_{|v|\le R_0}$ is the indicator function of $B(0,R_0)$.
For the $\Gamma(f,\mu^{\frac{1}{2}})$ part, it follows from \cite[Lemma 2.15]{Alexandre2012} that 
\begin{align}\label{esKf}
	\big(\Gamma(f,\mu^{\frac{1}{2}}),f\big)_{L^2_v} \le C\|\mu^{1/160}f\|_{L^2_v}\|\mu^{1/160}g\|_{L^2_v}. 
\end{align}
Moreover, by \cite[(6.6), p. 817]{Gressman2011}, for $\ga+2s>-\frac{d}{2}$, one has 
\begin{align}\label{Gammaes}
	|(\Gamma(f,g),h)_{L^2_v}|\le C\|f\|_{L^2_v}\|g\|_{L^2_D}\|h\|_{L^2_D}.
\end{align}
To obtain the desired commutator for $\Ga(f,g)-Q(\mu^{\frac{1}{2}}f,g)$ and for the weight function $\<v\>^{n}$, we borrow the following Lemma from \cite[Lemma 4.5]{Alexandre2011} and \cite[Lemma 6.15]{Morimoto2016}.
\begin{Lem}[{\cite[Lemma 4.5 and Lemma 2.12]{Alexandre2011} and \cite[Lemma 6.15]{Morimoto2016}}]
	Let $s\in(0,1)$ and $\ga>\max\{-d,-2s-\frac{d}{2}\}$. For any $n,n_1,n_2\in\R$, and any suitable functions $f,g,h$, we have 
\begin{align}\label{CommGaAndQ}\notag
	\big|\big(\Ga(f,g)-Q(\mu^{\frac{1}{2}}f,g),h\big)_{L^2_v}\big|
	&\le C\min\{\|\<v\>^{-N}f\|_{L^2_v}\|\<v\>^{n}g\|_{L^2_D}\|\<v\>^{\frac{\ga+2s}{2}-n}h\|_{L^2_v},\,\\
	&\qquad\qquad\notag\|\<v\>^{-N}f\|_{L^2_v}\|\<v\>^{n}g\|_{L^2_v}\|\<v\>^{\frac{\ga+2s}{2}-n}h\|_{L^2_D}\}
	\\&\kern-3em+C\|\<v\>^{-n_1}f\|_{L^2_v}\|\<v\>^{\frac{\ga+2s}{2}+n_1+n_2}g\|_{L^2_v}\|\<v\>^{\frac{\ga+2s}{2}-n_2}h\|_{L^2_v}, 
\end{align}
for some $C=C(m,n,n_1,n_2,\ga,s)>0$ and any $N\ge 0$. 
\end{Lem}
To estimate the commutator of weights, we borrow the estimate \eqref{CommvlTemp} below (whose proof is self-contained) together with \eqref{CommGaAndQ} to deduce that, for any $l,n\in\R$, 
\begin{align}\label{Commvl}
	|(\<v\>^l\Ga(f,g)-\Ga(f,\<v\>^lg),h)_{L^2_v}|
	&\le C
	\|f\|_{L^2_v}\|\<v\>^{l+n+\frac{\ga+2s}{2}}g\|_{L^2_v}\|\<v\>^{-n}h\|_{L^2_D}. 
\end{align}
By \cite[Prop. 6.10]{Morimoto2016}, for any $\ga>\max\{-d,-d/2-2s\}$ and $s_2,s_3\ge 0$ with $s_2+s_3=2s$, we have 
\begin{align}\label{es2sUpper}
	|(\Gamma(f,g),h)_{L^2_v}|
	\le C\|f\|_{L^2_v}\|\<D_v\>^{s_2}g\|_{L^2_v}\|\<v\>^{\ga+2s}\<D_v\>^{s_3}h\|_{L^2_v}.
\end{align}
By splitting $\Gamma(f,g)
=\<v\>^{-l_0}\Gamma(f,\<v\>^{l_0}g)
+\<v\>^{-l_0}\big(\<v\>^{l_0}\Gamma(f,g)
-\Gamma(f,\<v\>^{l_0}g)\big)$ with $l_0=l+\ga+2s$,
and applying estimates \eqref{Commvl} (with $n=-\frac{\ga+2s}{2}$ therein) and \eqref{es2sUpper}, we obtain, for any $\ga>\max\{-d,-d/2-2s\}$, $s_2\ge 0$, and $s_3\ge s$ with $s_2+s_3=2s$,  
\begin{align}
	\label{64eq}
	|(\Gamma(f,g),h)_{L^2_v}|
	\le C\|f\|_{L^2_v}\|\<v\>^{l+\ga+2s}\<D_v\>^{s_2}g\|_{L^2_v}\|\<v\>^{-l}\<D_v\>^{s_3}h\|_{L^2_v}.
\end{align}

\subsubsection{Basic facts in velocity variable and regular change of variable}
We first give some facts about the pre-post velocities.
Let $\mathbf{k}=\frac{v-v_*}{|v-v_*|}$ if $v\ne v_*$ and $\mathbf{k}=(1,0,0)$ if $v=v_*$.
Under the spherical coordinate, we write
\begin{align}\label{sigma}
	\sigma=\cos\th\,\mathbf{k}+\sin\th\,\omega,
\end{align}
with $\th\in[0,\pi/2]$, where $\omega\in\S^{d-2}(\mathbf{k}):=\{\omega\in\S^{d-1}_\si\,:\,\omega\cdot\mathbf{k}=0\}$, 
Then we have
\begin{align}
	\label{vprivth}
	\begin{cases}
		v'  =\cos^2\frac{\th}{2}v+\sin^2\frac{\th}{2}v_*+\frac{1}{2}|v-v_*|\sin\th\omega, \\
		v'_*  =\sin^2\frac{\th}{2}v+\cos^2\frac{\th}{2}v_*-\frac{1}{2}|v-v_*|\sin\th\omega, \\
		|v'-v|=|v'_*-v_*|=|v-v_*|\sin\frac{\th}{2}, \\
		|v'-v_*|=|v'_*-v|=|v-v_*|\cos\frac{\th}{2}.
	\end{cases}
\end{align}
Using \eqref{vprivth} and the fact that $\th\in\big(0,\frac{\pi}{2}\big]$, one can deduce 
\begin{align}\label{vstar3}
	\begin{cases}
	|v_*|^2\le C(|v'_*|^2+|v|^2),\ \ \ 
	|v'|^2\le C(|v|^2+|v'_*|^2),\\
	|v'|^2\le C(|v|^2+|v_*|^2),\ \ \ 
	|v|^2\le C(|v'|^2+|v_*|^2),
	\end{cases}
\end{align}
which, roughly speaking, mean that the distance between $v$ and $v'$ is controlled by $|v_*|$. Moreover, the following estimate is commonly used: for any $l\in\R$, one has 
\begin{align}
	\label{equivlam}
\<v\>^{l}\le2^{\frac{|l|}{2}}\<v_*\>^{l}\<v-v_*\>^{|l|}.
\end{align}
Moreover, we list some propositions for the collision operators. By means of \cite[Lemma 2.12 and its proof]{Alexandre2011}, for any $\ga>\max\{-d,-d/2-2s\}$, 
\begin{align}\label{ColliUpperTriple}
	\int_{\R^{2d}_{v,v_*}\times\S^{d-1}_\si}
	BF(v_*)(g-g')^2\,d\si dvdv_*
	\le C\big(\|\<v\>^{|\ga|}F\|_{L^2_v}+\|F\|_{L^1_{3|\ga|+2}}\big)
	\|g\|_{L^2_D}^2. 
\end{align}

Using the assumption on $b(\cos\th)$ in \eqref{ths} and the relation $|\eta^-|=|\eta|\sin\fr\th2$ as in \eqref{bobylevangle}, we have the following estimate for the angular integral. Note that angular integrals can be defined by their principal values.
\begin{Lem}[{Angular integrals for velocity and frequency}]\label{bcossinLem}
	Assume that $b(\cos\th)$ satisfies \eqref{ths}. 
\begin{itemize}[leftmargin=1.2em]
	\item 
Writing $b(\cos\th)=b\big(\frac{v-v_*}{|v-v_*|}\cdot\si\big)$, for any function $a(v,v_*)$ and any $v,v_*\in\R^d$, 
\begin{align}\label{angularintegral}
	\begin{aligned}
	&\int_{\S^{d-1}_\si}b(\cos\th)\sin^2\frac{\th}{2}\1_{\th\le a(v,v_*)^{-1}}\,d\si\le C\min\{1,a(v,v_*)^{2s-2}\},\\
	&\int_{\S^{d-1}_\si}b(\cos\th)\1_{\th> a(v,v_*)^{-1}}\,d\si\le Ca(v,v_*)^{2s}.
	\end{aligned}
\end{align}
\item 
Writing $b(\cos\th)=b\big(\frac{\eta}{|\eta|}\cdot\si\big)$, 
for any $\eta,\eta_*\in\R^d$, 
\begin{align}\label{angulareta}
	\begin{aligned}
		\int_{\S^{d-1}_\si}\1_{|\eta^-|\le C\<\eta_*\>}b(\cos\th)\sin^2\frac{\th}{2}\,d\si&\le C\min\{1,|\eta|^{-2+2s}\<\eta_*\>^{2-2s}\},\\
		\int_{\S^{d-1}_\si}\1_{|\eta^-|>\frac{1}{C}\<\eta_*\>}b(\cos\th)\,d\si&\le C|\eta|^{2s}\<\eta_*\>^{-2s}\1_{|\eta|>C\<\eta_*\>}.
	\end{aligned}
\end{align}
\end{itemize}
\end{Lem}

We also use the regular change of variables for $v'$ (and similarly for $z=v+\tau(v'-v)$). Note that in the following Lemma, the function $b$ may depend on several independent variables.
\begin{Lem}[{Regular change of variable}]\label{regularLem}
	Let $\mathbf{k}=\frac{v-v_*}{|v-v_*|}$.
	 For any function $f$ such that the integrals below are well-defined, we have the regular change of variables.
\begin{itemize}[leftmargin=1.1em]
	\item (The basic version) Writing $\cos\th=\mathbf{k}\cdot\sigma$, $\wt\omega\in\S^{d-2}(\mathbf{k})$ where $\S^{d-2}(\mathbf{k})$ is defined by \eqref{sigma}:
		\begin{multline}
			\label{regular1}
			\int_{\R^d_v}\int_{\S^{d-1}_\si}f\big(\sigma,\mathbf{k},\cos\th, v',|v-v_*|\big)\,d\sigma dv\\
			=\int_{\R^d_v\times\S^{d-2}_{\wt\omega}(\mathbf{k})}\int_0^{\frac{\pi}{2}}f\Big(\cos\frac{\th}{2}\,\mathbf{k}+\sin\frac{\th}{2}\,\wt\omega,\cos\frac{\th}{2}\,\mathbf{k}-\sin\frac{\th}{2}\,\wt\omega,\cos\th, v,\frac{|v-v_*|}{\cos\frac{\th}{2}}\Big)\frac{\sin^{d-2}\frac{\th}{2}\,d\th d\wt\omega dv}{2^{-(d-2)}\cos^2\frac{\th}{2}}.
		\end{multline}

		\item In particular, we have the simplified version:
		\begin{align}
			\label{regular}
			\int_{\R^d_v}\int_{\S^{d-1}_\si}f(\cos\th, v',|v-v_*|)\,d\sigma dv=\int_{\R^d_v}\int_{\S^{d-1}_\si}f\Big(\cos\th, v,\frac{|v-v_*|}{\cos\frac{\th}{2}}\Big)\,\frac{d\sigma dv}{\cos^{d}\frac{\th}{2}}. 
		\end{align}


		\item As a consequence, we have 
		\begin{align}\label{regularvprime}
			\int_{\R^d_v}\int_{\S^{d-1}_\si}f(\cos\th, v',|v-v_*|^\ga)(v'-v)\,d\sigma dv=0. 
		\end{align}
\end{itemize}
\end{Lem}
\begin{proof}
The proof is similar to \cite[Lemma 1]{Alexandre2000}, and a detailed proof is also given in \cite{Deng2023z}. Here, we provide a sketch of the proof for completeness.
To perform a change of variables $v\mapsto v'$ on the set $\{0\le\th\le\frac{\pi}{2}\}$, we calculate the Jacobian determinant, by the matrix determinant lemma, as
	\begin{align*}
		\Big|\frac{\pa v'}{\pa v}\Big|=\Big|\frac{1}{2}I+\frac{1}{2}\mathbf{k}\otimes\sigma\Big|
		=\frac{1}{2^d}(1+\mathbf{k}\cdot\sigma)=\frac{1}{2^{d-1}}(\mathbf{k}'\cdot\sigma)^2,
	\end{align*}
	where $\mathbf{k}'=\frac{v'-v_*}{|v'-v_*|}$ if $v'\ne v_*$ and $\mathbf{k}'=(1,0,0)$ if $v'=v_*$. The last identity can be seen from the geometry of binary collisions and deduced by using $\mathbf{k}\cdot\si=\cos\th$, $v'=\frac{v+v_*}{2}+\frac{|v-v_*|\sigma}{2}$ and $|v'-v_*|=|v-v_*|\cos\frac{\th}{2}$ from \eqref{vprivth}:
	\begin{align*}
		&\mathbf{k}'\cdot\si
		=\frac{(v-v_*)\cdot\si}{2|v'-v_*|}+\frac{|v-v_*|}{2|v'-v_*|}
		=\frac{\mathbf{k}\cdot\si}{2\cos\frac{\th}{2}}+\frac{1}{2\cos\frac{\th}{2}}=\cos\frac{\th}{2}\ge\frac{1}{\sqrt{2}},
		\\
		&\notag\mathbf{k}=\frac{v-v_*}{|v-v_*|}=\frac{2(v'-v_*)-|v-v_*|\sigma}{|v-v_*|}
		=2\cos\frac{\th}{2}\mathbf{k}'-\si=2(\mathbf{k}'\cdot\si)\mathbf{k}'-\si,\\
		&\notag|v-v_*|=\frac{|v'-v_*|}{\cos\frac{\th}{2}}=\frac{|v'-v_*|}{\mathbf{k}'\cdot\sigma}.
	\end{align*}
	Then we apply the change of variables $v \mapsto v'$ to \eqref{regular1} and relabel $v'$ as $v$ to deduce
	\begin{align}\label{341a}\notag
		& \int_{\S^{d-1}_\si}\int_{\R^d_v}f\big(\sigma,\mathbf{k},\mathbf{k}\cdot\sigma, v',|v-v_*|\big)\,dvd\sigma \\
		&
		=2^{d-1}\int_{\S^{d-1}_\si}\int_{\mathbf{k}\cdot\sigma\ge\frac{1}{\sqrt{2}}}f\Big(\sigma,2(\mathbf{k}\cdot\sigma)\mathbf{k}-\sigma,2(\mathbf{k}\cdot\sigma)^2-1, v,\frac{|v-v_*|}{\mathbf{k}\cdot\sigma}\Big)\,\frac{dvd\sigma}{(\mathbf{k}\cdot\sigma)^{2}}. 
	\end{align}
	Applying spherical coordinate with $\mathbf{k}$ as in \eqref{sigma}, i.e., $\sigma=\cos\th\,\mathbf{k}+\sin\th\,\wt\omega$, $\wt\omega\in\S^{d-2}(\mathbf{k})$,	
	where $\cos\th=\mathbf{k}\cdot\sigma$ and $\wt\omega\perp\mathbf{k}$, and then performing the change of variables $\th\mapsto 2\th$, we have
	\begin{align*}
		& 2^{d-1}\int_{\S^{d-1}_\si,\,\mathbf{k}\cdot\sigma\ge\frac{1}{\sqrt{2}}}f\Big(\sigma,2(\mathbf{k}\cdot\sigma)\mathbf{k}-\sigma,2(\mathbf{k}\cdot\sigma)^2-1, v,\frac{|v-v_*|}{\mathbf{k}\cdot\sigma}\Big)\,\frac{d\sigma}{(\mathbf{k}\cdot\sigma)^{2}}\\
		& =2^{d-2}\int_{\S^{d-2}_{\wt\omega}(\mathbf{k})}\int_0^{\frac{\pi}{2}}f\Big(\cos\frac{\th}{2}\,\mathbf{k}+\sin\frac{\th}{2}\,\wt\omega,\cos\frac{\th}{2}\,\mathbf{k}-\sin\frac{\th}{2}\,\wt\omega,\cos\th, v,\frac{|v-v_*|}{\cos\frac{\th}{2}}\Big)\frac{\sin^{d-2}\frac{\th}{2}\,d\th d\wt\omega}{\cos^2\frac{\th}{2}}.
	\end{align*}
	This, together with \eqref{341a}, implies \eqref{regular1}. Consequently, when $b$ depends only on $\cos\theta$, we obtain \eqref{regular}. 
	Finally, we apply the decomposition \eqref{vprime} to obtain
	$v'-v
	=-\frac{|v-v_*|\mathbf{k}}{2}+\frac{|v-v_*|\sigma}{2}$,
with $\mathbf{k}=\frac{v-v_*}{|v-v_*|}$.
Consequently, by regular change of variables \eqref{regular}, we have 
\begin{align*}
	&\int_{\R^d_v\times\S^{d-1}_\si}f(\cos\th, v',|v-v_*|)(v'-v)\,d\sigma dv\\
	&=2^{d-1}\int_{\R^d_v\times\S^{d-2}_{\wt\omega}(\mathbf{k})}\int_0^{\frac{\pi}{2}}\frac{|v-v_*|\sin^{d-1}\frac{\th}{2}}{\cos^{3}\frac{\th}{2}}f\Big(\cos\th, v,\frac{|v-v_*|}{\cos\frac{\th}{2}}\Big)\,\wt\omega\,d\th d\wt\omega dv. 
\end{align*}
Thus, by symmetry of $\wt\omega$, the integration vanishes.
	This concludes \eqref{regularvprime} and Lemma \ref{regularLem}.	
\end{proof}

The following change of variables can be found in \cite[p. 63]{Alexandre2010}. 
\begin{Lem}[Regular change of variables $v\mapsto z=v+\tau(v'-v)$]
	Assume the same conditions as in Lemma \ref{regularLem} and let $z=v+\tau(v'-v)$ with any $\tau\in[0,1]$, then we have 
\begin{align}
	\label{zminvstar}
	\begin{cases}
	\frac{|v-v_*|^2}{2}\le |z-v_*|^2\le \frac{5}{4}|v-v_*|^2,\\
	|z|^2\le C(|v|^2+|v_*|^2),\quad 
	|z|^2	\le C(|v'|^2+|v_*|^2),\\
	|v'|^2 \le C(|z|^2+|v_*|^2),\quad
	|v|^2\le C(|z|^2+|v_*|^2),
	\end{cases}
\end{align}
and the change-of-variables formula 
\begin{align}
			\label{regularz}
			\int_{\R^d_v\times\S^{d-1}_\si}b(\cos\th)f(z,|v-v_*|)\le C\int_{\R^d_v\times\S^{d-1}_\si}b(\cos\th)f\Big(v,C(\tau,\th)|v-v_*|\Big), 
		\end{align}
		for some constants $C>0$ and $C(\tau,\th)>0$ satisfying $\displaystyle 0<\inf_{\substack{\tau\in[0,1]\\
		\th\in[0,\frac{\pi}{2}]}}C(\tau,\th)\le \sup_{\substack{\tau\in[0,1]\\
		\th\in[0,\frac{\pi}{2}]}} C(\tau,\th)\le C$.
\end{Lem}
By \eqref{vstar3} and \eqref{zminvstar}, pairwise differences among $v,v'$ and $z=v+\tau(v'-v)$, $\tau\in[0,1]$, satisfy 
\begin{align}\label{distancevpriv}
	C^{-1}\max\{\<v\>\<v_*\>^{-1},\,\<z\>\<v_*\>^{-1}\}\le\<v'\>\le C\min\{\<v\>\<v_*\>,\,\<z\>\<v_*\>\}.
\end{align}

\subsubsection{Basic facts in frequency variable and pseudo-differential estimate}
We begin by listing some facts about $\eta^\pm:=\dfrac{\eta\pm|\eta|\si}{2}$ where $\si\in\S^{d-1}$ and let $\cos\th=\dfrac{\eta}{|\eta|}\cdot\si$ with $\th\in[0,\pi/2]$ in the Subsection. 
We then have the following basic identities for $\eta^+$, $\eta^-$, and $\eta$: for any $\tau\in[0,1]$ and $\th\in[0,\pi/2]$,
\begin{align}\label{bobylevangle}
\begin{cases}
	\eta^+\cdot\eta^-=0,\quad\eta-\eta^+=\eta^-,\\
	|\eta^-|^2
	=|\eta|^2\sin^2\frac{\th}{2},
	\quad
	|\eta^+|^2
	=|\eta|^2\cos^2\frac{\th}{2},\\
	|\eta^+|\le |\eta|\le 2|\eta^+|,
	\quad \frac{|\eta|}{\sqrt{2}}\le |\eta^++\tau(\eta-\eta^+)|\le|\eta|. 
\end{cases}
\end{align}
As in the case of velocity variable, we can decompose 
	\begin{align}\label{decomsitoj}
		\si=\cos\th\,\mathbf{j}+\sin\th\,\omega,
		\quad\text{with }\th\in[0,\pi/2],\ \mathbf{j}=\frac{\eta}{|\eta|},\ \omega\in\S^{d-2}(\mathbf{j}),
	\end{align}
	where $\S^{d-2}(\mathbf{j}):=\{\omega\in\S^{d-1}_\si\,:\,\omega\cdot\mathbf{j}=0\}$. Then 
	\begin{align}\label{symmetaomega}
		\big(\frac{\eta}{|\eta|}\cdot\si\big)\frac{\eta}{|\eta|}-\si
		&=\cos\th\mathbf{j}-\cos\th\,\mathbf{j}-\sin\th\,\omega
		=-\sin\th\,\omega,
	\end{align}
	is the symmetric part. 
Moreover, as in \cite[Prop. 2.16]{Alexandre2011}, it's direct to obtain 
	\begin{align}\label{equivetastar}
		\begin{cases}
			\text{when}~\<\eta_*\>\ge\frac{3}{2}|\eta|,&\<\eta_*\>\approx\<\eta-\eta_*\>\gtrsim\<\eta\>,\\
			\text{when}~\frac{1}{2}|\eta|\le\<\eta_*\>\le2|\eta|,&
			\<\eta\>\approx\<\eta_*\>\gtrsim\<\eta-\eta_*\>,\\
			\text{when}~\<\eta_*\>\le\frac{1}{2}|\eta|,&
			\<\eta\>\approx\<\eta-\eta_*\>\gtrsim\<\eta_*\>.
		\end{cases}
	\end{align}
To find the regular angular $\th$, note that if $f(\eta)$ is a radial function as $f(\eta)=\wt{f}(|\eta|^2)$, we have 
	$\na_\eta f(\eta)=2\eta\wt{f}'(|\eta|^2)$, 
and hence the first-order term in Taylor's expansion of $f(\eta)-f(\eta^+)$ vanishes since $\eta^+\cdot\eta^-=0$.
 by Taylor's expansion, 
Thus, by $|\eta^+|\le|\eta|\le2|\eta^+|$, we have 
\begin{align}\label{expanfetaplue}
	f(\eta)-f(\eta^+)
	&=\int^1_0(1-\tau)\na^2_\eta f(\eta^++\tau(\eta-\eta^+)):\eta^-\otimes\eta^-\,d\tau,\\
	\notag
	|f(\eta)-f(\eta^+)|
	&\notag\le |\eta^-|^2\sup_{\tau\in[0,1]}|\na^2_\eta f(\eta^++\tau(\eta-\eta^+))|
	\le |\eta|^2\sin^2\frac{\th}{2}\sup_{\frac{|\eta|}{2}\le|\wt\eta|\le|\eta|}|\na^2_\eta f(\wt\eta)|.
\end{align}
As in \cite[pp. 2043--2044]{Alexandre2008}, we can obtain the regular change of variables $\eta\mapsto\eta^+$ and $\eta\mapsto z=\eta^++\tau(\eta-\eta^+)$, $\tau\in(0,1)$, in the frequency variable:
\begin{align*}\notag
	\int_{\R^d_\eta\times\S^{d-1}_\si}b\big(\frac{\eta}{|\eta|}\cdot\si\big)f\big(\eta^+,|\eta|,|\eta^-|\big)
	&=2^d\int_{\R^d_\eta\times\S^{d-1}_\si}b(\cos\th)f\big(\eta,|\eta|\cos^{-1}\frac{\th}{2},|\eta|\sin\frac{\th}{2}\cos^{-1}\frac{\th}{2}\big)(1+\cos\th)^{-1},
\end{align*}
where $\cos\th=\frac{\eta}{|\eta|}\cdot\si$, 
and, 
for any $k,l\in\R$, 
\begin{align}\notag
	\label{xichangetoplustau}
	&\int_{\R^d_\eta\times\S^{d-1}_\si}\1_{0\le\th\le\frac{\pi}{2}}|\th|^{2k}|\eta|^{2l}
	f(\eta^++\tau\eta^-,|\eta|)\,d\eta d\si\\
	&\quad\le C\int_{\R^d_\eta\times\S^{d-1}_\si}\1_{0\le \th\le\frac{\pi}{2}}|\th|^{2k}|\eta|^{2l}f\Big(\eta,|\eta|\big(\cos^2\frac{\th}{2}+\tau^2\sin^2\frac{\th}{2}\big)^{-\frac{1}{2}}\Big)
	\,d\eta d\si,
\end{align}
where $\big(\cos^2\frac{\th}{2}+\tau^2\sin^2\frac{\th}{2}\big)^{-\frac{1}{2}}\approx C$.  

\smallskip\noindent{\bf Integration of $e^{2\pi iv\cdot\zeta}$ and the Dirac delta function.} 
To calculate \emph{commutator estimates}, we list some useful tools related to Bobylev formula; cf. \cite{Bobylev1988,Alexandre2000,Alexandre2010}. 
	Following \cite[Proof of Theorem 2.1]{Alexandre2010}, we have \begin{align}
	\label{deltafunction}
	\begin{cases}
	\int_{v}e^{2\pi iv\cdot (\eta-\zeta)}
	=\de(\eta-\zeta),\\
	\int_{v,\si}b\big(\frac{v-v_*}{|v-v_*|}\cdot\si\big)e^{2\pi i(\eta-\zeta)\cdot v}
	=\int_{\si}b\big(\frac{\eta}{|\eta|}\cdot\si\big)\de(\zeta-\eta),\\
	\int_{v,\si}b\big(\frac{v-v_*}{|v-v_*|}\cdot\si\big)e^{2\pi i\zeta\cdot v-2\pi i\eta\cdot v'}
	=\int_{\si}b\big(\frac{\eta}{|\eta|}\cdot\si\big)e^{-2\pi i\eta^-\cdot v_*}\de(\zeta-\eta^+).
	\end{cases}
\end{align}
This can be done by rotation in $\si$ and a limiting procedure for cutoff cross-sections, after combining the gain and loss terms $\Gamma(f,g)$ for sufficiently smooth functions $f$ and $g$ (once all norm estimates have been completed).

\smallskip\noindent{\bf The estimate of $\wh{\Th_c}$.}
Here, $\Th_c(v)=|v|^\ga\phi(v)$, where $\phi(v)\in[0,1]$ is a smooth radial function that is equal to $1$ when $|v|\le c$ and $0$ when $|v|\ge 2c$. 
It follows from \cite[Appendix A, Lemma A.1]{Alexandre2012} (although \cite{Alexandre2012} addresses the case $d=3$) that 
\begin{align}\label{nakThc}
	|D_\eta^k\wh{\Th_c}(\eta)|\le C_{d,\ga}\<\eta\>^{-d-\ga-k}.
\end{align}

\begingroup
\renewcommand{\addcontentsline}[3]{}
\section*{Declarations}
\endgroup
\noindent{\bf Funding.}
D.-Q. Deng was supported by JSPS KAKENHI Grant Number JP25K23329 and JP26K17011. S. Sakamoto was supported by JSPS KAKENHI Grant Number JP24K16952.

\noindent{\bf Data Availability Statements.}
We do not analyse or generate any datasets, because our work proceeds within a theoretical and mathematical approach.

\noindent{\bf Author Contribution Statement.}
D.D. conceived the study and proposed Theorem \ref{ThmNegativeHypo}. All authors developed the theory, proved the main results, wrote the manuscript, and reviewed the manuscript.

\noindent{\bf Acknowledgement.}
The authors would like to thank Prof. Weiran Sun for the valuable discussion. 

\medskip 
\providecommand{\bysame}{\leavevmode\hbox to3em{\hrulefill}\thinspace}
\providecommand{\MR}{\relax\ifhmode\unskip\space\fi MR }
\providecommand{\MRhref}[2]{%
  \href{http://www.ams.org/mathscinet-getitem?mr=#1}{#2}
}
\providecommand{\href}[2]{#2}


\end{document}